\documentclass[10pt]{article}
\usepackage{ amssymb, amsmath,amsthm, hyperref,
  color, tikz ,  geometry, fancyhdr, mathtools}
\usepackage[all,2cell]{xy}
   \UseAllTwocells
   \SilentMatrices
\usepackage[dvips]{epsfig}

\usetikzlibrary{arrows, decorations.markings, decorations, patterns,
positioning, decorations.pathreplacing, decorations.pathmorphing,
external, fit, calc}
\tikzset{->-/.style={decoration={markings, mark=at position .5 with {\arrow{>}}},postaction={decorate}}}
\tikzset{-<-/.style={decoration={markings, mark=at position .5 with {\arrow{<}}},postaction={decorate}}}

\usetikzlibrary{arrows, decorations.markings, decorations, patterns,
positioning, decorations.pathreplacing, decorations.pathmorphing, external}
\tikzset{->-/.style={decoration={markings, mark=at position .5 with {\arrow{>}}},postaction={decorate}}}
\tikzset{-<-/.style={decoration={markings, mark=at position .5 with {\arrow{<}}},postaction={decorate}}}

\newcommand{\soergel}{\ensuremath{\mathcal{S}}}
\newcommand{\koszul}{\ensuremath{\mathrm{K}}}

\newcommand{\Jones}{\ensuremath{\mathbb{J}}}

\newcommand{\myZ}{\ensuremath{\mathbb{Z}}}
\newcommand{\listk}[1]{\ensuremath{\underline{#1}}}
\newcommand{\polA}{\ensuremath{A}}
\newcommand{\imagesfolder}{Images}
\newcommand{\NB}[1]{\ensuremath{\vcenter{\hbox{#1}}}}

\newcommand{\rickardp}[2]{\ensuremath{T_{#1, #2}^+}}
\newcommand{\rickardm}[2]{\ensuremath{T_{#1, #2}^-}}

\newcommand{\mapX}{\ensuremath{\chi}}
\newcommand{\mapH}{\ensuremath{\eta}}
\newcommand{\mapB}{\ensuremath{\upsilon}}
\newcommand{\mapU}{\ensuremath{\zeta}}
\newcommand{\mapA}{\ensuremath{\alpha}}
\newcommand{\mapE}{\ensuremath{\epsilon}}
\newcommand{\mapD}{\ensuremath{\Delta}}
\newcommand{\mapN}{\ensuremath{\nabla}}

\newcommand{\cl}{\ensuremath{\colon\thinspace}}
\newcommand{\nopdg}{\ensuremath{\mathcal{F}}}

\newcommand{\somecategorynopdg}{{\mathcal{C}}}
\newcommand{\somecategorypdg}{{\somecategorynopdg^{\partial}}}

\normalfont\upshape

\theoremstyle{definition}
\newtheorem{thm}{Theorem}[section]
\newtheorem{cor}[thm]{Corollary}

\newtheorem{lem}[thm]{Lemma}
\newtheorem{rem}[thm]{Remark}
\newtheorem{prop}[thm]{Proposition}
\newtheorem{claim}[thm]{Claim}
\newtheorem{defn}[thm]{Definition}
\newtheorem{example}[thm]{Example}

\newtheorem*{thm*}{Theorem}

\numberwithin{equation}{section}

\def\N{{\mathbb N}}
\def\R{{\mathbb R}}
\def\Z{{\mathbb Z}}

\def\Q{{\mathbb Q}}

\def\F{{\mathbb F}}

\newcommand{\e}{\ensuremath{\mathbf{e}}}
\newcommand{\Fp}{\ensuremath{\mathbb{F}_p}}

\newcommand{\Hom}{{\rm Hom}}

\renewcommand{\to}{\rightarrow}

\newcommand{\End}{{\rm End}}

\newcommand{\rkq}{{\rm rk}_q}
\def\dif{\partial}

\def\lra{{\longrightarrow}}
   
\def\dmod{{\mathrm{\mbox{\textrm{-}}mod}}}   %

\def\Id{\mathrm{Id}}
\def\mc{\mathcal}

\def\shuffle{\,\raise 1pt\hbox{$\scriptscriptstyle\cup{\mskip
               -4mu}\cup$}\,}
\def\udmod{{\mathrm{\textrm{-}\underline{mod}}}}   %

\newcommand{\refequal}[1]{\xy {\ar@{=}^{#1}
(-1,0)*{};(1,0)*{}};
\endxy}

\def\K{\mathrm{K}} %

\newcommand{\mH}{\mathrm{H}} 
\newcommand{\mHH}{\mathrm{HH}}
\newcommand{\mHT}{\mathrm{HT}}
\newcommand{\mHHH}{\mathrm{HHH}}

\newcommand{\mC}{\mathrm{C}}

\newcommand{\stgamma}{\ensuremath{\vcenter{\hbox{\tikz[scale=0.3]{
\coordinate (B) at (0,0);
\coordinate (V1) at (0,1);
\coordinate (V2) at (1,2);
\coordinate (T1) at (-2,3);
\coordinate (T2) at (0,3);
\coordinate (T3) at (2,3);
\draw[>-] (B) -- (V1) node [at start, below] {\tiny{$a+b+c$}};
\draw[->] (V1) -- (T1) node [at end, above] {\tiny{$a$}};
\draw[->-] (V1)  -- (V2) node[midway, right] {\tiny{$b+c$}};
\draw[->] (V2) -- (T2) node[at end, above] {\tiny{$b$}};
\draw[->] (V2) -- (T3) node[at end, above] {\tiny{$c$}};
}}}}}

\newcommand{\stgammaprime}{\ensuremath{\vcenter{\hbox{\tikz[scale=0.3]{
\coordinate (B) at (0,0);
\coordinate (V1) at (0,1);
\coordinate (V2) at (-1,2);
\coordinate (T1) at (-2,3);
\coordinate (T2) at (0,3);
\coordinate (T3) at (2,3);
\draw[>-] (B) -- (V1) node [at start, below] {\tiny{$a+b+c$}};
\draw[->] (V1) -- (T3) node [at end, above] {\tiny{$c$}};
\draw[->-] (V1)  -- (V2) node[midway, left] {\tiny{$a+b$}};
\draw[->] (V2) -- (T1) node[at end, above] {\tiny{$a$}};
\draw[->] (V2) -- (T2) node[at end, above] {\tiny{$b$}};
}}}}}

\newcommand{\vertab}{\ensuremath{\vcenter{\hbox{\tikz[scale=0.4]{
\coordinate (B) at (0,0);
\coordinate (T) at (0,3);
\draw[white] (0, -0.5) -- (0, 3.5);
\draw[->] (B) -- (T) node[midway, right] {\tiny{$a+b$}};
}}}}}

\newcommand{\digonab}{\ensuremath{\vcenter{\hbox{\tikz[scale=0.4]{
\coordinate (B) at (0,0);
\coordinate (V1) at (0,0.5);
\coordinate (V2) at (0,2.5);
\coordinate (T) at (0,3);
\draw[white] (0, -0.5) -- (0, 3.5);
\draw[>-] (B) -- (V1) node[midway, left] {\tiny{$a+b$}};
\draw[->] (V2) -- (T) node[midway, left] {\tiny{$a+b$}};
\draw[->-] (V1) .. controls +(+0.5, +0.5) and +(+0.5, -0.5).. (V2) node[midway, right] {\tiny{$b$}};
\draw[->-] (V1) .. controls +(-0.5, +0.5) and +(-0.5, -0.5).. (V2) node[midway, left] {\tiny{$a$}};
}}}}}

\newcommand{\baddigonab}{\ensuremath{\vcenter{\hbox{\tikz[scale=0.4]{
\coordinate (B) at (0,0);
\coordinate (V1) at (0,0.5);
\coordinate (V2) at (0,2.5);
\coordinate (T) at (0,3);
\draw[white] (0, -0.5) -- (0, 3.5);
\draw[->] (B) -- (V1) node[midway, left] {\tiny{$a$}};
\draw[->] (V2) -- (T) node[midway, left] {\tiny{$a$}};
\draw[-<-] (V1) .. controls +(+0.5, +0.5) and +(+0.5, -0.5).. (V2) node[midway, right] {\tiny{$b$}};
\draw[->-] (V1) .. controls +(-0.5, +0.5) and +(-0.5, -0.5).. (V2) node[midway, left] {\tiny{$a+b$}};
}}}}}

\newcommand{\verta}{\ensuremath{\vcenter{\hbox{\tikz[scale=0.4]{
\coordinate (B) at (0,0);
\coordinate (T) at (0,3);
\draw[white] (0, -0.5) -- (0, 3.5);
\draw[->] (B) -- (T) node[midway, right] {\tiny{$a$}};
}}}}}

\newcommand{\squarea}{\ensuremath{\vcenter{\hbox{\tikz[scale=0.4]{
\coordinate (B1) at (-1,0);
\coordinate (B2) at (1,0);
\coordinate (C1) at (-1,1);
\coordinate (D1) at (-1,2);
\coordinate (C2) at (1,1);
\coordinate (D2) at (1,2);
\coordinate (T1) at (-1,3);
\coordinate (T2) at (1,3);
\draw[>-] (B1) -- (C1) node[at start, below] {\tiny{$1$}};
\draw[->-] (D1) -- (C1) node[midway, left   ] {\tiny{$a$}};
\draw[->] (D1) -- (T1) node[at end , above ] {\tiny{$1$}};
\draw[>-] (C2) -- (B2) node[at end, below] {\tiny{$a$}};
\draw[->-] (C2) -- (D2) node[midway, right] {\tiny{$1$}};
\draw[->] (T2) -- (D2) node[at start, above] {\tiny{$a$}};
\draw[->-] (D2) -- (D1) node[midway, above] {\tiny{$a+1$}};
\draw[->-] (C1) -- (C2) node[midway, below] {\tiny{$a+1$}};
}}}}}

\newcommand{\twoverta}{\ensuremath{\vcenter{\hbox{\tikz[scale=0.4]{
\coordinate (B1) at (-1,0);
\coordinate (T1) at (-1,3);
\coordinate (B2) at (1,0);
\coordinate (T2) at (1,3);
\draw[->] (B1) -- (T1) node[midway, left] {\tiny{$1$}};
\draw[->] (T2) -- (B2) node[midway, right] {\tiny{$a$}};
}}}}}

\newcommand{\doubleYa}{\ensuremath{\vcenter{\hbox{\tikz[scale=0.4]{
\coordinate (B1) at (-1,0);
\coordinate (T1) at (-1,3);
\coordinate (C) at (0,1);
\coordinate (D) at (0,2);
\coordinate (B2) at (1,0);
\coordinate (T2) at (1,3);
\draw[>-] (B1) -- (C) node[at start, below] {\tiny{$1$}};
\draw[->] (C) -- (B2) node[at end, below] {\tiny{$a$}};
\draw[->-] (D) -- (C) node[midway, left] {\tiny{$a-1$}};
\draw[>-] (T2) -- (D) node[at start, above] {\tiny{$a$}};
\draw[->] (D) -- (T1) node[at end, above] {\tiny{$1$}};
}}}}}

\newcommand{\vertavertb}{\ensuremath{\vcenter{\hbox{\tikz[scale=0.4]{
\coordinate (B1) at (-1,0);
\coordinate (T1) at (-1,3);
\coordinate (B2) at (1,0);
\coordinate (T2) at (1,3);
\draw[->] (B1) -- (T1) node[midway, left] {\tiny{$a$}};
\draw[->] (B2) -- (T2) node[midway, right] {\tiny{$b$}};
}}}}}

\newcommand{\squarec}{\ensuremath{\vcenter{\hbox{\tikz[scale=0.4]{
\coordinate (B1) at (-1,0);
\coordinate (B2) at (1,0);
\coordinate (C1) at (-1,1);
\coordinate (D1) at (-1,2);
\coordinate (C2) at (1,1);
\coordinate (D2) at (1,2);
\coordinate (T1) at (-1,3);
\coordinate (T2) at (1,3);
\draw[>-] (B1) -- (C1) node[at start, below] {\tiny{$a$}};
\draw[->-] (C1) -- (D1) node[midway, left   ] {\tiny{$a+1 $}};
\draw[->] (D1) -- (T1) node[at end , above ] {\tiny{$a$}};
\draw[>-] (B2) -- (C2) node[at start, below] {\tiny{$b$}};
\draw[->-] (C2) -- (D2) node[midway, right] {\tiny{$b-1$}};
\draw[->] (D2) -- (T2) node[at end, above] {\tiny{$b$}};
\draw[->-] (D1) -- (D2) node[midway, above] {\tiny{$1$}};
\draw[->-] (C2) -- (C1) node[midway, below] {\tiny{$1$}};
}}}}}

\newcommand{\squared}{\ensuremath{\vcenter{\hbox{\tikz[scale=0.4]{
\coordinate (B1) at (-1,0);
\coordinate (B2) at (1,0);
\coordinate (C1) at (-1,1);
\coordinate (D1) at (-1,2);
\coordinate (C2) at (1,1);
\coordinate (D2) at (1,2);
\coordinate (T1) at (-1,3);
\coordinate (T2) at (1,3);
\draw[>-] (B1) -- (C1) node[at start, below] {\tiny{$a$}};
\draw[->-] (C1) -- (D1) node[midway, left   ] {\tiny{$a-1$}};
\draw[->] (D1) -- (T1) node[at end , above ] {\tiny{$a$}};
\draw[>-] (B2) -- (C2) node[at start, below] {\tiny{$b$}};
\draw[->-] (C2) -- (D2) node[midway, right] {\tiny{$b+1$}};
\draw[->] (D2) -- (T2) node[at end, above] {\tiny{$b$}};
\draw[->-] (D2) -- (D1) node[midway, above] {\tiny{$1$}};
\draw[->-] (C1) -- (C2) node[midway, below] {\tiny{$1$}};
}}}}}

\title{A categorification of the colored Jones polynomial at a root of unity}
\author{You Qi, Louis-Hadrien Robert, Joshua Sussan and Emmanuel Wagner}

 \hypersetup{
    pdftoolbar=true,        
    pdfmenubar=true,        
    pdffitwindow=false,     
    pdfstartview={FitH},    
    pdftitle={A categorification of the colored Jones polynomial at a root of unity},    
    pdfauthor={You Qi, Louis-Hadrien Robert, Joshua Sussan and Emmanuel Wagner},     
    pdfsubject={A categorification of the colored Jones polynomial at a root of unity},   
    pdfcreator={You Qi, Louis-Hadrien Robert, Joshua Sussan and Emmanuel Wagner},   
    pdfproducer={You Qi, Louis-Hadrien Robert, Joshua Sussan and Emmanuel Wagner}, 
    pdfkeywords={}, 
    pdfnewwindow=true,      
    colorlinks=true,       
    linkcolor=darkgray,          
    citecolor=teal,        
    filecolor=magenta,      
    urlcolor=violet,          
    linkbordercolor=red,
    citebordercolor=teal,
    urlbordercolor=violet,  
    linktocpage=true
    }
\date{\today}

\begin{document}

\maketitle

\begin{abstract}
There is a $p$-differential on the triply-graded Khovanov--Rozansky homology of knots and links over a field of positive  characteristic $p$ that gives rise to an invariant in the homotopy category finite-dimensional $p$-complexes.

A differential on triply-graded homology discovered by Cautis is compatible with the $p$-differential structure.  As a consequence we get a categorification of the colored Jones polynomial evaluated at a $2p$th root of unity.
\end{abstract}

\setcounter{tocdepth}{2} \tableofcontents

\section{Introduction}

\subsection{Motivation}

Crane and Frenkel introduced the categorification program \cite{CF},
with the objective of lifting the Witten--Reshetikhin--Turaev (WRT)
$(2+1)$-dimensional TQFTs to $(3+1)$-dimensional theories.  The
first major success towards this goal was Khovanov's categorification
\cite{KhJones} of the Jones polynomial into a bigraded link homology
theory, now commonly known as \emph{Khovanov homology}.  The
graded Euler characteristic of Khovanov homology is equal to the Jones
polynomial where the quantum parameter may be taken to be a generic
complex number.  Khovanov homology is functorial and has had important
applications in low dimensional topology.

WRT $(2+1)$-dimensional TQFTs are obtained from the WRT $3$-manifold
invariants.  The simplest and most studied WRT invariants are the
ones associated with the Hopf algebras $U_\zeta(\mathfrak{sl}_2)$ for
$\zeta$ a root of unity. In this case, the WRT invariant of $M$ can  be defined by
summing up evaluations at $\zeta$ of the colored Jones polynomials of a
link $L\subset \mathbb{S}^3$ presenting $M$ by surgery.

The fact that these WRT invariants are defined using colored Jones
polynomials evaluated at roots of unity presents extra challenges, but
has led to interesting new structures in representation theory and
categorification.  The Jones polynomial at a root of unity may be
defined through the representation theory of quantum $\mathfrak{sl}_2$
at a root of unity. In contrast to the generic case (or even just the
representation theory of classic $\mathfrak{sl}_2$), at a root of
unity, categories of representations are not semisimple.

In \cite{Hopforoots}, Khovanov set up a program to categorify
structures at prime roots of unity.  He introduced the notion of
hopfological algebra, and more relevant to categorification at prime
roots of unity, the notion of a $p$-DG algebra.  Many technical
results generalizing facts from DG theory were proved in
\cite{QYHopf}.  A $p$-DG algebra is a $\Z$-graded algebra equipped
with a derivation $\partial$ of degree two such that $\partial^p=0$.
Equivalently, a $p$-DG algebra $A$ is a $\Z$-graded algebra which is a
module over $H=\Bbbk[\partial]/(\partial^p)$ with $\Bbbk$ a field of
characteristic $p$.  Khovanov proved that the Grothendieck ring of the
stable category $H \udmod$ is isomorphic to a ring closely related to
the cyclotomic ring $\mathcal{O}_p$ for a prime $p$.  Finding
$H$-module structures on categorical structures leads to
categorifications over the ring $\mathcal{O}_p$.

The first successful implementation of hopfological algebra was the
categorification of half of the small quantum group of
$\mathfrak{sl}_2$ at roots of unity.  This was later expanded to
the full small quantum group and big quantum group in \cite{EQ1, EQ2}.
Categorification of the representation theory of these algebras was
initiated in \cite{KQS, QiSussan3}, where tensor products of two
Weyl modules and tensor products of the natural representation were
categorified.  The first example of potential topological applications
occurred in \cite{QiSussan} where the Burau representation of the
braid group group of type $A$ at a root of unity was
categorified.

In \cite{QiSussanLink}, two of the authors categorified the Jones
polynomial at $2p$th roots of unity, where $p$ is odd.  Rather than searching for an
$H$-module structure on Khovanov homology, the authors considered a
different categorification of the Jones polynomial constructed in
\cite{Cautisremarks}, \cite{QRS}, and \cite{RW}.  This homology is
distinct from Khovanov homology.  This construction
builds upon Khovanov and Rozansky's categorification of the
HOMFLYPT polynomial using categories of Soergel bimodules.  Cautis
defined a differential $d_C$ \cite{Cautisremarks} on HOMFLYPT homology
leading to a categorification of the Jones polynomial.

The link homology introduced in \cite{Cautisremarks}, \cite{QRS}, and
\cite{RW} is in fact more general, and categorifies colored Jones
polynomials.  It is actually even more general than this and
categorifies link polynomials for higher rank quantum groups of type
$A$, but we do not consider these generalizations here.  The colored
constructions utilize categories of singular Soergel bimodules, and
this is the setting we work in here.  The $H$-module structure that we
use goes back to ideas of Khovanov and Rozansky \cite{KRWitt} where an
action of half of a Witt algebra on HOMFLYPT homology was constructed.

Our main result  is a categorification of the colored Jones polynomial
evaluated at roots of unity.

\begin{thm*}[Theorem \ref{thm:main-col}]
  Let $L$ be a colored link presented as the closure of a colored
  braid $\beta$.  The image of the $H$-module associated to the braid
  in the stable category, $\mH^{\dif}_/(\beta)$, is a 
  finite-dimensional bigraded framed link invariant whose Euler
  characteristic is the colored Jones polynomial evaluated at a $2p$th root of unity.
\end{thm*}

\subsection{Colored Jones categorifications}

We now give a brief history of various categorifications of colored
Jones polynomials at generic values of the quantum parameter.  These
categorifications are distinct from the ones considered in this work.

Using his categorification of the Jones polynomial and a certain
cabling technique, Khovanov categorified the Jones polynomial in
\cite{Khcolored}.  Lee-type deformations were considered in
\cite{BW}.

A categorification of tensor products of the natural representation of
$\mathfrak{sl}_2$ was achieved in \cite{BFK} using category
$\mathcal{O}$ for $\mathfrak{gl}_n$.  This was generalized to
arbitrary finite-dimensional representations (and extended to the
quantum case) in \cite{FKS}.  The authors used categories of
Harish-Chandra bimodules.  
Categorical Jones--Wenzl projectors were constructed in \cite{FSS} and
were used in \cite{StroppelSussanColorJ} to construct a Lie theoretic
categorification of the colored Jones polynomial.  These colored
homologies are in general infinite-dimensional.

Explicit constructions of categorical Jones--Wenzl projectors were
given in \cite{CoKr} and \cite{RoJW}.  The latter construction was
generalized in \cite{CautisClasp}. Infinite-dimensional colored Jones
homologies arose from these ideas.

Webster \cite{Webcombined} introduced diagrammatically defined
algebras, which may be viewed as combinatorial constructions of
certain categories of Harish-Chandra bimodules. These algebras are
extensions of KLR algebras introduced by Khovanov and Lauda \cite{KL1}
and Rouquier \cite{Rou2}.  Webster showed that there are categorical
quantum group representations on these categories which could be
extended to categorified tangle invariants. The resulting link
homology is infinite-dimensional as well, when the colors correspond
to non-minuscule representations.

These infinite-dimensional link homologies are difficult to
compute. There are interesting conjectures connecting these homologies
to the representation theory of certain infinite-dimensional algebras
\cite{GOR2}.  Some progress towards computations of these link
homologies and extracting a finite-dimensional functorial theory from
it was made by Hogancamp \cite{Hog2}.

While the infinite-dimensional link homologies categorifying colored
Jones polynomials seem to have rich representation-theoretic
structures, from the perspective of low dimensional topology, this
infinite-dimensionality poses some problems.  It is not possible to
have a functorial link homology if the invariant assigned to the
unknot is not finite-dimensional.  This is another advantage of
working with the categorification of the colored Jones polynomial
introduced by Cautis, Queffelec--Rose--Sartori,
and two of the authors of the present paper.

\subsection{Perspectives and future work}

The next immediate question is how to categorify the WRT $3$-manifold
invariant using this work.  Naively, one should take a direct sum of
the colored link homologies considered here.  More likely, one would
need a more subtle categorification of the so-called Kirby color.

Some categorified WRT invariants have appeared in the physics
literature \cite{GPPV}.  It would be interesting to connect these
structures to the link homologies we define here.

There should be other, and most likely distinct, categorifications of
the colored Jones polynomial at roots of unity based on constructions
mentioned earlier.  The most straightforward such link homology to
introduce would be based upon Webster's work.  It does appear
technically difficult to prove it is a link invariant.  A more
challenging direction would be to build upon the work of \cite{CoKr}
and \cite{RoJW} to define a link invariant categorifying the colored
Jones polynomial at a root of unity.  For generic values of $q$,
\cite{CoKr} and \cite{RoJW} are related to \cite{Webcombined} by
Koszul duality.  This duality becomes less apparent in the presence of
a $p$-differential.  Most likely, $p$-DG analogues of
$A_{\infty}$-algebras would need to be introduced.

\subsection{Outline}

In Section \ref{secbackground} we review some constructions known in
$p$-DG theory and homological algebra, such as the relative homotopy
category.

In Section \ref{sec:moy-graphs-soergel} we review singular Soergel
bimodules and related $p$-DG structures.  Important $p$-DG bimodule
maps are introduced, along with $p$-DG structures on Rickard
complexes.

In Section \ref{uncolored} we give a modification of the link homology
constructed in \cite{QiSussanLink} which categorifies the Jones
polynomial at $2p$th roots of unity.  This doubly-graded link
homology serves as the foundation upon which the colored versions are
built upon.  While the constructions and arguments are very similar to
those in \cite{QiSussanLink}, in Section \ref{uncolored} our
categorification of the Jones polynomial at $2p$th roots of unity are
doubly-graded, rather than their singly-graded relatives defined in
\cite{QiSussanLink}.  The major difference between these two
constructions is that here we take homology with respect to the $p$-differential at the very end.
In \cite{QiSussanLink}, we
form a total $p$-differential which is a sum of the Cautis
differential, the topological differential and the polynomial
$p$-differential.  This forced a collapse of all three original
gradings into one.

In order to bootstrap our arguments from the uncolored case to the
colored one, we use certain facts described in Section
\ref{sec:topred}.  In particular, we use the blist hypothesis from Section \ref{sec:blist}
throughout the manuscript.  Section \ref{sec:pitchfork-reduction}
contains useful reduction moves involving pitchforks and braiding
diagrams.  This section is completely combinatorial and reduces the
amount of checking that needs to be done later.
Categorical statements about Rickard complexes and certain singular
Soergel are proved in Section \ref{sec:pitchforks}.  In conjunction
with braiding in the uncolored case and the blist hypothesis, one
obtains a categorical $p$-DG braid group action on a relative homotopy
category in the colored case.
 We discuss how twists in the $p$-DG structure slide through crossings in Section \ref{sec:sliding-twists}.
In Section \ref{sec:forktwists}, we record some facts about how a
fork resolves a nearby crossing up to grading shift and a twist. 

In Section \ref{sec:coloredmarkII} we prove a Markov II move which
leads to a proof of the main theorem.  The proof of colored Markov II
relies upon the statement in the uncolored case, along with the blist
bypothesis.

We conclude in Section \ref{sec:hopflink} with the calculation of the
homology of the Hopf link where one component is colored by $2$ and
the other by $1$.

\subsection{Acknowledgements.}

We would like to thank Mikhail Khovanov for showing early and constant
interest in this project.

Y.Q.{} was partially supported by the NSF grant DMS-1947532. J.S.{} is
partially supported by the NSF grant DMS-1807161 and PSC CUNY Award
64012-00 52. L.H.R.{} is supported by the Luxembourg National Research
Fund PRIDE17/1224660/GPS.  E.W.{} is partially supported by the ANR
projects AlMaRe (ANR-19-CE40-0001-01), AHA (JCJC ANR-18-CE40-0001) and
CHARMES (ANR-19-CE40-0017).

\section{Background} \label{secbackground}

Throughout this work, $p$ is a prime number and $\F_p$ denotes $\Z/p\Z$,
the field with $p$ elements. In this section, we collect some necessary background material from
\cite{QiSussanLink}.

\subsection{On \texorpdfstring{$p$}{p}-complexes}

Let $H^\prime = \Z[\dif]$ be the graded polynomial algebra generated
by a degree $2$ generator $\dif$.  Define on $H^\prime$ a {\sl
comultiplication} $\Delta: H^\prime \lra H^\prime \otimes H^\prime$ by
setting
\begin{subequations}\label{eqn-Hopf-algebra-H-prime}
  \begin{equation}
    \Delta (\dif) = \dif \otimes 1 + 1\otimes \dif.
  \end{equation}
  Also set the {\sl counit} $\epsilon: H^\prime \lra \Z $ to be
  \begin{equation}
    \epsilon(\dif)=0,
  \end{equation}
  and {\sl antipode} $S:H^\prime \lra H^\prime$ to be
  \begin{equation}
    S(\dif)=-\dif.
  \end{equation}
\end{subequations}
Then $H^\prime$ is a graded Hopf algebra.

The ideal $(\dif^p, p)\subset H^\prime$ is a Hopf ideal, in the sense
that it is closed under $\Delta$, $\epsilon$ and $S$. The graded
quotient $H^\prime/(\dif^p,p)$ inherits a graded Hopf algebra
structure over $\Fp$ and is denoted $H$. %
The structure maps of $H$ are
still denoted $\Delta$, $\epsilon$ and $S$. If one prefer to work with
another field $\Bbbk$ of characteristic $p$ it is possible to tensor
to define $H$ to be $\left(H^\prime/(\dif^p,p)\right)\otimes_{\Fp}
\Bbbk$. In order to simplify the exposition, we do not do this.

\paragraph{$p$-Complexes.} Since $H$ (respectively $H^\prime$) is graded
commutative and cocommutative, its category of graded modules, denoted
$H\dmod$ (respectively $H^\prime\dmod$), is a symmetric monoidal abelian
category, where the monoidal structure is given by the usual vector
space tensor product $\otimes_{\Fp}$ (respectively $\otimes_\Z$). 

A graded module over $H$ will be referred to as a \emph{$p$-complex}
since, when $p=2$, such an object reduces to the usual notion of a
(cohomological) complex in characteristic two.

Note that $H$ being a $\F_p$-vector space, any $p$-complex is in
particular a $\Fp$-vector space. Hence when dealing with $p$-complexes, we
are working in characteristic $p$.

Graded $H$-modules constitute a Krull--Schmidt category, with
indecomposable $p$-complexes over $\Fp$ classified as follows. Set
\begin{equation}
  V_i:=H/(\dif^{i+1}) ,\quad \quad 0\leq i \leq p-1,
\end{equation}
so that $V_i$ has dimension $i+1$. Given any $a\in \Z$, denote by
$q^aV_i$ the module $V_i$ with grading shifted up by $a$. Then
$q^aV_i$ is concentrated in degrees $a, a+2,\dots, a+2i$:
\begin{equation}
  q^aV_i=\left(
    \NB{\tikz[scale = 1.5 ]{
        \node (Fa) at (0,0) {$\overset{a}{\Fp}$};
        \node (Fa2) at (1,0) {$\overset{a+2}{\Fp}$};
        \node (dots) at (2,0) {$\cdots$};
        \node (Fa2i) at (3,0) {$\overset{a+2i}{\Fp}$};
        \draw[double equal sign distance] (Fa) -- (Fa2);
        \draw[double equal sign distance] (Fa2) -- (dots);
        \draw[double equal sign distance] (dots) -- (Fa2i);
      }}\right) \ .
\end{equation}
The objects $q^a V_i$, $a\in \Z$ and $i\in \{0,\dots, p-1\}$
constitute a full list of indecomposable graded modules.

As a finite-dimensional Hopf algebra, $H$ is graded Frobenius
\cite{LaSw}, meaning that the full subcategory of graded injective
modules in $H\dmod$ coincides with that of graded projectives. The
\emph{(graded) stable module category} $H\udmod$ is the quotient of
$H\udmod$ by the class of graded projective-injective modules. It is a
triangulated category \cite[Theorem 2.6]{Hap88}. For any graded
$H$-module $M$, the \emph{homological shift functor} $[1]$ is defined
by taking a graded injective embedding $\jmath_M$:
\begin{equation}
  0 \lra M \stackrel{\jmath_M}{\lra} I_M \lra \mathrm{Coker}(\jmath_M) \lra 0, 
\end{equation}
and declaring $M[1]:=\mathrm{Coker}(\jmath_M)$. The inverse functor
$[-1]$ is constructed by taking a graded projective cover $\rho_M$:
\begin{equation}
  0 \lra \mathrm{Ker}(\rho_M) \lra P_M \stackrel{\rho_M}{\lra} M \lra 0,
\end{equation}
and setting $M[-1]:=\mathrm{Ker}(\rho_M)$.

The symmetric monoidal structure on $H\dmod$ descends to an exact
symmetric monoidal structure on $H\udmod$, which we still denote by
$\otimes_{\F_p}$. This follows from the fact that, given an $H$-module $M$ of
dimension $n$, $H\otimes_{\F_p} M$ is free of rank $n$ (see for instance,
\cite[Proposition 2.1]{QYHopf} for details).

In particular, the category $H\udmod$ inherits exact degree shift
functors from $H\dmod$, which we will write as $q^i(\mbox{-})$ for any
$i \in \Z$. For any module $M\in H\dmod$, $q^iM$ has its homogeneous
degree $k$ subspace equal to the original degree $k-i$ part of
$M$. One has functorial-in-$M$ isomorphisms
\begin{equation}
  q^iM \cong M\otimes_{\F_p} (q^i\Fp) \cong (q^i\Fp) \otimes_{\F_p} M.
\end{equation}

Let $V_{p-2}$ be the $(p-1)$-dimensional $p$-complex which is graded
isomorphic to $H/(\dif^{p-1})$. Then there is an isomorphism
\begin{equation}\label{eqn-V-p-2}
  M[1] \cong q^{2-2p}V_{p-2}\otimes_{\F_p} M,
\end{equation}
functorial in $M$. This is because, as alluded to above, $H\otimes M$
is free of rank $n$, and thus the embedding $M \lra H\otimes M$,
$m\mapsto \dif^{p-1}\otimes m$ can be used as $\jmath_M$. Furthermore,
this choice of $\jmath_M$ is clearly functorial.

Applying $[1]$ twice to the module $\Fp$ shows that
\begin{equation}
  \Fp[2]\cong q^{-2p} \Fp,
\end{equation}
and thus the functor isomorphism
$(\mbox{-})[2]\cong q^{-2p}(\mbox{-})$. This follows from the easily
verified fact that the tensor product $q^{4-2p}V_{p-2}\otimes_{\F_p} V_{p-2}$ decomposes into a
direct sum of $\Fp$ and graded free $H$-modules.

\paragraph{Slash cohomology.} Associated to a $p$-complex $M$, one can
define its {\sl slash cohomology} as in \cite[Section 2.1]{KQ}. This
is an analogue of the usual homology functor for chain complexes.

For each $0\le k \le p-2$, form the graded vector space
\[ \mH_{/k}(M) = \dfrac{\mathrm{Ker}(\dif^{k+1}_M)}{\mathrm{Im}
  (\dif^{p-k-1}_M)+\mathrm{Ker} (\dif^{k}_M)}.\] The original
$\Z$-grading on $M$ gives a decomposition
\[ \mH^{\bullet}_{/k}(M) = \bigoplus_{i\in \Z} \mH^{ i }_{/k}(M) .\]
The differential $\dif_M$ induces a map, also denoted $\dif_M$, which
takes $\mH^{i}_{/ k}(M)$ to $\mH^{ i+2}_{/k-1}(M)$.  Define the {\sl
slash cohomology} of $M$ as
\begin{equation}\label{eqnslashcohomology}
  \mH^{\bullet }_{/}(M) =  \bigoplus_{k=0}^{p-2} \mH^{\bullet }_{/k}(M).
\end{equation}
Also let
\[\mH^{i}_/(M) :=
  \bigoplus_{k=0}^{p-2}\mH_{/k}^{i}(M).\] We have the decompositions
\begin{equation}
  \mH^{\bullet}_/(M) = \bigoplus_{i\in \Z} \mH^{i}_/(M)  = \bigoplus_{k=0}^{p-2}
  \mH^{\bullet}_{/k}(M)=\bigoplus_{i\in \Z}\bigoplus_{k=0}^{p-2} \mH_{/k}^{ i}(M).
\end{equation}
$\mH_/^{\bullet}(M)$ is a bigraded $\Fp$-vector space, equipped with
an operator $\dif_M$ of bidegree $(-1,2)$,
$\dif_M:\mH_{/k}^{i}\lra \mH_{/k-1}^{i+2}$.

Forgetting the $k$-grading gives us a graded vector space
$\mH^{\bullet}_/(M)$ with differential $\dif_M$, which we can also
view as a graded $H$-module.  There is a decomposition
\begin{equation}\label{eqn-decomp-of-M-into-cohomology-and-free}
  M\cong \mH_/^{\bullet}(M)\oplus P(M)
\end{equation}
in the abelian category of $H$-modules, where $P(M)$ is a maximal
projective direct summand of $M$. It follows that $\mH_/^{\bullet}(M)$
is isomorphic to $M$ in the stable category $H\udmod$.

In particular, we have
\begin{equation}
  \mH_/^\bullet(q^aV_i) = 
  \begin{cases}
    q^aV_i , & a\in \Z, \quad i=0,\dots, p-2, \\
    0 & i=p-1 .
  \end{cases}
\end{equation}
The slash cohomology group $\mH_/^{\bullet}(M)$, viewed as an
$H$-module, does not contain any direct summand isomorphic to a graded
free $H$-module.

The assignment $M\mapsto \mH_{/}^\bullet(M)$ is functorial in $M$ and
can be viewed as a functor $H\dmod \lra H\udmod$ or as a functor
$H\udmod \lra H\udmod$. The latter functor is then isomorphic to the
identity functor.

\paragraph{Grothendieck ring.}
As a matter of notation for later, we will regard $\Z[q,q^{-1}]$ as the Grothendieck ring of graded (chain complexes of) abelian groups. We will often use the usual \emph{quantum integers and binomial coefficients}
\begin{equation}
    [n]:=q^{n-1}+q^{n-3}+\cdots+q^{1-n}, \quad \quad \genfrac[]{0pt}{}{n}{k}:= \dfrac{[n][n-1]\cdots [n-k+1]}{[k][k-1]\cdots [1]} ,
\end{equation}
for any $n\in \N$ and $k\in \{0,\dots, n\}$

The stable category $H\udmod$ is of particular interest to us because
it categorifies a ring closely related to the ring of cyclotomic integers at a $2p$th root of
unity.

\begin{lem}[{\cite[Proposition 5]{Hopforoots}}]
  The Grothendieck ring of the symmetric tensor triangulated category
  $H\udmod$ is isomorphic to
  \begin{equation}
    K_0(H\udmod) \cong \dfrac{\Z[q,q^{-1}]}{(1+q^2+\dots+q^{2(p-1)})}.
  \end{equation}
\end{lem}

For the rest of this work, we will denote by
\begin{equation}
  \mathbb{O}_p:= \dfrac{\Z[q,q^{-1}]}{(1+q^2+\dots+q^{2(p-1)})}.
\end{equation}
We will refer to the map
\begin{equation}\label{eqn:cyclotomicEuler}
  \chi_{\mathbb{O}}: H\dmod \lra \mathbb{O}_p,
\end{equation}
sending a graded $H$-modules $M$ to the symbol
of its slash cohomology $\mH^\bullet_/(M)$ in $K_0(H\udmod)$ as the \emph{cyclotomic
Euler characteristic} of $M$.

We will be considering a variant version of the cyclotomic Euler characteristic in this paper. To do this, let $M=\oplus_{i,j\in \Z} M_{i,j}$ be an $H$-module that is bigraded, with the second $t$ degree that is independent of the $q$ degree: $\dif(M_{i,j})\subset M_{i+2,j}$, for any $i,j\in \Z$. In this case
\begin{equation}\label{eqn-bigraded-Euler}
    \chi_{\mathbb{O},t}(M):=\sum_{j\in \Z}\chi_{\mathbb{O}}(\mH_{/}(M_{\bullet,j}))t^j \in \mathbb{O}_p[t,t^{-1}]
\end{equation}
is a Laurent polynomial in $t$ with coefficients in the cyclotomic ring.

\subsection{Relative homotopy category}

Throughout this work, we will study $H^\prime$-modules and $H$-modules. In order
to unify the results and the constructions, we adopt the following
terminology, partially borrowed from \cite[Definition 1.2]{EQ3}.
The symbol  $r$ denotes either $p$ or $\e$ (for even) and $H^r$ denotes
 $H$ if $r=p$ or $H'$ if $r=\e$.

Suppose $(A,\dif_A)$ is an \emph{$r$-differential graded ($r$-DG)
algebra}, i.e., a graded algebra equipped with a differential $\dif_A$
of degree two, satisfying
\begin{equation}
  \dif_A^p(a)\equiv 0 \quad\text{if $r = p$,} \qquad \text{and} \qquad \dif_A(ab)=\dif_A(a)b+a\dif_A(b),
\end{equation}
for all $a,b\in A$. In other words, $A$ is an algebra object in the
graded module category of the Hopf algebra $H^r$, where the primitive
degree-two generator $\dif\in H^r$ acts on $A$ by the differential
$\dif_A$. Clearly, a $p$-DG algebra is also an algebra object in the
graded module category of $H^\prime$.

Then, we may form the \emph{smash product algebra} $A\# H^\prime$ (respectively $A\# H$) in this
case. As an abelian group (respectively $\F_p$-vector space), $A\# H^\prime$ is isomorphic to
$A\otimes H^\prime$ (respectively $A\otimes_{\F_p} H$), subject to the multiplication rule determined by
\begin{equation}
  (a\otimes \dif)(b\otimes \dif)=ab\otimes \dif^2+ a\dif_A(b)\otimes \dif.
\end{equation}
Thus, for an $\e$-DG algebra $A$,  $A\otimes 1$ and $1\otimes H^\prime$ sit in
$A\# H^\prime$ as subalgebras by construction. Similarly, for a $p$-DG algebra, there are the natural $\F_p$-subalgebras $A\otimes_{\F_p} 1$ and $1\otimes_{\F_p} H$. We will often refer to modules and morphisms
in $A\#H^r\dmod$ as \emph{$H^r$-equivariant} $A$-modules (or $r$-DG
$A$-modules) and morphisms.

There is an exact forgetful functor between the usual homotopy
categories of chain complexes of graded $A\# H^r$-modules
\begin{equation}\label{eqn-forgetful-functor}
  \mc{F}: \mc{C}(A\# H^r)\lra \mc{C}(A).
\end{equation}
An object $K_\bullet$ in $\mc{C}(A\# H^r)$ lies inside the kernel of
the functor if and only if, when forgetting the $H^r$-module structure
on each term of $K_\bullet$, the complex of graded $A$-modules
$\mc{F}(K_\bullet)$ is null-homotopic. The null-homotopy map on
$\mc{F}(K_\bullet)$, though, is not required to intertwine
$H^r$-actions.

Since $H^r$ is a Hopf algebra, if $A$ is an $r$-DG algebra, then so is its opposite algebra $A^{\mathrm{op}}$. Endow
$A\otimes A^{\mathrm{op}}$ with a natural $H^r$-module structure. We
may consider the algebra $(A \otimes A^{\mathrm{op}})\# H^r$. Following the
definition given above, the multiplication is given by:
\[
  (a_1\otimes a_2 \otimes \dif)(b_1\otimes b_2 \otimes \dif)=a_1b_1
  \otimes b_2a_2\otimes \dif^2+ a_1\dif_A(b_1) \otimes b_2a_2 \otimes
  \dif +  a_1b_1 \otimes \dif_A(b_2)a_2 \otimes \dif.
\]
We will often refer to modules and morphisms in
$(A\otimes A^{\mathrm{op}})\#H^r\dmod$ as \emph{$H^r$-equivariant} bimodules
(or $r$-DG $A$-bimodules) and morphisms.

\begin{defn}\label{def-relative-homotopy-category}
  Given an $r$-DG algebra $(A,\dif_A)$, the \emph{relative homotopy
  category} is the Verdier quotient
  \[\mc{C}^\dif(A):=\dfrac{\mc{C}(A\#
    H^r)}{\mathrm{Ker}(\mc{F})}.\]
\end{defn}
The superscript $\dif$ in the definition is there to remind the reader of the
$H^r$-module structures on the objects.

The category $\mc{C}^\dif(A)$ is triangulated. By construction,
there is a factorization of the forgetful functor
\begin{gather}
  \NB{
    \tikz[xscale=3, yscale =0.8]{
      \node (AH) at (-1, 1) {$\mc{C}(A\# H^r)$};
      \node (A) at ( 1, 1) {$\mc{C}(A)$};
      \node (Adif) at ( 0, -1) {$\mc{C}^\dif(A)$};
      \draw[-to] (AH) -- (A) node[pos =0.5, above] {$\mc{F}$};
      \draw[-to] (AH) -- (Adif);
      \draw[-to] (Adif) -- (A);      
    }
  }  \ .
\end{gather}

Let us briefly discuss on the triangulated structure of the relative homotopy category $\mc{C}^\dif(A)$. By construction the shift functors $[\pm 1]$ are inherited from that of $\mc{C}(A\# H^r)$, which shifts complexes one step to the left or right. 

For the usual homotopy category $\mc{C}(A)$ of an algebra, standard distinguished triangles arise from short exact sequences
  \[
 0 \lra M_\bullet \stackrel{f}{\lra} N_\bullet \stackrel{g}{\lra} L_\bullet \lra 0
  \]
of $A$-modules that are termwise split exact. The class of distinguished triangles in $\mc{C}(A,d)$ are declared to be those that are isomorphic to standard ones.
For distinguished triangles in the relative homotopy category, similarly, we have the following construction.

\begin{lem}\label{lem-construction-of-triangle}
  A short exact sequence of chain complexes of $A\#H^r$-modules
  \[
 0 \lra M_\bullet \stackrel{f}{\lra} N_\bullet \stackrel{g}{\lra} L_\bullet \lra 0
  \]
 that is termwise $A$-split exact gives rise to a distinguished triangle in $\mc{C}^\dif(A)$. Conversely, any distinguished triangle in $\mc{C}^\dif(A)$ is isomorphic to one that arises in this form.
\end{lem}
\begin{proof}
For this proof, we will abbreviate complexes by $K_\bullet=K$ for the ease of notation. Let us prove the converse part first. By construction, distinguished triangles are those in $\mc{C}^{\dif}(A)$ that are isomorphic to standard distinguished triangles arising from  short exact sequences of $A\# H^r$-modules that are termwise $A\#H^r$-split exact. Forgetting about the $H^r$-actions, such sequences are also termwise $A$-split exact.

Now let $f: M \lra N$ be the injection as in the statement.  The cone of $f$ in $\mc{C}(A\# H^r)$ is given by
\[
C(f) \cong 
\left(
M[1]\oplus N, d_C:=
\begin{pmatrix}
d_{M[1]} & f \\
0 & d_N
\end{pmatrix}
\right).
\]  
The cone fits into a short exact sequence of $A\#H^r$-modules that are termwise $A\# H^r$-split:
\[
0\lra N \lra C(f)\lra M[1]\lra 0.
\]
Associated with this sequence is the standard distinguished triangle 
\[
N \lra C(f) \lra M[1] \lra N[1]
\]
in $\mc{C}(A\# H^r)$, which descends to a standard distinguished triangle in $\mc{C}^\dif(A)$.

To prove the statement, it then suffices to show that, in the relative homotopy category, we have an isomorphism
$C(f) \cong L$. Consider the map
\[
h: C(f) = (M[1]\oplus N, d_C) \lra L, \quad (m,n)\mapsto g(n).
\]
It is easily checked that $h$ is a surjective map of chain complexes, and the kernel is isomorphic to $C(\Id_M)$. Thus we have a short exact sequence of chain complexes of $A\# H^r$-modules
\[
0\lra C(\Id_M) \lra C(f) \stackrel{h}{\lra} L \lra 0
\]
Now, under $\mc{F}$, the sequence termwise splits over $A$: 
\[
\mc{F} (C(f)) \cong 
\left(
\mc{F}(M[1]\xrightarrow{\Id_M} M) \oplus \mc{F}(L)
\right)
\cong \left(\mc{F}(C(\Id_M))\oplus \mc{F}(L)\right).
\]
It follows that we have a distinguished triangle in $\mc{C}(A)$
\[
0\cong \mc{F}(C(\Id_M)) \lra \mc{F}(C(f)) \stackrel{\mc{F}(h)}{\lra} \mc{F}(L) \lra \mc{F}(C(\Id_M))[1]\cong 0,
\]
implying that $h$ is an isomorphism under $\mc{F}$. The result follows.
\end{proof}

\subsection{Relative Hochschild homology}

Because of the additional structures we consider on Soergel bimodules,
we will need a relative version of Hochschild homology, that is an
$H^\prime$-equivariant version of it. Since everything that follows about relative Hochschild homology works for $H$ as well,
we use the notation of the previous section and
will define $H^r$-equivariant Hochschild homology.

Recall that the usual \emph{simplicial bar complex} of a unital,
associative algebra $A$ is the complex:
\begin{subequations}\label{eqn-bar-complex}
  \begin{equation}
    \mathbf{p}_\bullet (A):= \left(\cdots \xrightarrow{d_{n+1}} A^{\otimes(n+2)} \stackrel{d_n}{\lra} A^{\otimes (n+1)}\xrightarrow{d_{n-1}} \cdots \stackrel{d_2}{\lra} A^{\otimes 3}\stackrel{d_1}{\lra} A^{\otimes 2} \lra 0 \right),
  \end{equation}
  where
  \begin{equation}
    d_i(a_0\otimes a_1\otimes\cdots \otimes a_{i+1})=\sum_{k=0}^{i+1} (-1)^k a_0\otimes \cdots a_{k-1}\otimes a_ka_{k+1}\otimes a_{k+2} \otimes \cdots \otimes a_{i+1}
  \end{equation}
\end{subequations}
The bar complex is a free bimodule resolution of $A$, as the augmented
complex
\begin{equation}\label{eqn-bar-complex-aug}
  \mathbf{p}_\bullet^\prime (A):= \left( \cdots \xrightarrow{d_{n+1}} A^{\otimes(n+2)} \stackrel{d_n}{\lra} A^{\otimes (n+1)}\xrightarrow{d_{n-1}} \cdots \stackrel{d_2}{\lra} A^{\otimes 3}\stackrel{d_1}{\lra} A^{\otimes 2} \lra \underline{A} \lra 0\right), 
\end{equation}
is acyclic (the underlined term sits in homological degree zero). This
can be seen by constructing a left $A$-module map
\begin{equation}\label{eqn-null-homotopy-map}
  \sigma: A^{\otimes n}\lra A^{\otimes (n+1)},\quad x\mapsto x \otimes 1 
\end{equation}
as the null-homotopy. If $M$ is an $(A,A)$-bimodule, then the
\emph{Hochschild homology} of $M$ is, by definition,
\begin{equation}\label{eqn-HH-definition}
  \mHH_\bullet (A,M):= \mH_\bullet (\mathbf{p}_\bullet(A)\otimes_{(A,A)}M).
\end{equation}

Another useful application of the bar resolution is to construct
projective resolutions: if $N$ is a left $A$-module, then the complex
$\mathbf{p}_\bullet (N):=\mathbf{p}_\bullet(A)\otimes_A N$ provides a
left projective resolution of $N$ over $A$. Likewise, if $M$ is a
right $A$-module, then we set
$\mathbf{p}_\bullet (M):=M\otimes_A\mathbf{p}_\bullet(A)$. One has the
chain complex
\begin{equation}\label{eqn-derived-tensor}
  M\otimes_A^{\mathbf{L}} N= M\otimes_A \mathbf{p}_\bullet (N)=\mathbf{p}_\bullet(M)\otimes_A N= M\otimes_A \mathbf{p}_\bullet(A) \otimes_A N,
\end{equation}
whose homology computes the derived tor of $M$ and $N$.

Assume now that $A$ carries also an $H^r$-module structure, and let $M$ be
an $r$-DG $A$-bimodule. Notice that the complex
\eqref{eqn-bar-complex} is $H^r$-equivariant. Indeed, the
differential maps $d_i$, $i\in \N$, are all built out of the
multiplication map on $A$, which respects the $\dif$-action.

Now, we temporarily forget the $H^r$-actions on $A$ and $M$, and
denote the resulting algebra and module by $A_0$ and $M_0$
respectively.  The usual Hochschild homology of $M_0$ over $A_0$ in
this case carries a natural $H^r$-action, since it is computed by the
homology of the complex
$\mathbf{p}_\bullet(A_0)\otimes_{(A_0,A_0)} M_0$, and this complex
carries an $H^r$-action coming from $A$ and $M$.

\begin{defn}\label{def-relative-HH}
  The \emph{relative Hochschild homology} of an $r$-DG bimodule
  $(M,\dif_M)$ over $(A,\dif_A)$ is the usual Hochschild homology of
  $M_0$ over $A_0$ equipped with the induced $H^r$-action from
  $\dif_M$ and $\dif_A$, and is denoted by
  \[
    \mHH^{\dif}_\bullet(M):=\mHH_\bullet(A_0,M_0).
  \]
\end{defn}

By construction, $\mHH_\bullet^\dif$ is a covariant functor in $M$: if
$f: M \lra N$ is an $H^r$-intertwining map of $r$-DG bimodules, then
there is an induced map
\begin{subequations}
  \begin{equation}\label{eqn-functoriality-HH}
    \mHH^\dif_\bullet(f)  : \mHH^\dif_\bullet (M)  \lra \mHH_\bullet^\dif (N),
  \end{equation}
  which is obtained by taking homology for the map of chain complexes
  \begin{equation}
    \Id\otimes f: \mathbf{p}_\bullet (A_0) \otimes_{(A_0,A_0)} M\lra \mathbf{p}_\bullet (A_0) \otimes_{(A_0,A_0)} N.
  \end{equation}
\end{subequations}
When no confusion can be caused, we will write for short
$\mHH(f):=\mHH_\bullet^\dif(f)$.

We record the following two basic properties of the relative
Hochschild homology that will be needed for later constructions.

The first property follows from the usual proof of trace-like
properties for Hochschild homology with respect to derived tensor
product of bimodules.

\begin{prop} \label{HHcyclprop} Given two $r$-DG bimodules $M$ and $N$
  over $A$, there is an isomorphism of $H^r$-modules
  \[
    \mHH^\dif_{\bullet}(M\otimes^{\mathbf{L}}_A N)\cong
    \mHH^\dif_{\bullet}(N\otimes^{\mathbf{L}}_A M).
  \]
\end{prop}

The second property allows one to compute relative Hochschild homology
groups with other bimodule resolutions whenever possible. After all,
the usual simplicial bar resolution is quite unwieldy for actual
computations due to its exponentially growing size.

\begin{prop} \label{prop-resolution-independence} Let $M$ be an $r$-DG
  bimodule over $A$. Suppose $f:K_\bullet \lra A_0$ is an
  $H^r$-equivariant resolution of $A_0$ such that each term of
  $K_\bullet$ is projective as an $(A_0, A_0)$-bimodule. Then $f$
  induces an isomorphism of $H^r$-modules
  \[
    \mH_\bullet(K_\bullet\otimes_{(A_0, A_0)}M)\cong
    \mHH^{\dif}_\bullet(M).
  \]
\end{prop}
\begin{proof}
  Combining \eqref{eqn-HH-definition} and
  \eqref{eqn-derived-tensor}, one may describe $\mHH_\bullet^\dif(M)$
  as the homology of $A_0$ tensored with the complex of
  $(A_0,A_0)$-bimodules
  \begin{equation}
    \mathbf{p}_{\bullet,\bullet}(M):= \mathbf{p}_\bullet(A_0) \otimes_{A_0} M \otimes_{A_0} \mathbf{p}_\bullet(A_0) ,
  \end{equation}
  where the short hand notation $\mathbf{p}_{\bullet,\bullet}(M)$ is
  just to emphasize that the complex is the total complex of a
  bicomplex. By construction, $\mathbf{p}_{\bullet, \bullet}(M)$ is an
  $H^r$-equivariant chain complex, such that each term is projective
  as a bimodule over $(A_0,A_0)$.

  Let $C_f$ be the cone of $f$, i.e., the complex
  \begin{equation}
    C_f= \left(
      \cdots \lra K_{1} \lra K_0 \lra \underline{A_0} \lra 0
    \right),
  \end{equation}
  where the underlined term $A_0$ sits in homological degree zero. Thus
  we have a short exact sequence of complexes of $(A_0,A_0)$-bimodules
  \begin{equation}
    0 \lra A_0 \lra C_f \lra K_\bullet[1] \lra 0
  \end{equation}
  that is $H^r$-equivariant. By assumption, $C_f$ is an acyclic
  complex.

  Tensoring the sequence with $\mathbf{p}_{\bullet , \bullet} ( M )$,
  the short exact sequence above remains exact:
  \begin{equation}
    0 \lra A_0 \otimes_{(A_0,A_0)} \mathbf{p}_{\bullet , \bullet} (M) \lra C_f \otimes_{(A_0,A_0)} \mathbf{p}_{\bullet , \bullet} ( M ) \lra  K_\bullet \otimes_{(A_0,A_0)}\mathbf{p}_{\bullet , \bullet} ( M )[1] \lra 0.
  \end{equation}
  The middle term
  $ C_f \otimes_{(A_0,A_0)} \mathbf{p}_{\bullet , \bullet} (M) $ is
  easily seen to be an acyclic complex. It follows from the long exact
  sequence in homology that we have an isomorphism
  \begin{equation}\label{eqn-HH-reduction1}
    \mHH_\bullet^\dif(M)= \mH_\bullet( A_0\otimes_{(A_0,A_0)}\mathbf{p}_{\bullet,\bullet} (M)) \cong \mH_\bullet (K_\bullet\otimes_{(A_0,A_0)}\mathbf{p}_{\bullet,\bullet} (M)) .
  \end{equation}

  The natural evaluation map $\mathbf{p}_{\bullet,\bullet}(M) \lra M$
  is a quasi-isomorphism that intertwines the $H^r$-actions, and thus the
  kernel of this map, denoted $N$, is an acyclic complex of
  $H^r$-equivariant bimodules:
  \begin{equation}
    0 \lra N \lra \mathbf{p}_{\bullet,\bullet}(M) \lra M \lra 0.
  \end{equation}
  Tensor this short exact sequence of bimodules with $K_\bullet$ and
  take homology. As above, $K_\bullet \otimes_{(A_0,A_0)} N$ remains
  acyclic. We obtain that
  \begin{equation}\label{eqn-HH-reduction2}
    \mH_\bullet (K_\bullet\otimes_{(A_0,A_0)}\mathbf{p}_{\bullet,\bullet} (M)) \cong
    \mH_\bullet (K_\bullet\otimes_{(A_0,A_0)}M) .
  \end{equation}
  Combining \eqref{eqn-HH-reduction1} and
  \eqref{eqn-HH-reduction2} gives us the desired result.
\end{proof}

\section{MOY graphs, Soergel bimodules, and
  \texorpdfstring{$H^\prime$}{H'}-structure}
\label{sec:moy-graphs-soergel}
Singular Soergel bimodules play a central role in categorification of
the colored HOMFLYPT polynomial \cite{WebWil,MSV2,Cautisremarks}.
In this section we introduce these bimodules, as well as
$H^\prime$-module 
structures on these objects.  Important maps between these bimodules
are reviewed, and we conclude this section with $H^\prime$-equivariant
versions of
Rickard complexes.  These complexes of $H^\prime$-equivariant 
bimodules will be shown
later in the paper to satisfy braid group relations and are used to
build a link homology.

\subsection{Polynomial algebras and Soergel bimodules}
Let $A = A_n = \myZ[x_1,\ldots,x_n]$
be the graded polynomial algebra, where
each generator $x_i$ has degree two.  The subspace of symmetric
polynomials %
$\myZ[x_1, \dots x_n]^{S_n} = \myZ[e_1, \dots, e_n] $
is also graded where the degree of the $i$th elementary symmetric
function $e_i$ is $2i$. The $i$th complete symmetric polynomial is
denoted $h_i$ and if $\lambda$ is a Young diagram, $s_\lambda$ is the
Schur polynomial associated with $\lambda$. 

For positive integers $\listk{k}= (k_1, \ldots, k_r) $, such that
$k_1+\cdots+k_r=n $, we consider more general subalgebras of
polynomial algebras:
\begin{equation} \label{defofsymmalg}
  A_{\listk{k}}:= \myZ[x_{1},
  \ldots, x_n]^{S_{k_1} \times \cdots \times S_{k_r}} \ .
\end{equation}
Throughout this work, we will often present such algebras, and
elements in these algebras in a graphical fashion.  The algebra
$A_{\listk{k}} $ will be denoted by $r$ vertical lines labeled $k_1$
to $k_r$ from left to right.  A generating elementary symmetric
function in the algebra will be denoted by a labeled dot. For example,
in \eqref{graphdefofalg}, the $i$th elementary symmetric function in the first
$k_1$ variables is given by a dot labeled $e_i$ on the first strand.
In general, one takes formal linear combinations of these diagrams to
represent a generic element of $A_{\listk{k}}$.

\begin{equation} \label{graphdefofalg} A_{\listk{k}} =
  \NB{\tikz[font=\tiny]{\begin{scope}
  \coordinate (bl) at (-0.5, -1);
  \coordinate (br) at ( 0.5, -1);
  \coordinate (tl) at (-0.5,  1);
  \coordinate (tr) at ( 0.5,  1);
   \node at (0,0) {$\dots$};
  \draw[>->] (bl) -- (tl) node[pos = 0, below] {$k_1$} node[pos = 1, above] {};
  \draw[>->] (br) -- (tr) node[pos = 0, below] {$k_r$} node[pos = 1, above] {};
\end{scope}}} \ , \quad \quad \quad A_{\listk{k}} \ni
  e_i(x_1, \ldots, x_{k_1})= \NB{\tikz[font=\tiny]{\begin{scope}
  \coordinate (bl) at (-0.5, -1);
  \coordinate (br) at ( 0.5, -1);
  \coordinate (tl) at (-0.5,  1);
  \coordinate (tr) at ( 0.5,  1);
   \node at (0,0) {$\dots$};
  \draw[>->] (bl) -- (tl) node[pos = 0, below] {$k_1$} node[pos = 1, above] {} coordinate[pos = 0.3] (ga);
  \draw[>->] (br) -- (tr) node[pos = 0, below] {$k_r$} node[pos = 1, above] {};
     \filldraw[] (ga) circle (1mm)
  node[left, ] {$e_i$};
\end{scope}

}} \ .
\end{equation}

Let $G_1$ and $G_2$ be Young subgroups of $S_n$ with
$G_2 \subseteq G_1$.  Then
$\Z[x_1,\ldots,x_n]^{G_1}$ %
is a subalgebra of $\Z[x_1,\ldots, x_n]^{G_2}.$ %
Then we may consider
$\Z[x_1,\ldots, x_n]^{G_2}$
as a
$(\myZ[x_1,\ldots, x_n]^{G_1},\myZ[x_1,\ldots,x_n]^{G_2})$-bimodule, where the right action is given by
multiplication and the left action is given by multiplication coming
from the inclusion of algebras.  In a similar way, we could consider
it as a
$(\myZ[x_1,\ldots, x_n]^{G_2},\myZ[x_1,\ldots,
x_n]^{G_1})$-bimodule.

In particular, suppose
\[\listk{k}=(k_1,\ldots, k_i, k_{i+1}, \ldots, k_{i+j}, k_{i+j+1},
  \ldots, k_r) \ , \quad \listk{\ell}=(k_1,\ldots, k_i,
  k_{i+1}+\cdots+ k_{i+j}, k_{i+j+1}, \ldots, k_r) \ .
\]
Then $A_{\listk{\ell}} \subseteq A_{\listk{k}}$ and we define the
$(A_{\listk{\ell}},A_{\listk{k}})$-bimodule
$A^{\listk{\ell}}_{\listk{k}} $ and the
$(A_{\listk{k}},A_{\listk{\ell}})$-bimodule
$A_{\listk{\ell}}^{\listk{k}} $ to both be $A_{\listk{k}}$ as vector
spaces.  A graphical representation of the bimodule
$A^{\listk{\ell}}_{\listk{k}} $ is given by vertical strands labeled
from left to right by $k_1$ to $k_i$, followed by vertical strands
labeled $k_{i+1}$ to $k_{i+j}$ merged into a vertex with a single
strand coming out of it labeled $k_{i+1}+\cdots+k_{i+j}$, followed by
vertical strands labeled $k_{i+j+1}$ to $k_r$ (see
\eqref{bimodmerge}).  Graphically, $A_{\listk{\ell}}^{\listk{k}} $ is
realized as the reflection of $A^{\listk{\ell}}_{\listk{k}} $ in a
horizontal axis.
\begin{equation} \label{bimodmerge} A^{\listk{\ell}}_{\listk{k}} ~=~
  \NB{\tikz[font=\tiny, yscale = 0.5, xscale =0.8]{\begin{scope}[yscale=1]
  \coordinate (a6) at (-3.5, -2.5);
  \coordinate (t6) at (-3.5, 1);
  \coordinate (a7) at (3.5, -2.5);
  \coordinate (t7) at (3.5, 1);
  \coordinate (t) at (0,1);
  \coordinate (o) at (0,-0);
  \coordinate (a0) at (-2, -2.5);
  \coordinate (a1) at (-1, -2.5);
  \coordinate (a2) at (-0.5, -2.5);
  \node[below] (a3) at (0, -2.5) {$\dots$};
  \coordinate (a5) at (2, -2.5);
  \coordinate (a4) at (1, -2.5);
  \node (a8) at (-2, -.75) {$\dots$};
  \node (a9) at (2, -.75) {$\dots$};
  \draw[->] (o) -- (t) node [above, pos = 1] {$k_{i+1}+\cdots +k_{i+j}$} coordinate[pos
  =0.5] (g);
  \draw[>-] (a0) .. controls +(0, 0.4) and +(0,0) .. (o) node [pos =0,
  below] {$k_{i+1}$};
  \draw[>-] (a5) .. controls +(0, 0.4)  and +(0,0) .. (o) node [pos =0, below] {$k_{i+j}$};
  \draw[>-] (a1) .. controls +(0, 0.4) and +(0,0) .. (o) node [pos =0,
  below] {$k_{i+2}$};
  \draw[>-] (a4) .. controls +(0, 0.4)  and +(0,0) .. (o) node [pos =0, below] {$k_{i+j-1}$};
  \draw[->] (a6) -- (t6) node [above, pos = 1] {$k_1$} coordinate[pos
  =0.5] (g);
    \draw[->] (a7) -- (t7) node [above, pos = 1] {$k_r$} coordinate[pos
  =0.5] (g);
\end{scope}

}}
\end{equation}
Elements of these bimodules are formal linear combinations of these
diagrams with dots labeled by elementary symmetric functions.  The
dots slide through vertices as indicated in \eqref{dotslide}. The
analogous relations hold in $A_{\listk{\ell}}^{\listk{k}}$.
\begin{equation} \label{dotslide} \NB{\tikz[yscale=0.5, xscale =
  0.8]{\begin{scope}[yscale=1]
  \coordinate (a6) at (-3.5, -2.5);
  \coordinate (t6) at (-3.5, 1);
  \coordinate (a7) at (3.5, -2.5);
  \coordinate (t7) at (3.5, 1);
  \coordinate (t) at (0,1);
  \coordinate (o) at (0,-0);
  \coordinate (a0) at (-1.5, -2.5);
  \coordinate (a1) at (-1, -2.5);
  \coordinate (a2) at (-0.5, -2.5);
  \node[below] (a3) at (0, -2.5) {$\dots$};
  \coordinate (a5) at (1.5, -2.5);
  \coordinate (a4) at (1, -2.5);
  \node (a8) at (-2, -.75) {$\dots$};
  \node (a9) at (2, -.75) {$\dots$};
  \draw[->] (o) -- (t) node [above, pos = 1] {$k_{i+1}+\cdots +k_{i+j}$} coordinate[pos =0.5] (X);
  \draw[>-] (a0) .. controls +(0, 0.4) and +(0,0) .. (o) node [pos =0,
  below] {$k_{i+1}$};
  \draw[>-] (a5) .. controls +(0, 0.4)  and +(0,0) .. (o) node [pos =0, below] {$k_{i+j}$};
  \draw[->] (a6) -- (t6) node [above, pos = 1] {$k_1$}; 
    \draw[->] (a7) -- (t7) node [above, pos = 1] {$k_r$};
     \filldraw[] (X) circle (1mm) node[left, ] {$e_d$};
\end{scope}

}} ~=~ \sum_{d_{1}+\cdots+d_j=d}~ \NB{\tikz[yscale=0.5,
  xscale = 0.8]{\begin{scope}[yscale=1]
  \coordinate (a6) at (-3.5, -2.5);
  \coordinate (t6) at (-3.5, 1);
  \coordinate (a7) at (3.5, -2.5);
  \coordinate (t7) at (3.5, 1);
  \coordinate (t) at (0,1);
  \coordinate (o) at (0,-0);
  \coordinate (a0) at (-1.5, -2.5);
  \coordinate (a1) at (-1, -2.5);
  \coordinate (a2) at (-0.5, -2.5);
  \node[below] (a3) at (0, -2.5) {$\dots$};
  \coordinate (a5) at (1.5, -2.5);
  \coordinate (a4) at (1, -2.5);
  \node (a8) at (-2, -.75) {$\dots$};
  \node (a9) at (2, -.75) {$\dots$};
  \draw[->] (o) -- (t) node [above, pos = 1] {$k_{i+1}+\cdots +k_{i+j}$} coordinate[pos =0.5] (X);
  \draw[>-] (a0) .. controls +(0, 0.4) and +(0,0) .. (o) node [pos =0,
  below] {$k_{i+1}$} coordinate[pos=0.3] (g1);
  \draw[>-] (a5) .. controls +(0, 0.4)  and +(0,0) .. (o) node [pos
  =0, below] {$k_{i+j}$} coordinate[pos=0.3] (g2);
  \draw[->] (a6) -- (t6) node [above, pos = 1] {$k_1$}; 
    \draw[->] (a7) -- (t7) node [above, pos = 1] {$k_r$};
     \filldraw[] (g1) circle (1mm) node[left, ] {$e_{d_1}$};
     \filldraw[] (g2) circle (1mm) node[right, ] {$e_{d_j}$};

\end{scope}

}}
\end{equation}

One may tensor these bimodules over intermediate algebras and obtain
new bimodules.  The idempotent completion of this monoidal category is
called the category of \emph{singular Soergel bimodules}.  In this
work, we will rarely consider summands of these tensor products.  The
category of singular Soergel bimodules plays an important role in
representation theory, and for a classification of the indecomposable
objects, see \cite{Wilthesis}.  Graphically, the tensor product
$M \otimes_A N$ corresponds to concatenating the diagram for $M$ on
top of the diagram for $N$.  We refer to the graphs obtained in this
way as \emph{MOY graphs}. We will often restrict to trivalent MOY
graphs, since up to canonical isomorphisms they encode Soergel
bimodules associated with general MOY graphs as illustrated by
\eqref{bimodmerge-general}. A MOY graph with only one (non-univalent)
vertex which is trivalent is said to be \emph{basic}.

\begin{equation} \label{bimodmerge-general} \NB{\tikz[font=\tiny, yscale = 0.5,
  xscale =0.8]{}} \ \ \simeq \ \ \NB{\tikz[font=\tiny,
  yscale = 0.5, xscale =0.8]{\begin{scope}[yscale=1]
  \coordinate (a6) at (-3.5, -2.5);
  \coordinate (t6) at (-3.5, 1);
  \coordinate (a7) at (3.5, -2.5);
  \coordinate (t7) at (3.5, 1);
  \coordinate (t) at (0,1);
  \coordinate (o) at (0,-0);
  \coordinate (a0) at (-2, -2.5);
  \coordinate (a1) at (-1, -2.5);
  \coordinate (a2) at (-0.5, -2.5);
  \node[below] (a3) at (0, -2.5) {$\dots$};
  \coordinate (a5) at (2, -2.5);
  \coordinate (a4) at (1, -2.5);
  \node (a8) at (-2, -.75) {$\dots$};
  \node (a9) at (2, -.75) {$\dots$};
  \draw[->] (o) -- (t) node [above, pos = 1] {$k_{i+1}+\cdots +k_{i+j}$} coordinate[pos
  =0.5] (g);
  \draw[>-] (a0) .. controls +(0, 0.4) and +(0,0) .. (o) node [pos =0,
  below] {$k_{i+1}$} coordinate [pos =0.7] (o2) coordinate[pos= 0.3] (o3);
  \draw[>-] (a5) .. controls +(0, 0.4)  and +(0,0) .. (o) node [pos =0, below] {$k_{i+j}$};
  \draw[>-] (a1) .. controls +(0, 0.4) and +(0,0) .. (o3) node [pos =0,
  below] {$k_{i+2}$};
  \draw[>-] (a4) .. controls +(0, 0.4)  and +(0,0) .. (o2) node [pos =0, below] {$k_{i+j-1}$};
  \draw[->] (a6) -- (t6) node [above, pos = 1] {$k_1$} coordinate[pos
  =0.5] (g);
    \draw[->] (a7) -- (t7) node [above, pos = 1] {$k_r$} coordinate[pos
  =0.5] (g);
\end{scope}

}}
\end{equation}

There are useful shifts of Soergel bimodules which appear in the
literature.  We now define these shifts.
\begin{defn}
  \label{defn:soergel}
  Let $\Gamma$ be a trivalent MOY graph.  For each (non-univalent)
  vertex $v \in \Gamma$, let $\ell_1(v)$ and $ \ell_2(v) $ be the
  labels of the thinner edges attached to $v$.  Let
  $\gamma(\Gamma)= \frac{1}{2} \sum_{v} \ell_1(v) \ell_2(v) \in
  \frac{1}{2}\Z$.  Then define the shifted bimodule
  $ \soergel(\Gamma)= q^{-\gamma(\Gamma)} \Gamma $.
\end{defn}

\begin{prop} \label{prop-grank} For a trivalent graph $\Gamma$, the
  graded rank of the corresponding singular Soergel bimodule as a left
  module is given by
  \begin{equation}
    \rkq(\soergel(\Gamma))=  \prod_{v_m,v_s}
    q^{\frac{1}{2} (\ell_1(v_m) \ell_2(v_m)-\ell_1(v_s) \ell_2(v_s))}
    \genfrac[]{0pt}{}{\ell_1(v_m)+\ell_2(v_m)}{\ell_1(v_m)} ,  
  \end{equation}
  where $v_m$ are merge vertices, $v_s$ are split vertices and $\rkq$
  denotes the graded rank.
\end{prop}

\begin{proof}
  First note that $\Z[x_1,\ldots,x_{a+b}]^{S_a \times S_b}$
  is a
  graded free left module over %
  $\Z[x_1,\ldots,x_{a+b}]^{S_{a+b}}$
  of graded rank $q^{ab} \genfrac[]{0pt}{}{a+b}{a} $. The proposition
  now follows from repeated use of this fact along with our grading
  conventions.
\end{proof}

Maps between Soergel bimodules are discussed in
Section~\ref{sec:pdgmaps}.

\subsection{\texorpdfstring{$H^\prime$}{H'}-equivariant Soergel
   bimodules}

\label{sec:dif-and-soergel}
The polynomial algebra %
$A=\Z[x_1,\ldots,x_n]$
carries an
$H^\prime$-module structure, where $\dif(x_i)=x_i^2$, and the action is
extended to the entire algebra by linearity and the Leibniz rule.

\begin{lem}\label{lem:dif-delta-nabla}
  \begin{enumerate}
  \item \label{it:dif-delta} Let
    $\Delta := \prod_{i<j} (x_j - x_i) \in A$. Then
    $\dif(\Delta) = (n-1)e_1(x_1, \dots, x_n) \Delta.$
  \item \label{it:dif-nabla} Let $k \in \{1, \dots, n\}$ and set
    \[
      \nabla:= \prod_{i=1}^k\prod_{j= k+1}^{n} (x_j - x_i) .
    \]
    Then
    \[\dif (\nabla) %
      = \left( (n-k)e_1(x_1, \dots, x_k) + ke_1(x_{k+1}, \dots, x_{n})
      \right) \nabla.\]
  \end{enumerate}
\end{lem}

\begin{proof}
  Let us start with item \ref{it:dif-delta}.  This follows from the
  fact that $\dif(x_i - x_j) = x_i^2 - x_j^2 = (x_i +x_j)(x_i -
  x_j)$. Indeed using this identity and the Leibniz rule, one obtains:
  \[
    \dif(\Delta)= \sum_{1\leq i<j \leq n}(x_i+x_j)\Delta =
    (n-1)e_1(x_1, \dots, x_n) \Delta.
  \]

  Item \ref{it:dif-nabla} is similar:
  \[
    \dif(\nabla) = \sum_{i=1}^k \sum_{j=k+1}^n(x_i + x_j)\nabla =
    \left( (n-k)e_1(x_1, \dots, x_k) + ke_1(x_{k+1}, \dots, x_{n})
    \right) \nabla.
  \]
\end{proof}
The $H^\prime$-module structure on $\Z[x_1, \ldots, x_n]$ descends to an
$H^\prime$-module structure on $A_{}=\Z[x_1,\ldots,x_n]^G$
where $G$ is
a Young subgroup of $S_n$.  In particular, algebras of symmetric
polynomials are $H^\prime$-equivariant. 

\begin{lem}\label{lem:dif-easy-sym}
  The derivation
  $\dif$ acts on elementary, complete, and Schur
  symmetric functions of 
  $\Z[x_1,\ldots,x_n]^{S_n}$
  as follows:
  \begin{enumerate}
  \item $\dif(e_i)=e_1 e_i - (i+1) e_i$,
  \item $\dif(h_i)=(i+1) h_{i+1} - h_1 h_i $,
  \item $\dif(s_{\lambda})= \sum_{\mu_i} (\lambda_i+1-i) s_{\mu_i} $
    where the sum is over all partitions $\mu_i$ obtained from
    $\lambda$ by adding a box to the $i$th row.
  \end{enumerate}
\end{lem}

\begin{proof}
  See \cite[Lemma 3.1]{EQ1} for the first and second equalities and
  \cite[Lemma 2.4]{EQ2} for the third equality.
\end{proof}

  From straightforward calculations, we obtain the following.%

\begin{lem}
  \label{lem:pdiff}
  \begin{enumerate}
  \item Let $\listk{k}=(k_1, \dots, k_\ell)$ be a finite sequence of
    non-negative integers with $k_1 + \dots+ k_\ell = k$. The algebra
    $\polA_{\listk{k}}:= \Z[x_1, \dots, x_k]^{S_{\listk{k}}}$
    with the inherited map $\dif$
    is an $H^\prime$-equivariant algebra.
  \item Basic singular Soergel bimodules over $\Z$ are
    $H^\prime$-equivariant bimodules. As a consequence, Soergel
    bimodules over $\Z$ associated with MOY graphs are
    $H^\prime$-equivariant bimodules. %
  \item An $H^\prime$-equivariant singular Soergel bimodule over $\Z$
    can be endowed with a new $H^\prime$-structure by twisting by a linear
    polynomial with coefficients in $\Z $.
  \end{enumerate}
\end{lem}

In order to categorify the colored Jones polynomial at a root of unity, we will later on tensor with $\Fp$ and consider every 
object as a $p$-DG module. For this to make sense, it is important to
check that $\dif^p$ acts trivially on these objects. This is ensured
by the following lemma. %

\begin{lem}\label{lem:pDG-ready}
  Let $M$ be a singular Soergel bimodule associated with a MOY graph
  and $x$ be an element of $M$, then $\dif^p(x)$ is divisible by $p$.
\end{lem}

\begin{proof}
  We consider $M$ as an $\e$-DG algebra. Since $\dif$ acts as a
  derivation, it is enough to prove this statement for
  generators of $M$ (as an algebra). Therefore we consider the
  polynomial algebra associated to an edge with label $\ell$. By
  definition, this is $\Z[x_1, \dots x_\ell]^{S_\ell}$ and it is
  generated by the elementary symmetric polynomials
  $e_1, \dots, e_\ell$. It is a subalgebra of $\Z[x_1, \dots x_\ell]$,
  so we may prove the statement for this algebra. Invoking the
  Leibniz rule one more time, it is enough to prove the statement for
  $x_i$ for $i \in \{1,\dots, \ell\}$. For all $k\geq 0$, one has
  $\dif^k(x_i) = k! x_i^{k+1}$, so that $\dif^p(x_i)$ is
  indeed a multiple of $p$.
\end{proof}

Let $M$ be an $H^\prime$-equivariant %
Soergel $(A_{\listk{k}},A_{\listk{\ell}})$-bimodule. %
Let
$f \in A_{\listk{\ell}} $ be a degree two element.  Then we define
$M^f$ to be the Soergel $(A_{\listk{k}},A_{\listk{\ell}})$-bimodule,
where for all $m \in M$, one has:
\[
\dif_{M^f}(m) = \dif_M(m) + m\cdot f.
\]
We
define ${}^f\!M $ in a similar way.

\begin{rem}
  On an $H^\prime$-equivariant Soergel bimodule $\soergel(\Gamma)$,
  twists are linear combinations of the first elementary symmetric
  functions on the variables for each edge.  This leads us to
  Definition \ref{defn:greendots} below.  Note however that on direct
  sums of $\soergel(\Gamma) $, more complicated twists could occur
  between the various summands.  Higher symmetric functions appear as
  we will see in the proof of Lemma~\ref{lem:twist-slides-11} and in
  the computation for the Hopf link in Section~\ref{sec:hopflink} .
\end{rem}

In order to facilitate calculations later on in a more compact way,
we introduce green dots which will encode twists pictorially.

\begin{defn}\label{defn:greendots}
  A \emph{green-dotted} MOY graph is a MOY graph $\Gamma$ with some
  green dots floating freely on its edges. Each green dot has a given
  multiplicity (an element of $\Z$). If a given edge carries several
  green dots, they may be replaced by one green dot on this edge with
  the sum of all multiplicities. Green dots with multiplicity equal to
  $0$ can be removed. Alternatively, these dots may be thought of as a
  map $g\colon\thinspace E(\Gamma) \to \Z$. However it is convenient
  to have a more pictorial approach to these green dots.
  
  A green-dotted MOY graph $\Gamma$ gives rise to a twisted $H^\prime$-equivariant
  Soergel bimodule $\soergel(\Gamma)$. The differential on the
  generator is given by a linear combination of terms with dots
  labeled by first elementary symmetric functions with coefficients
  coming from the integral labels of the green dots.
\end{defn}

\begin{example}
  \label{exa:greendots}
  For the Soergel bimodule associated to the green-dotted merge MOY
  graph $\Gamma$,
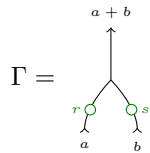
  \[
    \Gamma= \NB{\tikz[font=\tiny, scale =0.7]{\begin{scope}
  \coordinate (m) at (  0,  0);
  \coordinate (t) at (  0, 1);
  \coordinate (br) at (+.5,  -1);
  \coordinate (bl) at (-.5,  -1);
  \draw[>-] (bl) .. controls +(0,0.5) and + (0, 0) .. (m) node[pos =
  0, below] {$a$} coordinate[pos = 0.3] (ga);
  \draw[>-] (br) .. controls +(0,0.5) and + (0, 0) .. (m) node[pos =
  0, below] {$b$} coordinate[pos = 0.3] (gb);
  \draw[->] (m) -- (t) node[pos = 1, above] {$a+b$};
  \filldraw[draw= green!50!black, fill = white] (ga) circle (1mm)
  node[left, green!50!black] {$r$};
  \filldraw[draw= green!50!black, fill = white] (gb) circle (1mm)
  node[right, green!50!black] {$s$};
\end{scope}}} \
  \]
  the twist on the generator is given by
  \[
    \dif \left( {\NB{\tikz[font=\tiny, scale =0.7]{\begin{scope}
  \coordinate (m) at (  0,  0);
  \coordinate (t) at (  0, 1);
  \coordinate (br) at (+.5,  -1);
  \coordinate (bl) at (-.5,  -1);
  \draw[>-] (bl) .. controls +(0,0.5) and + (0, 0) .. (m) node[pos =
  0, below] {$a$};
  \draw[>-] (br) .. controls +(0,0.5) and + (0, 0) .. (m) node[pos =
  0, below] {$b$};
  \draw[->] (m) -- (t) node[pos = 1, above] {$a+b$};
\end{scope}}}} \right) = \quad
    r {\NB{\tikz[font=\tiny, scale =0.7]{\begin{scope}
  \coordinate (m) at (  0,  0);
  \coordinate (t) at (  0, 1);
  \coordinate (br) at (+.5,  -1);
  \coordinate (bl) at (-.5,  -1);
  \draw[>-] (bl) .. controls +(0,0.5) and + (0, 0) .. (m) node[pos =
  0, below] {$a$} coordinate[pos = 0.3] (ga);
  \draw[>-] (br) .. controls +(0,0.5) and + (0, 0) .. (m) node[pos =
  0, below] {$b$} coordinate[pos = 0.3] (gb);
  \draw[->] (m) -- (t) node[pos = 1, above] {$a+b$};
   \filldraw[] (ga) circle (1mm)
  node[left, ] {$e_1$};
\end{scope}}}} \quad + \quad s
    {\NB{\tikz[font=\tiny, scale =0.7]{\begin{scope}
  \coordinate (m) at (  0,  0);
  \coordinate (t) at (  0, 1);
  \coordinate (br) at (+.5,  -1);
  \coordinate (bl) at (-.5,  -1);
  \draw[>-] (bl) .. controls +(0,0.5) and + (0, 0) .. (m) node[pos =
  0, below] {$a$} coordinate[pos = 0.3] (ga);
  \draw[>-] (br) .. controls +(0,0.5) and + (0, 0) .. (m) node[pos =
  0, below] {$b$} coordinate[pos = 0.3] (gb);
  \draw[->] (m) -- (t) node[pos = 1, above] {$a+b$};
   \filldraw[] (gb) circle (1mm)
  node[right, ] {$e_1$};
\end{scope}}}} \ .
  \]
\end{example}

Our topological invariant uses Soergel bimodules associated with
green-dotted MOY graph, so we need an analogue of
Lemma~\ref{lem:pDG-ready} for that. %

\begin{lem}\label{pDG-ready-green}
  Let $M$ be a singular Soergel bimodule associated with a
  green-dotted MOY graph $\Gamma$ and $x$ and element of $M$, then $\dif^p(x)$
  is a multiple of $p$. 
\end{lem}

\begin{proof}
  We argue by induction on the sum of absolute values of the multiplicities of the green dots on the graph
  $\Gamma$. If this sum is $0$, the case is covered by
  Lemma~\ref{lem:pDG-ready} since green dots with multiplicity $0$ can
  be removed.

  Suppose now that $\Gamma$ has a green dot $g$ of non-trivial
  multiplicity $\alpha$. Denote by $\ell$ the label of the edge
  containing this green dot $g$. To this edge, there is an associated
  polynomial algebra $\Z[x_1, \dots, x_\ell]^{S_\ell}$ generated
  algebraically by the elementary symmetric polynomials
  $(e_i)_{1\leq i \leq \ell}$.  Note that the complete homogeneous
  symmetric polynomials $(h_i)_{1\leq i \leq \ell}$) also form a
  generating set.

  We distinguish two cases: $\alpha>0$ and $\alpha<0$. Suppose first
  that $\alpha>0$ and denote by $N$ the $H^\prime$-equivariant
  Soergel bimodule associated with $\Gamma$ where the multiplicity of
  the green dot $g$ is $\alpha-1$. By induction, we know that
  $\dif_N^p(x)$ is a multiple of $p$ for all $x$ in $N$.

  Let $x$ be an element of $M$. By definition, one has:
  \[
    \dif_M(x) = \dif_N(x) + xh_1.
  \]
  By an easy induction (using Lemma~\ref{lem:dif-easy-sym}) one can show that:
  \[
    \dif^k_M(x) = \sum_{i=0}^k\frac{k!}{i!}\dif^{i}_N(x) h_{k-i}.
  \]
  Hence for $k=p$ we obtain that $\dif^k_M(x)$ is a multiple of $p$.

  Suppose now
  that $\alpha<0$ and denote by $N$, the
  Soergel bimodule associated with $\Gamma$ where the multiplicity of
  the green dot $g$ is $\alpha+1$. By induction, we know that
  $\dif_N^p(x)$ is a multiple of $p$ for all $x$ in $N$.

  Let $x$ be an element of $M$. By definition, one has:
  \[
    \dif_M(x) = \dif_N(x) - xe_1.
  \]
  By an easy induction (using Lemma~\ref{lem:dif-easy-sym}) one can show that:
  \[
    \dif^k_M(x) = \sum_{i=0}^k(-1)^i\frac{k!}{i!}\dif^{i}_N(x) e_{k-i}.
  \]
  Hence for $k=p$ we obtain that $\dif^k_M(x)$ is a multiple of $p$.
\end{proof}

The next result is a convenient tool which allows us to move certain
configurations of greens dots through vertices of a MOY graph.  We
refer to these manipulations of twisted structures on Soergel
bimodules as \emph{green dot migrations}.
\begin{lem}
  \label{lem:iso-geen-dots}
  The following $H^\prime$-equivariant bimodules are isomorphic
  \[
    \soergel\left(\NB{\tikz[font=\tiny]{\begin{scope}
  \coordinate (m) at (  0,  0);
  \coordinate (t) at (  0, 1);
  \coordinate (br) at (+.5,  -1);
  \coordinate (bl) at (-.5,  -1);
  \draw[>-] (bl) .. controls +(0,0.5) and + (0, 0) .. (m) node[pos =
  0, below] {$a$} coordinate[pos = 0.3] (ga);
  \draw[>-] (br) .. controls +(0,0.5) and + (0, 0) .. (m) node[pos =
  0, below] {$b$} coordinate[pos = 0.3] (gb);
  \draw[->] (m) -- (t) node[pos = 1, above] {$a+b$};
  \filldraw[draw= green!50!black, fill = white] (ga) circle (1mm)
  node[left, green!50!black] {$r$};
  \filldraw[draw= green!50!black, fill = white] (gb) circle (1mm)
  node[right, green!50!black] {$r$};
\end{scope}}}\right) \cong
    \soergel\left(\NB{\tikz[font=\tiny]{\begin{scope}
  \coordinate (m) at (  0,  0);
  \coordinate (t) at (  0, 1);
  \coordinate (br) at (+.5,  -1);
  \coordinate (bl) at (-.5,  -1);
  \draw[>-] (bl) .. controls +(0,0.5) and + (0, 0) .. (m) node[pos =
  0, below] {$a$} coordinate[pos = 0.3] (ga);
  \draw[>-] (br) .. controls +(0,0.5) and + (0, 0) .. (m) node[pos =
  0, below] {$b$} coordinate[pos = 0.3] (gb);
  \draw[->] (m) -- (t) node[pos = 1, above] {$a+b$} coordinate[pos = 0.7] (gc);
    \filldraw[draw= green!50!black, fill = white] (gc) circle (1mm)
  node[right, green!50!black] {$r$};
\end{scope}}}\right) , \qquad \qquad
    \soergel\left(\NB{\tikz[font=\tiny]{\begin{scope}
  \coordinate (m) at (  0,  0);
  \coordinate (b) at (  0, -1);
  \coordinate (tr) at (+.5,  1);
  \coordinate (tl) at (-.5,  1);
  \draw[->] (m) .. controls +(0,0) and + (0, -0.5) .. (tl) node[pos =
  1, above] {$a$} coordinate[pos = 0.7] (ga);
  \draw[->] (m) .. controls +(0,0) and + (0, -0.5) .. (tr) node[pos =
  1, above] {$b$} coordinate[pos = 0.7] (gb);
  \draw[>-] (b) -- (m) node[pos = 0, below] {$a+b$} coordinate[pos = 0.3] (gc);
   \filldraw[draw= green!50!black, fill = white] (gc) circle (1mm)
  node[left, green!50!black] {$r$};
\end{scope}}}\right) =
    \soergel\left(\NB{\tikz[font=\tiny]{\begin{scope}
  \coordinate (m) at (  0,  0);
  \coordinate (b) at (  0, -1);
  \coordinate (tr) at (+.5,  1);
  \coordinate (tl) at (-.5,  1);
  \draw[->] (m) .. controls +(0,0) and + (0, -0.5) .. (tl) node[pos =
  1, above] {$a$} coordinate[pos = 0.7] (ga);
  \draw[->] (m) .. controls +(0,0) and + (0, -0.5) .. (tr) node[pos =
  1, above] {$b$} coordinate[pos = 0.7] (gb);
  \draw[>-] (b) -- (m) node[pos = 0, below] {$a+b$};
  \filldraw[draw= green!50!black, fill = white] (ga) circle (1mm)
  node[left, green!50!black] {$r$};
  \filldraw[draw= green!50!black, fill = white] (gb) circle (1mm)
  node[right, green!50!black] {$r$};
\end{scope}}}\right) .
  \]
\end{lem}

\begin{proof}
  This follows directly from the fact that
  \[
    e_1(x_1,\ldots,x_k)+e_1(x_{k+1}, \ldots, x_n) = e_1(x_1,\ldots,
    x_n) .
  \]
\end{proof}
Note that in the previous lemma $\soergel(\Gamma)$ denoted the
$H^\prime$-equivariant Soergel bimodule associated with the
green-dotted MOY graph $\Gamma$. We keep this notation for the rest of
the paper.

\subsection{\texorpdfstring{$H^\prime$}{H'}-equivariant maps between Soergel bimodules}
\label{sec:pdgmaps}
We will make frequent use of the morphisms of $H^\prime$-equivariant Soergel bimodules
that we define next.
\begin{lem}
  \label{lem:pdgmaps}
  \begin{enumerate}
  \item\label{it:pdg-maps-MV} The maps
    \[
      \begin{array}{crcl}
        \mapA \colon\thinspace &\soergel\left(  \NB{\tikz[font=\tiny,
                                 scale =0.7]{\begin{scope}
  \coordinate (b) at ( 0,-0.50);
  \coordinate (m) at (0,0);
  \coordinate (tr) at (+1,  1);
  \coordinate (tm) at (0,  1);
  \coordinate (tl) at (-1,  1);
  \draw[>-, >= to] (b) -- (m) node[pos=0, below] {$a+b+c$};
  \draw[-to] (m) .. controls +(0,0) and + (0, -0.5) .. (tl) node[pos =
  1, above] {$a$} coordinate[pos= 0.5] (ml) coordinate[pos = 0.25] (lml);
  \draw[-to] (m) .. controls +(0,0) and + (0, -0.5) .. (tr) node[pos =
  1, above] {$c$};
  \draw[-to] (ml)  .. controls +(0,0) and + (0, -0.5) .. (tm) node[pos =
  1, above] {$b$};
  \node[left] (lml) {$a+b$};
\end{scope}}} \right) &
                                                                      \to & \soergel\left(  \NB{\tikz[font=\tiny, scale =0.7]{\begin{scope}
  \coordinate (b) at ( 0,-0.50);
  \coordinate (m) at (0,0);
  \coordinate (tr) at (+1,  1);
  \coordinate (tm) at (0,  1);
  \coordinate (tl) at (-1,  1);
  \draw[>-, >= to] (b) -- (m) node[pos=0, below] {$a+b+c$};
  \draw[-to] (m) .. controls +(0,0) and + (0, -0.5) .. (tl) node[pos =
  1, above] {$a$};
  \draw[-to] (m) .. controls +(0,0) and + (0, -0.5) .. (tr) node[pos =
  1, above] {$c$}  coordinate[pos= 0.5] (mr) coordinate[pos = 0.25] (lmr);
  \draw[-to] (mr)  .. controls +(0,0) and + (0, -0.5) .. (tm) node[pos =
  1, above] {$b$};
  \node[right] (lmr) {$b+c$};
\end{scope}}} \right)  \\
                               &1  & \mapsto & 1
      \end{array}
    \]
    and
    \[
      \begin{array}{crcl}
        \mapA \colon\thinspace &\soergel\left( \NB{\tikz[font=\tiny, scale =0.7]{\begin{scope}
  \coordinate (t) at ( 0,0.50);
  \coordinate (m) at (0,0);
  \coordinate (br) at (+1,  -1);
  \coordinate (bm) at ( 0,  -1);
  \coordinate (bl) at (-1,  -1);
  \draw[-to] (m) -- (t) node[pos=1, above] {$a+b+c$};
  \draw[>-, >=to] (bl) .. controls +(0,0.5) and + (0, 0) .. (m) node[pos =
  0, below] {$a$} coordinate[pos= 0.5] (ml) coordinate[pos = 0.75] (lml);
  \draw[>-, >=to] (br) .. controls +(0,0.5) and + (0, -0) .. (m) node[pos =
  0, below] {$c$};
  \draw[>-, >=to] (bm)  .. controls +(0,0.5) and + (0, 0) .. (ml) node[pos =
  0, below] {$b$};
  \node[left] (lml) {$a+b$};
\end{scope}}} \right) & \to &
                                                                                                \soergel\left(  \NB{\tikz[font=\tiny, scale =0.7]{\begin{scope}
  \coordinate (t) at ( 0,0.50);
  \coordinate (m) at (0,0);
  \coordinate (br) at (+1,  -1);
  \coordinate (bm) at ( 0,  -1);
  \coordinate (bl) at (-1,  -1);
  \draw[-to] (m) -- (t) node[pos=1, above] {$a+b+c$};
  \draw[>-, >=to] (bl) .. controls +(0,0.5) and + (0, 0) .. (m) node[pos =
  0, below] {$a$};
  \draw[>-, >=to] (br) .. controls +(0,0.5) and + (0, -0) .. (m) node[pos =
  0, below] {$c$}  coordinate[pos= 0.5] (mr) coordinate[pos = 0.75] (lmr);
  \draw[>-, >=to] (bm)  .. controls +(0,0.5) and + (0, 0) .. (mr) node[pos =
  0, below] {$b$};
  \node[right] (lmr) {$b+c$};
\end{scope}}} \right)  \\
                               & 1 & \mapsto & 1
      \end{array}
    \]
    are isomorphisms of $H^\prime$-equivariant bimodules. Their inverses are also
    denoted by $\mapA$.
  \item\label{it:pdg-maps-splits} The maps
    \[
      \begin{array}{crcl}
        \mapH \colon\thinspace &\soergel\left(  \NB{\tikz[font=\tiny, scale =0.7]{\begin{scope}
  \coordinate (bl) at (-0.5, -1);
  \coordinate (br) at ( 0.5, -1);
  \coordinate (bm) at (  0,-0.3);
  \coordinate (tl) at (-0.5,  1);
  \coordinate (tr) at ( 0.5,  1);
  \coordinate (tm) at (  0, 0.3);
  \draw[>-]  (bl) .. controls +( 0, 0.5) and +(0,0) .. (bm)
  node[below, pos = 0] {$a$};
  \draw[>-]  (br) .. controls +( 0, 0.5) and +(0,0) .. (bm)
  node[below, pos = 0] {$b+c$};
  \draw[<-]  (tl) .. controls +( 0, -0.5) and +(0,0) .. (tm)
  node[above, pos = 0] {$a+c$};
  \draw[<-]  (tr) .. controls +( 0, -0.5) and +(0,0) .. (tm)
  node[above, pos = 0] {$b$};
  \draw [->-] (bm) -- (tm) node[left, pos = 0.5] {$a+b+c$};
\end{scope}}} \right) & \to & \soergel\left(  \NB{\tikz[font=\tiny, scale =0.7]{\begin{scope}
  \coordinate (bl) at (-0.5, -1);
  \coordinate (br) at ( 0.5, -1);
  \coordinate (tl) at (-0.5,  1);
  \coordinate (tr) at ( 0.5,  1);
  \draw[>->] (bl) -- (tl) node[pos = 0, below] {$a$} node[pos = 1,
  above] {$a+c$} coordinate[pos = 0.6] (ml);
  \draw[>->] (br) -- (tr) node[pos = 0, below] {$b+c$} node[pos = 1, above] {$b$} coordinate[pos = 0.4] (mr);
  \draw[->-] (mr) -- (ml) node [pos= 0.5, above] {$c$};
\end{scope}}} \right)  \\
                               &1  & \mapsto & 1
      \end{array}
    \]
    and
    \[
      \begin{array}{crcl}
        \mapH \colon\thinspace &\soergel\left(  \NB{\tikz[font=\tiny, scale =0.7]{\begin{scope}
  \coordinate (bl) at (-0.5, -1);
  \coordinate (br) at ( 0.5, -1);
  \coordinate (bm) at (  0,-0.3);
  \coordinate (tl) at (-0.5,  1);
  \coordinate (tr) at ( 0.5,  1);
  \coordinate (tm) at (  0, 0.3);
  \draw[>-]  (bl) .. controls +( 0, 0.5) and +(0,0) .. (bm)
  node[below, pos = 0] {$a+c$};
  \draw[>-]  (br) .. controls +( 0, 0.5) and +(0,0) .. (bm)
  node[below, pos = 0] {$b$};
  \draw[<-]  (tl) .. controls +( 0, -0.5) and +(0,0) .. (tm)
  node[above, pos = 0] {$a$};
  \draw[<-]  (tr) .. controls +( 0, -0.5) and +(0,0) .. (tm)
  node[above, pos = 0] {$b+c$};
  \draw [->-] (bm) -- (tm) node[left, pos = 0.5] {$a+b+c$};
\end{scope}}} \right) & \to &
                                                                                                   \soergel\left(  \NB{\tikz[font=\tiny, scale =0.7]{\begin{scope}
  \coordinate (bl) at (-0.5, -1);
  \coordinate (br) at ( 0.5, -1);
  \coordinate (tl) at (-0.5,  1);
  \coordinate (tr) at ( 0.5,  1);
  \draw[>->] (bl) -- (tl) node[pos = 0, below] {$a+c$} node[pos = 1,
  above] {$a$} coordinate[pos = 0.4] (ml);
  \draw[>->] (br) -- (tr) node[pos = 0, below] {$b$} node[pos = 1, above] {$b+c$} coordinate[pos = 0.6] (mr);
  \draw[->-] (ml) -- (mr) node [pos= 0.5, above] {$c$};
\end{scope}}} \right)  \\
                               & 1 & \mapsto & 1
      \end{array}
    \]
    are morphisms of $H^\prime$-equivariant bimodules.
  \item\label{it:pdg-maps-merges} The maps
    \[
      \begin{array}{crcl}
        \mapX \colon\thinspace &\soergel\left(  \NB{\tikz[font=\tiny, scale =0.7]{}} \right) & \to & \soergel\left(  \NB{\tikz[font=\tiny, scale =0.7]{\begin{scope}
  \coordinate (bl) at (-0.5, -1);
  \coordinate (br) at ( 0.5, -1);
  \coordinate (bm) at (  0,-0.3);
  \coordinate (tl) at (-0.5,  1);
  \coordinate (tr) at ( 0.5,  1);
  \coordinate (tm) at (  0, 0.3);
  \draw[>-]  (bl) .. controls +( 0, 0.5) and +(0,0) .. (bm)
  node[below, pos = 0] {$a$} coordinate[pos =0.3] (ga);
  \draw[>-]  (br) .. controls +( 0, 0.5) and +(0,0) .. (bm)
  node[below, pos = 0] {$b+c$};
  \draw[<-]  (tl) .. controls +( 0, -0.5) and +(0,0) .. (tm)
  node[above, pos = 0] {$a+c$};
  \draw[<-]  (tr) .. controls +( 0, -0.5) and +(0,0) .. (tm)
  node[above, pos = 0] {$b$} coordinate[pos =0.3] (gb);
  \draw [->-] (bm) -- (tm) node[left, pos = 0.5] {$a+b+c$};
  \filldraw[draw= green!50!black, fill = white] (gb) circle (1mm)
  node[right, green!50!black] {$-a$};
  \filldraw[draw= green!50!black, fill = white] (ga) circle (1mm)
  node[left, green!50!black] {$-b$};
\end{scope}}} \right)  \\
                               &1  & \mapsto & \displaystyle{\prod_{i=1}^a \prod_{j=1}^b(x_{a+c+j}\otimes 1 -
                                               1\otimes x_{i})}
      \end{array}
    \]
    and
    \[
      \begin{array}{crcl}
        \mapX \colon\thinspace &\soergel\left(  \NB{\tikz[font=\tiny, scale =0.7]{}} \right) & \to &
                                                                                              \soergel\left(  \NB{\tikz[font=\tiny, scale =0.7]{\begin{scope}
  \coordinate (bl) at (-0.5, -1);
  \coordinate (br) at ( 0.5, -1);
  \coordinate (bm) at (  0,-0.3);
  \coordinate (tl) at (-0.5,  1);
  \coordinate (tr) at ( 0.5,  1);
  \coordinate (tm) at (  0, 0.3);
  \draw[>-]  (bl) .. controls +( 0, 0.5) and +(0,0) .. (bm)
  node[below, pos = 0] {$a+c$};
  \draw[>-]  (br) .. controls +( 0, 0.5) and +(0,0) .. (bm)
  node[below, pos = 0] {$b$} coordinate[pos =0.3] (gb);
  \draw[<-]  (tl) .. controls +( 0, -0.5) and +(0,0) .. (tm)
  node[above, pos = 0] {$a$} coordinate[pos =0.3] (ga);
  \draw[<-]  (tr) .. controls +( 0, -0.5) and +(0,0) .. (tm)
  node[above, pos = 0] {$b+c$};
  \draw [->-] (bm) -- (tm) node[left, pos = 0.5] {$a+b+c$};
  \filldraw[draw= green!50!black, fill = white] (gb) circle (1mm)
  node[right, green!50!black] {$-a$};
  \filldraw[draw= green!50!black, fill = white] (ga) circle (1mm)
  node[left, green!50!black] {$-b$};
\end{scope}}} \right)  \\
                               & 1 & \mapsto & \displaystyle{\prod_{i=1}^a \prod_{j=1}^b(1\otimes x_{a+c+j} -
                                               x_{i}\otimes 1 )}
      \end{array}
    \]
    are morphisms of $H^\prime$-equivariant bimodules.
  \item\label{it:pdg-maps-merges-special} Let $a_1, a_2, b_1$ and
    $b_2$ be four integers such that $a_1 + a_2 = -a$ and
    $b_1 + b_2 = -b$. Then the map
    \[
      \begin{array}{crcl}
        \mapX \colon\thinspace &\soergel\left(  \NB{\tikz[font=\tiny, scale =0.7]{\begin{scope}
  \coordinate (bl) at (-0.5, -1);
  \coordinate (br) at ( 0.5, -1);
  \coordinate (tl) at (-0.5,  1);
  \coordinate (tr) at ( 0.5,  1);
  \draw[>->] (bl) -- (tl) node[pos = 0, below] {$a$} node[pos = 1, above] {$a$};
  \draw[>->] (br) -- (tr) node[pos = 0, below] {$b$} node[pos = 1, above] {$b$};
\end{scope}}} \right) & \to & \soergel\left(  \NB{\tikz[font=\tiny, scale =0.7]{\begin{scope}
  \coordinate (bl) at (-0.5, -1);
  \coordinate (br) at ( 0.5, -1);
  \coordinate (bm) at (  0,-0.3);
  \coordinate (tl) at (-0.5,  1);
  \coordinate (tr) at ( 0.5,  1);
  \coordinate (tm) at (  0, 0.3);
  \draw[>-]  (bl) .. controls +( 0, 0.5) and +(0,0) .. (bm)
  node[below, pos = 0] {$a$} coordinate[pos =0.3] (ga1);
  \draw[>-]  (br) .. controls +( 0, 0.5) and +(0,0) .. (bm)
  node[below, pos = 0] {$b$} coordinate[pos =0.3] (gb1);
  \draw[<-]  (tl) .. controls +( 0, -0.5) and +(0,0) .. (tm)
  node[above, pos = 0] {$a$} coordinate[pos =0.3] (ga2);
  \draw[<-]  (tr) .. controls +( 0, -0.5) and +(0,0) .. (tm)
  node[above, pos = 0] {$b$} coordinate[pos =0.3] (gb2);
  \draw [->-] (bm) -- (tm) node[left, pos = 0.5] {$a+b$};
  \filldraw[draw= green!50!black, fill = white] (gb1) circle (1mm)
  node[right, green!50!black] {$a_1$};
  \filldraw[draw= green!50!black, fill = white] (gb2) circle (1mm)
  node[right, green!50!black] {$a_2$};
  \filldraw[draw= green!50!black, fill = white] (ga1) circle (1mm)
  node[left, green!50!black] {$b_1$};
  \filldraw[draw= green!50!black, fill = white] (ga2) circle (1mm)
  node[left, green!50!black] {$b_2$};
\end{scope}}} \right)  \\
                               &1  & \mapsto & \displaystyle{\prod_{i=1}^a \prod_{j=1}^b(x_{a+j}\otimes 1 -
                                               1\otimes x_{i})} \\ &&& \displaystyle{= \prod_{i=1}^a \prod_{j=1}^b(1\otimes x_{a+j} -
                                                                       x_{i}\otimes 1 ) }
      \end{array}
    \]
    is a morphism of $H^\prime$-equivariant bimodules.
  \item\label{it:pdg-maps-digons} The maps
    \[
      \begin{array}{crcl}
        \mapB \colon\thinspace &\soergel\left(  \NB{\tikz[font=\tiny, scale =0.7]{\begin{scope}
  \coordinate (bm) at ( 0, -1);
  \coordinate (tm) at ( 0, 1);
  \draw[>->] (bm) -- (tm) node[above, pos =1] {$a+b$} node[below, pos =0] {$a+b$};
\end{scope}}} \right) & \to & \soergel\left(  \NB{\tikz[font=\tiny, scale =0.7]{\begin{scope}
  \coordinate (bm) at ( 0, -1);
  \coordinate (cm) at ( 0, -0.6);
  \coordinate (sm) at ( 0,  0.6);
  \coordinate (tm) at ( 0, 1);
  \draw[->] (sm) -- (tm) node[above, pos =1] {$a+b$};
  \draw[>-] (bm) -- (cm) node[below, pos =0] {$a+b$};
  \draw[->-] (cm) .. controls + ( 0.6, 0.6) and + ( 0.6, -0.6) .. (sm)
  node [pos = 0.5, right] {$b$};
  \draw[->-] (cm) .. controls + (-0.6, 0.6) and + (-0.6, -0.6) .. (sm)
  node [pos = 0.5,  left] {$a$};
\end{scope}}} \right)  \\
                               &1  & \mapsto & 1
      \end{array}
    \]
    and
    \[
      \begin{array}{crcl}
        \mapU \colon\thinspace
        & \soergel\left(\NB{\tikz[font=\tiny, scale =0.7]{\begin{scope}
  \coordinate (bm) at ( 0, -1);
  \coordinate (cm) at ( 0, -0.6);
  \coordinate (sm) at ( 0,  0.6);
  \coordinate (tm) at ( 0, 1);
  \draw[->] (sm) -- (tm) node[above, pos =1] {$a+b$};
  \draw[>-] (bm) -- (cm) node[below, pos =0] {$a+b$};
  \draw[->-] (cm) .. controls + ( 0.6, 0.6) and + ( 0.6, -0.6) .. (sm)
  node [pos = 0.5, right] {$b$} coordinate[pos = 0.8] (gb);
  \draw[->-] (cm) .. controls + (-0.6, 0.6) and + (-0.6, -0.6) .. (sm)
  node [pos = 0.5,  left] {$a$} coordinate[pos = 0.2] (ga);
  \filldraw[draw= green!50!black, fill = white] (ga) circle (1mm)
  node[below, green!50!black] {$-b$};
  \filldraw[draw= green!50!black, fill = white] (gb) circle (1mm)
  node[above,  green!50!black] {$-a$};
\end{scope}}}\right)
        & \to 
        & \soergel\left(  \NB{\tikz[font=\tiny, scale =0.7]{}} \right)  \\
        & P \otimes Q
        & \mapsto
        & \displaystyle{\sum_{\substack{I\sqcup J = \{1,\dots, a+b\}\\ |I| = a \\
        |J| =b}} \frac{P(x_{I})Q(x_J)}{\prod_{j \in J} \prod_{i \in
        I}(x_j - x_i)}}
      \end{array}
    \]
    are morphisms of $H^\prime$-equivariant bimodules.
  \item\label{it:pdg-maps-mult-e} The map
    \[
      \begin{array}{crcl}
        \mapE^i \colon\thinspace &\soergel\left(  \NB{\tikz[font=\tiny, scale =0.7]{\begin{scope}
  \coordinate (bm) at ( 0, -1);
  \coordinate (tm) at ( 0, 1);
  \draw[>->] (bm) -- (tm) node[above, pos =1] {$a$} node[below, pos =0] {$a$};
\end{scope}}} \right) & \to & \soergel\left(  \NB{\tikz[font=\tiny, scale =0.7]{\begin{scope}
  \coordinate (bm) at ( 0, -1);
  \coordinate (tm) at ( 0, 1);
  \draw[>->] (bm) -- (tm) node[above, pos =1] {$a$} node[below, pos
  =0] {$a$} coordinate[pos = 0.5] (g);
  \filldraw[draw= green!50!black, fill = white] (g) circle (1mm)
  node[left, green!50!black] {$-i$};
\end{scope}}} \right)  \\
                                 &1  & \mapsto & e_a^i
      \end{array}
    \]
    is a morphism of $H^\prime$-equivariant bimodules.
  \item\label{it:pdg-maps-mults-delta} The maps
    \[
      \begin{array}{crcl}
        \mapD \colon\thinspace &
                                 \soergel\left(\NB{\tikz[font=\tiny, scale =0.7]{}} \right) & \to &
                                                                                                  \soergel\left(\NB{\tikz[font=\tiny, scale =0.7]{\begin{scope}
  \coordinate (m) at (  0,  0);
  \coordinate (t) at (  0, 1);
  \coordinate (br) at (+.5,  -1);
  \coordinate (bl) at (-.5,  -1);
  \draw[>-] (bl) .. controls +(0,0.5) and + (0, 0) .. (m) node[pos =
  0, below] {$a$} coordinate[pos = 0.3] (ga);
  \draw[>-] (br) .. controls +(0,0.5) and + (0, 0) .. (m) node[pos =
  0, below] {$b$} coordinate[pos = 0.3] (gb);
  \draw[->] (m) -- (t) node[pos = 1, above] {$a+b$};
  \filldraw[draw= green!50!black, fill = white] (ga) circle (1mm)
  node[left, green!50!black] {$-b$};
  \filldraw[draw= green!50!black, fill = white] (gb) circle (1mm)
  node[right, green!50!black] {$-a$};
\end{scope}}} \right)  \\
                               &1  & \mapsto & \displaystyle{
                                               \prod_{i=1}^a
                                               \prod_{j=1}^b (1
                                               \otimes x_{a+j} - 1
                                               \otimes x_i)        }
      \end{array}
    \]
    and
    \[
      \begin{array}{crcl}
        \mapN \colon\thinspace & \soergel\left(\NB{\tikz[font=\tiny, scale =0.7]{\begin{scope}
  \coordinate (m) at (  0,  0);
  \coordinate (b) at (  0, -1);
  \coordinate (tr) at (+.5,  1);
  \coordinate (tl) at (-.5,  1);
  \draw[->] (m) .. controls +(0,0) and + (0, -0.5) .. (tl) node[pos =
  1, above] {$a$};
  \draw[->] (m) .. controls +(0,0) and + (0, -0.5) .. (tr) node[pos =
  1, above] {$b$};
  \draw[>-] (b) -- (m) node[pos = 0, below] {$a+b$};
\end{scope}}} \right) & \to     & \soergel\left(\NB{\tikz[font=\tiny, scale =0.7]{\begin{scope}
  \coordinate (m) at (  0,  0);
  \coordinate (b) at (  0, -1);
  \coordinate (tr) at (+.5,  1);
  \coordinate (tl) at (-.5,  1);
  \draw[->] (m) .. controls +(0,0) and + (0, -0.5) .. (tl) node[pos =
  1, above] {$a$} coordinate[pos = 0.7] (ga);
  \draw[->] (m) .. controls +(0,0) and + (0, -0.5) .. (tr) node[pos =
  1, above] {$b$} coordinate[pos = 0.7] (gb);
  \draw[>-] (b) -- (m) node[pos = 0, below] {$a+b$};
  \filldraw[draw= green!50!black, fill = white] (ga) circle (1mm)
  node[left, green!50!black] {$-b$};
  \filldraw[draw= green!50!black, fill = white] (gb) circle (1mm)
  node[right, green!50!black] {$-a$};
\end{scope}}} \right) \\
                               & 1                                                        & \mapsto & \displaystyle{\prod_{i=1}^a \prod_{j=1}^b (x_{a+j} \otimes 1 - x_i \otimes 1)}
      \end{array}
    \]
    are morphisms of $H^\prime$-equivariant bimodules.
  \end{enumerate}
\end{lem}

\begin{proof}

  The fact that these maps are indeed bimodule maps is not completely
  trivial but classical. We leave this to the reader and we focus on
  proving that these maps respect the $H^\prime$-module structures.  Items
  \ref{it:pdg-maps-MV} and \ref{it:pdg-maps-splits} are trivial.

  The argument for the two maps of item \ref{it:pdg-maps-merges} are
  similar (in fact the maps are intertwined by reordering
  variables). We only deal with the first one.
  Denote by $M$ and $M'$ the Soergel bimodules which are respectively
  the source and the target of the map. Since $\mapX$ is a bimodule
  map, it is enough to show that
  $\dif_{M'} \circ \mapX (1) = \mapX \circ \dif_M(1)$, in other words
  that $\dif_{M'} \circ \mapX (1) =0$.
  \begin{align*}
    \dif_{M'} \circ \mapX (1) &= \dif_{M'} \left( \prod_{i=1}^a \prod_{j=1}^b(x_{a+c+j}\otimes 1 -
                                1\otimes x_{i})\right) \\
                              &= \dif \left( \prod_{i=1}^a \prod_{j=1}^b(x_{a+c+j}\otimes 1 -
                                1\otimes x_{i}) \right) +
                                \dif_{M'}(1)\left(\prod_{i=1}^a \prod_{j=1}^b(x_{a+c+j}\otimes 1 -
                                1\otimes x_{i}) \right) \\
                              &= %
                                \left(a e_1(x_{a+c+1}, \dots, x_{a+c+b})\otimes 1 + b\otimes e_1(x_1, \dots, x_a) \right)
                                \left(
                                \prod_{i=1}^a \prod_{j=1}^b(x_{a+c+j}\otimes 1 -1\otimes x_{i}) 
                                \right) \\
                              &\qquad
                                + \left(1 \otimes (-be_1(x_1, \dots x_a)) -a e_1(x_{a+c+1}, \dots, x_{a+c+b}) \otimes 1 \right)
                                \left(
                                \prod_{i=1}^a \prod_{j=1}^b(x_{a+c+j}\otimes 1 -1\otimes x_{i})
                                \right) \\
                              &= 0,
  \end{align*}
  where the third identity comes from Lemma
  ~\ref{lem:dif-delta-nabla}.

  Item ~\ref{it:pdg-maps-merges-special} is a straightforward
  consequence of item~\ref{it:pdg-maps-merges} after further twisting
  the first tensor factor by $\gamma e_1$ (a multiple of the first
  elementary symmetric function in all of the $a+b$ variables) and further
  twisting the second factor by $-\gamma e_1$.

  For item~\ref{it:pdg-maps-digons}, compatibility of the map $\mapB$
  with $H^\prime$ is straightforward to show. We will now
  consider the map $\mapU$.  Denote by $M$ and $M'$ the Soergel
  bimodules which are respectively the source and the target of the
  map. We want to prove that
  $\dif_{M'} \circ \mapU (P\otimes Q) = \mapU \circ \dif_M(P\otimes
  Q)$.

  We first remark, that even though the expression $\mapU(P\otimes Q)$
  is written as a rational function, it is indeed a polynomial.  There
  is no difficulty in extending $\dif$ to rational fractions.  We
  compute:
  \begin{align*}
    \dif_{M'}\circ \mapU (P \otimes Q)
    &= \dif\left( \sum_{\substack{I\sqcup J = \{1,\dots, a+b\}\\ |I| = a,~
    |J| =b}} \frac{P(x_{I})Q(x_J)}{\prod_{j \in J} \prod_{i \in
    I}(x_j - x_i)} \right)
    \\
    &= \sum_{\substack{I\sqcup J = \{1,\dots, a+b\}\\ |I| = a,~
    |J| =b}} \dif \left(\frac{P(x_{I})Q(x_J)}{\prod_{j \in J} \prod_{i \in
    I}(x_j - x_i)} \right) \\
    &= \sum_{\substack{I\sqcup J = \{1,\dots, a+b\}\\ |I| = a,~
    |J| =b}}
    \frac{\dif\left(
    P(x_{I})Q(x_J)\right)
    \prod_{j \in J} \prod_{i \in I}(x_j - x_i)
    - P(x_{I})Q(x_J) \dif \left(\prod_{j \in J} \prod_{i \in
    I}(x_j - x_i) \right)
    }
    {
    \left(\prod_{j \in J} \prod_{i \in
    I}(x_j - x_i)\right)^2} \\
    &= \sum_{\substack{I\sqcup J = \{1,\dots, a+b\}\\ |I| = a,~
    |J| =b}}
    \frac{\dif\left(
    P(x_{I})Q(x_J)\right) 
    - P(x_{I})Q(x_J) (be_1(x_I) + ae_1(x_J)) 
    }
    {
    \left(\prod_{j \in J} \prod_{i \in
    I}(x_j - x_i)\right)}  \\
    &= \mapU(\dif_M(P\otimes Q)).
  \end{align*}
  
  Item~\ref{it:pdg-maps-mult-e} is
  straightforward. Item~\ref{it:pdg-maps-mults-delta} follows from
  computations similar to that for \ref{it:pdg-maps-merges} (and are
  in fact easier).
\end{proof}

\begin{rem}
  \label{rmk:nopdgmaps}
  If one forgets the $H^\prime$-module structures, all maps given in
  Lemma~\ref{lem:pdgmaps} are morphisms between Soergel
  bimodules. However the list given is not enough to generate all the
  morphism spaces. To fix this, one should generalize the map in
  item~\ref{it:pdg-maps-mult-e} to multiplication by an arbitrary
  homogeneous symmetric polynomial. However such a map, does not
  respect in general the $H^\prime$-module structure. In particular, the usual
  isomorphisms known as categorified MOY calculus (see
  \cite{MSV2} for instance) do not have analogues in general in the category of $H^\prime$-equivariant Soergel
  bimodules.
\end{rem}

\subsection{The uncolored Koszul complex and Cautis differential}
In this subsection, we let $A=A_n=\myZ[x_1,\dots, x_n]$. First we
consider the $H^\prime$-equivariant Koszul complex in one-variable:
\begin{equation}\label{eqn-ordinary-Koszul-with-Cautis-d}
  \K_1:  0 \lra a q^2 \myZ[x]^x\otimes \myZ[x]^{x} \xrightarrow{d_\K=x\otimes 1 -1\otimes x } \myZ[x]\otimes \myZ[x]\lra 0\ .
\end{equation}
Define on $\K_1$ the graded bimodule map $d_C$ introduced by Cautis
(see \cite[Section 6]{Cautisremarks} and \cite[Section 5.2.2]{RW}),
which we will call the \emph{Cautis differential} and is given by
\begin{equation}
  d_C(f)=(x^{2}\otimes 1) f
\end{equation} 
for any $f\in aq^2 \myZ[x]^x\otimes \myZ[x]^{x}$ and is zero on
$\myZ[x]\otimes \myZ[x]$. By construction, $d_C$ is an endomorphsim
of the Koszul complex $\K_1$ of bidegree $(-1,2)$.

\begin{lem} \label{lemma-acylicity-commutator-relation} The commutator
  of the endomorphisms $d_C$ and %
$\dif$ is null-homotopic on the Koszul complex $\K_1$.
\end{lem}
\begin{proof}
  The commutator $[\dif, d_C]$ is only nonzero on
  $aq^2\myZ[x]^x\otimes \myZ[x]^x$. Let $1^x\otimes 1^x$ be the
  generator of $aq^2\myZ[x]^x\otimes \myZ[x]^x$ and $1\otimes 1$ be
  the generator of $\myZ[x]\otimes \myZ[x]$. Then we have
  \begin{eqnarray}
    [\dif,d_C]( 1
    \otimes 1) & = & \dif(d_C( 1^x \otimes 1^x))-d_C \dif(1^x\otimes 1^x)  =  2x^31 \otimes 1- d_C(x1^x\otimes 1^x+1^x\otimes x1^x) \nonumber \\
               & = & x^3 1\otimes 1-x^21\otimes x1 = (x^2\otimes 1)d_{\K}(1^x\otimes 1^x).
  \end{eqnarray}
  Thus the homotopy can be taken to be
  \begin{equation}
    h: \myZ[x]\otimes \myZ[x] \lra aq^2\myZ[x]^x\otimes \myZ[x]^x, \quad \quad 1\otimes 1\mapsto x^21^x\otimes 1^x,
  \end{equation}
  and zero everywhere else. The lemma follows.
\end{proof}

The Koszul complex $\K_n=\K_1^{\otimes n}$ inherits the endomorphism
$d_C$ by forming the $n$-fold tensor product from the one-variable
case. It follows that, for a given $H^\prime$-equivariant bimodule $M$ over $A$, there
is an induced differential, still denoted by $d_C$, given via the
identification
\begin{equation}
  \mHH_\bullet^{\dif}(M)\cong \mH_\bullet (M\otimes_{(A, A)} \K_n),
\end{equation}
where the induced differential acts on the right hand side by $\Id_M \otimes
d_C$. By construction, $d_C$ has Hochschild degree $-1$ and $q$-degree
$2$.

\begin{cor}\label{cor-dC-commutes-with-H}
  The induced differential $d_C$ on $\mHH^{\dif}_\bullet(M)$ commutes
  with the $H^\prime$-action.
\end{cor}
\begin{proof}
  By construction, the induced maps arise respectively from taking
  homology of the endomorphisms $d_C \otimes \Id_M$ and
  $\dif_\K \otimes \Id_M +\Id_\K\otimes \dif_M$ on
  $(\K_n\otimes_{(A,A)} M, d_\K \otimes \Id_M)$.  Then
  \begin{eqnarray}
    [d_C \otimes \Id_M,\dif_\K \otimes \Id_M +\Id_\K\otimes \dif_M ] = [d_C,\dif_K]\otimes \Id_M
  \end{eqnarray}
  is null-homotopic as a consequence of the previous lemma.
\end{proof}

\begin{rem}
  The differential, first observed by Cautis \cite{Cautisremarks}, has
  the following more algebro-geometric meaning. Identifying
  $\mHH^1(A_n)$ as vector fields on $\mathrm{Spec}(A_n)=\mathbb{A}^n$,
  $\mHH^1(A_n)$ acts as differential operators on $\mHH_\bullet(M)$
  for any $R_n$-bimodule $M$, regarded as a coherent sheaf on
  $\mathbb{A}^n\times \mathbb{A}^n \cong T^* (\mathbb{A}^n)$. Under
  this identification, $d_C$ is given by, up to scaling by a nonzero
  number, contraction with the vector field
  \[
    \xi_C:=\sum_{i=1}^n x_i^2\frac{\dif}{\dif x_i}.
  \]
\end{rem}

\subsection{Koszul resolution in the colored case}
\label{sec:koszul-resolution-p}

In this subsection, let $k$ be a non-negative integer and consider the
$H^\prime$-equivariant algebra
\[A=A_{(k)}=\myZ[x_1, \dots x_k]^{S_k} = \myZ[e_1, \dots, e_k]\]
endowed with
the derivation $\dif$ (see Lemma
\ref{lem:dif-easy-sym} for the explicit action).  We want to have to
have an $H^\prime$-equivariant Koszul resolution of $A$, so that we can
explicitly compute the relative Hochschild homology of
$H^\prime$-equivariant bimodules over $A$. 

Denote by $V$ the graded $\myZ$-module generated by symbols
$e_1 ,\dots , e_k$ with $\deg_q e_i =2i$. The exterior algebra
$\Lambda^\bullet V$ is bigraded: one grading comes from the
$q$-degree of the $(e_i)_{i=1, \dots, k}$, while the other one, the
Hochschild grading, comes from the decomposition:
$\Lambda^\bullet V= \bigoplus_{i=0}^k \Lambda^i V,$ with $\Lambda^i V$
sitting in Hochschild degree $i$.

Recall that the usual \emph{Koszul resolution} $\K(A)$ of $A$ is the
bigraded projective $(A,A)$-bimodule
$A \otimes \Lambda^\bullet V \otimes A$ endowed with the bimodule
endomorphism $d_\K$ of degree $(0, -1)$ defined by:
\begin{eqnarray}\label{eqn-Koszul-differential}
  d_\K(1  \otimes (e_{i_1} \wedge \dots \wedge e_{i_\ell}) \otimes 1) & : =  &
                                                                               \sum_{j=1}^\ell (-1)^{j-1} \big( e_{i_j} \otimes (e_{i_1}
                                                                               \wedge \dots \wedge \widehat{e_{i_j}} \wedge \dots \wedge e_{i_\ell}) \otimes 1 \nonumber \\ 
                                                                      & & - 1 \otimes (e_{i_1}
                                                                          \wedge \dots \wedge \widehat{e_{i_j}} \wedge \dots \wedge e_{i_\ell}) \otimes e_{i_j} \big). 
\end{eqnarray}

One easily checks that $d_\K\circ d_\K = 0$, so that
$(\koszul(A) := A \otimes \Lambda^\bullet V \otimes A,
d_\K)$ is a chain complex. As a complex of bimodules, it is homotopically
equivalent to $A$ seen as a complex concentrated in Hochschild
degree $0$.

Using the (super)algebra structure on $\Lambda^\bullet V$ and
forgetting about $d_\K$, one can view $\koszul(A)$ as a
(super)algebra. With respect to this algebra structure $d_\K$ becomes
a derivation on $\koszul(A)$, so that $\K(A)$ is a usual (homological)
DG algebra.

On the subalgebra $A \otimes \Lambda^0 V \otimes A \cong A\otimes A$
of $\koszul(A)$, there is a derivation $\dif_{\koszul ,0}$ inherited from the
action of $\dif$ on $A$:
\begin{equation} \label{eq-pdiff-onkoszul} \dif_{\koszul,0}(P \otimes
  1 \otimes Q):= \dif(P) \otimes 1 \otimes 1 + 1 \otimes 1 \otimes
  \dif(Q) \ .
\end{equation}
Since $\koszul(A)$ is a projective resolution of $A$ as an
$(A,A)$-bimodule where $A \otimes \Lambda^0 V \otimes A $ is the
degree zero piece, $\dif_{\koszul,0}$ has an extension
$\dif_{\koszul} $ to the entire resolution $\koszul(A)$. %

More explicitly, by forcing $\partial$ to commute with the Koszul
differential \eqref{eqn-Koszul-differential},
one has the following differential $\dif_{\koszul}$ defined on
generators of $\koszul(A)$ by:
\begin{subequations}
  \begin{equation}
    \dif_{\koszul}(P \otimes 1 \otimes 1) = \dif (P) \otimes 1 \otimes 1, \quad \quad \quad
    \dif_{\koszul}(1 \otimes 1 \otimes Q) =  1\otimes 1 \otimes \dif(Q) ,
  \end{equation}
  \begin{equation}
    \dif_{\koszul}(1 \otimes e_i \otimes 1)=  e_1 \otimes
    e_i \otimes 1 + 1 \otimes e_1 \otimes e_i 
    -(i+1)1 \otimes e_{i+1}
    \otimes 1.
  \end{equation}
\end{subequations}
One can check that the map
$\dif_{\koszul}$ satisfies the Leibniz rule with respect to the
algebra structure on $\koszul(A)$.
Extending via the (super) Leibniz rule to the entire complex, we
have
\begin{align}
  \dif_{\koszul}(P \otimes e_{i_1} \wedge \dots \wedge e_{i_\ell}\otimes
  Q):=& \dif(P) \otimes e_{i_1} \wedge \dots \wedge e_{i_\ell}\otimes
        Q +
        P \otimes e_{i_1} \wedge \dots \wedge e_{i_\ell}\otimes
        \dif(Q)  \nonumber \\
      &+
        \sum_{j =1} ^\ell (i_j+1) (-1)^{j} P \otimes e_{i_{j}+1} \wedge
        e_{i_1} \wedge \dots\wedge \widehat{e_{i_j}} \wedge \dots \wedge e_{i_\ell}
        \otimes Q  \nonumber \\
      &+ \ell \left( 
        P e_1 \otimes e_{i_1} \wedge \dots \wedge e_{i_\ell} \otimes Q \right)  \label{eqn-dif-K} \\
      & +
        \sum_{j =1} ^\ell (-1)^{j+1} P \otimes e_1 \wedge
        e_{i_1} \wedge \dots \wedge \widehat{e_{i_j}}\wedge \dots \wedge e_{i_\ell}
        \otimes Q e_{i_j} \ . \nonumber %
\end{align}

\begin{lem}\label{lem:dif-dK-commute}
  The endomorphisms $\dif_\K$ and $d_\K$ commute on $\K(A)$.
\end{lem}

\begin{proof}
  It suffices to check the lemma on a bimodule generating set of
  $\koszul(A)$ since the commutator of derivations is a derivation.
  On the generator $P\otimes 1 \otimes Q $, one easily checks
  \[ (\dif_\K \circ d_\K)(P\otimes 1 \otimes Q) = 0 = (d_\K \circ
    \dif_\K)(P\otimes 1 \otimes Q).\]

  For the generator $1 \otimes e_i \otimes 1$, where
  $i \in \{1,\dots, k \}$, one computes
  \begin{eqnarray*}
    (d_\K\circ \dif_\K) (1 \otimes e_i \otimes 1)& = &
                                                       d_\K \left(e_1 \otimes e_i \otimes 1 + 1 \otimes
                                                       e_1 \otimes e_i \right) 
                                                       -(i+1)d_\K\left(1 \otimes e_{i+1}
                                                       \otimes 1\right)  \\
                                                 & = & e_1e_i\otimes 1 \otimes 1 - e_1 \otimes 1\otimes
                                                       e_i + e_1\otimes 1 \otimes e_i - 1 \otimes 1\otimes
                                                       e_1e_i  \\
                                                 & & -(i+1) \left(e_{i+1} \otimes 1 \otimes 1  - 1 \otimes 1 \otimes
                                                     e_{i+1}\right) \\
                                                 & = &\left(e_1e_i - (i+1)e_{i+1} \right) \otimes 1 \otimes 1 -
                                                       1 \otimes 1 \otimes \left(e_1e_i - (i+1)e_{i+1} \right) \\
                                                 & = & \dif (e_{i}) \otimes 1 \otimes 1 -
                                                       1 \otimes 1 \otimes \dif(e_i) \\ 
                                                 & = & \dif_\K( e_i \otimes 1 \otimes 1 - 1 \otimes 1 \otimes e_i) \\
                                                 & = & (\dif_\K \circ d_\K)(1 \otimes e_i \otimes 1). 
  \end{eqnarray*}
  The result follows.
\end{proof}

Lemma \ref{lem:dif-dK-commute} and the discussion before it establish the following result.

\begin{cor}
  The Koszul complex $(\K(A),d_\K)$ with respect to the
  $\dif_\K$-action is an $H^\prime$-equivariant bimodule resolution of
  $A$. \hfill$\square$
\end{cor}

This corollary allows us to use the $H^\prime$-equivariant resolution
$(\K(A),\dif_\K)$ to compute the relative Hochschild homology of an
$H^\prime$-equivariant bimodule $(M,\dif_M)$ over $A$. The map
\[
  \dif_{\K}\otimes \Id_M +\Id_{\K}\otimes \dif_M :
  \K(A)\otimes_{(A,A)} M \lra \K(A)\otimes_{(A,A)} M,
\]
upon taking homology, descends to a map on
$\mHH_\bullet^\dif(M)$ which we still denote by $\dif_M$.

Next we recall the Cautis differential $d_C$ on $\K(A)$ (see
\cite[Section 6]{Cautisremarks} and \cite[Section 5.2.2]{RW}).
\begin{align}
  d_C(P \otimes e_{i_1} \wedge \dots \wedge e_{i_\ell}\otimes
  Q):=&
        \sum_{j =1} ^\ell (-1)^{j-1} (e_1e_{i_j}-(i_j+1)e_{i_j+1})P \otimes  
        e_{i_1} \wedge \dots\wedge \widehat{e_{i_j}} \wedge \dots \wedge e_{i_\ell}
        \otimes Q \ .
\end{align}
This is an endomorphism of $\K(A)$ of $q$-degree $2$ and Hochschild
degree $-1$.

\begin{lem}\label{lem-commutator-null-homotopic}
  \begin{enumerate}
  \item The endomorphism $d_C$ commutes with the Koszul differential
    $d_\K$.
  \item The endomorphisms $d_C$ and $\dif_\K$ commute up to homotopy
    on $\K(A)$.
  \end{enumerate}
  
\end{lem}
\begin{proof}
  The first statement is an easy computation and is established in
  \cite[Section 6]{Cautisremarks} and \cite[Appendix B.3]{RW}. Let us verify the second
  one. We compute the effect of the commutator $[\dif_K, d_C ]$ on an
  element of the form $P\otimes e_i\otimes Q$, for $P,Q\in A$. We calculate
  \begin{eqnarray*}
    d_C\circ \dif_\K (P\otimes e_i \otimes Q )& = & d_C\big( \dif(P) \otimes e_i \otimes Q+  P \otimes e_i \otimes \dif(Q) \\
                                              & & + Pe_1\otimes e_i \otimes Q + P \otimes e_1\otimes e_iQ -(i+1) P \otimes e_{i+1} \otimes Q \big) \\
                                              & = & \dif(P)(e_1e_i-(i+1)e_{i+1}) \otimes 1\otimes Q+ P(e_1e_i-(i+1)e_{i+1}) \otimes 1 \otimes \dif(Q) \\
                                              & & + Pe_1(e_1e_i-(i+1)e_{i+1})\otimes 1 \otimes Q  + P(e_1^2-2e_2) \otimes 1 \otimes e_iQ  \\
                                              & & - (i+1) P(e_1e_{i+1}-(i+2)e_{i+2}) \otimes 1 \otimes Q,
  \end{eqnarray*}
  and
  \begin{eqnarray*}
    \dif_\K \circ d_C  (P\otimes e_i \otimes Q )& = & \dif_\K \big( P(e_1e_i - (i+1)e_{i+1})\otimes 1 \otimes Q \big)\\
                                                & = & \dif(P)(e_1e_i-(i+1)e_{i+1}) \otimes 1\otimes Q+ P(e_1^2-2e_{2})e_i \otimes 1 \otimes Q \\
                                                & & + Pe_1(e_1e_i-(i+1)e_{i+1})\otimes 1 \otimes Q  -(i+1) P(e_1e_{i+1}-(i+2)e_{i+2}) \otimes 1 \otimes Q  \\
                                                & & +  P(e_1e_{i}-(i+1)e_{i+1}) \otimes 1 \otimes \dif(Q).
  \end{eqnarray*}
  Thus
  \begin{equation}\label{eqn-null-homotopy-for-commutator-1}
    [\dif_\K ,  d_C]  (P\otimes e_i \otimes Q ) =  P(e_1^2-2e_2)e_i\otimes 1 \otimes Q- P(e_1^2-2e_2)\otimes 1 \otimes e_iQ = (e_1^2-2e_2) d_\K(P\otimes e_i \otimes Q) .
  \end{equation}
  Define a homotopy map $h: \K(A)\lra \K(A)$ by declaring on
  tensor factors, the bimodule map to be the slanted arrow
  \begin{equation}\label{eqn-null-homotopy-for-commutator-2}
      \NB{\tikz[xscale=4.5, yscale =1.5]{
      \node (TL) at (0,1) {$ \myZ[e_i] \otimes e_i \otimes \myZ[e_i]$};
      \node (TR) at (1,1) {$ \myZ[e_i] \otimes 1 \otimes \myZ[e_i]$};
      \node (BL) at (0,0) {$ \myZ[e_i] \otimes e_i \otimes \myZ[e_i]$};
      \node (BR) at (1,0) {$ \myZ[e_i] \otimes 1 \otimes \myZ[e_i]$};
    \draw[-to] (TL) -- (TR) node[pos=0.5, above] {$d_\K$};
    \draw[-to] (BL) -- (BR) node[pos=0.5, below] {$d_\K$};
    \draw[double equal sign distance ] (TL) -- (BL);
    \draw[double equal sign distance]  (TR) -- (BR);
    \draw[-to] (TR) -- (BL) node[left=0.3cm, pos=0.3, scale =0.7] {$(e_1^2-2e_2)\otimes e_i \otimes 1$};
  }}  
  \end{equation}
  and zero everywhere else.  Then
  \eqref{eqn-null-homotopy-for-commutator-1} shows that
  $[\dif_\K, \dif_C]=d_\K h + hd_\K$ is null-homotopic. The lemma
  follows.
\end{proof}

The first part of the lemma implies that $d_C$ descends to the level
of homology.  For ease of notation, we will still denote the induced
endomorphism $\mHH(d_C)$ on $\mHH^\dif_\bullet(M)$ by $d_C$. We record
this observation in the following corollary.

\begin{cor} \label{cor:dCdKcomHH}
  Given an $H^\prime$-equivariant bimodule $M$ over $A$, the induced Cautis
  differential $d_C$ commutes with $\dif_M$ on
  $\mHH_\bullet^\dif(M)$.
\end{cor}
\begin{proof} 
  The proof is entirely analogous to the proof of Corollary
  \ref{cor-dC-commutes-with-H}, using Lemma
  \ref{lem-commutator-null-homotopic} instead.
\end{proof}

\begin{example} \label{ex:1varoverZ}
This example is of particular importance because it will later serve as the link homology of the unknot.
Here we take $M$ to be $A$ (the algebra generated by elementary symmetric polynomials $e_1,\ldots,e_n$) itself. In this case,
  $\mHH_\bullet^\dif(A) \simeq A \otimes \Lambda^\bullet V$ and it has
  a natural bigraded algebra structure. One has:
  \begin{align*}
    d_C(P \otimes 1) &= 0, &&&
                               d_C(1 \otimes e_i) &= (e_1e_i-(i+1)e_{i+1})\otimes 1, \\
    \dif_A (P \otimes 1) &= \dif(P) \otimes 1, &&&
                                                   \dif_A (1 \otimes e_i) &= e_1 \otimes e_i + e_i\otimes e_1 -
                                                                            (i+1) 1 \otimes e_{i+1}.
  \end{align*}
  
We now investigate the differential $d_C$ on $\mHH_\bullet^\dif(A)$ and study $\mHH_0^{\dif}(A) / d_C(\mHH_1^{\dif}(A)) $.

For simplicity, assume $n=2$.  Then in this quotient, there are relations $e_1^2=2e_2$ and $e_1 e_2=0$.
As a $\Z$-module, the space is generated by elements $1,e_1,e_2^b$ where $b \geq 1$.
Since $2 e_2^2=e_1^2 e_2=0$, we get that $e_2^b$ is $2$-torsion for $b \geq 2$. Thus
\[
\mHH_0^{\dif}(A) / d_C(\mHH_1^{\dif}(A)) 
\cong
\Z1 \oplus \Z e_1 \oplus \Z e_2 \oplus \bigoplus_{j=2}^{\infty} \Z_2 e_2^j \ . 
\]

For more general $n$, by \cite[Proposition 5.23]{RW}, the space $(\mHH_\bullet^\dif(A)/d_C(\mHH_{\bullet-1}^\dif(A)))\otimes \Q$ is concentrated in Hochschild degree zero only. A direct computation shows that a $\Q$-basis of this space is given by
$\{1,e_1,e_2,\dots, e_n\}$. It follows that the higher integral homology of $(\mHH_\bullet^\dif(A),d_C)$ is always torsion, while the zeroth homology has free part $\Z 1\oplus \Z e_1 \oplus \dots \oplus \Z e_n$ and an infinite torsion subgroup.
\end{example}

For later constructions, we will need the following. 
\begin{cor}\label{cor-p-nilpotency}
  Let $M$ be a singular Soergel bimodule associated with a
  green-dotted MOY graph $\Gamma$. Then, on the $\F_p$-space $\mHH^\dif_\bullet(M) \otimes \F_p$, we have $\dif_M^p\equiv 0$.
\end{cor}
\begin{proof}
  By construction, $\dif_M$ is the induced operator on homology 
  \[
  \dif_{\K}\otimes \Id_M +\Id_{\K}\otimes \dif_M :
  \K(A)\otimes_{(A,A)} M \lra \K(A)\otimes_{(A,A)} M,
  \]
whose $p$th power is equal to 
  \[
  \left(\dif_{\K}\otimes \Id_M +\Id_{\K}\otimes \dif_M \right)^p=\sum_{i=0}^p \genfrac(){0pt}{0}{p}{i} \dif_{\K}^i \otimes \dif_M^{p-i}.
  \]
Let us analyze the summands of this operator.  There are three cases.

When $i=1,\dots, p-1$, the binomial coefficients $\genfrac(){0pt}{2}{p}{i}$ are all $p$-divisible. Thus $\genfrac(){0pt}{2}{p}{i} \dif_{\K}^i \otimes \dif_M^{p-i}$ acts by zero on $\mHH^\dif_\bullet(M)\otimes \F_p$ for any such $i$. 

Next, for $i=0$, we have, by Lemma \ref{pDG-ready-green}, that the term $\Id_\K \otimes \dif_M^p$ is also $p$-divisible on $\mHH^\dif_\bullet(M)$. It follows that the induced map on homology also vanishes when base changed to $\F_p$.

Finally, for $\dif_\K^p\otimes \Id_M$, observe that $\dif_A^p$ is also $p$-divisible (this can be seen either via a straightforward computation, or as the special case of Lemma \ref{pDG-ready-green} when $M=A$). As $\K(A)$ is an $H^\prime$-equivariant resolution of $A$, $\dif_{\K}^p$ must then be homotopic to a $p$-multiple of an operator on $\K(A)$, which we denote by $D$.
It follows that is $(\dif_{\K}^p-pD)\otimes \Id_M$ is null-homotopic on $\K(A)\otimes_{(A,A)}M$, and thus $\dif_\K^p\otimes \Id_M=pD\otimes \Id_M$ on homology. The result now follows by base changing to $\F_p$. 
\end{proof}

\begin{rem}
 Note that even when tensoring with $\Fp$, the map $\dif_\K$ is still not necessarily $p$-nilpotent. For instance, take $p=3$,
$k=2$ and $A= \Fp[e_1,e_2]$. One could then calculate
\begin{align*}
  \partial_{\koszul}^3(1 \otimes e_2 \otimes 1) =
  e_1\otimes e_1\otimes e_1e_2+ 1\otimes e_1\otimes e_1^2e_2-1\otimes e_1\otimes e_2^2 \neq 0.
\end{align*}
So, in general, $\partial_{\koszul}^p$ is not identically zero on the
Koszul complex tensored by $\Fp$.
\end{rem}

\subsection{Rickard complexes}
\label{sec:rickard-complexes}

\begin{defn}\label{def:rickard}
Let $a$ and $b$ be two positive integers. The \emph{Rickard complexes} $\rickardp{a}{b}$ and $\rickardm{a}{b}$ are two complexes of $H^\prime$-equivariant $(A_{a,b}, A_{b,a})$-bimodules. They have the following diagrammatic definitions. 

\begin{align}\label{eq:rickard-plus}
\rickardp{a}{b}&:=q^{-ab-1} \left( \cdots \longrightarrow
q^{-j} \soergel\left(\NB{\tikz[font=\tiny]{\begin{scope}
  \coordinate (rt) at (+1,   1.2);
  \coordinate (lt) at (-1,   1.2);
  \coordinate (rb) at (+1,  -1.2);
  \coordinate (lb) at (-1,  -1.2);
  \coordinate (r1) at (+1, -0.6);
  \coordinate (l1) at (-1, -0.4);
  \coordinate (r2) at (+1,  0.6);
  \coordinate (l2) at (-1,  0.4);
  \draw[>-] (rb) -- (r1) node[pos=0, below] {$a$};
  \draw[>-] (lb) -- (l1) node[pos=0, below] {$b$}; 
  \draw[->-] (r1) -- (r2) node[pos=0.3, right] {$j$}   coordinate [pos=0.7] (g1);
  \draw[->-] (l1) -- (l2) node[pos=0.5, right] {$a+b-j$};
  \draw[->] (r2) -- (rt) node[pos=1, above] {$b$};
  \draw[->] (l2) -- (lt) node[pos=1, above] {$a$};
  \draw [->-] (r1) -- (l1) node[pos =0.5, below, sloped] {$a-j$};
  \draw [->-] (l2) -- (r2) node[pos =0.5, above, sloped] {$b-j$};
  \filldraw[draw= green!50!black, fill = white] (g1) circle (1mm)
  node[right, green!50!black] {$j-1$};
\end{scope}

}}\right)\longrightarrow
q^{-j-1} \soergel\left(\NB{\tikz[font=\tiny]{\begin{scope}
  \coordinate (rt) at (+1,   1.2);
  \coordinate (lt) at (-1,   1.2);
  \coordinate (rb) at (+1,  -1.2);
  \coordinate (lb) at (-1,  -1.2);
  \coordinate (r1) at (+1, -0.6);
  \coordinate (l1) at (-1, -0.4);
  \coordinate (r2) at (+1,  0.6);
  \coordinate (l2) at (-1,  0.4);
  \draw[>-] (rb) -- (r1) node[pos=0, below] {$a$};
  \draw[>-] (lb) -- (l1) node[pos=0, below] {$b$}; 
  \draw[->-] (r1) -- (r2) node[pos=0.3, right] {$j+1$}   coordinate [pos=0.7] (g1);
  \draw[->-] (l1) -- (l2) node[pos=0.5, right] {$a+b-j-1$};
  \draw[->] (r2) -- (rt) node[pos=1, above] {$b$};
  \draw[->] (l2) -- (lt) node[pos=1, above] {$a$};
  \draw [->-] (r1) -- (l1) node[pos =0.5, below, sloped] {$a-j-1$};
  \draw [->-] (l2) -- (r2) node[pos =0.5, above, sloped] {$b-j-1$};
  \filldraw[draw= green!50!black, fill = white] (g1) circle (1mm)
  node[right, green!50!black] {$j$};
\end{scope}

}}\right)\longrightarrow \cdots \right)
\\ \label{eq:rickard-minus}
  \rickardm{a}{b}&:= q^{ab+1}\left( \cdots \longrightarrow
q^{j+1} \soergel\left(\NB{\tikz[font=\tiny]{\begin{scope}
  \coordinate (rt) at (+1,   1.2);
  \coordinate (lt) at (-1,   1.2);
  \coordinate (rb) at (+1,  -1.2);
  \coordinate (lb) at (-1,  -1.2);
  \coordinate (r1) at (+1, -0.6);
  \coordinate (l1) at (-1, -0.4);
  \coordinate (r2) at (+1,  0.6);
  \coordinate (l2) at (-1,  0.4);
  \draw[>-] (rb) -- (r1) node[pos=0, below] {$a$} coordinate [pos=0.7] (g1);
  \draw[>-] (lb) -- (l1) node[pos=0, below] {$b$} coordinate [pos=0.7] (g2); 
  \draw[->-] (r1) -- (r2) node[pos=0.3, right] {$j+1$}   coordinate [pos=0.7] (g3);
  \draw[->-] (l1) -- (l2) node[pos=0.3, right] {$a+b-j-1$}   coordinate [pos=0.7] (g4);
  \draw[->] (r2) -- (rt) node[pos=1, above] {$b$} coordinate [pos=0.3] (g5);
  \draw[->] (l2) -- (lt) node[pos=1, above] {$a$} coordinate [pos=0.3] (g6);
  \draw [->-] (r1) -- (l1) node[pos =0.6, below, sloped] {$a-j-1$} coordinate [pos=0.3] (g7);
  \draw [->-] (l2) -- (r2) node[pos =0.6, above, sloped] {$b-j-1$} coordinate [pos=0.3] (g8);
  \filldraw[draw= green!50!black, fill = white] (g6) circle (1mm)
  node[left, green!50!black] {$-b$};
  \filldraw[draw= green!50!black, fill = white] (g5) circle (1mm)
  node[right, green!50!black] {$-a$};
  \filldraw[draw= green!50!black, fill = white] (g4) circle (1mm)
  node[left, green!50!black] {$j+1$};
  \filldraw[draw= green!50!black, fill = white] (g3) circle (1mm)
  node[right, green!50!black] {$1$};
  \filldraw[draw= green!50!black, fill = white] (g7) circle (1mm)
  node[above, green!50!black] {$-j-1$};
  \filldraw[draw= green!50!black, fill = white] (g8) circle (1mm)
  node[below, green!50!black] {$-j-1$};
\end{scope}

}}\right)\longrightarrow
q^j  \soergel\left(\NB{\tikz[font=\tiny]{\begin{scope}
  \coordinate (rt) at (+1,   1.2);
  \coordinate (lt) at (-1,   1.2);
  \coordinate (rb) at (+1,  -1.2);
  \coordinate (lb) at (-1,  -1.2);
  \coordinate (r1) at (+1, -0.6);
  \coordinate (l1) at (-1, -0.4);
  \coordinate (r2) at (+1,  0.6);
  \coordinate (l2) at (-1,  0.4);
  \draw[>-] (rb) -- (r1) node[pos=0, below] {$a$} coordinate [pos=0.7] (g1);
  \draw[>-] (lb) -- (l1) node[pos=0, below] {$b$} coordinate [pos=0.7] (g2); 
  \draw[->-] (r1) -- (r2) node[pos=0.3, right] {$j$}   coordinate [pos=0.7] (g3);
  \draw[->-] (l1) -- (l2) node[pos=0.3, right] {$a+b-j$}   coordinate [pos=0.7] (g4);
  \draw[->] (r2) -- (rt) node[pos=1, above] {$b$} coordinate [pos=0.3] (g5);
  \draw[->] (l2) -- (lt) node[pos=1, above] {$a$} coordinate [pos=0.3] (g6);
  \draw [->-] (r1) -- (l1) node[pos =0.6, below, sloped] {$a-j$} coordinate [pos=0.3] (g7);
  \draw [->-] (l2) -- (r2) node[pos =0.6, above, sloped] {$b-j$} coordinate [pos=0.3] (g8);
  \filldraw[draw= green!50!black, fill = white] (g6) circle (1mm)
  node[left, green!50!black] {$-b$};
  \filldraw[draw= green!50!black, fill = white] (g5) circle (1mm)
  node[right, green!50!black] {$-a$};
  \filldraw[draw= green!50!black, fill = white] (g4) circle (1mm)
  node[left, green!50!black] {$j$};
  \filldraw[draw= green!50!black, fill = white] (g3) circle (1mm)
  node[right, green!50!black] {$1$};
  \filldraw[draw= green!50!black, fill = white] (g7) circle (1mm)
  node[above, green!50!black] {$-j$};
  \filldraw[draw= green!50!black, fill = white] (g8) circle (1mm)
  node[below, green!50!black] {$-j$};
\end{scope}

}}\right)\longrightarrow \cdots \right) 
\end{align}
In both cases, $j$ is an integer which runs from $0$ to
$\min(a,b)$. Note that in the first Rickard complex, $j$ increases from
left to right while in the second it decreases. 
For both complexes, the term for $j=0$ is concentrated in cohomological degree zero.
The differentials of
$\rickardp{a}{b}$ and $\rickardm{a}{b}$ are given by a composition of
elementary morphisms described by~(\ref{eq:diff-rickard-plus})
and~(\ref{eq:diff-rickard-minus}) respectively.%
\begin{equation}
  \label{eq:diff-rickard-plus}
  \begin{tikzpicture}[xscale =4, yscale=2.5 ]
    \node (A) at (-1,1) {$ \soergel\left(\NB{\tikz[font=\tiny]{}}\right)$};
    \node (B) at (1,1) {$ q^{-j+a-1} \soergel\left(\NB{\tikz[font=\tiny]{\begin{scope}[font=\tiny]
  \coordinate (rt) at (+1,   1.5);
  \coordinate (lt) at (-1,   1.5);
  \coordinate (rb) at (+1,  -1.5);
  \coordinate (lb) at (-1,  -1.5);
  \coordinate (r1) at (+1, -0.9);
  \coordinate (l1) at (-1, -0.7);
  \coordinate (r2) at (+1,  0.9);
  \coordinate (l2) at (-1,  0.7);
  \draw[>-] (rb) -- (r1) node[pos=0, below] {$a$} coordinate [pos=0.7] (g1);
  \draw[>-] (lb) -- (l1) node[pos=0, below] {$b$} coordinate [pos=0.7] (g2); 
  \draw[->-] (r1) -- (r2) node[pos=0.5, right] {$j$} coordinate  [pos=0.8] (g3) coordinate [pos = 0.2] (d1); 
  \draw[->-] (l1) -- (l2) node[pos=0.6, right] {$a+b-j$}   coordinate [pos=0.7] (g4)  coordinate[pos = 0.257] (d2);
  \draw[->] (r2) -- (rt) node[pos=1, above] {$b$} coordinate [pos=0.3] (g5);
  \draw[->] (l2) -- (lt) node[pos=1, above] {$a$} coordinate [pos=0.3] (g6);
  \draw [->-] (r1) -- (l1) node[pos =0.5, below, sloped] {$a-j-1$} coordinate [pos=0.3] (g7);
  \draw [->-] (l2) -- (r2) node[pos =0.5, above, sloped] {$b-j$} coordinate [pos=0.3] (g8);
  \draw [->-] (d1) -- (d2) node[pos =0.3, above] {$1$} coordinate [pos=0.7] (g9);
  \filldraw[draw= green!50!black, fill = white] (g3) circle (1mm)
  node[right, green!50!black] {$j-1$};
\end{scope}

}}\right)$};
    \node (C) at (-1,-1) {$q^{-j-1} \soergel\left(\NB{\tikz[font=\tiny]{\begin{scope}[font=\tiny]
  \coordinate (rt) at (+1,   1.5);
  \coordinate (lt) at (-1,   1.5);
  \coordinate (rb) at (+1,  -1.5);
  \coordinate (lb) at (-1,  -1.5);
  \coordinate (r1) at (+1, -0.9);
  \coordinate (l1) at (-1, -0.7);
  \coordinate (r2) at (+1,  0.9);
  \coordinate (l2) at (-1,  0.7);
  \draw[>-] (rb) -- (r1) node[pos=0, below] {$a$} coordinate [pos=0.7] (g1);
  \draw[>-] (lb) -- (l1) node[pos=0, below] {$b$} coordinate [pos=0.7] (g2); 
  \draw[->-] (r1) -- (r2) node[pos=0.5, right] {$j$} coordinate  [pos=0.8] (g3) coordinate [pos = 0.2] (d1); 
  \draw[->-] (l1) -- (l2) node[pos=0.5, right] {$a+b-j-1$}   coordinate [pos=0.5] (g4);
  \draw[->] (r2) -- (rt) node[pos=1, above] {$b$} coordinate [pos=0.3] (g5);
  \draw[->] (l2) -- (lt) node[pos=1, above] {$a$} coordinate [pos=0.3] (g6);
  \draw [->-] (r1) -- (l1) node[pos =0.5, below, sloped] {$a-j-1$} coordinate [pos=0.3] (g7);
  \draw [->-] (l2) -- (r2) node[pos =0.4, above, sloped] {$b-j-1$} coordinate [pos=0.3] (g8) coordinate[pos = 0.75] (d2);
  \draw [->-] (d1) -- (d2) node[pos =0.3, left] {$1$} coordinate [pos=0.7] (g9);
  \filldraw[draw= green!50!black, fill = white] (g3) circle (1mm)
  node[right, green!50!black] {$j-1$};
\end{scope}

}}\right)$};
    \node (D) at (1,-1)
    {$q^{-1} \soergel\left(\NB{\tikz[font=\tiny]{}}\right)$};
    \draw[->] (A) -- (B) node[pos =0.5, above] {$\mapA \circ \mapA \circ  \mapB$};
    \draw (B.east) -- +(0.1, 0) arc (90: 0:0.2) -- +(0, -0.6) arc
    (0:-90:0.2) -- (0,0) node [pos =1, above] {$\mapH$};
    \draw[<-] (C.west) -- +(-0.1,0) arc (270:180:0.2) --+(0, +0.6) arc
    (180:90:0.2) -- (0,0);
    \draw[->] (C) -- (D) node[pos =0.5, above] {$\mapU \circ \mapA$};
  \end{tikzpicture}
\end{equation}
\begin{equation}
  \label{eq:diff-rickard-minus}
    \centering
    \begin{tikzpicture}[xscale =4, yscale=2.5 ]
      \node (A) at (-1,1) {$\soergel\left(\NB{\tikz[font=\tiny]{}}\right)$};

      \node (B) at (1, 1) {$q^j \soergel\left(\NB{\tikz[font=\tiny]{\begin{scope}[font=\tiny]
  \coordinate (rt) at (+1,   1.5);
  \coordinate (lt) at (-1,   1.5);
  \coordinate (rb) at (+1,  -1.5);
  \coordinate (lb) at (-1,  -1.5);
  \coordinate (r1) at (+1, -0.9);
  \coordinate (l1) at (-1, -0.7);
  \coordinate (r2) at (+1,  0.9);
  \coordinate (l2) at (-1,  0.7);
  \draw[>-] (rb) -- (r1) node[pos=0, below] {$a$} coordinate [pos=0.7] (g1);
  \draw[>-] (lb) -- (l1) node[pos=0, below] {$b$} coordinate [pos=0.7] (g2); 
  \draw[->-] (r1) -- (r2) node[pos=0.5, right] {$j$} coordinate  [pos=0.8] (g3) coordinate [pos = 0.2] (d1); 
  \draw[->-] (l1) -- (l2) node[pos=0.3, right] {$a+b-j-1$}   coordinate [pos=0.7] (g4);
  \draw[->] (r2) -- (rt) node[pos=1, above] {$b$} coordinate [pos=0.3] (g5);
  \draw[->] (l2) -- (lt) node[pos=1, above] {$a$} coordinate [pos=0.3] (g6);
  \draw [->-] (r1) -- (l1) node[pos =0.6, below, sloped] {$a-j-1$} coordinate [pos=0.3] (g7);
  \draw [->-] (l2) -- (r2) node[pos =0.4, above, sloped] {$b-j-1$} coordinate [pos=0.3] (g8) coordinate[pos = 0.75] (d2);
  \draw [->-] (d1) -- (d2) node[pos =0.3, left] {$1$} coordinate [pos=0.7] (g9);
  \filldraw[draw= green!50!black, fill = white] (g6) circle (1mm)
  node[left, green!50!black] {$-b+j+1$};
  \filldraw[draw= green!50!black, fill = white] (g5) circle (1mm)
  node[right, green!50!black] {$-a$};
  \filldraw[draw= green!50!black, fill = white] (g3) circle (1mm)
  node[right, green!50!black] {$1$};
  \filldraw[draw= green!50!black, fill = white] (g9) circle (1mm)
  node[left, green!50!black] {$1$};
  \filldraw[draw= green!50!black, fill = white] (g7) circle (1mm)
  node[above, green!50!black] {$-j-1$};
\end{scope}

}}\right)$};
      \node (C) at (-1, -1){$q^{j-a} \soergel\left(\NB{\tikz[font=\tiny]{\begin{scope}[font=\tiny]
  \coordinate (rt) at (+1,   1.5);
  \coordinate (lt) at (-1,   1.5);
  \coordinate (rb) at (+1,  -1.5);
  \coordinate (lb) at (-1,  -1.5);
  \coordinate (r1) at (+1, -0.9);
  \coordinate (l1) at (-1, -0.7);
  \coordinate (r2) at (+1,  0.9);
  \coordinate (l2) at (-1,  0.7);
  \draw[>-] (rb) -- (r1) node[pos=0, below] {$a$} coordinate [pos=0.7] (g1);
  \draw[>-] (lb) -- (l1) node[pos=0, below] {$b$} coordinate [pos=0.7] (g2); 
  \draw[->-] (r1) -- (r2) node[pos=0.5, right] {$j$} coordinate  [pos=0.8] (g3) coordinate [pos = 0.2] (d1); 
  \draw[->-] (l1) -- (l2) node[pos=0.6, right] {$a+b-j$}   coordinate [pos=0.7] (g4)  coordinate[pos = 0.257] (d2);
  \draw[->] (r2) -- (rt) node[pos=1, above] {$b$} coordinate [pos=0.3] (g5);
  \draw[->] (l2) -- (lt) node[pos=1, above] {$a$} coordinate [pos=0.3] (g6);
  \draw [->-] (r1) -- (l1) node[pos =0.6, below, sloped] {$a-j-1$} coordinate [pos=0.3] (g7);
  \draw [->-] (l2) -- (r2) node[pos =0.4, above, sloped] {$b-j$} coordinate [pos=0.3] (g8);
  \draw [->-] (d1) -- (d2) node[pos =0.3, above] {$1$} coordinate [pos=0.7] (g9);
  \filldraw[draw= green!50!black, fill = white] (g6) circle (1mm)
  node[left, green!50!black] {$-b+j$};
  \filldraw[draw= green!50!black, fill = white] (g5) circle (1mm)
  node[right, green!50!black] {$-a$};
  \filldraw[draw= green!50!black, fill = white] (g3) circle (1mm)
  node[right, green!50!black] {$1$};
  \filldraw[draw= green!50!black, fill = white] (g9) circle (1mm)
  node[above, green!50!black] {$1-a$};
  \filldraw[draw= green!50!black, fill = white] (g7) circle (1mm)
  node[above, green!50!black] {$-j-1$};
\end{scope}

}}\right)$};
      \node (D) at (1, -1)
      {$q^{-1} \soergel\left(\NB{\tikz[font=\tiny]{}}\right)$};
      \draw[->] (A) -- (B) node[pos =0.5, above] {$\mapA \circ  \mapB$};
      \draw (B.east) -- +(0.1, 0) arc (90: 0:0.2) -- +(0, -0.6) arc
      (0:-90:0.2) -- (0,0) node [pos =1, above] {$\mapX$};
      \draw[<-] (C.west) -- +(-0.1,0) arc (270:180:0.2) --+(0, +0.6) arc
      (180:90:0.2) -- (0,0);
      \draw[->] (C) -- (D) node[pos =0.5, above] {$\mapU \circ \mapA \circ  \mapA$};
    \end{tikzpicture}
  \end{equation}    
\end{defn}

\begin{rem}
\label{rmk:rickard-is-a-complex}
\begin{enumerate}
\item These are indeed chain complexes, since when one forgets the
  $H^\prime$-module structure, they are precisely the maps used to
  define the classical Rickard complexes, see for instance \cite{ETW}. 
\item In the non-$H^\prime$-equivariant setting, $\rickardp{a}{b}$ and
  $\rickardm{a}{b}$ are homotopy equivalent to other complexes, where
  the two vertical rungs in the squares are oriented oppositely. With
  the $H^\prime$-structure, these other versions of Rickard complexes
  are not homotopy equivalent in the relative homotopy category to the
  ones defined in Definition~\ref{def:rickard}, so we made an
  arbitrary choice here. We believe that any choice gives rise to the
  same homology theory.
\end{enumerate}
\end{rem}

Using $T_{a,b}^+$ and $T_{a,b}^-$, we extend $\soergel(\cdot)$ to MOY
graphs with crossings. At this stage it is not clear at all that this
makes sense topologically.

\begin{equation} \label{eqn:posab}
T_{a,b}^+ = \soergel\left(
\NB{\tikz[]{\begin{scope}
  \draw (0, 0) -- +(0,1);
  \node at (1,0.5) {$\dots$};
    \draw (3, 0) .. controls +(0, 0.2) and +(0, -0.2) ..  +(-1,1)
    node[pos=0, below] {$b$} node[pos=1, above, white] {$b$};
  \fill[white] (2.5, 0.5) circle (2mm);
  \draw (2, 0) .. controls +(0, 0.2) and +(0, -0.2) ..  +(1,1) node[pos=0, below] {$a$};
  \node at (4,0.5) {$\dots$};
  \draw (5, 0) -- +(0,1); 
\end{scope}

}} \right)
\end{equation}

\begin{equation} \label{eqn:negab}
T_{a,b}^- = \left(
\NB{\tikz[xscale = -1, yscale = 1]{\begin{scope}
  \draw (0, 0) -- +(0,1);
  \node at (1,0.5) {$\dots$};
    \draw (3, 0) .. controls +(0, 0.2) and +(0, -0.2) ..  +(-1,1) node[pos=0, below] {$a$};
  \fill[white] (2.5, 0.5) circle (2mm);
  \draw (2, 0) .. controls +(0, 0.2) and +(0, -0.2) ..  +(1,1)
  node[pos=0, below] {$b$} node[pos=1, above, white] {$b$};
  \node at (4,0.5) {$\dots$};
  \draw (5, 0) -- +(0,1); 
\end{scope}}} \right)
\end{equation}

\begin{lem}\label{lem:exact-sequence-square}
  The sequences
  \begin{align*}
\cdots \longrightarrow
 \soergel\left(\NB{\tikz[font=\tiny]{\begin{scope}
  \coordinate (rt) at (+1,   1.2);
  \coordinate (lt) at (-1,   1.2);
  \coordinate (rb) at (+1,  -1.2);
  \coordinate (lb) at (-1,  -1.2);
  \coordinate (r1) at (+1, -0.6);
  \coordinate (l1) at (-1, -0.4);
  \coordinate (r2) at (+1,  0.6);
  \coordinate (l2) at (-1,  0.4);
  \draw[>-] (rb) -- (r1) node[pos=0, below] {$a$};
  \draw[>-] (lb) -- (l1) node[pos=0, below] {$b$}; 
  \draw[->-] (r1) -- (r2) node[pos=0.3, right] {$j$}   coordinate [pos=0.7] (g1);
  \draw[->-] (l1) -- (l2) node[pos=0.5, right] {$a+b-j$};
  \draw[->] (r2) -- (rt) node[pos=1, above] {$b+1$};
  \draw[->] (l2) -- (lt) node[pos=1, above] {$a-1$};
  \draw [->-] (r1) -- (l1) node[pos =0.5, below, sloped] {$a-j$};
  \draw [->-] (l2) -- (r2) node[pos =0.5, above, sloped] {$b+1-j$};
  \filldraw[draw= green!50!black, fill = white] (g1) circle (1mm)
  node[left, green!50!black] {$j-1$};
\end{scope}

}}\right)\longrightarrow
\soergel\left(\NB{\tikz[font=\tiny]{\begin{scope}
  \coordinate (rt) at (+1,   1.2);
  \coordinate (lt) at (-1,   1.2);
  \coordinate (rb) at (+1,  -1.2);
  \coordinate (lb) at (-1,  -1.2);
  \coordinate (r1) at (+1, -0.6);
  \coordinate (l1) at (-1, -0.4);
  \coordinate (r2) at (+1,  0.6);
  \coordinate (l2) at (-1,  0.4);
  \draw[>-] (rb) -- (r1) node[pos=0, below] {$a$};
  \draw[>-] (lb) -- (l1) node[pos=0, below] {$b$}; 
  \draw[->-] (r1) -- (r2) node[pos=0.3, left] {$j+1$}   coordinate [pos=0.7] (g1);
  \draw[->-] (l1) -- (l2) node[pos=0.5, right] {$a+b-j-1$};
  \draw[->] (r2) -- (rt) node[pos=1, above] {$b+1$};
  \draw[->] (l2) -- (lt) node[pos=1, above] {$a-1$};
  \draw [->-] (r1) -- (l1) node[pos =0.5, below, sloped] {$a-j-1$};
  \draw [->-] (l2) -- (r2) node[pos =0.5, above, sloped] {$b-j$};
  \filldraw[draw= green!50!black, fill = white] (g1) circle (1mm)
  node[right, green!50!black] {$j$};
\end{scope}

}}\right)\longrightarrow \cdots
  \end{align*}
  and
  \begin{align*}
\cdots \longrightarrow
\soergel\left(\NB{\tikz[font=\tiny]{\begin{scope}
  \coordinate (rt) at (+1,   1.2);
  \coordinate (lt) at (-1,   1.2);
  \coordinate (rb) at (+1,  -1.2);
  \coordinate (lb) at (-1,  -1.2);
  \coordinate (r1) at (+1, -0.6);
  \coordinate (l1) at (-1, -0.4);
  \coordinate (r2) at (+1,  0.6);
  \coordinate (l2) at (-1,  0.4);
  \draw[>-] (rb) -- (r1) node[pos=0, below] {$a$};
  \draw[>-] (lb) -- (l1) node[pos=0, below] {$b$}; 
  \draw[->-] (r1) -- (r2) node[pos=0.3, right] {$j+1$}   coordinate [pos=0.7] (g1);
  \draw[->-] (l1) -- (l2) node[pos=0.5, right] {$a+b-j-1$};
  \draw[->] (r2) -- (rt) node[pos=1, above] {$b+1$};
  \draw[->] (l2) -- (lt) node[pos=1, above] {$a-1$} coordinate [pos = 0.5] (g2);
  \draw [->-] (r1) -- (l1) node[pos =0.4, below, sloped] {$a-j-1$} coordinate [pos = 0.7] (g3);
  \draw [->-] (l2) -- (r2) node[pos =0.6, below, sloped] {$b-j$} coordinate [pos = 0.3] (g4);
  \filldraw[draw= green!50!black, fill = white] (g1) circle (1mm) node[right, green!50!black] {$1$};
  \filldraw[draw= green!50!black, fill = white] (g2) circle (1mm) node[right, green!50!black] {$j+1-b$};
  \filldraw[draw= green!50!black, fill = white] (g3) circle (1mm) node[above, green!50!black] {$-j-1$};
  \filldraw[draw= green!50!black, fill = white] (g4) circle (1mm) node[below, green!50!black] {$1$};
\end{scope}

}}\right)\longrightarrow
 \soergel\left(\NB{\tikz[font=\tiny]{\begin{scope}
  \coordinate (rt) at (+1,   1.2);
  \coordinate (lt) at (-1,   1.2);
  \coordinate (rb) at (+1,  -1.2);
  \coordinate (lb) at (-1,  -1.2);
  \coordinate (r1) at (+1, -0.6);
  \coordinate (l1) at (-1, -0.4);
  \coordinate (r2) at (+1,  0.6);
  \coordinate (l2) at (-1,  0.4);
  \draw[>-] (rb) -- (r1) node[pos=0, below] {$a$};
  \draw[>-] (lb) -- (l1) node[pos=0, below] {$b$}; 
  \draw[->-] (r1) -- (r2) node[pos=0.3, right] {$j$}   coordinate [pos=0.7] (g1);
  \draw[->-] (l1) -- (l2) node[pos=0.5, right] {$a+b-j$};
  \draw[->] (r2) -- (rt) node[pos=1, above] {$b+1$};
  \draw[->] (l2) -- (lt) node[pos=1, above] {$a-1$} coordinate [pos = 0.5] (g2);
  \draw [->-] (r1) -- (l1) node[pos =0.4, below, sloped] {$a-j$} coordinate [pos = 0.7] (g3);
  \draw [->-] (l2) -- (r2) node[pos =0.6, above, sloped] {$b+1-j$} coordinate [pos = 0.3] (g4);
  \filldraw[draw= green!50!black, fill = white] (g1) circle (1mm) node[right, green!50!black] {$1$};
  \filldraw[draw= green!50!black, fill = white] (g2) circle (1mm) node[right, green!50!black] {$j-b$};
  \filldraw[draw= green!50!black, fill = white] (g3) circle (1mm) node[above, green!50!black] {$-j$};
  \filldraw[draw= green!50!black, fill = white] (g4) circle (1mm) node[below, green!50!black] {$1$};
\end{scope}

}}\right)\longrightarrow \cdots
  \end{align*}
  are exact. The arrows in these sequences are given by compositions of
elementary morphisms given by~(\ref{eq:exactness0-2}) and (\ref{eq:exactness0-1}).
\end{lem}
\begin{equation}
  \label{eq:exactness0-2}
  \NB{
    \begin{tikzpicture}[xscale = 3.5, yscale=2 ]
      \node (A) at (-1,1) {$\soergel\left(\NB{\tikz[font=\tiny]{}}\right)$};
      \node (B) at (1,1) {$q^{a-1-j} \soergel\left(\NB{\tikz[font=\tiny]{\begin{scope}[font=\tiny]
  \coordinate (rt) at (+1,   1.5);
  \coordinate (lt) at (-1,   1.5);
  \coordinate (rb) at (+1,  -1.5);
  \coordinate (lb) at (-1,  -1.5);
  \coordinate (r1) at (+1, -0.9);
  \coordinate (l1) at (-1, -0.7);
  \coordinate (r2) at (+1,  0.9);
  \coordinate (l2) at (-1,  0.7);
  \draw[>-] (rb) -- (r1) node[pos=0, below] {$a$} coordinate [pos=0.7] (g1);
  \draw[>-] (lb) -- (l1) node[pos=0, below] {$b$} coordinate [pos=0.7] (g2); 
  \draw[->-] (r1) -- (r2) node[pos=0.5, right] {$j$} coordinate  [pos=0.8] (g3) coordinate [pos = 0.2] (d1); 
  \draw[->-] (l1) -- (l2) node[pos=0.6, right] {$a+b-j$}   coordinate [pos=0.7] (g4)  coordinate[pos = 0.257] (d2);
  \draw[->] (r2) -- (rt) node[pos=1, above] {$b+1$} coordinate [pos=0.3] (g5);
  \draw[->] (l2) -- (lt) node[pos=1, above] {$a-1$} coordinate [pos=0.3] (g6);
  \draw [->-] (r1) -- (l1) node[pos =0.5, below, sloped] {$a-j-1$} coordinate [pos=0.3] (g7);
  \draw [->-] (l2) -- (r2) node[pos =0.5, above, sloped] {$b+1-j$} coordinate [pos=0.3] (g8);
  \draw [->-] (d1) -- (d2) node[pos =0.3, above] {$1$} coordinate [pos=0.7] (g9);
  \filldraw[draw= green!50!black, fill = white] (g3) circle (1mm)
  node[left, green!50!black] {$j-1$};
\end{scope}

}}\right)$};
      \node (C) at (-1,-1) {$q^{-j}\soergel\left(\NB{\tikz[font=\tiny]{\begin{scope}[font=\tiny]
  \coordinate (rt) at (+1,   1.5);
  \coordinate (lt) at (-1,   1.5);
  \coordinate (rb) at (+1,  -1.5);
  \coordinate (lb) at (-1,  -1.5);
  \coordinate (r1) at (+1, -0.9);
  \coordinate (l1) at (-1, -0.7);
  \coordinate (r2) at (+1,  0.9);
  \coordinate (l2) at (-1,  0.7);
  \draw[>-] (rb) -- (r1) node[pos=0, below] {$a$} coordinate [pos=0.7] (g1);
  \draw[>-] (lb) -- (l1) node[pos=0, below] {$b$} coordinate [pos=0.7] (g2); 
  \draw[->-] (r1) -- (r2) node[pos=0.5, right] {$j$} coordinate  [pos=0.8] (g3) coordinate [pos = 0.2] (d1); 
  \draw[->-] (l1) -- (l2) node[pos=0.5, right] {$a+b-j-1$}   coordinate [pos=0.5] (g4);
  \draw[->] (r2) -- (rt) node[pos=1, above] {$b+1$} coordinate [pos=0.3] (g5);
  \draw[->] (l2) -- (lt) node[pos=1, above] {$a-1$} coordinate [pos=0.3] (g6);
  \draw [->-] (r1) -- (l1) node[pos =0.5, below, sloped] {$a-j-1$} coordinate [pos=0.3] (g7);
  \draw [->-] (l2) -- (r2) node[pos =0.4, above, sloped] {$b-j$} coordinate [pos=0.3] (g8) coordinate[pos = 0.75] (d2);
  \draw [->-] (d1) -- (d2) node[pos =0.3, left] {$1$} coordinate [pos=0.7] (g9);
  \filldraw[draw= green!50!black, fill = white] (g3) circle (1mm)
  node[right, green!50!black] {$j-1$};
\end{scope}

}}\right)$};
      \node (D) at (1,-1)
      {$\soergel\left(\NB{\tikz[font=\tiny]{}}\right)$};
      \draw[->] (A) -- (B) node[pos =0.5, above] {$\mapA \circ \mapA \circ  \mapB$};
      \draw (B.east) -- +(0.1, 0) arc (90: 0:0.2) -- +(0, -0.6) arc
      (0:-90:0.2) -- (0,0) node [pos =1, above] {$\mapH$};
      \draw[<-] (C.west) -- +(-0.1,0) arc (270:180:0.2) --+(0, +0.6) arc
      (180:90:0.2) -- (0,0);
      \draw[->] (C) -- (D) node[pos =0.5, above] {$\mapU \circ \mapA$};
    \end{tikzpicture}}
\end{equation}
\begin{equation}
    \label{eq:exactness0-1}
  \NB{
\begin{tikzpicture}[xscale =4, yscale=2 ]
      \node (A) at (-1,1) {$\soergel\left(\NB{\tikz[font=\tiny]{}}\right)$};

      \node (B) at (1, 1) {$q^j \soergel\left(\NB{\tikz[font=\tiny]{\begin{scope}[font=\tiny]
  \coordinate (rt) at (+1,   1.5);
  \coordinate (lt) at (-1,   1.5);
  \coordinate (rb) at (+1,  -1.5);
  \coordinate (lb) at (-1,  -1.5);
  \coordinate (r1) at (+1, -0.9);
  \coordinate (l1) at (-1, -0.7);
  \coordinate (r2) at (+1,  0.9);
  \coordinate (l2) at (-1,  0.7);
  \draw[>-] (rb) -- (r1) node[pos=0, below] {$a$} coordinate [pos=0.7] (g1);
  \draw[>-] (lb) -- (l1) node[pos=0, below] {$b$} coordinate [pos=0.7] (g2); 
  \draw[->-] (r1) -- (r2) node[pos=0.5, right] {$j$} coordinate  [pos=0.8] (g3) coordinate [pos = 0.2] (d1); 
  \draw[->-] (l1) -- (l2) node[pos=0.3, right] {$a+b-j-1$}   coordinate [pos=0.7] (g4);
  \draw[->] (r2) -- (rt) node[pos=1, above] {$b+1$} coordinate [pos=0.3] (g5);
  \draw[->] (l2) -- (lt) node[pos=1, above] {$a-1$} coordinate [pos=0.3] (g6);
  \draw [->-] (r1) -- (l1) node[pos =0.6, below, sloped] {$a-j-1$} coordinate [pos=0.3] (g7);
  \draw [->-] (l2) -- (r2) node[pos =0.4, below, sloped] {$b-j$} coordinate [pos=0.3] (g8) coordinate[pos = 0.75] (d2);
  \draw [->-] (d1) -- (d2) node[pos =0.3, left] {$1$} coordinate [pos=0.7] (g9);
  \filldraw[draw= green!50!black, fill = white] (g6) circle (1mm)
  node[right, green!50!black] {$j+1-b$};
  \filldraw[draw= green!50!black, fill = white] (g5) circle (1mm)
  node[right, green!50!black] {$1$};
  \filldraw[draw= green!50!black, fill = white] (g7) circle (1mm)
  node[above, green!50!black] {$-j-1$};
\end{scope}

}}\right)$};
      \node (C) at (-1, -1){$q^{j-a+1} \soergel\left(\NB{\tikz[font=\tiny]{\begin{scope}[font=\tiny]
  \coordinate (rt) at (+1,   1.5);
  \coordinate (lt) at (-1,   1.5);
  \coordinate (rb) at (+1,  -1.5);
  \coordinate (lb) at (-1,  -1.5);
  \coordinate (r1) at (+1, -0.9);
  \coordinate (l1) at (-1, -0.7);
  \coordinate (r2) at (+1,  0.9);
  \coordinate (l2) at (-1,  0.7);
  \draw[>-] (rb) -- (r1) node[pos=0, below] {$a$} coordinate [pos=0.7] (g1);
  \draw[>-] (lb) -- (l1) node[pos=0, below] {$b$} coordinate [pos=0.7] (g2); 
  \draw[->-] (r1) -- (r2) node[pos=0.5, right] {$j$} coordinate  [pos=0.8] (g3) coordinate [pos = 0.2] (d1); 
  \draw[->-] (l1) -- (l2) node[pos=0.6, right] {$a+b-j$}   coordinate [pos=0.7] (g4)  coordinate[pos = 0.257] (d2);
  \draw[->] (r2) -- (rt) node[pos=1, above] {$b+1$} coordinate [pos=0.3] (g5);
  \draw[->] (l2) -- (lt) node[pos=1, above] {$a-1$} coordinate [pos=0.3] (g6);
  \draw [->-] (r1) -- (l1) node[pos =0.6, below, sloped] {$a-j-1$} coordinate [pos=0.3] (g7);
  \draw [->-] (l2) -- (r2) node[pos =0.4, below, sloped] {$b+1-j$} coordinate [pos=0.3] (g8);
  \draw [->-] (d1) -- (d2) node[pos =0.3, above] {$1$} coordinate [pos=0.7] (g9);
  \filldraw[draw= green!50!black, fill = white] (g6) circle (1mm)
  node[right, green!50!black] {$j-b$};
  \filldraw[draw= green!50!black, fill = white] (g5) circle (1mm)
  node[right, green!50!black] {$1$};
  \filldraw[draw= green!50!black, fill = white] (g9) circle (1mm)
  node[above, green!50!black] {$1-a$};
  \filldraw[draw= green!50!black, fill = white] (g7) circle (1mm)
  node[above, green!50!black] {$-j-1$};
\end{scope}

}}\right)$};
      \node (D) at (1, -1)
      {$\soergel\left(\NB{\tikz[font=\tiny]{}}\right)$};
      \draw[->] (A) -- (B) node[pos =0.5, above] {$\mapA \circ  \mapB$};
      \draw (B.east) -- +(0.1, 0) arc (90: 0:0.2) -- +(0, -0.6) arc
      (0:-90:0.2) -- (0,0) node [pos =1, above] {$\mapX$};
      \draw[<-] (C.west) -- +(-0.1,0) arc (270:180:0.2) --+(0, +0.6) arc
      (180:90:0.2) -- (0,0);
      \draw[->] (C) -- (D) node[pos =0.5, above] {$\mapU \circ \mapA \circ  \mapA$};
    \end{tikzpicture}} \\[0.5cm]
  \end{equation}

\begin{proof}
  The fact that the composition of two successive maps is zero, is a tedious but
  relatively straightforward computation. In order to prove that the sequences
  are exact, we gives explicit (non $H^\prime$-equivariant) homotopies between the
  identities of these complexes and the zero maps. The elementary pieces
  of these homotopies are given by~(\ref{eq:exactness-2}) and (\ref{eq:exactness-1}) . It is a straightforward but lengthy computation to show that these are indeed homotopy equivalences between the
    given complexes and the zero complex. 
  \begin{equation}
        \label{eq:exactness-2}
        \NB{
  \begin{tikzpicture}[xscale =4, yscale=2 ]
      \node (A) at (-1,1) {$\soergel\left(\NB{\tikz[font=\tiny]{\begin{scope}
  \coordinate (rt) at (+1,   1.2);
  \coordinate (lt) at (-1,   1.2);
  \coordinate (rb) at (+1,  -1.2);
  \coordinate (lb) at (-1,  -1.2);
  \coordinate (r1) at (+1, -0.6);
  \coordinate (l1) at (-1, -0.4);
  \coordinate (r2) at (+1,  0.6);
  \coordinate (l2) at (-1,  0.4);
  \draw[>-] (rb) -- (r1) node[pos=0, below] {$a$};
  \draw[>-] (lb) -- (l1) node[pos=0, below] {$b$}; 
  \draw[->-] (r1) -- (r2) node[pos=0.3, left] {$j+1$}   coordinate [pos=0.7] (g1);
  \draw[->-] (l1) -- (l2) node[pos=0.5, right] {$a+b-j-1$};
  \draw[->] (r2) -- (rt) node[pos=1, above] {$b+1$};
  \draw[->] (l2) -- (lt) node[pos=1, above] {$a-1$};
  \draw [->-] (r1) -- (l1) node[pos =0.5, below, sloped] {$a-j-1$};
  \draw [->-] (l2) -- (r2) node[pos =0.5, above, sloped] {$b-j$};
\end{scope}

}}\right)$};
      \node (B) at ( 1,1) {$q^j \soergel\left(\NB{\tikz[font=\tiny]{\begin{scope}[font=\tiny]
  \coordinate (rt) at (+1,   1.5);
  \coordinate (lt) at (-1,   1.5);
  \coordinate (rb) at (+1,  -1.5);
  \coordinate (lb) at (-1,  -1.5);
  \coordinate (r1) at (+1, -0.9);
  \coordinate (l1) at (-1, -0.7);
  \coordinate (r2) at (+1,  0.9);
  \coordinate (l2) at (-1,  0.7);
  \draw[>-] (rb) -- (r1) node[pos=0, below] {$a$} coordinate [pos=0.7] (g1);
  \draw[>-] (lb) -- (l1) node[pos=0, below] {$b$} coordinate [pos=0.7] (g2); 
  \draw[->-] (r1) -- (r2) node[pos=0.5, right] {$j$} coordinate  [pos=0.8] (g3) coordinate [pos = 0.2] (d1); 
  \draw[->-] (l1) -- (l2) node[pos=0.5, right] {$a+b-j-1$}   coordinate [pos=0.5] (g4);
  \draw[->] (r2) -- (rt) node[pos=1, above] {$b+1$} coordinate [pos=0.3] (g5);
  \draw[->] (l2) -- (lt) node[pos=1, above] {$a-1$} coordinate [pos=0.3] (g6);
  \draw [->-] (r1) -- (l1) node[pos =0.5, below, sloped] {$a-j-1$} coordinate [pos=0.3] (g7);
  \draw [->-] (l2) -- (r2) node[pos =0.4, above, sloped] {$b-j$} coordinate [pos=0.3] (g8) coordinate[pos = 0.75] (d2);
  \draw [->-] (d1) -- (d2) node[pos =0.3, left] {$1$} coordinate [pos=0.7] (g9);
\end{scope}}}\right)$};
      \node (C) at (-1,-1) {$q^{-a+j-1}\soergel\left(\NB{\tikz[font=\tiny]{\begin{scope}[font=\tiny]
  \coordinate (rt) at (+1,   1.5);
  \coordinate (lt) at (-1,   1.5);
  \coordinate (rb) at (+1,  -1.5);
  \coordinate (lb) at (-1,  -1.5);
  \coordinate (r1) at (+1, -0.9);
  \coordinate (l1) at (-1, -0.7);
  \coordinate (r2) at (+1,  0.9);
  \coordinate (l2) at (-1,  0.7);
  \draw[>-] (rb) -- (r1) node[pos=0, below] {$a$} coordinate [pos=0.7] (g1);
  \draw[>-] (lb) -- (l1) node[pos=0, below] {$b$} coordinate [pos=0.7] (g2); 
  \draw[->-] (r1) -- (r2) node[pos=0.5, right] {$j$} coordinate  [pos=0.8] (g3) coordinate [pos = 0.2] (d1); 
  \draw[->-] (l1) -- (l2) node[pos=0.6, right] {$a+b-j$}   coordinate [pos=0.7] (g4)  coordinate[pos = 0.257] (d2);
  \draw[->] (r2) -- (rt) node[pos=1, above] {$b+1$} coordinate [pos=0.3] (g5);
  \draw[->] (l2) -- (lt) node[pos=1, above] {$a-1$} coordinate [pos=0.3] (g6);
  \draw [->-] (r1) -- (l1) node[pos =0.5, below, sloped] {$a-j-1$} coordinate [pos=0.3] (g7);
  \draw [->-] (l2) -- (r2) node[pos =0.5, above, sloped] {$b+1-j$} coordinate [pos=0.3] (g8);
  \draw [->-] (d1) -- (d2) node[pos =0.3, above] {$1$} coordinate [pos=0.7] (g9);
\end{scope}}}\right)$};
      \node (D) at (1,-1)
      {$q^{-2}\soergel\left(\NB{\tikz[font=\tiny]{\begin{scope}
  \coordinate (rt) at (+1,   1.2);
  \coordinate (lt) at (-1,   1.2);
  \coordinate (rb) at (+1,  -1.2);
  \coordinate (lb) at (-1,  -1.2);
  \coordinate (r1) at (+1, -0.6);
  \coordinate (l1) at (-1, -0.4);
  \coordinate (r2) at (+1,  0.6);
  \coordinate (l2) at (-1,  0.4);
  \draw[>-] (rb) -- (r1) node[pos=0, below] {$a$};
  \draw[>-] (lb) -- (l1) node[pos=0, below] {$b$}; 
  \draw[->-] (r1) -- (r2) node[pos=0.3, right] {$j$}   coordinate [pos=0.7] (g1);
  \draw[->-] (l1) -- (l2) node[pos=0.5, right] {$a+b-j$};
  \draw[->] (r2) -- (rt) node[pos=1, above] {$b+1$};
  \draw[->] (l2) -- (lt) node[pos=1, above] {$a-1$};
  \draw [->-] (r1) -- (l1) node[pos =0.5, below, sloped] {$a-j$};
  \draw [->-] (l2) -- (r2) node[pos =0.5, above, sloped] {$b+1-j$};
\end{scope}

}}\right)$};
      \draw[->] (A) -- (B) node[pos =0.5, above] {$\mapA \circ  \mapB$};
      \draw (B.east) -- +(0.1, 0) arc (90: 0:0.2) -- +(0, -0.6) arc
      (0:-90:0.2) -- (0,0) node [pos =1, above] {$\mapE \circ \mapX$};
      \draw[<-] (C.west) -- +(-0.1,0) arc (270:180:0.2) --+(0, +0.6) arc
      (180:90:0.2) -- (0,0);
      \draw[->] (C) -- (D) node[pos =0.5, above] {$\mapU \circ \mapA \circ  \mapA$};
      \end{tikzpicture}}
  \end{equation}
    \begin{equation}
          \label{eq:exactness-1}
\NB{
      \begin{tikzpicture}[xscale =3.5, yscale=2 ]
      \node (A) at (-1,1) {$\soergel\left(\NB{\tikz[font=\tiny]{}}\right)$};
      \node (B) at ( 1,1) {$q^{a-1-j} \soergel\left(\NB{\tikz[font=\tiny]{\begin{scope}[font=\tiny]
  \coordinate (rt) at (+1,   1.5);
  \coordinate (lt) at (-1,   1.5);
  \coordinate (rb) at (+1,  -1.5);
  \coordinate (lb) at (-1,  -1.5);
  \coordinate (r1) at (+1, -0.9);
  \coordinate (l1) at (-1, -0.7);
  \coordinate (r2) at (+1,  0.9);
  \coordinate (l2) at (-1,  0.7);
  \draw[>-] (rb) -- (r1) node[pos=0, below] {$a$} coordinate [pos=0.7] (g1);
  \draw[>-] (lb) -- (l1) node[pos=0, below] {$b$} coordinate [pos=0.7] (g2); 
  \draw[->-] (r1) -- (r2) node[pos=0.5, right] {$j$} coordinate  [pos=0.8] (g3) coordinate [pos = 0.2] (d1); 
  \draw[->-] (l1) -- (l2) node[pos=0.6, right] {$a+b-j$}   coordinate [pos=0.7] (g4)  coordinate[pos = 0.257] (d2);
  \draw[->] (r2) -- (rt) node[pos=1, above] {$b+1$} coordinate [pos=0.3] (g5);
  \draw[->] (l2) -- (lt) node[pos=1, above] {$a-1$} coordinate [pos=0.3] (g6);
  \draw [->-] (r1) -- (l1) node[pos =0.5, below, sloped] {$a-j-1$} coordinate [pos=0.3] (g7);
  \draw [->-] (l2) -- (r2) node[pos =0.5, above, sloped] {$b+1-j$} coordinate [pos=0.3] (g8);
  \draw [->-] (d1) -- (d2) node[pos =0.3, above] {$1$} coordinate [pos=0.7] (g9);
\end{scope}}}\right)$};
      \node (C) at (-1,-1) {$q^{-j-2}\soergel\left(\NB{\tikz[font=\tiny]{\begin{scope}[font=\tiny]
  \coordinate (rt) at (+1,   1.5);
  \coordinate (lt) at (-1,   1.5);
  \coordinate (rb) at (+1,  -1.5);
  \coordinate (lb) at (-1,  -1.5);
  \coordinate (r1) at (+1, -0.9);
  \coordinate (l1) at (-1, -0.7);
  \coordinate (r2) at (+1,  0.9);
  \coordinate (l2) at (-1,  0.7);
  \draw[>-] (rb) -- (r1) node[pos=0, below] {$a$} coordinate [pos=0.7] (g1);
  \draw[>-] (lb) -- (l1) node[pos=0, below] {$b$} coordinate [pos=0.7] (g2); 
  \draw[->-] (r1) -- (r2) node[pos=0.5, right] {$j$} coordinate  [pos=0.8] (g3) coordinate [pos = 0.2] (d1); 
  \draw[->-] (l1) -- (l2) node[pos=0.5, right] {$a+b-j-1$}   coordinate [pos=0.5] (g4);
  \draw[->] (r2) -- (rt) node[pos=1, above] {$b+1$} coordinate [pos=0.3] (g5);
  \draw[->] (l2) -- (lt) node[pos=1, above] {$a-1$} coordinate [pos=0.3] (g6);
  \draw [->-] (r1) -- (l1) node[pos =0.5, below, sloped] {$a-j-1$} coordinate [pos=0.3] (g7);
  \draw [->-] (l2) -- (r2) node[pos =0.4, above, sloped] {$b-j$} coordinate [pos=0.3] (g8) coordinate[pos = 0.75] (d2);
  \draw [->-] (d1) -- (d2) node[pos =0.3, left] {$1$} coordinate [pos=0.7] (g9);
\end{scope}}}\right)$};
      \node (D) at ( 1,-1)
      {$q^{-2}\soergel\left(\NB{\tikz[font=\tiny]{}}\right)$};
      \draw[->] (A) -- (B) node[pos =0.5, above] {$\mapA \circ \mapA \circ  \mapB$};
      \draw (B.east) -- +(0.1, 0) arc (90: 0:0.2) -- +(0, -0.6) arc
      (0:-90:0.2) -- (0,0) node [pos =1, above] {$\mapH \circ \mapE$};
      \draw[<-] (C.west) -- +(-0.1,0) arc (270:180:0.2) --+(0, +0.6) arc
      (180:90:0.2) -- (0,0);
      \draw[->] (C) -- (D) node[pos =0.5, above] {$\mapU \circ  \mapA$};
    \end{tikzpicture}}
  \end{equation}

      \end{proof}

\section{Uncolored homology theories} \label{uncolored} In this
section we review the construction of a link homology categorifying
the Jones polynomial at a $2p$th root of unity given in
\cite{QiSussanLink}.  Strictly speaking, the homology that we describe
here is different from the one in \cite{QiSussanLink}, but the
definition and proofs are similar.  In particular, the construction we
now provide is a bigraded theory instead of a singly-graded
theory. We use the invariance of this uncolored homology in the proof
of invariance of the colored theory later on.

As before, $A = A_n = \myZ[x_1,\ldots,x_n]$
denotes the graded polynomial algebra, where
each generator $x_i$ has degree two.
We let $B_i$ denote the $(A,A)$-bimodule associated to MOY graph
\eqref{dumb11} and so $\soergel(B_i)=q^{-1} B_i$.
\begin{equation} \label{dumb11}
  \NB{\tikz[scale = 0.7, font=\tiny]{\begin{scope}[font=\tiny, scale=0.7]
  \coordinate (bll) at (-1.5, -1);
  \coordinate (brr) at ( 1.5, -1);
  \node (mll) at (-2.25, 0) {$\dots$};
  \node (mrr) at ( 2.25, 0) {$\dots$};
  \coordinate (blll) at (-3, -1);
  \coordinate (brrr) at ( 3, -1);
  \coordinate (bl) at (-0.5, -1);
  \coordinate (br) at ( 0.5, -1);
  \coordinate (bm) at (  0,-0.3);
  \coordinate (tl) at (-0.5,  1);
  \coordinate (tr) at ( 0.5,  1);
  \coordinate (tm) at (  0, 0.3);
  \draw [->-] (brr) -- +(0,2) node[pos=0, below] {$1$};
  \draw [->-] (brrr) -- +(0,2) node[pos=0, below] {$1$};
  \draw [->-] (blll) -- +(0,2) node[pos=0, below] {$1$};
  \draw [->-] (bll) -- +(0,2) node[pos=0, below] {$1$};
  \draw[>-]  (bl) .. controls +( 0, 0.5) and +(0,0) .. (bm)
  node[below, pos = 0] {$1$};
  \draw[>-]  (br) .. controls +( 0, 0.5) and +(0,0) .. (bm)
  node[below, pos = 0] {$1$};
  \draw[<-]  (tl) .. controls +( 0, -0.5) and +(0,0) .. (tm)
  node[above, pos = 0] {$1$};
  \draw[<-]  (tr) .. controls +( 0, -0.5) and +(0,0) .. (tm)
  node[above, pos = 0] {$1$};
  \draw [->-] (bm) -- (tm) node[left, pos = 0.5] {$2$};
  \draw
  [decorate,decoration={brace,amplitude=10pt,mirror,raise=4pt}]
  ($(blll)+(0, -0.5)$)  -- ($(bl)+(0, -0.5)$) node[midway, yshift =
  -0.5cm, below] {$i$};
  \draw
  [decorate,decoration={brace,amplitude=10pt,mirror,raise=4pt}]
  ($(br)+(0, -0.5)$)  -- ($(brrr)+(0, -0.5)$) node[midway, yshift =
  -0.5cm, below] {$n-i$};

\end{scope}
}}
\end{equation}

\subsection{Elementary braiding complexes}
In \cite{KRWitt} (see \cite{QiSussanLink} for the $H$-equivariant setting), it is shown that there are
$(A,A) \# H^\prime$-module homomorphisms
\begin{enumerate}
\item[(i)]
  $rb_i \colon A \longrightarrow q^{-2} B_i^{-(x_i+x_{i+1})}$, where
  $1 \mapsto (x_{i+1} \otimes 1 - 1 \otimes x_{i}) $;
\item[(ii)] $br_i \colon B_i \longrightarrow A$, where
  $1 \otimes 1 \mapsto 1$.
\end{enumerate}
Note that $rb_i$ is precisely the map $\mapX$ given in
Lemma~\ref{lem:pdgmaps}.\ref{it:pdg-maps-merges-special} for $a=b=1$
and that
$br_i$ is the map $\mapH$ given in
Lemma~\ref{lem:pdgmaps}.\ref{it:pdg-maps-splits} for $a=b=1$ and $c=0$.

Thus we have complexes of $(A,A) \# H^\prime$-modules
\begin{equation}\label{eqn-elementary-braids-half-grading}
  T_i^+ :=
  q^{-3} \left( B_i  \xrightarrow{br_i}  A\right)
  ,
  \quad \quad \quad
  T_i^- :=  q^3\left(A \xrightarrow{rb_i} q^{-2}  B_i^{-(x_i+x_{i+1})}\right)
  .
\end{equation}
These complexes are cohomologically graded (\emph{i.e}.{} the
differential has degree $1$), the homological degree is called the
\emph{topological degree} or \emph{$t$-degree}. In both cases the
bimodule $B_i$ sits in $t$-degree $0$. 
Pictorially, these complexes are:
\begin{equation} \label{T+picture}
   T_i^+=
q^{-3} \left(  \soergel\left(\NB{\tikz[scale = 0.7, font
    =\tiny]{\begin{scope}
  \coordinate (bl) at (-0.5, -1);
  \coordinate (br) at ( 0.5, -1);
  \coordinate (bm) at (  0,-0.3);
  \coordinate (tl) at (-0.5,  1);
  \coordinate (tr) at ( 0.5,  1);
  \coordinate (tm) at (  0, 0.3);
  \draw[>-]  (bl) .. controls +( 0, 0.5) and +(0,0) .. (bm)
  node[below, pos = 0] {$1$};
  \draw[>-]  (br) .. controls +( 0, 0.5) and +(0,0) .. (bm)
  node[below, pos = 0] {$1$};
  \draw[<-]  (tl) .. controls +( 0, -0.5) and +(0,0) .. (tm)
  node[above, pos = 0] {$1$};
  \draw[<-]  (tr) .. controls +( 0, -0.5) and +(0,0) .. (tm)
  node[above, pos = 0] {$1$};
  \draw [->-] (bm) -- (tm) node[left, pos = 0.5] {$2$};
\end{scope}}} \right)
  \rightarrow 
 \soergel\left(    \NB{\tikz[scale = 0.7, font =\tiny]{\begin{scope}
  \coordinate (bl) at (-0.5, -1);
  \coordinate (br) at ( 0.5, -1);
  \coordinate (tl) at (-0.5,  1);
  \coordinate (tr) at ( 0.5,  1);
    \coordinate (ml) at (-0.5,  -.8);
        \coordinate (Ml) at (-0.5,  .8);
 \coordinate (mr) at (0.5,  -.6);
\coordinate (Mr) at (0.5,  .6);

   \draw[>->] (bl) -- (tl) node[pos = 0, below] {$1$} node[pos = 1,
  above] {$1$};

    \draw[>->] (br) -- (tr) node[pos = 0, below] {$1$} node[pos = 1, above] {$1$};

\end{scope}}} \right)
    \right) , 
    \quad \quad 
    T_i^-= q^3 
    \left(
 \soergel\left(\NB{\tikz[scale = 0.7, font =\tiny]{}} \right)
     \rightarrow
     q^{-2}
\soergel\left(     \NB{\tikz[scale = 0.7, font =\tiny]{\begin{scope}
  \coordinate (bl) at (-0.5, -1);
  \coordinate (br) at ( 0.5, -1);
  \coordinate (bm) at (  0,-0.3);
  \coordinate (tl) at (-0.5,  1);
  \coordinate (tr) at ( 0.5,  1);
  \coordinate (tm) at (  0, 0.3);
  \draw[>-]  (bl) .. controls +( 0, 0.5) and +(0,0) .. (bm)
  node[below, pos = 0] {$1$};
  \draw[>-]  (br) .. controls +( 0, 0.5) and +(0,0) .. (bm)
  node[below, pos = 0] {$1$};
  \draw[<-]  (tl) .. controls +( 0, -0.5) and +(0,0) .. (tm)
  node[above, pos = 0] {$1$} coordinate[pos = 0.25] (ga) ;
    \filldraw[draw= green!50!black, fill = white] (ga) circle (1mm)
  node[left, green!50!black] {$-1$};

  \draw[<-]  (tr) .. controls +( 0, -0.5) and +(0,0) .. (tm)
  node[above, pos = 0] {$1$} coordinate[pos = 0.25] (gb) ;
    \filldraw[draw= green!50!black, fill = white] (gb) circle (1mm)
  node[left, green!50!black] {$-1$};
  \draw [->-] (bm) -- (tm) node[left, pos = 0.5] {$2$};
 
\end{scope}}} \right)
    \right) \ .
\end{equation}

In the coming sections we will, for presentation reasons, often omit
the various shifts built into the definitions of $T_i^+$ and $T_i^-$.

We associate respectively to the positive and negative crossings $\sigma_i^+$
and $\sigma_i^{-}$ between the $i$th and $(i+1)$st strands in
\eqref{2crossings} the chain complexes of $(A,A)\# H^\prime $-bimodules $T_i^+$
and $T_i^-$:
\begin{equation} \label{2crossings} \sigma_i^+:=\NB{\tikz[scale =
  1]{\begin{scope}
  \draw (0, 0) -- +(0,1);
  \node at (1,0.5) {$\dots$};
  \draw (3, 0) .. controls +(0, 0.2) and +(0, -0.2) ..  +(-1,1);
  \fill[white] (2.5, 0.5) circle (2mm);
  \draw (2, 0) .. controls +(0, 0.2) and +(0, -0.2) ..  +(1,1);
  \node at (4,0.5) {$\dots$};
  \draw (5, 0) -- +(0,1); 
\end{scope}}} \ ,
  \hspace{1in}
  \sigma_i^-:=
  \NB{\tikz[xscale = -1, yscale = 1]{}} \ .
\end{equation}
More generally, if $\beta\in \mathrm{Br}_n$ is a braid group element
written as a product in the elementary generators
$\sigma_{i_i}^{\epsilon_1}\cdots \sigma_{i_k}^{\epsilon_k}$, where
$\epsilon_i\in \{+, -\}$, we assign the chain complex of
$(A,A)\# H^\prime$-bimodules
\begin{equation}
   T_\beta:=T_{i_1}^{\epsilon_1}\otimes_A\cdots \otimes_A T_{i_k}^{\epsilon_k}.
\end{equation}

\begin{thm}[{\cite[Section 4]{KRWitt} and \cite[Section
  2]{QiSussanLink}}]
  \label{thm-braid-invariant}
  The complexes $T_i^+$, $T_i^-$ are mutually inverse complexes
  in the relative homotopy category $\mc{C}^{\dif}(A,A,)$. They
  satisfy the braid relations
  \begin{itemize}
  \item $T_i^+T_j^+\cong T_j^+T_i^+$ if $|i-j|>1$,
  \item $T_i^+T_{i+1}^+ T_i^+ \cong T_{i+1}^+T_i^+ T_{i+1}^+$ for all
    $i=1,\dots, n-2$.
  \end{itemize}
  Consequently, given any braid group element
  $\beta\in \mathrm{Br}_n$, the chain complex of $T_\beta$ associated
  to it is a well-defined element of the relative homotopy category
  $\mc{C}^{\dif}(A,A)$.
\end{thm}

\subsection{Definition and theorem}
\label{sec:def-uncolored}
In this section we categorify the Jones polynomial of any link using
analogous arguments from \cite{Cautisremarks}, \cite{RW} and
\cite{Roulink} adapted to the $p$-DG setting.
We follow the convention of \cite{KR3} that
Hochschild ($a$) degrees are homological and topological ($t$) degrees
are cohomological.

Let $\beta\in \mathrm{Br}_n$ be a braid group element in $n$
strands. By Theorem \ref{thm-braid-invariant}, there is a chain
complex of $(A_n,A_n)\# H^\prime$-bimodules $T_\beta$, well defined up to
homotopy, associated with $\beta$. Let
\begin{equation}\label{eqn-chain-complex-for-braid}
  T_\beta = \left(\dots\stackrel{d_0}{\lra} T_\beta^{i-1} \stackrel{d_0}{\lra} T_\beta^{i} \stackrel{d_0}{\lra} T_\beta^{i+1}\stackrel{d_0}{\lra}\dots\right).
\end{equation}
Then consider the complex
\[
  \mHH_\bullet^{\dif}(T_\beta^{})= \dots \lra
  \mHH_\bullet^{\dif}(T_\beta^{i-1}) \xrightarrow{d_t}
  \mHH_\bullet^{\dif}(T_\beta^{i}) \xrightarrow{d_t}
  \mHH_\bullet^{\dif}(T_\beta^{i+1})\lra \dots
\]
where $d_t:=\mHH_\bullet^{\dif}(d_0)$, called the \emph{topological
  differential}, is the induced map of $d_0$ on Hochschild homology.

The Cautis differential $d_C$ has Hochschild degree $-1$ and
$q$-degree $2$. It acts on Hochschild homology of each $T_\beta^i$,
$i\in \Z$. We now form a chain complex with respect to the \emph{total differential}
$d_T:=d_t+d_C$, yielding a complex of $H^\prime$-modules. 
\begin{equation}\label{eqn-HH-complex-dT}
  \dots \lra  \mHH_\bullet^{\dif}(T_\beta^{i-1}) \xrightarrow{d_T}  \mHH_\bullet^{\dif}(T_\beta^{i}) \xrightarrow{d_T}        \mHH_\bullet^{\dif}(T_\beta^{i+1})\lra \dots 
\end{equation}
Furthermore, because the topological differential preserves $a$ and $q$ degrees while the Cautis differential is inhomogeneous, we are
forced to collapse the $a$-grading onto $t$ and $q$ by
\begin{equation}\label{eqnatqdegreecollapse}
  a=t^{-1}q^2.
\end{equation}
As a result, the complex \eqref{eqn-HH-complex-dT} is a bigraded complex of $H^\prime$-modules.

\begin{rem}[Homology vs cohomology]
  Despite the fact that the complex $T_\beta$ is usually written as a
  cochain complex in the literature (see, for instance, \cite{KR3}),
  we will be treating $\mHT$, the homology with respect to $d_T$, as a
  homology theory for $\beta$, mainly because it is covariant. Thus
  $HT_j$ will stand for the homology in $t$-degree $-j$.
\end{rem}

\begin{defn}
\label{def-HHH}
Let $\beta$ be a braid.
\begin{enumerate}
   \item[(1)]  The \emph{$H^\prime$-module} of $\beta$ is the bigraded
  $H^\prime$-representation
  \[
    \mC^{\prime}(\beta):=q^{-n} \mHT_\bullet \left( \mHH_\bullet^\dif(T_\beta), d_T \right) .
  \]
Here the notation $\mHT$ on the right hand side emphasizes that the homology is taken with respect to the total differential $d_T$.  
 \item[(2)] The \emph{$H$-module} of $\beta$ is the bigraded $H$-representation
 \[
    \mC^{\dif}(\beta):=\mC^\prime(\beta)^{\mathrm{free}}\otimes_\Z \F_p,
  \]
  where $\mC^\prime(\beta)^{\mathrm{free}}$ denotes the free part of the $\Z$-module.
 \item[(3)]
  The \emph{slash cohomology} of $\beta$ is the slash
  cohomology (see \eqref{eqnslashcohomology}) of $\mC^\dif(\beta)$:
  \[
    \mH^{\dif}_/(\beta):=\mH_/ \left( \mC^\dif(\beta)
    \right).
  \]
  Recall that it is the image of the bigraded module $\mC^{\dif}(\beta)$ in
  the stable category $H\udmod$.

\end{enumerate}
\end{defn}

\begin{rem}[Bockstein construction]\label{rmk-Bockstein}
We now record some important comments on Definition \ref{def-HHH}.
 
The first part of Definition \ref{def-HHH} is simply an integral version of the main construction of \cite{RW}, with a compatible $H^\prime$-module structure. In the second part of Definition \ref{def-HHH}, taking the free part of the homology of a chain complex of $\Z$-modules and then tensoring with $\F_p$ is usually known as the infinity page of the \emph{Bockstein spectral sequence} in algebraic topology (see \cite[Example 2.3]{MayPrimerSS} for a nice description), which we will abbreviate as the \emph{Bockstein construction} in what follows. 
  
Using the Bockstein construction in (2) may seem to be unmotivated. However, this is the key tool that retains the finiteness
property of the symmetric homology theory of \cite{RW} in finite
characteristic (see Example~\ref{ex:1varoverFp}). Recall that the
main theorem of \cite[Theorem 6.2]{RW} utilizes Soergel bimodules over
$\Q$, and the invariant $\mHT_\bullet(\mHH_\bullet(\beta))$ is a
finite-dimensional invariant for the closure of $\beta$. This implies
that the (torsion) free part of $\mC^\prime(\beta)$ is a finitely
generated abelian group, so that $\mC^\dif(\beta)$ is a
finite-dimensional $\F_p$-vector space by construction. By Corollary
\ref{cor-p-nilpotency}, $\dif^p\equiv 0$ on $\mC^\dif(\beta)$, so that
part (3) is well-defined.
  
In contrast, if one starts by working with Soergel bimodules in positive characteristic
from the beginning without resorting to the Bockstein construction, then
  the resulting homology theory of a braid might be of infinite
  dimension. To see this phenomena for the unknot, see Example
  \ref{ex:1varoverFp}.
\end{rem}

\begin{example} \label{ex:1varoverFp}
Recall from Example \ref{ex:1varoverZ} that 
$\mHH_\bullet^\dif(A) \simeq A \otimes \Lambda^\bullet V$ 
and that the Cautis differential $d_C$ and the derivation $\dif_A$ act on it.  In this example we consider $\mHH_0^{\dif}(A) / d_C(\mHH_1^{\dif}(A)) $ over $\Fp$ by base changing $A$ over $\F_p$. We will see that this space is sensitive to the characteristic $p$.

Assume that $n<p$.  From the construction earlier, we see that $d_C(1 \otimes e_i)=e_1 e_i - (i+1) e_{i+1}$ (where we assume $e_{n+1}=0$).
This implies that in $\mHH_0^\dif(A)$ there are relations $e_j=\frac{1}{j!} e_1^j $ and $e_1^{n+1}=0$.
Thus $\mHH_0^{\dif}(A) / d_C(\mHH_1^{\dif}(A)) \cong \Fp[e_1]/(e_1^{n+1}) $ and is thus a finite-dimensional algebra over $\Fp$.  It is easy to see that the $p$-differential $\dif_A$ is trivial on this quotient.

The situation is quite different for $n \geq p$ and for concreteness, we restrict to the case $n=p$.
For $j<p$, then again relations in the quotient 
$\mHH_0^{\dif}(A) / d_C(\mHH_1^{\dif}(A)) $ imply that
$e_j=\frac{1}{j!} e_1^j $.  
However the relation $e_1 e_{p-1} = p e_p = 0 $ implies that $e_1^p=0$ in the quotient and gives rise to new phenomena.  We still have $e_1 e_p = 0$ in the quotient. Thus we get that  
$\mHH_0^{\dif}(A) / d_C(\mHH_1^{\dif}(A)) \cong \Fp[e_1,e_p]/(e_1^{p}, e_1 e_p) $.
In contrast to the case $n<p$, this is an infinite-dimensional algebra.  It is easy to see again that the $p$-differential is trivial on it and so even the slash cohomology of $\mHH_0^{\dif}(A) / d_C(\mHH_1^{\dif}(A)) $ in this case is infinite-dimensional.  In fact, when $n=p$, there is non-trivial homology with respect to $d_C$ in Hochschild degree one.  This is in stark contrast to the $n<p$ case as well as the situation when working with rational coefficients.
\end{example}

By definition, the space $\mC^\prime(\beta)$ (resp. $\mC^\dif(\beta)$, $\mH^{\dif}_/(\beta)$) is doubly-graded by
topological ($t$) degree and quantum ($q$) degree. When necessary to
emphasize each graded piece of the space, we will write
$\mC^{\prime}_{i,k}(\beta)$ (respectively $\mC^{\dif}_{i,k}(\beta)$, $\mH^\dif_{/i,k}(\beta)$) to denote the homogeneous component
concentrated in $t$-degree $i$ and $q$-degree $k$.

The following theorem is a particular case of the main result of
\cite{KRWitt}, where we have only kept track of the degree two
derivation. The detailed
verification given in Section~\ref{secmarkov}, however, uses the main ideas of
\cite{Roulink} and \cite{RW} and differs from that of
\cite{KRWitt}. This proof serves as the model for the other link
homology theories in this paper.

\begin{thm}\label{thm-untwisted-HOMFLY}
  The slash cohomology of $\beta$ depends only on the braid
  closure of $\beta$ as a framed link in $\R^3$.
\end{thm}

As a convention for the framing number of a braid closure, if a strand
for a component of link is altered as in the left of \eqref{framing},
then we say that the framing of the component is increased by $1$
(with respect to the blackboard framing).  If a strand for a component
of link is altered as in the right of \eqref{framing}, then we say
that the framing of the component is decreased by $1$.

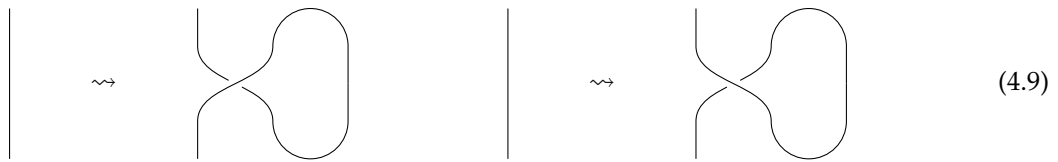
\begin{equation} \label{framing}
  \NB{\tikz[yscale = 1]{\begin{scope}
  \draw (0, -1) -- +(0,2);
\end{scope}
\node at (1.25, 0) {$\rightsquigarrow$};
\begin{scope}[xshift = 2.5cm]
  \draw (0, 1) -- (0, 0.5) .. controls +(0,-0.5) and +(0, 0.5)
  .. (1, -0.5) arc (180:360:0.5) -- (2,0);  
    \fill[white] (0.5, 0) circle (1mm);
  \draw (0, -1) -- (0, -0.5) .. controls +(0,0.5) and +(0, -0.5)
  .. (1, 0.5) arc (180:0:0.5) -- (2,0);
\end{scope}}} \qquad\qquad\qquad
  \NB{\tikz[yscale =-1]{}}
\end{equation}

Denote by $\mathtt{f}_i(L)$ the framing number of the $i$th strand of
a link $L$. Then, under the Reidemeister moves of \eqref{framing},
$\mathtt{f}_i(L)$ adds or subtracts one when changing from the
corresponding left local picture to the right local picture.

\subsection{Markov moves} \label{secmarkov} The usual HOMFLYPT
homologies of two braid compositions $\beta_1\beta_2$ and
$\beta_2\beta_1$ are isomorphic due to the trace-like property of the
usual Hochschild homology functor. The relative Hochschild homology
also remembers the $H^\prime$-action. The following result is then immediate.

\begin{prop}
  Let $\beta_1$ and $\beta_2$ be two braids on $n$ strands.  Then
  $\mC^{\prime}(\beta_1 \beta_2) \cong \mC^{\prime}(\beta_2
  \beta_1)$. \qedhere
\end{prop}

To establish the second Markov move, we start by analyzing the
(relative) Hochschild homology of the  $H^\prime$-complexes of bimodules
associated with the two diagrams of \eqref{markov2pic}.
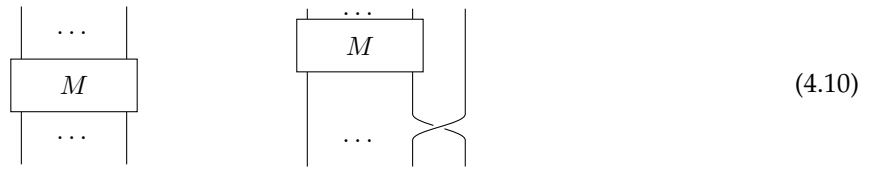
\begin{equation}
  \NB{\tikz[scale=0.7]{\begin{scope}
  \draw (-1, -1.5) -- +(0,3);
  \draw ( 1, -1.5) -- +(0,3);
  \node at (0, -1) {$\dots$};
  \node at (0,  1) {$\dots$};
  \filldraw[fill = white, draw = black] (-1.2, 0.5) rectangle (1.2,
  -0.5) node [pos =0.5] {$M$};
\end{scope}}} \qquad \qquad \qquad 
  \NB{\tikz[scale=0.7]{\begin{scope}
  \draw (-1, -1.5) -- +(0,3);
  \draw ( 2, -1.5) -- (2,-1 ) .. controls +(0, 0.2) and + (0, -0.2)
  .. (1, -0.5) -- (1, 1.5);
  \filldraw[white] (1.5, -0.75) circle (1mm);
  \draw ( 1, -1.5) -- (1,-1 ) .. controls +(0, 0.2) and + (0, -0.2)
  .. (2, -0.5) -- (2, 1.5);
  \node at (0, -1) {$\dots$};
  \node at (0,  1.4) {$\dots$};
  \filldraw[fill = white, draw = black] (-1.2, 1.3) rectangle (1.2,
  0.3) node [pos =0.5] {$M$};
\end{scope}

}}
  \label{markov2pic}
\end{equation}
Let $\Lambda \langle x_{n+1} \rangle$ be the exterior algebra in the
variable $x_{n+1}$.  Recall that $A_n=\myZ[x_1,\ldots,x_n]$ and let
$M \in (A_n,A_n) \#H^\prime \dmod$. Set
$\K_1^\prime=\myZ[x_{n+1}] \otimes \Lambda \langle x_{n+1} \rangle
\otimes \myZ[x_{n+1}] \cong \K_1$.  Letting $\K_n$ denote the Koszul
resolution of $A_n$, we have that
$\K_{n+1}=\K_n \otimes \K_1^\prime$.

As done in \cite{KR3}, we depict the Hochschild homology of $M$ by the
diagram closure
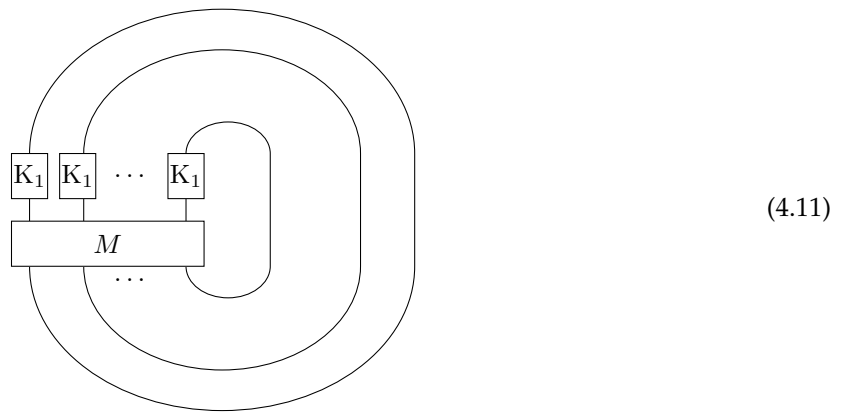
\begin{equation} \label{CMpic} \NB{\tikz[xscale=0.8, yscale =0.6]{\begin{scope}
    \draw (-1.3, -0.5) -- ++(0, 2.5) arc (180:0:3.2) -- ++(0,-2.5) arc
  (360:180:3.2);
  \draw (1.3, -0.5) -- ++(0, 2.5) arc (180:0:0.7) -- ++(0,-2.5) arc
  (360:180:0.7);
  \draw (-0.4, -0.5) -- ++(0, 2.5) arc (180:0:2.3) -- ++(0,-2.5) arc
  (360:180:2.3);
  \node at (0.4, 1.5) {$\dots$};
  \node at (0.4, -0.8) {$\dots$};
  \filldraw[fill=white, draw= black] (-1.6, 1) rectangle (-1, 2)
  node[pos=0.5] {$\K_1$};
  \filldraw[fill=white, draw= black] (-0.8, 1) rectangle (-0.2, 2)
  node[pos=0.5] {$\K_1$};
  \filldraw[fill=white, draw= black] (1, 1) rectangle (1.6, 2) node[pos=0.5] {$\K_1$};
  \filldraw[fill=white, draw= black] (-1.6, -0.5) rectangle (1.6, 0.5)
  node[pos=0.5] {$M$};
\end{scope}}}
\end{equation}
where the single strands connecting the boxes indicate tensor products
over the one-variable polynomial rings labeling those
strands. %

Then the proof of second Markov move essentially reduces to a
computation of the partial Hochschild homology with respect to the
last variable $x_{n+1}$.  This operation is diagramatically
represented in \eqref{picpartialtrace}.
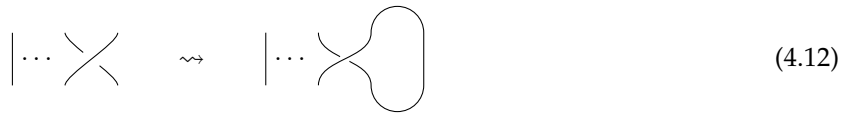
\begin{equation}
  \label{picpartialtrace}
  \NB{\tikz[scale = 0.7]{\begin{scope}
  \draw (1, 0) -- +(0,1);
  \node at (1.5,0.5) {$\dots$};
  \draw (3, 0) .. controls +(0, 0.2) and +(0, -0.2) ..  +(-1,1);
  \fill[white] (2.5, 0.5) circle (2mm);
  \draw (2, 0) .. controls +(0, 0.2) and +(0, -0.2) ..  +(1,1);
\end{scope}}}
  \qquad
  \rightsquigarrow
  \qquad
  \NB{\tikz[scale = 0.7]{%
\begin{scope}[xshift = 2.5cm]
 \draw (-1, -.5) -- +(0,1);
  \node at (-.5,0) {$\dots$};
  \draw  (0, 0.5) .. controls +(0,-0.5) and +(0, 0.5)
  .. (1, -0.5) arc (180:360:0.5) -- (2,0);  
    \fill[white] (0.5, 0) circle (1mm);
  \draw  (0, -0.5) .. controls +(0,0.5) and +(0, -0.5)
  .. (1, 0.5) arc (180:0:0.5) -- (2,0);
\end{scope}}} 
\end{equation}
This requires an analysis of
$\K_1^\prime \otimes_{(\myZ[x_{n+1}],\myZ[x_{n+1}])} T_n^+$ in
Proposition \ref{prop-Markov-II-for-HHH}, which will be the heart of
establishing the invariance under the Markov II moves.

Let $\beta$ be a braid with $n$ strands and $\beta \mathrm{I}$ be the
braid with $n+1$ strands obtained from $\beta$ by adjoining a trivial strand
to the right.
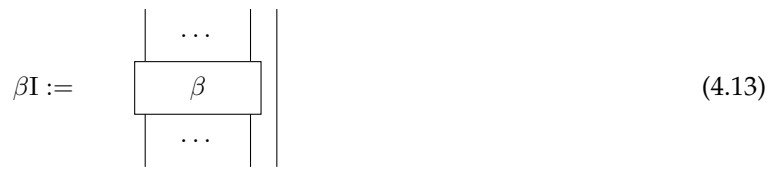
\begin{equation}
  \beta \mathrm{I}:= \quad \quad    \NB{\tikz[scale=0.7]{\begin{scope}
  \draw (-1, -1.5) -- +(0,3);
  \draw ( 1, -1.5) -- +(0,3);
  \draw ( 1.5, -1.5) -- +(0,3);
  \node at (0, -1) {$\dots$};
  \node at (0,  1) {$\dots$};
  \filldraw[fill = white, draw = black] (-1.2, 0.5) rectangle (1.2,
  -0.5) node [pos =0.5] {$\beta$};
\end{scope}}} 
  \label{markov2pic-2}
\end{equation}

\begin{prop}\label{prop-Markov-II-for-HHH}
There are isomorphisms of bigraded $H^\prime$-modules:
  \begin{enumerate}
  \item[(i)]
    $\mC^{\prime}((\beta \mathrm{I})\cdot \sigma_n^+) \cong
    \mC^{\prime}(\beta)^{2x_n}$,
  \item[(ii)]
    $\mC^{\prime}((\beta \mathrm{I})\cdot \sigma_n^-) \cong
    \mC^{\prime}(\beta)^{-2x_n}$,
  \end{enumerate}
  where $\mC^{\prime}(\beta)^{\pm 2x_n}$ denotes overall twistings in
  the $H^\prime$-action on the modules.
\end{prop}

Applying the Bockstein construction and slash homology, the proposition immediately implies the
following.
\begin{cor} 
  Let $\beta$ be a braid. Then there are isomorphisms of bigraded $H$-modules and their slash homology groups
  \begin{enumerate}
  \item[(i)]
    $\mC^{\dif}((\beta \mathrm{I})\cdot \sigma_n^+) \cong
    \mC^{\dif}(\beta)^{2x_n}$ and
    $\mC^{\dif}((\beta \mathrm{I})\cdot \sigma_n^-) \cong
    \mC^{\dif}(\beta)^{-2x_n}$,
  \item[(ii)]
   $\mH^{\dif}_/((\beta \mathrm{I})\cdot \sigma_n^+) \cong
    \mH^{\dif}_/\left((\beta)^{2x_n}\right)$ and
    $\mH^{\dif}_/((\beta \mathrm{I})\cdot \sigma_n^-) \cong
    \mH^{\dif}_/\left((\beta)^{-2x_n}\right)$.
  \end{enumerate}
 Here the ${\pm 2x_n}$ denotes overall twistings in the $H^\prime$-action on
  the corresponding (complexes of) bigraded $H^\prime$-modules.  \hfill \qedhere
\end{cor}

Set $M:=T_\beta$, the chain complex of Soergel bimodules
associated with $\beta$. Both identities of Proposition \ref{prop-Markov-II-for-HHH} are proved in a similar way.  For the first statement of the proposition,
we start by observing that
\begin{equation*}
  \mHH^{\dif}_\bullet((M\otimes \myZ[x_{n+1}])\otimes_{A_{n+1}} T_n^+) =
  \mH_\bullet^v(\K_{n+1} \otimes_{(A_{n+1},A_{n+1})} ((M \otimes \myZ[x_{n+1}])\otimes_{A_{n+1}} T_n^+)),
\end{equation*}
where the (vertical) homology $\mH_\bullet^v$ above is taken with
respect to the differential coming from the Koszul complex $\K_{n+1}$.

We have
\begin{align}\label{eqn-bimod-iso-tensor-with-Tn}
  \K_{n+1} \otimes_{(A_{n+1},A_{n+1})} ((M \otimes \myZ[x_{n+1}]) \otimes_{A_{n+1}} T_n^+) &=
                                                                                            (\K_n \otimes \K_1^\prime) \otimes_{(A_{n+1},A_{n+1})} ((M \otimes \myZ[x_{n+1}]) \otimes_{A_{n+1}} T_n^+) \nonumber \\
                                                                                          & \cong 
                                                                                            \K_n \otimes_{(A_n,A_n)} (M \otimes_{A_n} (\K_1^\prime \otimes_{(\myZ[x_{n+1}],\myZ[x_{n+1}])} T_n^+)).
\end{align}
These isomorphisms, in terms of diagrammatics, can be interpreted as
taking closures of the following diagrammatic equalities:
\begin{equation} \label{CMinductpic}
  \NB{\tikz[scale=0.6]{\begin{scope}
  \draw (0, 0) -- +(0,5);
  \draw (1, 0) -- +(0,5);
  \draw (3, 0) -- +(0,5);
  \draw (5,0) ..controls + (0, 0.3) and +(0,-0.3) .. ++(-1, 1) coordinate
  [pos =0.5] (Xing) -- +(0,4);
  \fill[white] (Xing) circle (1mm);
  \draw (4,0).. controls + (0, 0.3) and +(0,-0.3) .. ++(1, 1) -- +(0,4);
  \filldraw[fill= white] (-0.1, 1.5) rectangle (4.1, 2.5) node[pos=
  0.5] {$M$} coordinate[pos =0.5, yshift = -0.7cm] (d1);
  \filldraw[fill= white] (-0.1, 3.5 ) rectangle (5.1, 4.5) node[pos=
  0.5] {$K_{n+1}$} coordinate[pos =0.5, yshift = 0.7cm, xshift = -0.5cm] (d2);
  \node at (d1) {$\dots$};
  \node at (d2) {$\dots$};
\end{scope}}}
  ~=~
  \NB{\tikz[scale=0.6]{\begin{scope}
  \draw (0, 0) -- +(0,5);
  \draw (1, 0) -- +(0,5);
  \draw (3, 0) -- +(0,5);
  \draw (5,0) ..controls + (0, 0.3) and +(0,-0.3) .. ++(-1, 1) coordinate
  [pos =0.5] (Xing) -- +(0,4);
  \fill[white] (Xing) circle (1mm);
  \draw (4,0).. controls + (0, 0.3) and +(0,-0.3) .. ++(1, 1) -- +(0,4);
  \filldraw[fill= white] (-0.1, 1.5) rectangle (4.1, 2.5) node[pos=
  0.5] {$M$} coordinate[pos =0.5, yshift = -0.7cm] (d1);
  \filldraw[fill= white] (-0.1, 3.5 ) rectangle (4.1, 4.5) node[pos=
  0.5] {$\K_{n}$} coordinate[pos =0.5, yshift = 0.7cm] (d2);
  \filldraw[fill= white] (4.6, 3.5 ) rectangle (5.4, 4.5) node[pos=
  0.5] {$\K_{1}'$};
  \node at (d1) {$\dots$};
  \node at (d2) {$\dots$};
\end{scope}}}
  ~=~
  \NB{\tikz[scale=0.6]{\begin{scope}
  \draw (0, 0) -- +(0,5);
  \draw (1, 0) -- +(0,5);
  \draw (3, 0) -- +(0,5);
  \draw (5,0) ..controls + (0, 0.3) and +(0,-0.3) .. ++(-1, 1) coordinate
  [pos =0.5] (Xing) -- +(0,4);
  \fill[white] (Xing) circle (1mm);
  \draw (4,0).. controls + (0, 0.3) and +(0,-0.3) .. ++(1, 1) -- +(0,4);
  \filldraw[fill= white] (-0.1, 2.4) rectangle (4.1, 3.4) node[pos=
  0.5] {$M$} coordinate[pos =0.5, yshift = -0.7cm] (d1);
  \filldraw[fill= white] (-0.1, 3.6 ) rectangle (4.1, 4.6) node[pos=
  0.5] {$\K_{n}$} coordinate[pos =0.5, yshift = 0.65cm] (d2);
  \filldraw[fill= white] (4.6, 1.2 ) rectangle (5.4, 2.2) node[pos=
  0.5] {$\K_{1}'$};
  \node at (d1) {$\dots$};
  \node at (d2) {$\dots$};
\end{scope}}}.
\end{equation}
Note that
$\K_1^\prime \otimes_{(\myZ[x_{n+1}],\myZ[x_{n+1}])} T_n^+$ is an
$(a,t)$-bicomplex of $(A_{n+1},A_{n+1})$-bimodules
\begin{equation*}
  \NB{\tikz[xscale=4, yscale =1.5]{
      \node (TL) at (0,1) {$q\left({}^{x_{n+1}}B_n^{x_{n+1}}\right)$};
      \node (TR) at (1,1) {$q A_{n+1}^{2x_{n+1}}$};
      \node (BL) at (0,0) {$q^{-3} B_n $};
      \node (BR) at (1,0) {$q^{-3}A_{n+1}$};
    \draw[-to] (TL) -- (TR) node[pos=0.5, above] {$br$};
    \draw[-to] (BL) -- (BR) node[pos=0.5, below] {$br$};
    \draw[-to] (TL) -- (BL) node[pos=0.5, left, scale = 0.7] {$\begin{array}{c}x_{n+1}\otimes 1 \\- 1\otimes{x_{n+1}}\end{array}$};
    \draw[-to] (TR) -- (BR) node[pos=0.5, right] {$0$};
  }}
\end{equation*}
where the object in the upper-right corner is in $t$-degree $0$.
Here the grading shift conventions follow from
\eqref{eqn-elementary-braids-half-grading} together with the collapse
\eqref{eqnatqdegreecollapse}. For ease of notation, we will mostly
ignore them in this section.

It follows that there is a short exact sequence of bicomplexes of
$(A_{n+1},A_{n+1})$-bimodules
\begin{equation*}
  0 \longrightarrow Y_1 \longrightarrow \K_1^\prime \otimes_{(\myZ[x_{n+1}],\myZ[x_{n+1}])} T_n^+
  \longrightarrow Y_2 \longrightarrow 0
\end{equation*}
where the terms of the sequence are defined by
\begin{equation} \label{sesbicomplexes}
  \begin{tikzpicture}[yscale=1.5, xscale=2]
  \node (O1)  at (4.5,1.5) {$0$}; 
  \node (Y1)  at (4.5,0.5) {$Y_1$}; 
  \node (FX)  at (4.5,-1.5) {$\K_1^\prime \otimes_{(\myZ[x_{n+1}],\myZ[x_{n+1}])} T_n^-$};
  \node (Y2)  at (4.5,-3.5) {$Y_2$};
  \node (O2)  at (4.5,-4.5) {$0$};  
 \begin{scope}
  \node (A1)  at (0,1) {$A_{n+1}^{x_n+3x_{n+1}}$}; 
  \node (B1) at (2,1)   {$A_{n+1}^{2x_{n+1}}$}; 
  \node (C1)  at (0,0) {$0$};       
  \node (D1) at (2,0)   {$0$};        
  \draw[-to] (A1)--(B1) node [midway,above] {\tiny{$(x_{n+1} - x_{n})$}};
  \draw[-to] (A1)--(C1) node [midway,left] {};
  \draw[-to] (B1)--(D1) node [midway,right] {};
  \draw[-to] (C1)--(D1) node [midway,below] {};
  \draw [dotted, rounded corners] (-0.5,-0.3) -- ++(3,0) -- ++(0,1.6) -- ++(-3,0) -- cycle;
  \draw [dotted] (2.5, 0.5) -- (Y1);%
 \end{scope}
 \begin{scope}[yshift =0cm]
  \node (A2)  at (0,-1) {${}^{x_{n+1}}B_n^{x_{n+1}}$};   
  \node (B2) at (2,-1)  {$A_{n+1}^{2x_{n+1}}$};  
  \node (C2)  at (0,-2) {$B_n $};   
  \node (D2) at (2,-2)  {$A_{n+1}$};  
  \draw[-to] (A2)--(B2) node [midway,above, scale =0.7] {{$br$}};
  \draw[-to] (A2)--(C2) node [midway,right, scale= 0.7] {$\begin{array}{c}x_{n+1}\otimes 1 \\- 1\otimes x_{n+1}\end{array}$};
  \draw[-to] (B2)--(D2) node [midway,right, scale=0.7] {{$0$}};
  \draw[-to] (C2)--(D2) node [midway,below, scale=0.7] {{$br$}};
  \draw [dotted, rounded corners] (-0.5,-2.3) -- ++(3,0) -- ++(0,1.6) -- ++(-3,0) -- cycle;
  \draw [dotted] (2.5, -1.5) -- (FX);%
 \end{scope}
 \begin{scope}[yshift = 0cm]
  \node (A3)  at (0,-3) {$\widetilde{A}_{n+1}^{x_{n}+x_{n+1}}$};
  \node (B3) at (2,-3)  {$0$};
  \node (C3)  at (0,-4) {$B_n $};
  \node (D3) at (2,-4) {$A_{n+1}$};
  \draw[-to] (A3)--(B3) node [midway,above] {};%
  \draw[-to] (A3)--(C3) node [midway,right, scale=0.7] 
        {$\begin{array}{c}x_{n+1}\otimes 1 \\ - 1\otimes x_{n+1}\end{array}$};
  \draw[-to] (B3)--(D3) node [midway,right, scale=0.7] {{$0$}};
  \draw[-to] (C3)--(D3) node [midway,below , scale=0.7] {{$br$}};
  \draw [dotted, rounded corners] (-0.5,-4.3) -- ++(3,0) -- ++(0,1.6) -- ++(-3,0) -- cycle;
  \draw [dotted] (2.5, -3.5) -- (Y2);
 \end{scope}
\draw[-to] (A1) .. controls +(-0.5,-0.7) and +(-0.5, 0.7) .. (A2) node[midway,
left, scale =0.7]
        {$\begin{array}{c}x_{n+1}\otimes 1 \\ - 1\otimes x_{n}\end{array}$};
\draw[-to] (A2) .. controls +(-0.5,-0.7) and +(-0.5, 0.7) .. (A3) node[midway,
left, scale=0.7] {$\widetilde{br}$};
\draw[-to] (C2) .. controls +(-0.7,-0.7) and +(-0.7, 0.7) .. (C3) node[midway,
left, scale=0.7] {{$\Id$}};
\draw[-to] (B1) .. controls +( 0.5,-0.7) and +( 0.5, 0.7) .. (B2) node[midway,
right, scale=0.7] {{$\Id$}};
\draw[-to] (D2) .. controls +( 0.5,-0.7) and +( 0.5, 0.7) .. (D3) node[midway,
right, scale= 0.7] {{$\Id$}};
\draw[->] (O1) -- (Y1);
\draw[->] (Y1) -- (FX);
\draw[->] (FX) -- (Y2);
\draw[->] (Y2) -- (O2);
\end{tikzpicture}
\end{equation}
where $\widetilde{A}_{n+1}$ is equal to $A_{n+1}$ as a left
$A_{n+1}$-module but the right action of $A_{n+1}$ is twisted by the
permutation $s_n \in S_{n+1}$, and
$\widetilde{br}(a \otimes b)=br(a s_n(b))$.  It is a straightforward
exercise to check that all maps above are equivariant with respect to
the $H^\prime$-action.  We show it for the map
\begin{equation} \label{defphi}
  \phi := x_{n+1} \otimes 1 - 1 \otimes x_n \colon     
  A_{n+1}^{x_n+3x_{n+1}} \longrightarrow {}^{x_{n+1}}B_n^{x_{n+1}}.
\end{equation}
One calculates
\begin{align*}
  \phi(\dif(1))
  & = \phi(x_n+x_{n+1}+2x_{n+1})
  \\
  & =
    x_{n+1}\otimes(x_n+x_{n+1}) + 2x_{n+1}^2\otimes 1 - 1  \otimes
    (x_n(x_n+x_{n+1})) -2 x_{n+1}\otimes x_n \\
  &=
    x_{n+1}\otimes x_{n+1} +2x_{n+1}^2\otimes 1 -x_{n+1}\otimes x_n -1 \otimes
    x_n^2 - 1\otimes x_nx_{n+1}\\
  &= x_{n+1}^2 \otimes 1 -1 \otimes x_n^2 + (x_{n+1} \otimes 1 +
    1\otimes x_{n+1})(x_{n+1} \otimes 1 - 1\otimes x_n )\\
  &= \dif(\phi(1)).
\end{align*}
There is a splitting of the short exact sequence
\eqref{sesbicomplexes} regarded as a short exact sequence of
$(A_n,A_n)$-bimodules, given by
\begin{equation} \label{sesbicomplexessplitting}
  \begin{tikzpicture}[yscale=1.5, xscale=2]
 \begin{scope}[yshift =0cm]
  \node (A2)  at (0,-1) {${}^{x_{n+1}}B_n^{x_{n+1}}$};   
  \node (B2) at (2,-1)  {$A_{n+1}^{2x_{n+1}}$};  
  \node (C2)  at (0,-2) {$B_n $};   
  \node (D2) at (2,-2)  {$A_{n+1}$};  
  \draw[-to] (A2)--(B2) node [midway,above, scale =0.7] {{$br$}};
  \draw[-to] (A2)--(C2) node [midway,right, scale= 0.7] {$\begin{array}{c}x_{n+1}\otimes 1 \\- 1\otimes x_{n+1}\end{array}$};
  \draw[-to] (B2)--(D2) node [midway,right, scale=0.7] {{$0$}};
  \draw[-to] (C2)--(D2) node [midway,below, scale=0.7] {{$br$}};
  \draw [dotted, rounded corners] (-0.5,-2.3) -- ++(3,0) -- ++(0,1.6) -- ++(-3,0) -- cycle;
 \end{scope}
 \begin{scope}[yshift = 0cm]
  \node (A3)  at (0,-3) {$\widetilde{A}_{n+1}^{x_{n}+x_{n+1}}$};
  \node (B3) at (2,-3)  {$0$};
  \node (C3)  at (0,-4) {$B_n $};
  \node (D3) at (2,-4) {$A_{n+1}$};
  \draw[-to] (A3)--(B3) node [midway,above] {};%
  \draw[-to] (A3)--(C3) node [midway,right, scale=0.7] 
        {$\begin{array}{c}x_{n+1}\otimes 1 \\ - 1\otimes x_{n+1}\end{array}$};
  \draw[-to] (B3)--(D3) node [midway,right, scale=0.7] {{$0$}};
  \draw[-to] (C3)--(D3) node [midway,below , scale=0.7] {{$br$}};
  \draw [dotted, rounded corners] (-0.5,-4.3) -- ++(3,0) -- ++(0,1.6) -- ++(-3,0) -- cycle;
 \end{scope}
\draw[to-] (A2) .. controls +(-0.5,-0.7) and +(-0.5, 0.7) .. (A3) node[midway,
left, scale=0.9] {$\theta$};
\draw[to-] (C2) .. controls +(-0.7,-0.7) and +(-0.7, 0.7) .. (C3) node[midway,
left, scale=0.7] {{$\Id$}};
\draw[to-] (D2) .. controls +( 0.5,-0.7) and +( 0.5, 0.7) .. (D3) node[midway,
right, scale= 0.7] {{$\Id$}};
\end{tikzpicture}  
\end{equation}
where
\begin{equation*}
  \theta(f(x_1,\ldots,x_{n-1})x_n^i x_{n+1}^j)
  =f(x_1,\ldots,x_{n-1}) x_n^i \otimes x_{n}^j.
\end{equation*}
We briefly explain why $\theta$ is a well-defined bimodule
homomorphism. By definition
$\theta(x_n^i x_{n+1}^j)=x_n^i \otimes x_{n}^j$.  Note that
$\theta(x_n^i x_{n+1}^j)=\theta(x_n^i \cdot x_{n+1}^j)=x_n^i
\theta(x_{n+1}^j)=x_n^i \otimes x_{n}^j $ where we viewed $x_n^i$ as
acting on the left of $x_{n+1}^j$.  Similarly,
$\theta(x_n^i x_{n+1}^j)=\theta(x_n^i \cdot
x_{n}^j)=\theta(x_n^i)x_n^j= x_n^i \otimes x_{n}^j $ where we viewed
$x_n^j$ as acting on the right of $x_{n}^i$.  However, a direct
computation shows that $\theta$ does not intertwine the $H^\prime$-actions,
and thus this splitting is not $H^\prime$-equivariant.

The short exact sequence \eqref{sesbicomplexes} plugged back into
\eqref{eqn-bimod-iso-tensor-with-Tn} gives us a short exact sequence
\begin{equation}\label{eqn-split-exact-sequence}
  0\rightarrow  \K_n \otimes_{(A_n,A_n)} (M \otimes_{A_n} Y_1) \rightarrow  \K_n \otimes_{(A_n,A_n)} ((M \otimes \myZ[x_{n+1}])\otimes_{A_{n+1}} T_n^+) \rightarrow
  \K_n \otimes_{(A_n,A_n)} (M \otimes_{A_n} Y_2) \rightarrow 0,
\end{equation}
which is split as a sequence of $(A_n,A_n)$-bimodules. Taking homology
with respect to the vertical differentials gives rise to a long exact
sequence
\begin{equation*}
  \cdots
  \rightarrow \mH_i^v(\K_n \otimes_{(A_n,A_n)} (M \otimes_{A_n} Y_1)) \rightarrow \mHH_i^{\dif}((M\otimes \myZ[x_{n+1}]) \otimes_{A_{n+1}} T_n^+) \rightarrow 
  \mH_i^v(\K_n \otimes_{(A_n,A_n)} (M \otimes_{A_n} Y_2))
  \rightarrow \cdots.
\end{equation*}
Due to the splitting exactness of \eqref{eqn-split-exact-sequence},
the long exact sequence breaks up into short exact sequences of the
form
\begin{equation}\label{eqn-ses-in-homology}
  0
  \rightarrow \mH_i^v(\K_n \otimes_{(A_n,A_n)} (M \otimes_{A_n} Y_1)) \rightarrow \mHH_i^{\dif}((M\otimes \myZ[x_{n+1}]) \otimes_{A_{n+1}} T_n^+) \rightarrow 
  \mH_i^v(\K_n \otimes_{(A_n,A_n)} (M \otimes_{A_n} Y_2))
  \rightarrow 0,
\end{equation}
one for each $i\in \Z$. When forgetting about the $H^\prime$-module
structure, the splitting \eqref{sesbicomplexessplitting} gives a
decomposition
\begin{equation}\label{eqndirectsumHH}
  \mHH_\bullet^{\dif}((M\otimes \myZ[x_{n+1}]) \otimes_{A_{n+1}} T_n^+) \cong
  \mH_\bullet^v(\K_n \otimes_{(A_n,A_n)} (M \otimes_{A_n} Y_1)) \oplus
  \mH_\bullet^v(\K_n \otimes_{(A_n,A_n)} (M \otimes_{A_n} Y_2)) .
\end{equation}

We will need the following further analysis about the term involving
$Y_2$.

\begin{lem}\label{lem-technical-1}
  There is an $H^\prime$-equivariant short exact sequence of bimodules
  \[
    0\lra Y_2^{\prime \prime} \lra Y_2\stackrel{\psi}{\lra} Y_2^\prime
    \lra 0,
  \]
  where the surjective map $\psi$ is given by
  \begin{equation} \label{bicomplexmor}
  \NB{\tikz[xscale=2.5, yscale =1.5]{
      \node (TL) at (0,1) {$\widetilde{A}_{n+1}^{x_n+x_{n+1}} $};
      \node (TR) at (1,1) {$0$};
      \node (BL) at (0,0) {$B_n $};
      \node (BR) at (1,0) {$A_{n+1}$};
    \draw[-to] (TL) -- (TR);
    \draw[-to] (BL) -- (BR) node[pos=0.5, above] {$br$};
    \draw[-to] (TL) -- (BL) node[pos=0.5, left, font=\tiny] {$\begin{array}{c}x_{n+1}\otimes 1 \\- 1\otimes{x_{n+1}}\end{array}$};
    \draw[-to] (TR) -- (BR);
    \node (Y2) at (-1.1, 0.5) {$\psi:Y_2:=$};
    \begin{scope}[xshift =2cm]
      \node (TL2) at (0,1) {$0$};
      \node (TR2) at (1,1) {$0$};
      \node (BL2) at (0,0) {$A_{n+1}$};
      \node (BR2) at (1,0) {$A_{n+1}$};
    \draw[-to] (TL2) -- (TR2);
    \draw[-to] (BL2) -- (BR2) node[pos=0.5, above] {$\Id$};
    \draw[-to] (TL2) -- (BL2);
    \draw[-to] (TR2) -- (BR2);
    \node[left] (Yp2) at (1.7, 0.5) {$=:Y'_2,$};
  \end{scope}
  \draw[-to] (BL) .. controls +(+1, -0.5) and +(-1, -0.5) .. (BL2)
  node[pos=0.5, below] {$br$};
  \draw[-to] (BR) .. controls +(+1, -0.5) and +(-1, -0.5) .. (BR2)
  node[pos=0.5, below] {$\Id$};
}}
  \end{equation}
  and the kernel of $\psi$ is given by
  \begin{equation}
  \NB{\tikz[xscale=2.5, yscale =1.5]{
      \node (TL) at (0,1) {$ \widetilde{A}_{n+1}^{x_n+x_{n+1}}$};
      \node (BL) at (0,0) {$ \widetilde{A}_{n+1}^{x_n+x_{n+1}}$};
      \node (BR) at (-0.7,0.5) {$Y_2^{\prime \prime} :=$};
      \draw [-to] (TL)-- (BL) node[pos =.5, right] {$\Id$};
    }}. 
  \end{equation}
\end{lem}
\begin{proof}
  This is a straightforward exercise.
\end{proof}

Now let us turn on the total differential $d_T$ on
$\mHH_\bullet^{\dif}((M\otimes \myZ[x_{n+1}]) \otimes_{A_{n+1}} T_n^+)
$. We have the following.

\begin{lem}\label{lem-technical-2}
  There is a short exact exact sequence of $H^\prime$-equivariant graded modules
  with respect to the total differential $d_T$:
  \begin{equation}\label{eqn-dc-filtration}
    0\rightarrow
    \mH_\bullet^v(\K_n \otimes_{(A_n,A_n)} (M \otimes_{A_n} Y_2))
    \rightarrow
    \mHH_\bullet^{\dif}((M\otimes \myZ[x_{n+1}]) \otimes_{A_{n+1}} T_n^+) 
    \rightarrow
    \mH_\bullet^v(\K_n \otimes_{(A_n,A_n)} (M \otimes_{A_n} Y_1)) 
    \rightarrow 0.
  \end{equation}
\end{lem}

\begin{proof}
Since the topological differential $d_t$ preserves the $a$ and $q$ degrees, we only need to analyze the effect of the Cautis differential $d_C$.

Let us examine how $d_C$ lifts to the Koszul
  complex
  \begin{equation*}
    \K_{n+1}=\K_n\otimes \K_1^\prime=\K_n \otimes \myZ[x_{n+1}] \otimes \Lambda \langle dx_{n+1} \rangle \otimes \myZ[x_{n+1}].
  \end{equation*}
  Under this identification, the differential $d_C$, which depends on
  the number of polynomial variables $n+1$ and will be written as
  $d_{n+1}$, may be inductively defined as
  \begin{equation} \label{Drecursive} d_{n+1}=d_n \otimes\mathrm{Id} +
    \mathrm{Id}\otimes d_1^{\prime}.
  \end{equation}
  Here we have set
  $\K_1^\prime=\myZ[x_{n+1}] \otimes \Lambda \langle dx_{n+1} \rangle
  \otimes \myZ[x_{n+1}] $ equipped with its own part of the Cautis
  differential
  \[ {d}^\prime_{1}:=x_{n+1}^2 \otimes \iota_{\frac{\partial}{\partial
        x_{n+1}} }\otimes 1 .
  \]
  The notation $\iota_{\frac{\partial}{\partial x_{n+1}}}$ denotes the
  contraction of $dx_{n+1}$ with $\frac{\partial}{\partial x_{n+1}}$.

  Now since $d_{n+1}$ acts on the $Y_1$ and $Y_2$ tensor factors via
  $d_1^\prime$, it suffices to verify that $d_1^\prime$ preserves the
  submodule arising from $Y_2$ and presents the part arising from
  $Y_1$ as a quotient. To do this, we re-examine the sequence
  \eqref{sesbicomplexes} under vertical (Hochschild) homology. The
  part $Y_2$, under vertical homotopy equivalence, contributes to the
  horizontal (topological) complex (see  \eqref{bicomplexmor})
  \begin{subequations}
    \begin{equation}
      {Y}^\prime_2:=\left(A_{n+1} \xrightarrow{d_t=\mathrm{Id}} A_{n+1}\right)
    \end{equation}
    sitting entirely in Hochschild degree $0$. Likewise, the part
    $Y_1$ contributes to the horizontal complex
    \begin{equation}
      {Y}^\prime_1:= \left(A_{n+1}^{x_n+3x_{n+1}} \xrightarrow{d_t=(x_{n+1}-x_n)}A_{n+1}^{2x_{n+1}} \right)
    \end{equation}
  \end{subequations}
  sitting entirely in Hochschild degree $1$.  Since $d_1^{\prime}$
  decreases Hochschild degree by one, $ {Y}^\prime_1 $ must be
  preserved under $d_1^\prime$, acting upon it trivially, and
  $ {Y}^\prime_2 $ is equipped with the (zero) quotient action by
  $d_1^\prime$. The claim follows from this degree reasoning.
\end{proof}

Taking homology with respect to $d_T$ of the short exact sequence
\eqref{eqn-dc-filtration} gives us a long exact sequence
\begin{equation}\label{eqnHCHHles}
\NB{      \begin{tikzpicture}[xscale=6.9, yscale=0.5]
        \node (dot1) at (0.4,1) {$\cdots$};
        \node (HTHjY2) at (1,1) {$\mHT_j ( \mH_\bullet ^v(\K_n \otimes_{(A_n,A_n)} (M \otimes_{A_n} Y_2)))$};
        \node (HTHHj) at (2,1) {$\mHT_{j}(\mHH_\bullet^{\dif} ((M\otimes \myZ[x_{n+1}]) \otimes_{A_{n+1}} T_n^+))$};
        \node (HTHjY1) at (1,-1) {$\mHT_j(\mH_\bullet ^v(\K_n \otimes_{(A_n,A_n)} (M\otimes_{A_n} Y_1))) $};
        \node (HTHj1Y2) at (2,-1) {$\mHT_{j-1} ( \mH_{\bullet}^v(\K_n\otimes_{(A_n,A_n)} (M \otimes_{A_n} Y_2)) )  $};
        \node (dot2) at (2.7,-1) {$\cdots$};
        \draw[-to] (dot1) -- (HTHjY2);
        \draw[-to] (HTHjY2) -- (HTHHj);
        \draw[-to] (HTHHj) -- ++(0.5,0) arc (90:-90:0.071 and 0.5) -- ++(-2,0)
        arc (90:270:0.071 and 0.5) -- (HTHjY1);
        \draw[-to] (HTHjY1) -- (HTHj1Y2);
        \draw[-to] (HTHj1Y2) -- (dot2);
      \end{tikzpicture}}\ .
\end{equation}

\begin{lem}\label{lem-technical-3}
  There is an isomorphism of bigraded
  $H^\prime$-modules
  \begin{equation}\label{eqnHCHHfromY1}
    \mHT_\bullet ( \mH_\bullet ^v(\K_n \otimes_{(A_n,A_n)} (M \otimes_{A_n} Y_1))) \cong \mHT_\bullet (\mHH_\bullet^{\dif} ((M\otimes \myZ[x_{n+1}]) \otimes_{A_{n+1}} T_n^+)).
  \end{equation}
\end{lem}
\begin{proof}
As a chain complex of $H^\prime$-modules, \eqref{eqnHCHHles} splits into short exact sequences
\begin{equation}\label{eqnHCHHcomplex}
\NB{      \begin{tikzpicture}[xscale=6.9, yscale=0.5]
        \node (zero1) at (0.4,1) {$0$};
        \node (HTHjY2) at (1,1) {$\mHT_j ( \mH_\bullet ^v(\K_n \otimes_{(A_n,A_n)} (M \otimes_{A_n} Y_2)))$};
        \node (HTHHj) at (2,1) {$\mHT_{j}(\mHH_\bullet^{\dif} ((M\otimes \myZ[x_{n+1}]) \otimes_{A_{n+1}} T_n^+))$};
        \node (HTHjY1) at (1,-1) {$\mHT_j(\mH_\bullet ^v(\K_n \otimes_{(A_n,A_n)} (M\otimes_{A_n} Y_1))) $};
        \node (zero2) at (2,-1) {$0$};
        \draw[-to] (zero1) -- (HTHjY2);
        \draw[-to] (HTHjY2) -- (HTHHj);
        \draw[-to] (HTHHj) -- ++(0.5,0) arc (90:-90:0.071 and 0.5) -- ++(-2,0)
        arc (90:270:0.071 and 0.5) -- (HTHjY1);
        \draw[-to] (HTHjY1) -- (zero2);
      \end{tikzpicture}}\ .
  \end{equation}
since the splitting map \eqref{sesbicomplexessplitting} intertwines $H^\prime$-actions after taking vertical homology. Note that the only non $H^\prime$-intertwining part $\theta$
  becomes zero after taking vertical homology.  Furthermore, with respect to
  the total differential, we have that
\begin{equation}
    \mHT_\bullet(\mH_\bullet ^v(\K_n \otimes_{(A_n,A_n)} (M \otimes_{A_n} Y_2))) \cong 
    \mHT_\bullet(\mH_\bullet ^v(\K_n \otimes_{(A_n,A_n)} (M \otimes_{A_n} Y_2^\prime)))
\end{equation}
is a contractible complex of $H^\prime$-modules. The desired isomorphism follows.
\end{proof}

Let us now finish the proof of the proposition.

\begin{proof}[Proof of Proposition \ref{prop-Markov-II-for-HHH}]
  By Lemma \ref{lem-technical-3}, the term
  $\mC^\prime((\beta \mathrm{I})\cdot \sigma_n^+)$ is isomorphic to the
  homology of the bigraded $H^\prime$-complex
  \[
  \mHT_\bullet\left(  \mH_\bullet ^v(\K_n \otimes_{(A_n,A_n)} (M
      \otimes_{A_n} Y_1)), d_T \right).
\]

  Consider the short exact sequence involving $Y_1$:
  \begin{equation}
    0\lra  A_{n+1}^{x_n+3x_{n+1}} \xrightarrow{x_{n+1}-x_n} A_{n+1}^{2x_{n+1}} \lra A_n^{2x_n} \lra 0 .
  \end{equation}
  Again, the sequence splits as an $H^\prime$-equivariant sequence of
  $A_n$-modules (a splitting can be given, for instance, by
  identifying $A_n^{2x_n}$ with the subalgebra generated by
  $x_1,\dots, x_{n-1}, x_{n+1}$, matching the $H^\prime$-twistings).  This,
  in turn, shows that
  \begin{equation}
    \mHT_\bullet\left(  \mH_\bullet ^v(\K_n \otimes_{(A_n,A_n)} (M \otimes_{A_n} Y_1)), d_T \right)\cong 
    \mHT_\bullet ( \mH_\bullet ^v(\K_n \otimes_{(A_n,A_n)} M^{2x_n}  ))
    \cong \mC^{\prime}(\beta)^{2x_n}.
  \end{equation}
  The first statement of the proposition now follows.

  The second statement is proved in an analogous way.  We just sketch
  the adjustments needed and also refer the reader to the second half
  of the proof of \cite[Proposition 4.12]{QiSussanLink} for
  comparison.

  The computation of
  $\mHH_\bullet^{\dif}((M\otimes \myZ[x_{n+1}]) \otimes_{A_{n+1}}
  T_n^-) $ is very similar. We outline the necessary changes for the
  interested reader.

  Again, first note that
  $\K_1^\prime \otimes_{(\myZ[x_{n+1}],\myZ[x_{n+1}])} T_n^-$
  is a bicomplex of $(A_{n+1},A_{n+1})$-bimodules
  \begin{equation}
        \NB{\tikz[xscale=4, yscale =1.8]{
      \node (TL) at (0,1) {$q^7 A_{n+1}^{2x_{n+1}}$};
      \node (TR) at (1,1) {$q^5 ({}^{x_{n+1}}B_{n}^{-x_n})$};
      \node (BL) at (0,0) {$q^3 A_{n+1}$};
      \node (BR) at (1,0) {$q B_{n}^{-(x_n+x_{n+1})} $.};
    \draw[-to] (TL) -- (TR) node[pos=0.5, above] {$rb$};
    \draw[-to] (BL) -- (BR) node[pos=0.5, below] {$rb$};
    \draw[-to] (TL) -- (BL) node[pos=0.5, left] {$0$};
    \draw[-to] (TR) -- (BR) node[pos=0.5, right, scale = 0.7] {$\begin{array}{c}x_{n+1}\otimes 1 \\- 1\otimes{x_{n+1}}\end{array}$};
  }}
  \end{equation}
Note that the object in the bottom-right corner is in $t$-degree $0$.  Ignore the grading shifts for now for ease of notation. There is a
  short exact sequence of bicomplexes of $(A_{n+1},A_{n+1})$-bimodules
  \[
    0\lra Z_1 \lra \K_1^\prime
    \otimes_{(\myZ[x_{n+1}],\myZ[x_{n+1}])} T_n^- \lra Z_2\lra
    0,
  \]
  whose terms are defined by
  \begin{equation} \label{sesbicomplexes2}
\NB{
  \begin{tikzpicture}[yscale=1.5, xscale=2]
  \node (O1)  at (4.5,1.5) {$0$}; 
  \node (Y1)  at (4.5,0.5) {$Z_1$}; 
  \node (FX)  at (4.5,-1.5) {$\K_1^\prime \otimes_{(\myZ[x_{n+1}],\myZ[x_{n+1}])} T_n^+$};
  \node (Y2)  at (4.5,-3.5) {$Z_2$};
  \node (O2)  at (4.5,-4.5) {$0$};  
 \begin{scope}
  \node (A1)  at (0,1) {$A_{n+1}^{2x_{n+1}}$}; 
  \node (B1) at (2,1)   {$A_{n+1}^{2x_{n+1}}$}; 
  \node (C1)  at (0,0) {$0$};       
  \node (D1) at (2,0)   {$0$};        
  \draw[-to] (A1)--(B1) node [midway,above, scale=0.7] {{$\Id$}};
  \draw[-to] (A1)--(C1) node [midway,left] {};
  \draw[-to] (B1)--(D1) node [midway,right] {};
  \draw[-to] (C1)--(D1) node [midway,below] {};
  \draw [dotted, rounded corners] (-0.5,-0.3) -- ++(3,0) -- ++(0,1.6) -- ++(-3,0) -- cycle;
  \draw [dotted] (2.5, 0.5) -- (Y1);
 \end{scope}
 \begin{scope}[yshift =0cm]
  \node (A2)  at (0,-1) {$A_{n+1}^{2x_{n+1}}$};   
  \node (B2) at (2,-1)  {${}^{x_{n+1}}B_n^{-x_{n}} $};  
  \node (C2)  at (0,-2) {$A_{n+1} $};   
  \node (D2) at (2,-2)  {$B_n^{-x_{n}-x_{n+1}}$};  
  \draw[-to] (A2)--(B2) node [midway,above, scale =0.7] {{$rb$}};
  \draw[-to] (A2)--(C2) node [midway,right, scale= 0.7] {$0$};
  \draw[-to] (B2)--(D2) node [midway,left, scale=0.7] {$\begin{array}{c}x_{n+1}\otimes 1 \\- 1\otimes x_{n+1}\end{array}$};
  \draw[-to] (C2)--(D2) node [midway,below, scale=0.7] {{$rb$}};
  \draw [dotted, rounded corners] (-0.5,-2.3) -- ++(3,0) -- ++(0,1.6) -- ++(-3,0) -- cycle;
  \draw [dotted] (2.5, -1.5) -- (FX);
 \end{scope}
 \begin{scope}[yshift = 0cm]
  \node (A3)  at (0,-3) {$0 $};
  \node (B3) at (2,-3)  {$\widetilde{A}_{n+1}$};
  \node (C3)  at (0,-4) {$A_{n+1}$};
  \node (D3) at (2,-4) {$B_n^{-(x_{n+1} +x_n)}$};
  \draw[-to] (A3)--(B3) node [midway,above] {};%
  \draw[-to] (A3)--(C3) node [midway,left, scale=0.7] {};
   
  \draw[-to] (B3)--(D3) node [midway,left, scale=0.7]      {$\begin{array}{c}x_{n+1}\otimes 1 \\ - 1\otimes x_{n+1}\end{array}$};
  \draw[-to] (C3)--(D3) node [midway,below , scale=0.7] {{$rb$}};
  \draw [dotted, rounded corners] (-0.5,-4.3) -- ++(3,0) -- ++(0,1.6) -- ++(-3,0) -- cycle;
  \draw [dotted] (2.5, -3.5) -- (Y2);
 \end{scope}
\draw[-to] (A1) .. controls +(-0.5,-0.7) and +(-0.5, 0.7) .. (A2) node[midway,
left, scale =0.7] {$\Id$};
 \draw[-to] (B2) .. controls +(0.7,-0.7) and +(0.7, 0.7) .. (B3) node[midway,
 right, scale=0.7] {$\widetilde{br}$};
\draw[-to] (C2) .. controls +(-0.5,-0.7) and +(-0.5, 0.7) .. (C3) node[midway,
left, scale=0.7] {{$\Id$}};
\draw[-to] (B1) .. controls +( 0.5,-0.7) and +( 0.5, 0.7) .. (B2) node[midway,
right, scale=0.7] {{$rb$}};
\draw[-to] (D2) .. controls +( 0.5,-0.7) and +( 0.5, 0.7) .. (D3) node[midway,
right, scale= 0.7] {{$\Id$}};
\draw[->] (O1) -- (Y1);
\draw[->] (Y1) -- (FX);
\draw[->] (FX) -- (Y2);
\draw[->] (Y2) -- (O2);
\end{tikzpicture}} \ .
  \end{equation}

  As in the previous part, there is a splitting of bicomplexes of
  $(A_n,A_n)$-bimodules given by
  \begin{equation} \label{bicomplexsplitting'}
\NB{
  \begin{tikzpicture}[yscale=1.5, xscale=2]
 \begin{scope}[yshift =0cm]
  \node (A2)  at (0,-1) {$A_{n+1}^{2x_{n+1}}$};   
  \node (B2) at (2,-1)  {${}^{x_{n+1}}B_n^{-x_{n}} $};  
  \node (C2)  at (0,-2) {$A_{n+1} $};   
  \node (D2) at (2,-2)  {$B_n^{-x_{n}-x_{n+1}}$};  
  \draw[-to] (A2)--(B2) node [midway,above, scale =0.7] {{$rb$}};
  \draw[-to] (A2)--(C2) node [midway,right, scale= 0.7] {$0$};
  \draw[-to] (B2)--(D2) node [midway,left, scale=0.7] {$\begin{array}{c}x_{n+1}\otimes 1 \\- 1\otimes x_{n+1}\end{array}$};
  \draw[-to] (C2)--(D2) node [midway,below, scale=0.7] {{$rb$}};
  \draw [dotted, rounded corners] (-0.5,-2.3) -- ++(3,0) -- ++(0,1.6) -- ++(-3,0) -- cycle;
 \end{scope}
 \begin{scope}[yshift = 0cm]
  \node (A3)  at (0,-3) {$0 $};
  \node (B3) at (2,-3)  {$\widetilde{A}_{n+1}$};
  \node (C3)  at (0,-4) {$A_{n+1}$};
  \node (D3) at (2,-4) {$B_n^{-(x_{n+1} +x_n)}$};
  \draw[-to] (A3)--(B3) node [midway,above] {};%
  \draw[-to] (A3)--(C3) node [midway,left, scale=0.7] {};
   
  \draw[-to] (B3)--(D3) node [midway,left, scale=0.7]      {$\begin{array}{c}x_{n+1}\otimes 1 \\ - 1\otimes x_{n+1}\end{array}$};
  \draw[-to] (C3)--(D3) node [midway,below , scale=0.7] {{$rb$}};
  \draw [dotted, rounded corners] (-0.5,-4.3) -- ++(3,0) -- ++(0,1.6) -- ++(-3,0) -- cycle;
 \end{scope}
 \draw[to-] (B2) .. controls +(0.7,-0.7) and +(0.7, 0.7) .. (B3) node[midway,
 right, scale=0.7] {$\phi$};
\draw[to-] (C2) .. controls +(-0.5,-0.7) and +(-0.5, 0.7) .. (C3) node[midway,
left, scale=0.7] {{$\Id$}};
\draw[to-] (D2) .. controls +( 0.5,-0.7) and +( 0.5, 0.7) .. (D3) node[midway,
right, scale= 0.7] {{$\Id$}};
\end{tikzpicture}}
  \end{equation}
  where $\phi$ was defined in \eqref{defphi}.  Thus we get short exact sequences
  of the form
  \begin{equation}\label{eqn-ses-in-homology-prime}
    0
    \rightarrow \mH_i^v(\K_n \otimes_{(A_n,A_n)} (M \otimes_{A_n} Z_1)) \rightarrow \mHH_i^{\dif}((M\otimes \myZ[x_{n+1}]) \otimes_{A_{n+1}} T_n^-) \rightarrow 
    \mH_i^v(\K_n \otimes_{(A_n,A_n)} (M \otimes_{A_n} Z_2))
    \rightarrow 0
  \end{equation}
  for each $i\in \Z$.
  Taking horizontal homology for this short exact sequence gives us a
  long exact sequence. However, since the term involving $Z_1$ will
  only contribute to a contractible complex of $H^\prime$-modules, we can
  safely remove it, and focus on the homology contribution arising
  from $Z_2$.

  There is a morphism of $(a,t)$-bicomplexes
  \begin{equation} \label{bicomplexmor'}
    \NB{\tikz[xscale=3, yscale =1.8]{
      \node (TL) at (0,1) {$0 $};
      \node (TR) at (1,1) {$\widetilde{A}_{n+1}^{x_n+x_{n+1}}$};
      \node (BL) at (0,0) {$B_n^{-(x_n+x_{n+1})}$};
      \node (BR) at (1,0) {$A_{n+1}$};
    \draw[-to] (TL) -- (TR);
    \draw[-to] (BL) -- (BR) node[pos=0.5, above] {$rb$};
    \draw[-to] (TL) -- (BL);
    \draw[-to] (TR) -- (BR) node[pos=0.5, right, scale=0.7] {$\begin{array}{c}x_{n+1}\otimes 1 \\- 1\otimes{x_{n+1}}\end{array}$};
    \node (Z2) at (-0.7, 0.5) {$Z_2:=$};
    \begin{scope}[xshift =2cm]
      \node (TL2) at (0,1) {$0$};
      \node (TR2) at (1,1) {$0$};
      \node (BL2) at (0,0) {$A_{n+1}$};
      \node (BR2) at (1,0) {$A_{n+1}^{-(x_n+x_{n+1})}$};
    \draw[-to] (TL2) -- (TR2);
    \draw[-to] (BL2) -- (BR2) node[pos=0.5, above, scale=0.7] {$x_{n+1}-x_n$};
    \draw[-to] (TL2) -- (BL2);
    \draw[-to] (TR2) -- (BR2);
    \node[left] (Zp2) at (1.7, 0.5) {$=:Z'_2,$};
  \end{scope}
  \draw[-to] (BL) .. controls +(+1, -0.5) and +(-1, -0.5) .. (BL2)
  node[pos=0.5, below] {$\Id$};
  \draw[-to] (BR) .. controls +(+1, -0.5) and +(-1, -0.5) .. (BR2)
  node[pos=0.5, below] {$br$};
}}
  \end{equation}
  whose kernel is a vertical complex connected by the identity map.
  Thus there is an isomorphism of bigraded $H^\prime$-modules
\[\mHT_\bullet(\mHH_\bullet^{\dif} ((M\otimes \myZ[x_{n+1}]) \otimes_{A_{n+1}} T_n^-) )\cong \mHT_\bullet ( \mH_i^v((M\otimes \myZ[x_{n+1}])\otimes Z_2^\prime))\cong \mHT_\bullet(\mHH_\bullet^{\dif} (M^{-2x_n})).\]
The result then follows.
\end{proof}

\subsection{Unlinks and unframed invariants}\label{subsecHOMFLYunlink}
In this section, we compute $\mC^{\prime}$, $\mC^{\dif}$ and $\mH^{\dif}_/$ for the identity element of
the braid group $\mathrm{Br}_n$.

For the unknot, recall from the previous section that the Koszul resolution
$\K_1$ of $\Z[x]$ as bimodules is given by
\begin{equation*}
  \NB{\tikz[xscale=4.5, yscale =1.5]{
      \node (TL) at (0,1) {$q^2a\Z[x]^{x} \otimes \Z[x]^{x}$};
      \node (TR) at (1,1) {$\Z^{}[x] \otimes \Z^{}[x]$};
    \draw[-to] (TL) -- (TR) node[pos=0.5, above, scale=0.8] {$x\otimes 1-1\otimes x$};
  }}
  \ .
\end{equation*}
Tensoring this complex with $\Z[x]$ as a bimodule yields
\begin{equation*}
  \NB{\tikz[xscale=3, yscale =1.5]{
      \node (TL) at (0,1) {$q^2a\Z[x]^{2x}$};
      \node (TR) at (1,1) {$\Z[x]$};
    \draw[-to] (TL) -- (TR) node[pos=0.5, above] {$0$};
  }}
  \ .
\end{equation*}
Thus the homology of the unknot (up to shift) is identified with the
triply-graded $H^\prime$-module ($t$-degree zero)
\begin{equation*}
  q^2 a\Z[x]^{2x} \oplus  \Z[x]
  \ .
\end{equation*}
Impose the Cautis differential $d_C$ ($d_t=0$ for the unknot) on the above module and collapse
the $a$ degree according to \eqref{eqnatqdegreecollapse}. Then $d_C$
sends an element $aq^2 x^i \mapsto x^{i+2} $.  Thus after taking
homology with respect to $d_C$, we get that 
\[
\mC^\prime(\mathrm{I})\cong q^{-1}\Z\langle 1, x \rangle 
\] 
with the trivial
derivation $\dif$. Here the $q^{-1}$ factor comes from the overall
$q$-degree shift in Definition \ref{def-HHH}.

Applying the Bockstein construction and taking slash homology, we obtain
\[
\mC^\dif(\mathrm{I})\cong q^{-1}\F_p\langle 1, x \rangle ,\quad \quad
\mH_/^\dif(\mathrm{I})\cong q^{-1}\F_p\langle 1, x \rangle.
\] 

More generally, via the Koszul complex $\K_n=\K_1^{\otimes n}$, we
have that the homology (before taking $d_C$) of the $n$-component
unlink $L_0$ is equal to
\begin{equation}
  q^{-n}\bigotimes_{i=1}^n \left(  aq^2 \Z[x_i]^{2x_i} \oplus \Z[x_i]  \right).
\end{equation}
Then upon taking homology with respect to $d_C$, we obtain that the slash
homology of the unlink is isomorphic to
\[
  q^{-n}\F_p[x_1,\ldots,x_n] / (x_1^2, \ldots, x_n^2)
\]
with the trivial $p$-differential.

\begin{rem}
  Using the computation of the unlink above, we could introduce a
  twisting framing factor on the homology of a link.  As done in
  \cite{QiSussanLink}, using this framing factor, one could modify the
  definition of $\mH^{\dif}_/$ and obtain a homological invariant of
  unframed links. In \cite[Section 4.3 and 5.3]{QiSussanLink}, it was
  shown that this invariant of a link is represented by a
  finite-dimensional $p$-complex, whose symbol in $K_0(H\udmod)$ is
  equal to the Jones polynomial specialized at a $2p$th root of unity.
\end{rem}

\begin{rem}
  This approach to a categorification of the Jones polynomial, at
  generic values of $q$, was first developed by Cautis
  \cite{Cautisremarks}.  We have followed the exposition %
  of \cite{RW} and the closely related approach of Queffelec,
  Rose, and Sartori \cite{QRS}.
\end{rem}

\section{Braid diagram moves} 
\label{sec:topred}
This section constitutes the technical heart of the  construction of categorical colored link invariants at roots of unity, and contains useful tools to help prove our main theorem. 

Section \ref{sec:blist} contains an algebro-diagrammatic reduction, which we call the \emph{blist hypothesis}, that allows us to deduce the colored invariance from the uncolored case described in Section \ref{uncolored}. In Section \ref{sec:pitchfork-reduction} we move on to show that in any algebraic category, where the blist hypothesis holds, and is the target category of a functor from colored braids and trivalent graphs, colored braid invariance can be reduced to checking eight moves of sliding pitchforks. Then in Section \ref{sec:pitchforks} we show that the pitchforks sliding moves do hold in the relative homotopy category of singular Soergel bimodules. Finally, in Sections \ref{sec:sliding-twists} and \ref{sec:forktwists}, we explore how twistings of $H^\prime$-actions interact with crossings and pitchforks.

For simplicity, we will often abbreviate the polynomial ring $A=A_{\underline{k}}$ when $\underline{k}$ is clear from the context.

\subsection{The blist hypothesis}
\label{sec:blist}

In what follows, we will prove invariance of some quantities under some local
moves on green-dotted MOY graphs which contain crossings. It will often be
convenient to \emph{blist} either partially or completely an edge at the
boundary (or \emph{leg}) of diagrams we consider, and prove invariance
on these blisted diagrams.

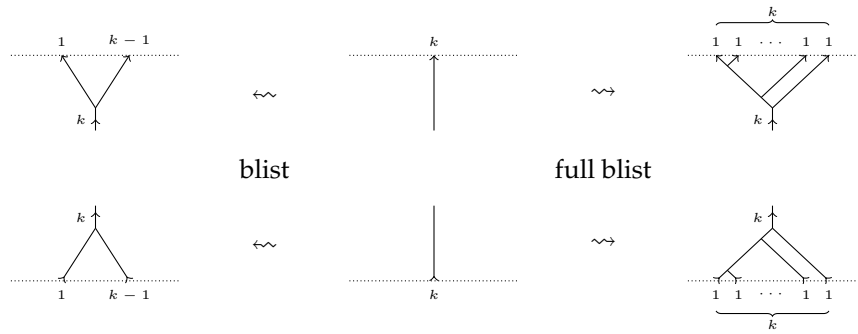
\[
  \NB{\tikz[]{\begin{scope}[xscale =1.5, font =\tiny]
  \begin{scope}[yshift = 0.5cm]
    \begin{scope}
      \draw[densely dotted] (-0.75,1) -- +(1.5,0);
      \draw[->] (0,0) -- +(0,1) node[pos =1, above] {$k$};
    \end{scope}
    \begin{scope}[xshift = -3cm]
      \draw[densely dotted] (-0.75,1) -- +(1.5,0);
      \draw[->-] (0,0) -- +(0,0.3) node[pos =0.5, left] {$k$};
      \draw[->] (0,0.3) -- ( -0.3,1) node [pos =1, above] {$1$};
      \draw[->] (0,0.3) -- ( 0.3,1) node [pos =1, above] {$k-1$};
    \end{scope}
    \begin{scope}[xshift = 3cm]
      \draw[densely dotted] (-0.75,1) -- +(1.5,0);
      \draw[->-] (0,0) -- +(0,0.3) node[pos =0.5, left] {$k$};
      \draw[->] (0,0.3) -- ( -0.5,1) node [pos =1, above] {$1$}
      coordinate[pos =0.2] (a) coordinate[pos =0.8] (b);
      \draw[->] (b) -- ( -0.3,1) node [pos =1, above] {$1$};
      \draw[->] (a) -- (0.3,1) node [pos =1, above] {$1$};
      \draw[->] (0,0.3) -- ( 0.5,1) node [pos =1, above] {$1$};
      \node[above] at (0,1) {$\cdots$};
      \draw [decorate,decoration={brace,amplitude=2pt},yshift=0.4cm] (-0.5,1) -- (0.5,1) node [black,pos=0.5,above] {$k$};
    \end{scope}
    \node[font = \normalsize, rotate=180] at (-1.5, 0.5)  {$\rightsquigarrow$};
    \node[font = \normalsize] at ( 1.5, 0.5) {$\rightsquigarrow$};
  \end{scope}
  \node[font = \normalsize] at (-1.5, 0) {blist};
  \node[font = \normalsize] at ( 1.5, 0) {full blist};
  \begin{scope}[yshift = -0.5cm, yscale = -1]
    \node[font = \normalsize, rotate=180] at (-1.5, 0.5) {$\rightsquigarrow$};
    \node[font = \normalsize] at ( 1.5, 0.5) {$\rightsquigarrow$};
    \begin{scope}
      \draw[densely dotted] (-0.75,1) -- +(1.5,0);
      \draw[-<] (0,0) -- +(0,1) node[pos =1, below] {$k$};
    \end{scope}
    \begin{scope}[xshift = -3cm]
      \draw[densely dotted] (-0.75,1) -- +(1.5,0);
      \draw[-<-] (0,0) -- +(0,0.3) node[pos =0.5, left] {$k$};
      \draw[-<] (0,0.3) -- ( -0.3,1) node [pos =1, below] {$1$};
      \draw[-<] (0,0.3) -- ( 0.3,1) node [pos =1, below] {$k-1$};
    \end{scope}
    \begin{scope}[xshift = 3cm]
      \draw[densely dotted] (-0.75,1) -- +(1.5,0);
      \draw[-<-] (0,0) -- +(0,0.3) node[pos =0.5, left] {$k$};
      \draw[-<] (0,0.3) -- ( -0.5,1) node [pos =1, below] {$1$}
      coordinate[pos =0.2] (a) coordinate[pos =0.8] (b);
      \draw[-<] (b) -- ( -0.3,1) node [pos =1, below] {$1$};
      \draw[-<] (a) -- (0.3,1) node [pos =1, below] {$1$};
      \draw[-<] (0,0.3) -- ( 0.5,1) node [pos =1, below] {$1$};
      \node[below] at (0,1) {$\cdots$};
      \draw [decorate,decoration={brace,amplitude=2pt},yshift=0.4cm] (0.5,1) -- (-0.5,1) node [black,pos=0.5,below] {$k$};
    \end{scope}
  \end{scope}
\end{scope}
}}
\]

Under circumstances we shall explain below, this \emph{blisting
  procedure} is legitimate. These circumstances will be referred to as
the \emph{blist hypothesis} in the upcoming sections.

We let $\somecategorypdg$ denote the relative homotopy category of
$H^\prime$-equivariant Soergel bimodules, $\somecategorynopdg$ be the homotopy
category of Soergel bimodules and $\nopdg: \somecategorypdg \to
\somecategorynopdg$ be the forgetful functor. 

Let $D$ be a diagram such that $\soergel(D)$ is an $H^\prime$-equivariant Soergel bimodule.
In this case, the graded endomorphism space $\End_{\mc{C}}(\mc{F}(\soergel(D)))$ carries a natural $H^\prime$-action, given by
\begin{equation}
    \dif(\phi) (x):=\dif(\phi(x))-\phi(\dif(x)),
\end{equation}
where $\phi\in \End_{\mc{C}}(\mc{F}\soergel(D))$ and $x\in \mc{F}(\soergel(D))$. Then $\phi$ is $H^\prime$-equivariant if and only if $\dif(\phi)=0$.

\begin{prop}\label{prop:blist-hyp}
  Let $D$ and $D'$ be two diagrams with identical boundaries and
  suppose that  $\End_\somecategorynopdg(\nopdg(\soergel(D)))$ is positively
  graded and that $\End_\somecategorynopdg(\nopdg(\soergel(D)))_0$ is free and has rank
  $1$. Blist the same leg of thickness $k$ on both
  $D$ and $D'$ and denote by $D_b$ and $D_{b}'$ the resulting diagrams.
  Further assume that $\soergel(D_b) \simeq \soergel(D'_b)$ (in
  $\somecategorypdg$). Then we have the following consequences.
  \begin{itemize}
  \item[1] The space $\End_\somecategorynopdg(\nopdg(D_b))$ is positively
  graded and  $\End_\somecategorynopdg(\nopdg(D_b))_0$ is free and has rank $1$.
  \item[2] There is an isomorphism $\soergel(D) \simeq \soergel(D')$.
  \end{itemize}
\end{prop}

\begin{proof}
  Recall that $A_{(k)}\cong \Z[x_1,\dots, x_k]^{S_k}$ is the ring of symmetric polynomials in $k$
  variables and $A_{(k-1,1)}\cong \Z[x_1,\dots, x_{k-1}]^{S_{k-1}}\otimes \Z[x_k]$ is the ring of polynomials in $k$
  variables which are symmetric in the first $(k-1)$  variables.
  We identify $A_{(k-1,1)}$ as a free $A_{(k)}$-module of graded rank $q^{k-1}[k]$ via the decomposition (see Lemma~\ref{lem:some-dim-end-space} for generalizations)
  \begin{equation}\label{eqn-decompostion-k-1-k}
      A_{(k-1,1)}\cong \bigoplus_{i=0}^{k-1} A_{(k)}x_k^i.
  \end{equation}
  Observe that this decomposition does not respect the $H^\prime$ actions.

  Now we have
  $\nopdg(\soergel(D_b)) \cong \nopdg(\soergel(D)) \otimes_{A_{(k)}} A_{(k-1, 1)}\cong q^{k-1}[k]
  \nopdg(\soergel(D))$ and
  $\nopdg(\soergel(D'_b)) \cong \nopdg(\soergel(D')) \otimes_{A_{(k)}} A_{(k-1, 1)} \cong q^{k-1}[k]
  \nopdg(\soergel(D'))$. In both cases for the second isomorphisms, the
  $A_{(k-1,1)}$-module structures are restricted to $A_{(k)}$-modules
  structures. At this stage it is important to notice that there is
  \emph{a priori} no $H^\prime$-equivariant direct sum decomposition for $\soergel(D_b)$ and $\soergel(D'_b)$.

  The first statement about the rank of the space of degree zero endomorphisms directly follows from the decomposition \eqref{eqn-decompostion-k-1-k} and the resulting base change isomorphism
  \begin{equation}\label{eqn-blist-base-change}
  \End_\somecategorynopdg(\nopdg(\soergel(D_b))) \cong \End_\somecategorynopdg(\nopdg(\soergel(D))) \otimes_{A_{(k)}} A_{(k-1,1)}\cong q^{k-1}[k] \End_\somecategorynopdg(\nopdg(\soergel(D))).
  \end{equation}
  Here, the last isomorphism is, again, as right modules over $A_{(k)}$.
  One then obtains that
  \[\rkq^{} \Hom_\somecategorynopdg(\nopdg(\soergel(D_b)), \nopdg(\soergel(D'_b)))_0 =1.
  \]
  This finishes the proof of the first statement.
  
  Next, let $\varphi \cl \soergel(D_b) \to \soergel(D_b')$ be an
  isomorphism in $\somecategorypdg$. By the discussion before the
  proposition, $\dif(\varphi)=0$.  Applying the forgetful functor
  $\nopdg: \somecategorypdg \to \somecategorynopdg$, one gets an
  isomorphism
  $\varphi \cl \nopdg(\soergel(D_b)) \to \nopdg(\soergel(D'_b))$ in
  $\somecategorynopdg$. Using (not necessarily $H^\prime$-equivariant)
  projections and injections induced from identifying $A_{(k)} $ with
  $A_{(k)}x_k^0$ in \eqref{eqn-decompostion-k-1-k}, one obtains an
  isomorphism $\phi\cl \nopdg(\soergel(D)) \to \nopdg(\soergel(D'))$
  in $\somecategorynopdg$. To show the second statement, it suffices
  to prove that $\dif(\phi)=0 $.

  On the one hand, we have, under the identification of \eqref{eqn-blist-base-change}, that
  $ \phi\otimes_{A_{(k)}} 1 $ is an isomorphism from
  $\nopdg(\soergel(D_b))$ to $\nopdg(\soergel(D'_b))$ in $\somecategorynopdg$ (here $1$ stands for the unit element of both $A_{(k-1,1)}$ and $A_{(k)}$). On the other
  hand, $\varphi$ is another one. Hence one has
  $\phi \otimes_{A_{(k)}} 1 = \pm \varphi$ by the first statement, so that
  \[
  0 =\pm \dif(\varphi)= \dif(\phi \otimes_{A_{(k)}}1) =\dif(\phi)\otimes_{A_{(k)}}1.
  \]
As $A_{(k-1,1)}$ is flat over $A_{(k)}$, it follows that $\dif(\phi)=0$ and thus $\phi$ is an $H^\prime$-equivariant isomorphism.
\end{proof}

\begin{rem}
  \begin{enumerate}
  \item Property~\ref{prop:blist-hyp} can be used inductively, so that
    we can fully blist a leg and/or blist several legs.
  \item This argument should be compared with Bar-Natan's strategy for
    proving projective functoriality of Khovanov homology \cite{BN}.
  \end{enumerate}
\end{rem}

Computing the rank of the space of endomorphisms of a complex of
Soergel bimodules associated with a tangle diagram is often easier
than it may first appear. The cases which will be needed are gathered
in the next lemma which in particular shows that these diagrams
satisfy the dimension hypothesis of Proposition~\ref{prop:blist-hyp}.

\begin{lem}
  The following identities holds:
  \label{lem:some-dim-end-space}
  \begin{enumerate}
    \item \begin{equation}
    \rkq\End_\somecategorynopdg\left( \nopdg\left(
    \NB{\tikz[font=\tiny, yscale= 0.5]{}}
    \right)\right)
     = \frac{1}{\prod_{k=1}^a(1-q^{2k})} ,  \label{eq:dim-end-strand}
     \end{equation}
    \item 
    \begin{equation} \rkq\End_\somecategorynopdg\left( \nopdg\left(
    \NB{\tikz[font=\tiny, yscale=0.6]{}} 
    \right)\right)
    = \frac{1}{\prod_{k=1}^a(1-q^{2k}) \prod_{k=1}^b(1-q^{2k})} , \label{eq:dim-end-Y} \end{equation}
    \item
    \begin{equation}
     \rkq\End_\somecategorynopdg\left( \nopdg\left(
    \NB{\tikz[font=\tiny, xscale =0.7]{\begin{scope}[font=\tiny]
  \draw[->] (0.5, -0.5) ..controls +(0,0.3) and +(0,-0.3) .. (-0.5,
  0.5) node[pos=1, above] {$a$} coordinate[pos =0.2] (t1);
  \fill[white] (0,0) circle (2mm);
  \draw[->] (-0.5, -0.5) ..controls +(0,0.3) and +(0,-0.3) .. (0.5,
  0.5) node[pos=1, above] {$b$} coordinate[pos =0.2] (t2);
\end{scope}}}  
    \right)\right) =
    \rkq\End_\somecategorynopdg\left( \nopdg\left(
    \NB{\tikz[font=\tiny, yscale=0.5, xscale =0.7]{\begin{scope}[font=\tiny]
  \draw[->] (0.5, -0.5) ..controls +(0,0.3) and +(0,-0.3) .. (0.5,
  1.5) node[pos=1, above] {$b$} coordinate[pos =0.2] (t1);
  \draw[->] (-0.5, -0.5) ..controls +(0,0.3) and +(0,-0.3) .. (-0.5,
  1.5) node[pos=1, above] {$a$} coordinate[pos =0.2] (t2);
\end{scope}}}  
    \right)\right)
    = \frac{1}{\prod_{k=1}^a(1-q^{2k})\prod_{k=1}^b(1-q^{2k})} ,  \label{eq:dim-end-crossing}
    \end{equation}
    \item\begin{equation}
    \rkq\End_\somecategorynopdg\left( \nopdg\left(
    \NB{\tikz[font=\tiny, yscale=0.6]{\begin{scope}[yscale=-1]
  \coordinate (b) at (0, -1);
  \coordinate (o) at (0, -0.5);
  \coordinate (tl) at (-0.5, 1);
  \coordinate (tr) at ( 0.5, 1);
  \coordinate (bl) at (-0.5, -1);
  \coordinate (trr) at ( 1,  1);
  \draw[<-<] (bl) .. controls (+0,1) and +(0, -0.5) .. (trr)
  node[pos=1, below] {$c$} node[pos =0, above] {$c$}  coordinate[pos =0.19] (d1)  coordinate[pos =0.5] (d2); 
  \fill[white] (d1) circle (0.1);
  \fill[white] (d2) circle (0.1);
  \draw[<-] (b) -- (o) node [pos= 0, above] {$a+b$};
  \draw[-<] (o) .. controls +(0,0) and +(0, -0.5) .. (tl) node[pos= 1, below] {$a$}; 
  \draw[-<] (o) .. controls +(0,0) and +(0,  -0.5) .. (tr) node[pos= 1, below] {$b$};
\end{scope}
}}
    \right)\right)
    =\frac{1}{\prod_{k=1}^a(1-q^{2k})\prod_{k=1}^b(1-q^{2k})\prod_{k=1}^c(1-q^{2k})} ,
    \label{eq:dim-end-pf}\end{equation}
    \item\begin{equation}  \label{eq:dim-end-r3}
     \rkq\End_\somecategorynopdg\left( \nopdg\left(
    \NB{\tikz[font=\tiny, xscale =0.6, yscale=0.6]{\begin{scope}[font=\tiny]
  \draw (0.5, -0.5) ..controls +(0,0.3) and +(0,-0.3) .. (-0.5,
  0.5);%
  \fill[white] (0,0) circle (2mm);
  \draw (-0.5, -0.5) ..controls +(0,0.3) and +(0,-0.3) .. (0.5,
  0.5);%
  \draw (1.5, 0.5) ..controls +(0,0.3) and +(0,-0.3) .. (0.5,
  1.5);%
  \fill[white] (1,1) circle (2mm);
  \draw (0.5, 0.5) ..controls +(0,0.3) and +(0,-0.3) .. (1.5,
  1.5);%
  \draw[->] (0.5, 1.5) ..controls +(0,0.3) and +(0,-0.3) .. (-0.5,
  2.5) node[pos=1, above] {$c$} coordinate[pos =0.2] (t2);
  \fill[white] (0,2) circle (2mm);
  \draw[->] (-0.5, 1.5) ..controls +(0,0.3) and +(0,-0.3) .. (0.5,
  2.5) node[pos=1, above] {$b$} coordinate[pos =0.2] (t1);
  \draw (-0.5, 0.5) -- (-0.5, 1.5);
  \draw (1.5, -0.5) -- (1.5, 0.5);
  \draw[->] (1.5, 1.5) -- (1.5, 2.5) node[pos=1, above] {$a$};
\end{scope}}}
    \right)\right)
    =\frac{1}{\prod_{k=1}^a(1-q^{2k})\prod_{k=1}^b(1-q^{2k})\prod_{k=1}^c(1-q^{2k})}.
      \end{equation}
  \end{enumerate}
  Moreover taking mirror images and changing crossings in the diagrams
  give the same graded ranks. All the endomorphisms spaces considered
  are free as $\Z$-modules.
\end{lem}

\begin{proof}
In order to compute these graded ranks, first take one of the diagrams $D$ above and reflect it in the bottom horizontal axis resulting in a diagram $D^{\dagger}$.
Then one computes the graded ranks of the bimodule associated to the concatenated diagram $D D^{\dagger}$.

The first item follows trivially from the graded rank of the algebra of symmetric polynomials in $a$ variables.

The second item follows from the graded rank of partially symmetric polynomials $A_{(a,b)}$ as a module over symmetric polynomials $A_{(a+b)}$ along with the graded rank calculated in the first item.

For the third item, note that $D D^{\dagger}$ gives rise to a complex in the homotopy category of Soergel bimodules isomorphic to $A_{(a,b)}$.  The graded rank now follows from the first item.

The graded ranks for the fourth and fifth items are computed in a similar way to the third item.  One also uses the results from the first two items.

\end{proof}

\subsection{Forkslide reductions}
\label{sec:pitchfork-reduction}
In this subsection we consider certain diagrams which contain both crossings and trivalent vertices.  In order to prove braid invariance later on, we will need certain relations to hold between these diagrams.  We introduce these relations and prove a fact that reduces the set of relations to a simpler set of relations.

By \emph{forkslide moves}, we mean one of the following 8 local changes
in diagrams:
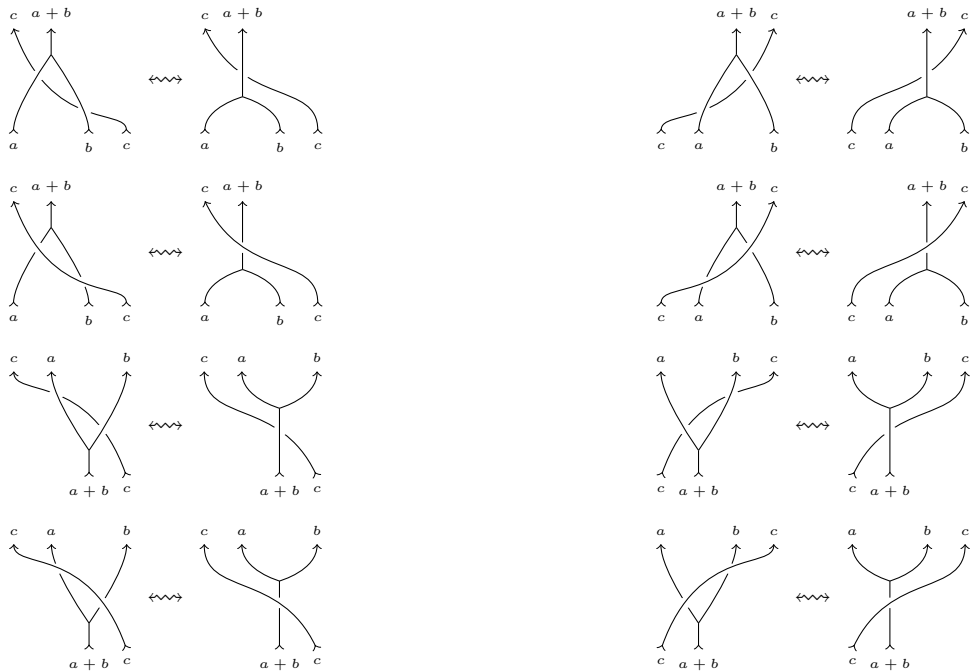
\begin{align*}
    \NB{\tikz[font=\tiny, yscale =0.7]{}} 
  &\leftrightsquigarrow
  \NB{\tikz[font=\tiny, yscale =0.7]{\begin{scope}[yscale=-1]
  \coordinate (b) at (0, -1);
  \coordinate (o) at (0, 0.3);
  \coordinate (tl) at (-0.5, 1);
  \coordinate (tr) at ( 0.5, 1);
  \coordinate (bl) at ( -0.5, -1);
  \coordinate (trr) at (1,  1);
  \draw[<-<] (bl) .. controls (+0,0.5) and +(0, -1) .. (trr)
  node[pos=1, below] {$c$} node[pos =0, above] {$c$}  coordinate[pos =0.278] (d0);
  \fill[white] (d0) circle (0.1);
  \draw[<-] (b) -- (o) node [pos= 0, above] {$a+b$};
  \draw[-<] (o) .. controls +(0,0) and +(0, -0.5) .. (tl) node[pos= 1, below] {$a$}; 
  \draw[-<] (o) .. controls +(0,0) and +(0,  -0.5) .. (tr) node[pos= 1, below] {$b$};
\end{scope}

}}  &&\qquad&
  \NB{\tikz[font=\tiny, yscale =0.7]{\begin{scope}[yscale=-1]
  \coordinate (b) at (0, -1);
  \coordinate (o) at (0, -0.5);
  \coordinate (tl) at (-0.5, 1);
  \coordinate (tr) at ( 0.5, 1);
  \coordinate (br) at ( 0.5, -1);
  \coordinate (tll) at (-1,  1);
  \draw[<-<] (br) .. controls (+0,1) and +(0, -0.5) .. (tll)
  node[pos=1, below] {$c$} node[pos =0, above] {$c$}  coordinate[pos =0.19] (d1)  coordinate[pos =0.5] (d2);
  \fill[white] (d1) circle (0.1);
  \fill[white] (d2) circle (0.1);
  \draw[<-] (b) -- (o) node [pos= 0, above] {$a+b$};
  \draw[-<] (o) .. controls +(0,0) and +(0, -0.5) .. (tl) node[pos= 1, below] {$a$};
  \draw[-<] (o) .. controls +(0,0) and +(0,  -0.5) .. (tr) node[pos= 1, below] {$b$};

\end{scope}
}} 
  &\leftrightsquigarrow
  \NB{\tikz[font=\tiny, yscale =0.7]{\begin{scope}[yscale=-1]
  \coordinate (b) at (0, -1);
  \coordinate (o) at (0, 0.3);
  \coordinate (tl) at (-0.5, 1);
  \coordinate (tr) at ( 0.5, 1);
  \coordinate (br) at ( 0.5, -1);
  \coordinate (tll) at (-1,  1);
  \draw[<-<] (br) .. controls (+0,0.5) and +(0, -1) .. (tll)
  node[pos=1, below] {$c$} node[pos =0, above] {$c$}  coordinate[pos =0.278] (d0);
  \fill[white] (d0) circle (0.1);
  \draw[<-] (b) -- (o) node [pos= 0, above] {$a+b$};
  \draw[-<] (o) .. controls +(0,0) and +(0, -0.5) .. (tl) node[pos= 1, below] {$a$}; 
  \draw[-<] (o) .. controls +(0,0) and +(0,  -0.5) .. (tr) node[pos= 1, below] {$b$};
\end{scope}
}} \\
  \NB{\tikz[font=\tiny, yscale =0.7]{\begin{scope}[yscale=-1]
  \coordinate (b) at (0, -1);
  \coordinate (o) at (0, -0.5);
  \coordinate (tl) at (-0.5, 1);
  \coordinate (tr) at ( 0.5, 1);
  \coordinate (bl) at (-0.5, -1);
  \coordinate (trr) at ( 1,  1);
  \draw[<-] (b) -- (o) node [pos= 0, above] {$a+b$};
  \draw[-<] (o) .. controls +(0,0) and +(0, -0.5) .. (tl) node[pos= 1, below] {$a$} coordinate[pos =0.4] (d1);
  \draw[-<] (o) .. controls +(0,0) and +(0,  -0.5) .. (tr) node[pos= 1, below] {$b$} coordinate[pos =0.74] (d2); 
  \fill[white] (d1) circle (0.1);
  \fill[white] (d2) circle (0.1);
  \draw[<-<] (bl) .. controls (+0,1) and +(0, -0.5) .. (trr) node[pos=1, below] {$c$} node[pos =0, above] {$c$}; 
\end{scope}}} 
  &\leftrightsquigarrow
  \NB{\tikz[font=\tiny, yscale =0.7]{    \begin{scope}[yscale=-1]
  \coordinate (b) at (0, -1);
  \coordinate (o) at (0, 0.3);
  \coordinate (tl) at (-0.5, 1);
  \coordinate (tr) at ( 0.5, 1);
  \coordinate (bl) at ( -0.5, -1);
  \coordinate (trr) at (1,  1);
  \draw[<-] (b) -- (o) node [pos= 0, above] {$a+b$}  coordinate[pos =0.68] (d0);
  \draw[-<] (o) .. controls +(0,0) and +(0, -0.5) .. (tl) node[pos= 1, below] {$a$}; 
  \draw[-<] (o) .. controls +(0,0) and +(0,  -0.5) .. (tr) node[pos= 1, below] {$b$};
  \fill[white] (d0) circle (0.1);
  \draw[<-<] (bl) .. controls (+0,0.5) and +(0, -1) .. (trr) node[pos=1, below] {$c$} node[pos =0, above] {$c$}; 
\end{scope}
}} &&\qquad&
  \NB{\tikz[font=\tiny, yscale =0.7]{\begin{scope}[yscale=-1]
  \coordinate (b) at (0, -1);
  \coordinate (o) at (0, -0.5);
  \coordinate (tl) at (-0.5, 1);
  \coordinate (tr) at ( 0.5, 1);
  \coordinate (br) at ( 0.5, -1);
  \coordinate (tll) at (-1,  1);
  \draw[<-] (b) -- (o) node [pos= 0, above] {$a+b$};
  \draw[-<] (o) .. controls +(0,0) and +(0, -0.5) .. (tl) node[pos= 1, below] {$a$} coordinate[pos =0.74] (d1);
  \draw[-<] (o) .. controls +(0,0) and +(0,  -0.5) .. (tr) node[pos= 1, below] {$b$} coordinate[pos =0.4] (d2); %
  \fill[white] (d1) circle (0.1);
  \fill[white] (d2) circle (0.1);
  \draw[<-<] (br) .. controls (+0,1) and +(0, -0.5) .. (tll) node[pos=1, below] {$c$} node[pos =0, above] {$c$}; 
\end{scope}}} 
  &\leftrightsquigarrow
  \NB{\tikz[font=\tiny, yscale =0.7]{\begin{scope}[yscale=-1]
  \coordinate (b) at (0, -1);
  \coordinate (o) at (0, 0.3);
  \coordinate (tl) at (-0.5, 1);
  \coordinate (tr) at ( 0.5, 1);
  \coordinate (br) at ( 0.5, -1);
  \coordinate (tll) at (-1,  1);
  \draw[<-] (b) -- (o) node [pos= 0, above] {$a+b$}  coordinate[pos =0.68] (d0);
  \draw[-<] (o) .. controls +(0,0) and +(0, -0.5) .. (tl) node[pos= 1, below] {$a$}; 
  \draw[-<] (o) .. controls +(0,0) and +(0,  -0.5) .. (tr) node[pos= 1, below] {$b$};
  \fill[white] (d0) circle (0.1);
  \draw[<-<] (br) .. controls (+0,0.5) and +(0, -1) .. (tll) node[pos=1, below] {$c$} node[pos =0, above] {$c$}; 
\end{scope}
}} \\
  \NB{\tikz[font=\tiny, yscale =0.7]{\begin{scope}
  \coordinate (b) at (0, -1);
  \coordinate (o) at (0, -0.5);
  \coordinate (tl) at (-0.5, 1);
  \coordinate (tr) at ( 0.5, 1);
  \coordinate (br) at ( 0.5, -1);
  \coordinate (tll) at (-1,  1);
  \draw[>->] (br) .. controls (+0,1) and +(0, -0.5) .. (tll)
  node[pos=1, above] {$c$} node[pos =0, below] {$c$}  coordinate[pos =0.19] (d1)  coordinate[pos =0.5] (d2);
  \fill[white] (d1) circle (0.1);
  \fill[white] (d2) circle (0.1);
  \draw[>-] (b) -- (o) node [pos= 0, below] {$a+b$};
  \draw[->] (o) .. controls +(0,0) and +(0, -0.5) .. (tl) node[pos= 1, above] {$a$};
  \draw[->] (o) .. controls +(0,0) and +(0,  -0.5) .. (tr) node[pos= 1, above] {$b$};

\end{scope}
}} 
  &\leftrightsquigarrow
  \NB{\tikz[font=\tiny, yscale =0.7]{\begin{scope}
  \coordinate (b) at (0, -1);
  \coordinate (o) at (0, 0.3);
  \coordinate (tl) at (-0.5, 1);
  \coordinate (tr) at ( 0.5, 1);
  \coordinate (br) at ( 0.5, -1);
  \coordinate (tll) at (-1,  1);
  \draw[>->] (br) .. controls (+0,0.5) and +(0, -1) .. (tll)
  node[pos=1, above] {$c$} node[pos =0, below] {$c$}  coordinate[pos =0.278] (d0);
  \fill[white] (d0) circle (0.1);
  \draw[>-] (b) -- (o) node [pos= 0, below] {$a+b$};
  \draw[->] (o) .. controls +(0,0) and +(0, -0.5) .. (tl) node[pos= 1, above] {$a$}; 
  \draw[->] (o) .. controls +(0,0) and +(0,  -0.5) .. (tr) node[pos= 1, above] {$b$};
\end{scope}
}}  &&\qquad&
  \NB{\tikz[font=\tiny, yscale =0.7]{\begin{scope}
  \coordinate (b) at (0, -1);
  \coordinate (o) at (0, -0.5);
  \coordinate (tl) at (-0.5, 1);
  \coordinate (tr) at ( 0.5, 1);
  \coordinate (bl) at (-0.5, -1);
  \coordinate (trr) at ( 1,  1);
  \draw[>->] (bl) .. controls (+0,1) and +(0, -0.5) .. (trr)
  node[pos=1, above] {$c$} node[pos =0, below] {$c$}  coordinate[pos =0.19] (d1)  coordinate[pos =0.5] (d2); 
  \fill[white] (d1) circle (0.1);
  \fill[white] (d2) circle (0.1);
  \draw[>-] (b) -- (o) node [pos= 0, below] {$a+b$};
  \draw[->] (o) .. controls +(0,0) and +(0, -0.5) .. (tl) node[pos= 1, above] {$a$}; 
  \draw[->] (o) .. controls +(0,0) and +(0,  -0.5) .. (tr) node[pos= 1, above] {$b$};
\end{scope}
}} 
  &\leftrightsquigarrow
  \NB{\tikz[font=\tiny, yscale =0.7]{\begin{scope}
  \coordinate (b) at (0, -1);
  \coordinate (o) at (0, 0.3);
  \coordinate (tl) at (-0.5, 1);
  \coordinate (tr) at ( 0.5, 1);
  \coordinate (bl) at ( -0.5, -1);
  \coordinate (trr) at (1,  1);
  \draw[>->] (bl) .. controls (+0,0.5) and +(0, -1) .. (trr)
  node[pos=1, above] {$c$} node[pos =0, below] {$c$}  coordinate[pos =0.278] (d0);
  \fill[white] (d0) circle (0.1);
  \draw[>-] (b) -- (o) node [pos= 0, below] {$a+b$};
  \draw[->] (o) .. controls +(0,0) and +(0, -0.5) .. (tl) node[pos= 1, above] {$a$}; 
  \draw[->] (o) .. controls +(0,0) and +(0,  -0.5) .. (tr) node[pos= 1, above] {$b$};
\end{scope}
}} \\
  \NB{\tikz[font=\tiny, yscale =0.7]{\begin{scope}
  \coordinate (b) at (0, -1);
  \coordinate (o) at (0, -0.5);
  \coordinate (tl) at (-0.5, 1);
  \coordinate (tr) at ( 0.5, 1);
  \coordinate (br) at ( 0.5, -1);
  \coordinate (tll) at (-1,  1);
  \draw[>-] (b) -- (o) node [pos= 0, below] {$a+b$};
  \draw[->] (o) .. controls +(0,0) and +(0, -0.5) .. (tl) node[pos= 1, above] {$a$} coordinate[pos =0.74] (d1);
  \draw[->] (o) .. controls +(0,0) and +(0,  -0.5) .. (tr) node[pos= 1, above] {$b$} coordinate[pos =0.4] (d2); %
  \fill[white] (d1) circle (0.1);
  \fill[white] (d2) circle (0.1);
  \draw[>->] (br) .. controls (+0,1) and +(0, -0.5) .. (tll) node[pos=1, above] {$c$} node[pos =0, below] {$c$}; 
\end{scope}}} 
  &\leftrightsquigarrow
  \NB{\tikz[font=\tiny, yscale =0.7]{\begin{scope}
  \coordinate (b) at (0, -1);
  \coordinate (o) at (0, 0.3);
  \coordinate (tl) at (-0.5, 1);
  \coordinate (tr) at ( 0.5, 1);
  \coordinate (br) at ( 0.5, -1);
  \coordinate (tll) at (-1,  1);
  \draw[>-] (b) -- (o) node [pos= 0, below] {$a+b$}  coordinate[pos =0.68] (d0);
  \draw[->] (o) .. controls +(0,0) and +(0, -0.5) .. (tl) node[pos= 1, above] {$a$}; 
  \draw[->] (o) .. controls +(0,0) and +(0,  -0.5) .. (tr) node[pos= 1, above] {$b$};
  \fill[white] (d0) circle (0.1);
  \draw[>->] (br) .. controls (+0,0.5) and +(0, -1) .. (tll) node[pos=1, above] {$c$} node[pos =0, below] {$c$}; 
\end{scope}}}  &&\qquad&
  \NB{\tikz[font=\tiny, yscale =0.7]{\begin{scope}
  \coordinate (b) at (0, -1);
  \coordinate (o) at (0, -0.5);
  \coordinate (tl) at (-0.5, 1);
  \coordinate (tr) at ( 0.5, 1);
  \coordinate (bl) at (-0.5, -1);
  \coordinate (trr) at ( 1,  1);
  \draw[>-] (b) -- (o) node [pos= 0, below] {$a+b$};
  \draw[->] (o) .. controls +(0,0) and +(0, -0.5) .. (tl) node[pos= 1, above] {$a$} coordinate[pos =0.4] (d1);
  \draw[->] (o) .. controls +(0,0) and +(0,  -0.5) .. (tr) node[pos= 1, above] {$b$} coordinate[pos =0.74] (d2); 
  \fill[white] (d1) circle (0.1);
  \fill[white] (d2) circle (0.1);
  \draw[>->] (bl) .. controls (+0,1) and +(0, -0.5) .. (trr) node[pos=1, above] {$c$} node[pos =0, below] {$c$}; 
\end{scope}}} 
  &\leftrightsquigarrow
  \NB{\tikz[font=\tiny, yscale =0.7]{    \begin{scope}
  \coordinate (b) at (0, -1);
  \coordinate (o) at (0, 0.3);
  \coordinate (tl) at (-0.5, 1);
  \coordinate (tr) at ( 0.5, 1);
  \coordinate (bl) at ( -0.5, -1);
  \coordinate (trr) at (1,  1);
  \draw[>-] (b) -- (o) node [pos= 0, below] {$a+b$}  coordinate[pos =0.68] (d0);
  \draw[->] (o) .. controls +(0,0) and +(0, -0.5) .. (tl) node[pos= 1, above] {$a$}; 
  \draw[->] (o) .. controls +(0,0) and +(0,  -0.5) .. (tr) node[pos= 1, above] {$b$};
  \fill[white] (d0) circle (0.1);
  \draw[>->] (bl) .. controls (+0,0.5) and +(0, -1) .. (trr) node[pos=1, above] {$c$} node[pos =0, below] {$c$}; 
\end{scope}}} \\
\end{align*}
for any $a,b, c$ in $\Z_{>0}$. 

In what follows, we will show that $\soergel(\cdot)$ is invariant
under forkslide moves in the relative homotopy category of singular Soergel bimodules. It turns out that the  $H^\prime$-equivariant complexes associated to all these diagrams satisfy
the hypothesis of Proposition~\ref{prop:blist-hyp} (see
\eqref{eq:dim-end-pf}). This simplifies considerably our task as stated
by the claim below.
\begin{claim}
It is enough to prove invariance of
$\soergel(\cdot)$ under the following moves:
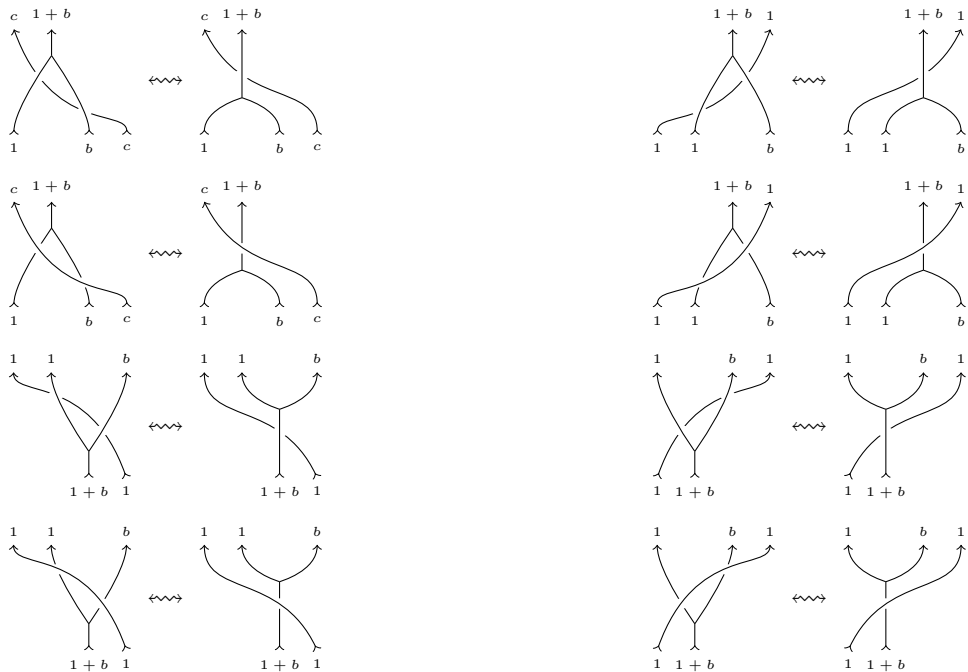
\begin{align*}
  \NB{\tikz[font=\tiny, yscale =0.7]{\begin{scope}[yscale=-1]
  \coordinate (b) at (0, -1);
  \coordinate (o) at (0, -0.5);
  \coordinate (tl) at (-0.5, 1);
  \coordinate (tr) at ( 0.5, 1);
  \coordinate (bl) at (-0.5, -1);
  \coordinate (trr) at ( 1,  1);
  \draw[<-<] (bl) .. controls (+0,1) and +(0, -0.5) .. (trr)
  node[pos=1, below] {$ c$} node[pos =0, above] {$ c$}  coordinate[pos =0.19] (d1)  coordinate[pos =0.5] (d2); 
  \fill[white] (d1) circle (0.1);
  \fill[white] (d2) circle (0.1);
  \draw[<-] (b) -- (o) node [pos= 0, above] {$1+b $};
  \draw[-<] (o) .. controls +(0,0) and +(0, -0.5) .. (tl) node[pos= 1, below] {$1$}; 
  \draw[-<] (o) .. controls +(0,0) and +(0,  -0.5) .. (tr) node[pos=
  1, below] {$b $};
\end{scope}

}} 
  &\leftrightsquigarrow
  \NB{\tikz[font=\tiny, yscale =0.7]{\begin{scope}[yscale=-1]
  \coordinate (b) at (0, -1);
  \coordinate (o) at (0, 0.3);
  \coordinate (tl) at (-0.5, 1);
  \coordinate (tr) at ( 0.5, 1);
  \coordinate (bl) at ( -0.5, -1);
  \coordinate (trr) at (1,  1);
  \draw[<-<] (bl) .. controls (+0,0.5) and +(0, -1) .. (trr)
  node[pos=1, below] {$ c$} node[pos =0, above] {$ c$}  coordinate[pos =0.278] (d0);
  \fill[white] (d0) circle (0.1);
  \draw[<-] (b) -- (o) node [pos= 0, above] {$1+b $};
  \draw[-<] (o) .. controls +(0,0) and +(0, -0.5) .. (tl) node[pos= 1, below] {$1$}; 
  \draw[-<] (o) .. controls +(0,0) and +(0,  -0.5) .. (tr) node[pos=
  1, below] {$b $};
\end{scope}

}}  &&\qquad&
  \NB{\tikz[font=\tiny, yscale =0.7]{\begin{scope}[yscale=-1]
  \coordinate (b) at (0, -1);
  \coordinate (o) at (0, -0.5);
  \coordinate (tl) at (-0.5, 1);
  \coordinate (tr) at ( 0.5, 1);
  \coordinate (br) at ( 0.5, -1);
  \coordinate (tll) at (-1,  1);
  \draw[<-<] (br) .. controls (+0,1) and +(0, -0.5) .. (tll)
  node[pos=1, below] {$1$} node[pos =0, above] {$1$}  coordinate[pos =0.19] (d1)  coordinate[pos =0.5] (d2);
  \fill[white] (d1) circle (0.1);
  \fill[white] (d2) circle (0.1);
  \draw[<-] (b) -- (o) node [pos= 0, above] {$1+b$};
  \draw[-<] (o) .. controls +(0,0) and +(0, -0.5) .. (tl) node[pos= 1, below] {$1$};
  \draw[-<] (o) .. controls +(0,0) and +(0,  -0.5) .. (tr) node[pos= 1, below] {$b$};

\end{scope}
}} 
  &\leftrightsquigarrow
  \NB{\tikz[font=\tiny, yscale =0.7]{\begin{scope}[yscale=-1]
  \coordinate (b) at (0, -1);
  \coordinate (o) at (0, 0.3);
  \coordinate (tl) at (-0.5, 1);
  \coordinate (tr) at ( 0.5, 1);
  \coordinate (br) at ( 0.5, -1);
  \coordinate (tll) at (-1,  1);
  \draw[<-<] (br) .. controls (+0,0.5) and +(0, -1) .. (tll)
  node[pos=1, below] {$1$} node[pos =0, above] {$1$}  coordinate[pos =0.278] (d0);
  \fill[white] (d0) circle (0.1);
  \draw[<-] (b) -- (o) node [pos= 0, above] {$1+b$};
  \draw[-<] (o) .. controls +(0,0) and +(0, -0.5) .. (tl) node[pos= 1, below] {$1$}; 
  \draw[-<] (o) .. controls +(0,0) and +(0,  -0.5) .. (tr) node[pos= 1, below] {$b$};
\end{scope}
}} \\
  \NB{\tikz[font=\tiny, yscale =0.7]{\begin{scope}[yscale=-1]
  \coordinate (b) at (0, -1);
  \coordinate (o) at (0, -0.5);
  \coordinate (tl) at (-0.5, 1);
  \coordinate (tr) at ( 0.5, 1);
  \coordinate (bl) at (-0.5, -1);
  \coordinate (trr) at ( 1,  1);
  \draw[<-] (b) -- (o) node [pos= 0, above] {$1+b$};
  \draw[-<] (o) .. controls +(0,0) and +(0, -0.5) .. (tl) node[pos= 1, below] {$1$} coordinate[pos =0.4] (d1);
  \draw[-<] (o) .. controls +(0,0) and +(0,  -0.5) .. (tr) node[pos= 1, below] {$b$} coordinate[pos =0.74] (d2); 
  \fill[white] (d1) circle (0.1);
  \fill[white] (d2) circle (0.1);
  \draw[<-<] (bl) .. controls (+0,1) and +(0, -0.5) .. (trr)
  node[pos=1, below] {$ c$} node[pos =0, above] {$ c$}; 
\end{scope}
}} 
  &\leftrightsquigarrow
  \NB{\tikz[font=\tiny, yscale =0.7]{    \begin{scope}[yscale=-1]
  \coordinate (b) at (0, -1);
  \coordinate (o) at (0, 0.3);
  \coordinate (tl) at (-0.5, 1);
  \coordinate (tr) at ( 0.5, 1);
  \coordinate (bl) at ( -0.5, -1);
  \coordinate (trr) at (1,  1);
  \draw[<-] (b) -- (o) node [pos= 0, above] {$1+b$}  coordinate[pos =0.68] (d0);
  \draw[-<] (o) .. controls +(0,0) and +(0, -0.5) .. (tl) node[pos= 1, below] {$1$}; 
  \draw[-<] (o) .. controls +(0,0) and +(0,  -0.5) .. (tr) node[pos= 1, below] {$b$};
  \fill[white] (d0) circle (0.1);
  \draw[<-<] (bl) .. controls (+0,0.5) and +(0, -1) .. (trr)
  node[pos=1, below] {$ c$} node[pos =0, above] {$ c$}; 
\end{scope}
}} &&\qquad&
  \NB{\tikz[font=\tiny, yscale =0.7]{\begin{scope}[yscale=-1]
  \coordinate (b) at (0, -1);
  \coordinate (o) at (0, -0.5);
  \coordinate (tl) at (-0.5, 1);
  \coordinate (tr) at ( 0.5, 1);
  \coordinate (br) at ( 0.5, -1);
  \coordinate (tll) at (-1,  1);
  \draw[<-] (b) -- (o) node [pos= 0, above] {$1+b$};
  \draw[-<] (o) .. controls +(0,0) and +(0, -0.5) .. (tl) node[pos= 1, below] {$1$} coordinate[pos =0.74] (d1);
  \draw[-<] (o) .. controls +(0,0) and +(0,  -0.5) .. (tr) node[pos= 1, below] {$b$} coordinate[pos =0.4] (d2); %
  \fill[white] (d1) circle (0.1);
  \fill[white] (d2) circle (0.1);
  \draw[<-<] (br) .. controls (+0,1) and +(0, -0.5) .. (tll) node[pos=1, below] {$1$} node[pos =0, above] {$1$}; 
\end{scope}}} 
  &\leftrightsquigarrow
  \NB{\tikz[font=\tiny, yscale =0.7]{\begin{scope}[yscale=-1]
  \coordinate (b) at (0, -1);
  \coordinate (o) at (0, 0.3);
  \coordinate (tl) at (-0.5, 1);
  \coordinate (tr) at ( 0.5, 1);
  \coordinate (br) at ( 0.5, -1);
  \coordinate (tll) at (-1,  1);
  \draw[<-] (b) -- (o) node [pos= 0, above] {$1+b$}  coordinate[pos =0.68] (d0);
  \draw[-<] (o) .. controls +(0,0) and +(0, -0.5) .. (tl) node[pos= 1, below] {$1$}; 
  \draw[-<] (o) .. controls +(0,0) and +(0,  -0.5) .. (tr) node[pos= 1, below] {$b$};
  \fill[white] (d0) circle (0.1);
  \draw[<-<] (br) .. controls (+0,0.5) and +(0, -1) .. (tll) node[pos=1, below] {$1$} node[pos =0, above] {$1$}; 
\end{scope}
}} \\
  \NB{\tikz[font=\tiny, yscale =0.7]{\begin{scope}
  \coordinate (b) at (0, -1);
  \coordinate (o) at (0, -0.5);
  \coordinate (tl) at (-0.5, 1);
  \coordinate (tr) at ( 0.5, 1);
  \coordinate (br) at ( 0.5, -1);
  \coordinate (tll) at (-1,  1);
  \draw[>->] (br) .. controls (+0,1) and +(0, -0.5) .. (tll)
  node[pos=1, above] {$1$} node[pos =0, below] {$1$}  coordinate[pos =0.19] (d1)  coordinate[pos =0.5] (d2);
  \fill[white] (d1) circle (0.1);
  \fill[white] (d2) circle (0.1);
  \draw[>-] (b) -- (o) node [pos= 0, below] {$1+b$};
  \draw[->] (o) .. controls +(0,0) and +(0, -0.5) .. (tl) node[pos= 1, above] {$1$};
  \draw[->] (o) .. controls +(0,0) and +(0,  -0.5) .. (tr) node[pos= 1, above] {$b$};

\end{scope}
}} 
  &\leftrightsquigarrow
  \NB{\tikz[font=\tiny, yscale =0.7]{\begin{scope}
  \coordinate (b) at (0, -1);
  \coordinate (o) at (0, 0.3);
  \coordinate (tl) at (-0.5, 1);
  \coordinate (tr) at ( 0.5, 1);
  \coordinate (br) at ( 0.5, -1);
  \coordinate (tll) at (-1,  1);
  \draw[>->] (br) .. controls (+0,0.5) and +(0, -1) .. (tll)
  node[pos=1, above] {$1$} node[pos =0, below] {$1$}  coordinate[pos =0.278] (d0);
  \fill[white] (d0) circle (0.1);
  \draw[>-] (b) -- (o) node [pos= 0, below] {$1+b$};
  \draw[->] (o) .. controls +(0,0) and +(0, -0.5) .. (tl) node[pos= 1, above] {$1$}; 
  \draw[->] (o) .. controls +(0,0) and +(0,  -0.5) .. (tr) node[pos= 1, above] {$b$};
\end{scope}
}}  &&\qquad&  \NB{\tikz[font=\tiny, yscale =0.7]{\begin{scope}
  \coordinate (b) at (0, -1);
  \coordinate (o) at (0, -0.5);
  \coordinate (tl) at (-0.5, 1);
  \coordinate (tr) at ( 0.5, 1);
  \coordinate (bl) at (-0.5, -1);
  \coordinate (trr) at ( 1,  1);
  \draw[>->] (bl) .. controls (+0,1) and +(0, -0.5) .. (trr)
  node[pos=1, above] {$1$} node[pos =0, below] {$1$}  coordinate[pos =0.19] (d1)  coordinate[pos =0.5] (d2); 
  \fill[white] (d1) circle (0.1);
  \fill[white] (d2) circle (0.1);
  \draw[>-] (b) -- (o) node [pos= 0, below] {$1+b$};
  \draw[->] (o) .. controls +(0,0) and +(0, -0.5) .. (tl) node[pos= 1, above] {$1$}; 
  \draw[->] (o) .. controls +(0,0) and +(0,  -0.5) .. (tr) node[pos= 1, above] {$b$};
\end{scope}
}} 
  &\leftrightsquigarrow
  \NB{\tikz[font=\tiny, yscale =0.7]{\begin{scope}
  \coordinate (b) at (0, -1);
  \coordinate (o) at (0, 0.3);
  \coordinate (tl) at (-0.5, 1);
  \coordinate (tr) at ( 0.5, 1);
  \coordinate (bl) at ( -0.5, -1);
  \coordinate (trr) at (1,  1);
  \draw[>->] (bl) .. controls (+0,0.5) and +(0, -1) .. (trr)
  node[pos=1, above] {$1$} node[pos =0, below] {$1$}  coordinate[pos =0.278] (d0);
  \fill[white] (d0) circle (0.1);
  \draw[>-] (b) -- (o) node [pos= 0, below] {$1+b$};
  \draw[->] (o) .. controls +(0,0) and +(0, -0.5) .. (tl) node[pos= 1, above] {$1$}; 
  \draw[->] (o) .. controls +(0,0) and +(0,  -0.5) .. (tr) node[pos= 1, above] {$b$};
\end{scope}
}} 
\\
  \NB{\tikz[font=\tiny, yscale =0.7]{\begin{scope}
  \coordinate (b) at (0, -1);
  \coordinate (o) at (0, -0.5);
  \coordinate (tl) at (-0.5, 1);
  \coordinate (tr) at ( 0.5, 1);
  \coordinate (br) at ( 0.5, -1);
  \coordinate (tll) at (-1,  1);
  \draw[>-] (b) -- (o) node [pos= 0, below] {$1+b$};
  \draw[->] (o) .. controls +(0,0) and +(0, -0.5) .. (tl) node[pos= 1, above] {$1$} coordinate[pos =0.74] (d1);
  \draw[->] (o) .. controls +(0,0) and +(0,  -0.5) .. (tr) node[pos= 1, above] {$b$} coordinate[pos =0.4] (d2); %
  \fill[white] (d1) circle (0.1);
  \fill[white] (d2) circle (0.1);
  \draw[>->] (br) .. controls (+0,1) and +(0, -0.5) .. (tll) node[pos=1, above] {$1$} node[pos =0, below] {$1$}; 
\end{scope}}} 
  &\leftrightsquigarrow
  \NB{\tikz[font=\tiny, yscale =0.7]{\begin{scope}
  \coordinate (b) at (0, -1);
  \coordinate (o) at (0, 0.3);
  \coordinate (tl) at (-0.5, 1);
  \coordinate (tr) at ( 0.5, 1);
  \coordinate (br) at ( 0.5, -1);
  \coordinate (tll) at (-1,  1);
  \draw[>-] (b) -- (o) node [pos= 0, below] {$1+b$}  coordinate[pos =0.68] (d0);
  \draw[->] (o) .. controls +(0,0) and +(0, -0.5) .. (tl) node[pos= 1, above] {$1$}; 
  \draw[->] (o) .. controls +(0,0) and +(0,  -0.5) .. (tr) node[pos= 1, above] {$b$};
  \fill[white] (d0) circle (0.1);
  \draw[>->] (br) .. controls (+0,0.5) and +(0, -1) .. (tll) node[pos=1, above] {$1$} node[pos =0, below] {$1$}; 
\end{scope}}}  &&\qquad&  \NB{\tikz[font=\tiny, yscale =0.7]{\begin{scope}
  \coordinate (b) at (0, -1);
  \coordinate (o) at (0, -0.5);
  \coordinate (tl) at (-0.5, 1);
  \coordinate (tr) at ( 0.5, 1);
  \coordinate (bl) at (-0.5, -1);
  \coordinate (trr) at ( 1,  1);
  \draw[>-] (b) -- (o) node [pos= 0, below] {$1+b$};
  \draw[->] (o) .. controls +(0,0) and +(0, -0.5) .. (tl) node[pos= 1, above] {$1$} coordinate[pos =0.4] (d1);
  \draw[->] (o) .. controls +(0,0) and +(0,  -0.5) .. (tr) node[pos= 1, above] {$b$} coordinate[pos =0.74] (d2); 
  \fill[white] (d1) circle (0.1);
  \fill[white] (d2) circle (0.1);
  \draw[>->] (bl) .. controls (+0,1) and +(0, -0.5) .. (trr) node[pos=1, above] {$1$} node[pos =0, below] {$1$}; 
\end{scope}}} 
  &\leftrightsquigarrow
  \NB{\tikz[font=\tiny, yscale =0.7]{    \begin{scope}
  \coordinate (b) at (0, -1);
  \coordinate (o) at (0, 0.3);
  \coordinate (tl) at (-0.5, 1);
  \coordinate (tr) at ( 0.5, 1);
  \coordinate (bl) at ( -0.5, -1);
  \coordinate (trr) at (1,  1);
  \draw[>-] (b) -- (o) node [pos= 0, below] {$1+b$}  coordinate[pos =0.68] (d0);
  \draw[->] (o) .. controls +(0,0) and +(0, -0.5) .. (tl) node[pos= 1, above] {$1$}; 
  \draw[->] (o) .. controls +(0,0) and +(0,  -0.5) .. (tr) node[pos= 1, above] {$b$};
  \fill[white] (d0) circle (0.1);
  \draw[>->] (bl) .. controls (+0,0.5) and +(0, -1) .. (trr) node[pos=1, above] {$1$} node[pos =0, below] {$1$}; 
\end{scope}}}.
\\ 
\end{align*} 
\end{claim}

\begin{proof}
  The proof of this claim does not really depend on how $\soergel$ is
  defined but rather on the fact that we can use the blist hypothesis (\emph{i.e.{}}
  Proposition~\ref{prop:blist-hyp}). It permits us to deduce invariance under
  certain moves from the invariance under blisted versions of these
  moves. %
  
  We start by proving invariance under
  \[
    \NB
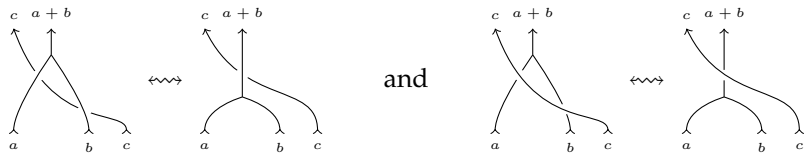
{\tikz[font=\tiny, yscale =0.7]{}} 
  \leftrightsquigarrow
  \NB{\tikz[font=\tiny, yscale =0.7]{}}  \qquad \text{and}
  \qquad
  \NB{\tikz[font=\tiny, yscale =0.7]{\begin{scope}[yscale=-1]
  \coordinate (b) at (0, -1);
  \coordinate (o) at (0, -0.5);
  \coordinate (tl) at (-0.5, 1);
  \coordinate (tr) at ( 0.5, 1);
  \coordinate (bl) at (-0.5, -1);
  \coordinate (trr) at ( 1,  1);
  \draw[<-] (b) -- (o) node [pos= 0, above] {$a+b$};
  \draw[-<] (o) .. controls +(0,0) and +(0, -0.5) .. (tl) node[pos= 1, below] {$a$} coordinate[pos =0.4] (d1);
  \draw[-<] (o) .. controls +(0,0) and +(0,  -0.5) .. (tr) node[pos= 1, below] {$b$} coordinate[pos =0.74] (d2); 
  \fill[white] (d1) circle (0.1);
  \fill[white] (d2) circle (0.1);
  \draw[<-<] (bl) .. controls (+0,1) and +(0, -0.5) .. (trr) node[pos=1, below] {$c$} node[pos =0, above] {$c$}; 
\end{scope}}} 
  \leftrightsquigarrow
  \NB{\tikz[font=\tiny, yscale =0.7]{}}
  \]
  by induction on $a$. Both cases are similar, so we only deal with the
  first one. Due to the blist hypothesis, it is enough to show
  invariance under
  \[
    \NB
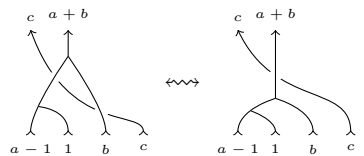
{\tikz[font=\tiny, yscale =0.7]{\begin{scope}[yscale=-1]
  \coordinate (b) at (0, -1);
  \coordinate (o) at (0, -0.5);
  \coordinate (tl) at (-0.5, 1);
  \coordinate (tm) at (   0, 1);
  \coordinate (tr) at ( 0.5, 1);
  \coordinate (bl) at (-0.5, -1);
  \coordinate (trr) at ( 1,  1);
  \draw[<-<] (bl) .. controls (+0,1) and +(0, -0.5) .. (trr)
  node[pos=1, below] {$c$} node[pos =0, above] {$c$}  coordinate[pos =0.19] (d1)  coordinate[pos =0.5] (d2); 
  \fill[white] (d1) circle (0.1);
  \fill[white] (d2) circle (0.1);
  \draw[<-] (b) -- (o) node [pos= 0, above] {$a+b$};
  \draw[-<] (o) .. controls +(0,0) and +(0, -0.5) .. (tl) node[pos= 1,
  below] {$a-1$} coordinate[pos=0.7] (m);
  \draw[-<] (m) .. controls +(0,0) and +(0, -0.5) .. (tm) node[pos= 1,
  below] {$1$};
  \draw[-<] (o) .. controls +(0,0) and +(0,  -0.5) .. (tr) node[pos= 1, below] {$b$};
\end{scope}

}} 
  \leftrightsquigarrow
  \NB{\tikz[font=\tiny, yscale =0.7]{\begin{scope}[yscale=-1]
  \coordinate (b) at (0, -1);
  \coordinate (o) at (0, 0.3);
  \coordinate (tl) at (-0.5, 1);
  \coordinate (tm) at ( 0, 1);
  \coordinate (tr) at ( 0.5, 1);
  \coordinate (bl) at ( -0.5, -1);
  \coordinate (trr) at (1,  1);
  \draw[<-<] (bl) .. controls (+0,0.5) and +(0, -1) .. (trr)
  node[pos=1, below] {$c$} node[pos =0, above] {$c$}  coordinate[pos =0.278] (d0);
  \fill[white] (d0) circle (0.1);
  \draw[<-] (b) -- (o) node [pos= 0, above] {$a+b$};
  \draw[-<] (o) .. controls +(0,0) and +(0, -0.5) .. (tl) node[pos= 1,
  below] {$a-1$} coordinate[pos = 0.6] (m); 
  \draw[-<] (o) .. controls +(0,0) and +(0,  -0.5) .. (tr) node[pos= 1, below] {$b$};
  \draw[-<] (m) .. controls +(0,0) and +(0, -0.3) .. (tm) node[pos= 1,
  below] {$1$};
\end{scope}
}},
  \]
  which follows from the sequence:
  \[
    \NB
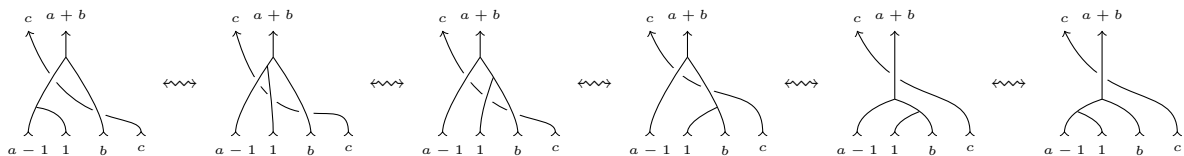
{\tikz[font=\tiny, yscale =0.7]{}} 
  \leftrightsquigarrow
  \NB{\tikz[font=\tiny, yscale =0.7]{\begin{scope}[yscale=-1]
  \coordinate (b) at (0, -1);
  \coordinate (o) at (0, -0.5);
  \coordinate (tl) at (-0.5, 1);
  \coordinate (tm) at (   0, 1);
  \coordinate (tr) at ( 0.5, 1);
  \coordinate (bl) at (-0.5, -1);
  \coordinate (trr) at ( 1,  1);
  \draw[<-<] (bl) .. controls (+0,1.5) and +(0, -1) .. (trr)
  node[pos=1, below] {$c$} node[pos =0, above] {$c$}  coordinate[pos
  =0.172] (d1)  coordinate[pos =0.5] (d2) coordinate[pos=0.27] (d3); 
  \fill[white] (d1) circle (0.1);
  \fill[white] (d2) circle (0.1);
  \fill[white] (d3) circle (0.1);
  \draw[<-] (b) -- (o) node [pos= 0, above] {$a+b$};
  \draw[-<] (o) .. controls +(0,0) and +(0, -0.5) .. (tl) node[pos= 1,
  below] {$a-1$} coordinate[pos=0.25] (m);
  \draw[-<] (o) .. controls +(0,0) and +(0,  -0.5) .. (tr) node[pos=
  1, below] {$b$};
  \draw[-<] (m) .. controls +(0,0) and +(0, -0.5) .. (tm) node[pos= 1,
  below] {$1$};

\end{scope}

}}
  \leftrightsquigarrow
  \NB{\tikz[font=\tiny, yscale =0.7]{\begin{scope}[yscale=-1]
  \coordinate (b) at (0, -1);
  \coordinate (o) at (0, -0.5);
  \coordinate (tl) at (-0.5, 1);
  \coordinate (tm) at (   0, 1);
  \coordinate (tr) at ( 0.5, 1);
  \coordinate (bl) at (-0.5, -1);
  \coordinate (trr) at ( 1,  1);
  \draw[<-<] (bl) .. controls (+0,1) and +(0, -0.5) .. (trr)
  node[pos=1, below] {$c$} node[pos =0, above] {$c$}  coordinate[pos
  =0.19] (d1)  coordinate[pos =0.5] (d2) coordinate[pos=0.325] (d3); 
  \fill[white] (d1) circle (0.1);
  \fill[white] (d2) circle (0.1);
  \fill[white] (d3) circle (0.1);
  \draw[<-] (b) -- (o) node [pos= 0, above] {$a+b$};
  \draw[-<] (o) .. controls +(0,0) and +(0, -0.5) .. (tl) node[pos= 1,
  below] {$a-1$};
  \draw[-<] (o) .. controls +(0,0) and +(0,  -0.5) .. (tr) node[pos=
  1, below] {$b$}  coordinate[pos=0.4] (m);
  \draw[-<] (m) .. controls +(0,0) and +(0, -0.5) .. (tm) node[pos= 1,
  below] {$1$};

\end{scope}

}}
  \leftrightsquigarrow
  \NB{\tikz[font=\tiny, yscale =0.7]{\begin{scope}[yscale=-1]
  \coordinate (b) at (0, -1);
  \coordinate (o) at (0, -0.5);
  \coordinate (tl) at (-0.5, 1);
  \coordinate (tm) at (   0, 1);
  \coordinate (tr) at ( 0.5, 1);
  \coordinate (bl) at (-0.5, -1);
  \coordinate (trr) at ( 1,  1);
  \draw[<-<] (bl) .. controls (+0,0.7) and +(0, -1.2) .. (trr)
  node[pos=1, below] {$c$} node[pos =0, above] {$c$}  coordinate[pos
  =0.215] (d1)  coordinate[pos =0.41] (d2);
  \fill[white] (d1) circle (0.1);
  \fill[white] (d2) circle (0.1);
  \draw[<-] (b) -- (o) node [pos= 0, above] {$a+b$};
  \draw[-<] (o) .. controls +(0,0) and +(0, -0.5) .. (tl) node[pos= 1,
  below] {$a-1$};
  \draw[-<] (o) .. controls +(0,0) and +(0,  -0.5) .. (tr) node[pos=
  1, below] {$b$}  coordinate[pos=0.7] (m);
  \draw[-<] (m) .. controls +(0,0) and +(0, -0.3) .. (tm) node[pos= 1,
  below] {$1$};
\end{scope}

}}
  \leftrightsquigarrow
  \NB{\tikz[font=\tiny, yscale =0.7]{\begin{scope}[yscale=-1]
  \coordinate (b) at (0, -1);
  \coordinate (o) at (0, 0.3);
  \coordinate (tl) at (-0.5, 1);
  \coordinate (tm) at ( 0, 1);
  \coordinate (tr) at ( 0.5, 1);
  \coordinate (bl) at ( -0.5, -1);
  \coordinate (trr) at (1,  1);
  \draw[<-<] (bl) .. controls (+0,0.5) and +(0, -1) .. (trr)
  node[pos=1, below] {$c$} node[pos =0, above] {$c$}  coordinate[pos =0.278] (d0);
  \fill[white] (d0) circle (0.1);
  \draw[<-] (b) -- (o) node [pos= 0, above] {$a+b$};
  \draw[-<] (o) .. controls +(0,0) and +(0, -0.5) .. (tl) node[pos= 1,
  below] {$a-1$};
  \draw[-<] (o) .. controls +(0,0) and +(0,  -0.5) .. (tr) node[pos=
  1, below] {$b$}  coordinate[pos = 0.6] (m); 
  \draw[-<] (m) .. controls +(0,0) and +(0, -0.3) .. (tm) node[pos= 1,
  below] {$1$};
\end{scope}
}}
  \leftrightsquigarrow
  \NB{\tikz[font=\tiny, yscale =0.7]{}}. 
\]

Let us now prove invariance under the following two moves:
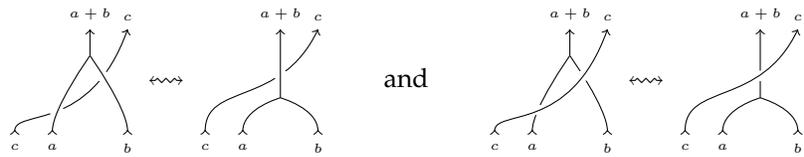
\[
  \NB{\tikz[font=\tiny, yscale =0.7]{}} 
  \leftrightsquigarrow
  \NB{\tikz[font=\tiny, yscale =0.7]{}}
  \qquad \text{and} \qquad
  \NB{\tikz[font=\tiny, yscale =0.7]{\begin{scope}[yscale=-1]
  \coordinate (b) at (0, -1);
  \coordinate (o) at (0, -0.5);
  \coordinate (tl) at (-0.5, 1);
  \coordinate (tr) at ( 0.5, 1);
  \coordinate (br) at ( 0.5, -1);
  \coordinate (tll) at (-1,  1);
  \draw[<-] (b) -- (o) node [pos= 0, above] {$a+b$};
  \draw[-<] (o) .. controls +(0,0) and +(0, -0.5) .. (tl) node[pos= 1, below] {$a$} coordinate[pos =0.74] (d1);
  \draw[-<] (o) .. controls +(0,0) and +(0,  -0.5) .. (tr) node[pos= 1, below] {$b$} coordinate[pos =0.4] (d2); %
  \fill[white] (d1) circle (0.1);
  \fill[white] (d2) circle (0.1);
  \draw[<-<] (br) .. controls (+0,1) and +(0, -0.5) .. (tll) node[pos=1, below] {$c$} node[pos =0, above] {$c$}; 
\end{scope}}} 
  \leftrightsquigarrow
  \NB{\tikz[font=\tiny, yscale =0.7]{}}   .
\]
Both cases are similar, so we only deal with the first one. We only need
to prove invariance under the move
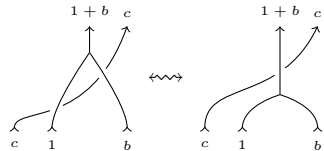
\[
  \NB{\tikz[font=\tiny, yscale =0.7]{\begin{scope}[yscale=-1]
  \coordinate (b) at (0, -1);
  \coordinate (o) at (0, -0.5);
  \coordinate (tl) at (-0.5, 1);
  \coordinate (tr) at ( 0.5, 1);
  \coordinate (br) at ( 0.5, -1);
  \coordinate (tll) at (-1,  1);
  \draw[<-<] (br) .. controls (+0,1) and +(0, -0.5) .. (tll)
  node[pos=1, below] {$c$} node[pos =0, above] {$c$}  coordinate[pos =0.19] (d1)  coordinate[pos =0.5] (d2);
  \fill[white] (d1) circle (0.1);
  \fill[white] (d2) circle (0.1);
  \draw[<-] (b) -- (o) node [pos= 0, above] {$1+b$};
  \draw[-<] (o) .. controls +(0,0) and +(0, -0.5) .. (tl) node[pos= 1, below] {$1$};
  \draw[-<] (o) .. controls +(0,0) and +(0,  -0.5) .. (tr) node[pos= 1, below] {$b$};

\end{scope}
}} 
  \leftrightsquigarrow
  \NB{\tikz[font=\tiny, yscale =0.7]{\begin{scope}[yscale=-1]
  \coordinate (b) at (0, -1);
  \coordinate (o) at (0, 0.3);
  \coordinate (tl) at (-0.5, 1);
  \coordinate (tr) at ( 0.5, 1);
  \coordinate (br) at ( 0.5, -1);
  \coordinate (tll) at (-1,  1);
  \draw[<-<] (br) .. controls (+0,0.5) and +(0, -1) .. (tll)
  node[pos=1, below] {$c$} node[pos =0, above] {$c$}  coordinate[pos =0.278] (d0);
  \fill[white] (d0) circle (0.1);
  \draw[<-] (b) -- (o) node [pos= 0, above] {$1+b$};
  \draw[-<] (o) .. controls +(0,0) and +(0, -0.5) .. (tl) node[pos= 1, below] {$1$}; 
  \draw[-<] (o) .. controls +(0,0) and +(0,  -0.5) .. (tr) node[pos= 1, below] {$b$};
\end{scope}
}}
  \]
  since we can then argue as we did before to recover the general
  case. We proceed by induction on $c$. Due to the blist
  hypothesis, it is enough to prove invariance under the following
  move:
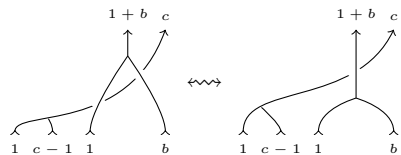
\[
  \NB{\tikz[font=\tiny, yscale =0.7]{\begin{scope}[yscale=-1]
  \coordinate (b) at (0, -1);
  \coordinate (o) at (0, -0.5);
  \coordinate (tl) at (-0.5, 1);
  \coordinate (tr) at ( 0.5, 1);
  \coordinate (br) at ( 0.5, -1);
  \coordinate (tll) at (-1,  1);
  \coordinate (tlll) at (-1.5,  1);
  \draw[<-<] (br) .. controls (+0,1) and +(0, -0.5) .. (tlll)
  node[pos=1, below] {$1$} node[pos =0, above] {$c$}  coordinate[pos
  =0.177] (d1)  coordinate[pos =0.387] (d2) coordinate[pos = 0.65] (m);
  \draw[-<] (m) .. controls +(0,0) and +(0, -0.2) .. (tll) node[below] {$c-1$};
  \fill[white] (d1) circle (0.1);
  \fill[white] (d2) circle (0.1);
  \draw[<-] (b) -- (o) node [pos= 0, above] {$1+b$};
  \draw[-<] (o) .. controls +(0,0) and +(0, -0.5) .. (tl) node[pos= 1, below] {$1$};
  \draw[-<] (o) .. controls +(0,0) and +(0,  -0.5) .. (tr) node[pos= 1, below] {$b$};

\end{scope}
}} 
  \leftrightsquigarrow
  \NB{\tikz[font=\tiny, yscale =0.7]{\begin{scope}[yscale=-1]
  \coordinate (b) at (0, -1);
  \coordinate (o) at (0, 0.3);
  \coordinate (tl) at (-0.5, 1);
  \coordinate (tr) at ( 0.5, 1);
  \coordinate (br) at ( 0.5, -1);
  \coordinate (tll) at (-1,  1);
  \coordinate (tlll) at (-1.5,  1);
  \draw[<-<] (br) .. controls (+0,0.5) and +(0, -1) .. (tlll)
  node[pos=1, below] {$1$} node[pos =0, above] {$c$}  coordinate[pos
  =0.25] (d0) coordinate[pos = 0.75] (m);
  \draw[-<] (m) .. controls +(0,0) and +(0, -0.2) .. (tll) node[below] {$c-1$};
  \fill[white] (d0) circle (0.1);
  \draw[<-] (b) -- (o) node [pos= 0, above] {$1+b$};
  \draw[-<] (o) .. controls +(0,0) and +(0, -0.5) .. (tl) node[pos= 1, below] {$1$}; 
  \draw[-<] (o) .. controls +(0,0) and +(0,  -0.5) .. (tr) node[pos= 1, below] {$b$};
\end{scope}
}}.
  \]
  This is obtained from the sequence:
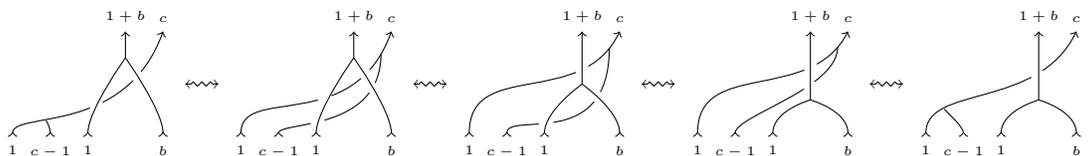
\[
  \NB{\tikz[font=\tiny, yscale =0.7]{}} 
  \leftrightsquigarrow
  \NB{\tikz[font=\tiny, yscale =0.7]{\begin{scope}[yscale=-1]
  \coordinate (b) at (0, -1);
  \coordinate (o) at (0, -0.5);
  \coordinate (tl) at (-0.5, 1);
  \coordinate (tr) at ( 0.5, 1);
  \coordinate (br) at ( 0.5, -1);
  \coordinate (tll) at (-1,  1);
  \coordinate (tlll) at (-1.5,  1);
  \draw[<-<] (br) .. controls (+0,1) and +(0, -1) .. (tlll)
  node[pos=1, below] {$1$} node[pos =0, above] {$c$}  coordinate[pos
  =0.177] (d1)  coordinate[pos =0.387] (d2) coordinate[pos = 0.08] (m);
  \draw[-<] (m) .. controls +(0,1.5) and +(0, -0.3) .. (tll) node[below]
  {$c-1$} coordinate[pos =0.177] (d3)  coordinate[pos =0.577] (d4);
  \fill[white] (d1) circle (0.1);
  \fill[white] (d2) circle (0.1);
  \fill[white] (d3) circle (0.1);
  \fill[white] (d4) circle (0.1);
  \draw[<-] (b) -- (o) node [pos= 0, above] {$1+b$};
  \draw[-<] (o) .. controls +(0,0) and +(0, -0.5) .. (tl) node[pos= 1, below] {$1$};
  \draw[-<] (o) .. controls +(0,0) and +(0,  -0.5) .. (tr) node[pos= 1, below] {$b$};

\end{scope}
}} 
  \leftrightsquigarrow
  \NB{\tikz[font=\tiny, yscale =0.7]{\begin{scope}[yscale=-1]
  \coordinate (b) at (0, -1);
  \coordinate (o) at (0, 0);
  \coordinate (tl) at (-0.5, 1);
  \coordinate (tr) at ( 0.5, 1);
  \coordinate (br) at ( 0.5, -1);
  \coordinate (tll) at (-1,  1);
  \coordinate (tlll) at (-1.5,  1);
  \draw[<-<] (br) .. controls (+0,0.5) and +(0, -1.5) .. (tlll)
  node[pos=1, below] {$1$} node[pos =0, above] {$c$}  coordinate[pos
  =0.24] (d1)  coordinate[pos = 0.08] (m);
  \draw[-<] (m) .. controls +(0,2) and +(0, -0.5) .. (tll) node[below]
  {$c-1$} coordinate[pos =0.21] (d3)  coordinate[pos =0.577] (d4);
  \fill[white] (d1) circle (0.1);
  \fill[white] (d3) circle (0.1);
  \fill[white] (d4) circle (0.1);
  \draw[<-] (b) -- (o) node [pos= 0, above] {$1+b$};
  \draw[-<] (o) .. controls +(0,0) and +(0, -0.5) .. (tl) node[pos= 1, below] {$1$};
  \draw[-<] (o) .. controls +(0,0) and +(0,  -0.5) .. (tr) node[pos= 1, below] {$b$};

\end{scope}
}} 
  \leftrightsquigarrow
  \NB{\tikz[font=\tiny, yscale =0.7]{\begin{scope}[yscale=-1]
  \coordinate (b) at (0, -1);
  \coordinate (o) at (0, 0.3);
  \coordinate (tl) at (-0.5, 1);
  \coordinate (tr) at ( 0.5, 1);
  \coordinate (br) at ( 0.5, -1);
  \coordinate (tll) at (-1,  1);
  \coordinate (tlll) at (-1.5,  1);
  \draw[<-<] (br) .. controls (+0,0.4) and +(0, -1.5) .. (tlll)
  node[pos=1, below] {$1$} node[pos =0, above] {$c$}  coordinate[pos
  =0.24] (d0) coordinate[pos = 0.08] (m);
  \draw[-<] (m) .. controls +(0,0.8) and +(0, -0.4) .. (tll)
  node[below] {$c-1$} coordinate[pos   =0.345] (d1);
  \fill[white] (d0) circle (0.1);
  \fill[white] (d1) circle (0.1);
  \draw[<-] (b) -- (o) node [pos= 0, above] {$1+b$};
  \draw[-<] (o) .. controls +(0,0) and +(0, -0.5) .. (tl) node[pos= 1, below] {$1$}; 
  \draw[-<] (o) .. controls +(0,0) and +(0,  -0.5) .. (tr) node[pos= 1, below] {$b$};
\end{scope}
}} 
  \leftrightsquigarrow
  \NB{\tikz[font=\tiny, yscale =0.7]{}}.
  \]

  Invariance of the remaining moves is proved in a similar way:
  introduce a blist on the vertex-less strand and argue by induction.
\end{proof}

\subsection{Forkslide isomorphisms}
\label{sec:pitchforks}

This section contains proofs of various categorical forkslide moves. They are key for proving braid group relations for colored braids.

\begin{prop} \label{prop:allpitches}
There are isomorphisms in the relative homotopy category of all the forkslide reduction moves.  
\end{prop}

\begin{proof}
The details of the two most difficult cases are provided below in Propositions \ref{prop:pitch-1} and \ref{prop:pitch-3}.
The other cases are proved in analogous fashions and in fact most of the details simplify.
\end{proof}

\begin{prop} \label{prop:pitch-1}
  The complexes of $H^\prime$-equivariant bimodules 
  \[
    \soergel\left(
  \NB{\tikz[font=\tiny, yscale =0.7]{}} 
    \right)
    \qquad \text{and} \qquad
    \soergel\left(
  \NB{\tikz[font=\tiny, yscale =0.7]{}}   \right)
  \]
are isomorphic in the relative homotopy category.
\end{prop}

\begin{proof}
  To save space, in this proof we will often omit $\soergel\left( \cdot
  \right)$ around diagrams and give grading shifts only in a couple of key places.  First note that
\[  
C:= \soergel\left(
  \NB{\tikz[font=\tiny, yscale =0.7]{}} 
    \right)
\]
  is the flattening (total complex) of the bicomplex of  $H^\prime$-equivariant bimodules
  \[
    C':=\NB{
}
  \]
  obtained by plugging in Rickard complexes associated with the two
  crossings in two different directions. Denote the vertical maps by $d_T$.

  Consider the complex \[
    H_-: =\NB{
    %
}
  \] 
  obtained from the first complex of bimodules of
  Lemma~\ref{lem:exact-sequence-square}. We may consider
  $H_-$ as a bicomplex with trivial vertical differentials. By
  Lemma~\ref{lem:exact-sequence-square}, it is null-homotopic.
  
  The map $\iota \cl H_- \to C'$, given on each diagram by the
  following composition
  \[
    \iota:=\NB{
      %
}
  \]
  is a map of bicomplexes of $H^\prime$-equivariant bimodules. This implies
  that $C(\iota)\cong C^\prime \oplus tH_-$, the cone of $\iota$,
  is isomorphic to $C'$ in the relative homotopy category of $H^\prime$-equivariant
  bimodules. %

  Consider now the complex
  \[
D:=     \soergel\left(
  \NB{\tikz[font=\tiny, yscale =0.7]{}} 
\right) .
  \]
It can be thought of as a bicomplex with the zero vertical
differential. More concretely, one has
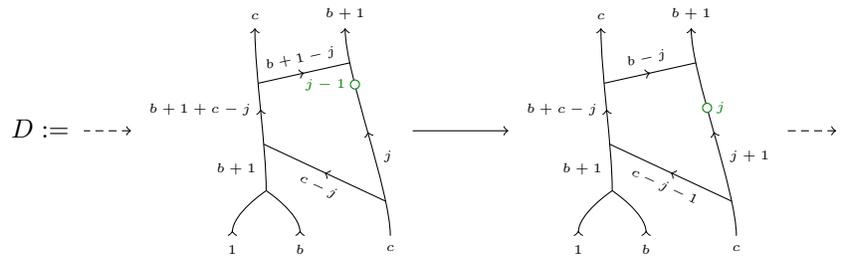
\[
  D:=\NB{
    %
} \ .
  \]
  Let us now define two chain maps $\varphi\cl D \to H_-$ 
  and $\psi\cl D \to C'$. For each diagram, $\varphi$ is given by the
following composition:
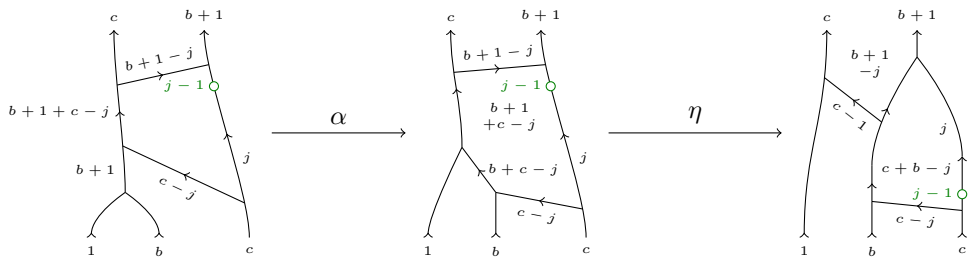
\[
\NB{
          %
 ,
    }
\]
and $\psi$ by the following composition:
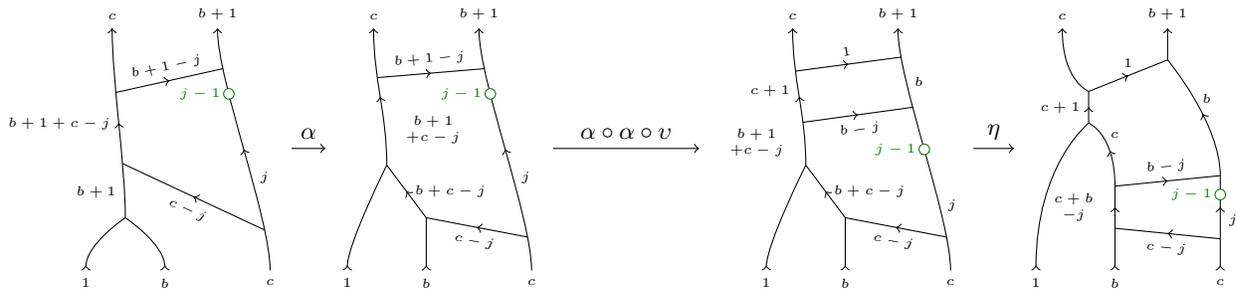
\[
\NB{
          %
 .
  }
\]

We put back in the grading shifts and claim that the following complex
of $H^\prime$-equivariant bimodules is null-homotopic in the usual homotopy category of Soergel
bimodules (forgetting about $H^\prime$-structures). %

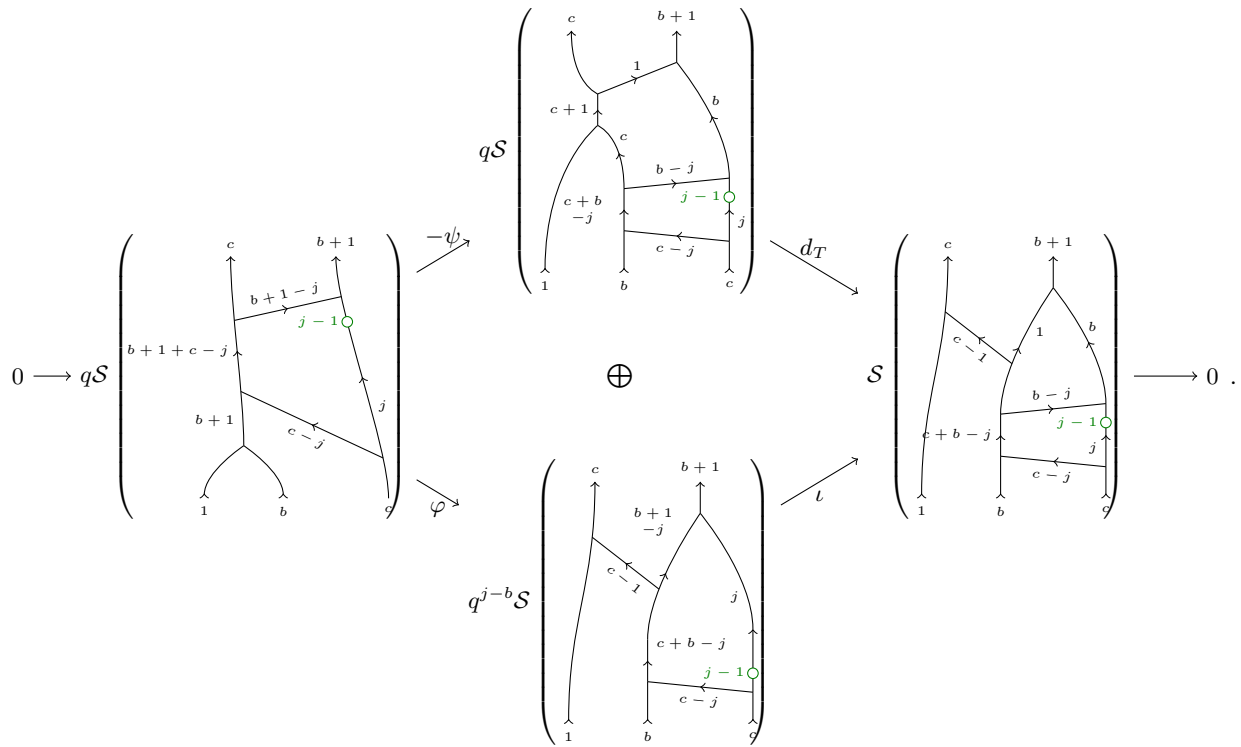
\[
  %
 
\]
First note that the composition of maps is indeed $0$, since one can
check that the square anti-commutes (this is the fundamental purpose of
the minus sign). For further arguments, we will abbreviate the relevant bimodules by rewriting the square above as follows:
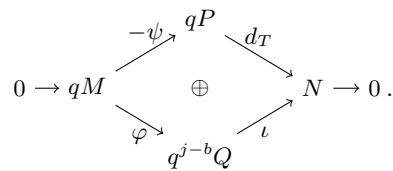
\[
  %
 
\]
Forgetting the $H^\prime$-structure, we have that $N\simeq [b+1-j]
Q$ and $P\simeq M \oplus [b-j] Q$. Rewriting the the maps above using these
isomorphisms, we obtain:
\[
  %
 
\]
which is null-homotopic.
This null-homotopic chain complex constitutes the homogeneous terms of the mapping cone of
\[
D \xrightarrow{(-\psi,\varphi)} C(\iota).
\]
This implies that $D$ and
$C(\iota)$ are isomorphic in the relative homotopy category. Keeping track of overall shifts, it follows that
\[
    \soergel\left(
  \NB{\tikz[font=\tiny, yscale =0.7]{}} 
    \right)
\simeq    \soergel\left(
  \NB{\tikz[font=\tiny, yscale =0.7]{}}   \right).
  \]  
\end{proof}

\begin{prop} \label{prop:pitch-3}
    The complexes of $H^\prime$-equivariant bimodules 
  \[
    \soergel\left(
  \NB{\tikz[font=\tiny, yscale =0.7]{}} 
    \right)
    \qquad \text{and} \qquad
    \soergel\left(
  \NB{\tikz[font=\tiny, yscale =0.7]{}} 
    \right)
  \]
are isomorphic in the relative homotopy category.
\end{prop}

  \begin{proof}
  As before, we will often omit $\soergel\left( \cdot
  \right)$ around diagrams. 
Consider 
\[  
C:= \soergel\left(
  \NB{\tikz[font=\tiny, yscale =0.7]{}}
    \right)
\]
  as the flattening (total complex) of the bicomplex of  $H^\prime$-equivariant bimodules
  \[
    C':=\NB{
}
  \]
  obtained by plugging in Rickard complexes associated with the two
  crossings in two different directions. Denote the vertical map by $d_T$.

  Consider the complex \[
    H_+: =\NB{
    %
}
  \] 
  obtained from the second complex of bimodules of
  Lemma~\ref{lem:exact-sequence-square}. We may consider
  $H_+$ as a bicomplex with trivial vertical differentials. By
  Lemma~\ref{lem:exact-sequence-square}, it is null-homotopic.

    The map $\pi \cl C' \to H_+$, %
    given on each diagram by the
  following composition
  \[
    \pi:=\NB{
      %
}
  \]
  is a map of bicomplexes of $H^\prime$-equivariant bimodules. This implies
  that $C(\pi)\cong C^\prime \oplus tH_+$, the cone of $\pi$, is isomorphic to $C'$ in the relative homotopy category of $H^\prime$-equivariant bimodules.

  Consider now the complex
  \[
D:=     \soergel\left(
  \NB{\tikz[font=\tiny, yscale =0.7]{}} 
    \right) \ .
  \]
It can be thought of as a bicomplex with the zero vertical
differential. More concretely, one has
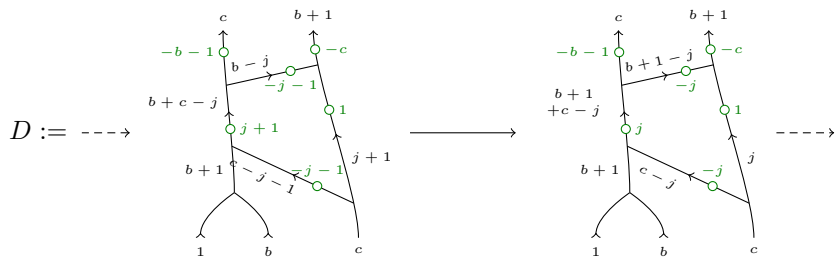
\[
  D:=\NB{
    %
} \ .
  \]
  Let us now define two chain maps $\varphi\cl H_+ \to D$ 
  and $\psi\cl C' \to D$. For each diagram, $\varphi$ is given by the
following composition:
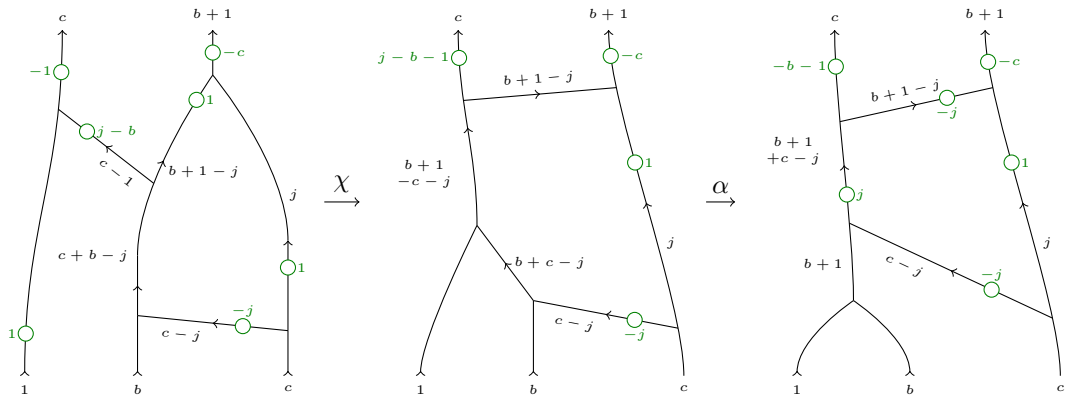
\[
\NB{
          %
    } \ ,
\]
and $\psi$ by the following composition:
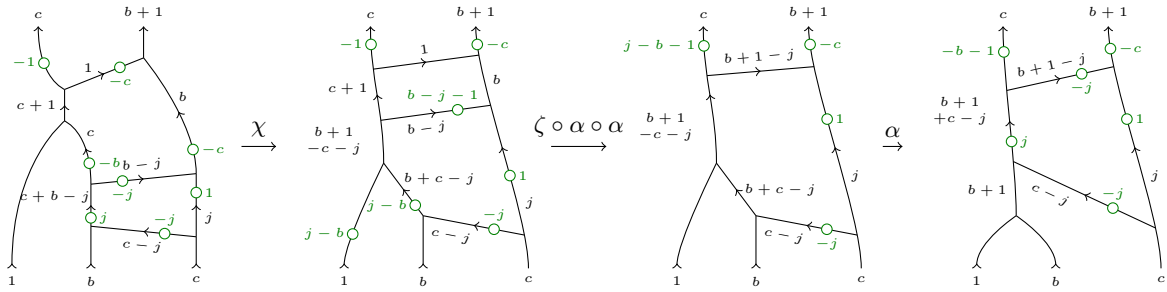
\[
\NB{
          %
 \ .
    }
\]

We put back in the grading shifts and claim that the following complex
of $H^\prime$-equivariant bimodules is null-homotopic in the usual homotopy category of Soergel
bimodules (forgetting about $H^\prime$-structures).
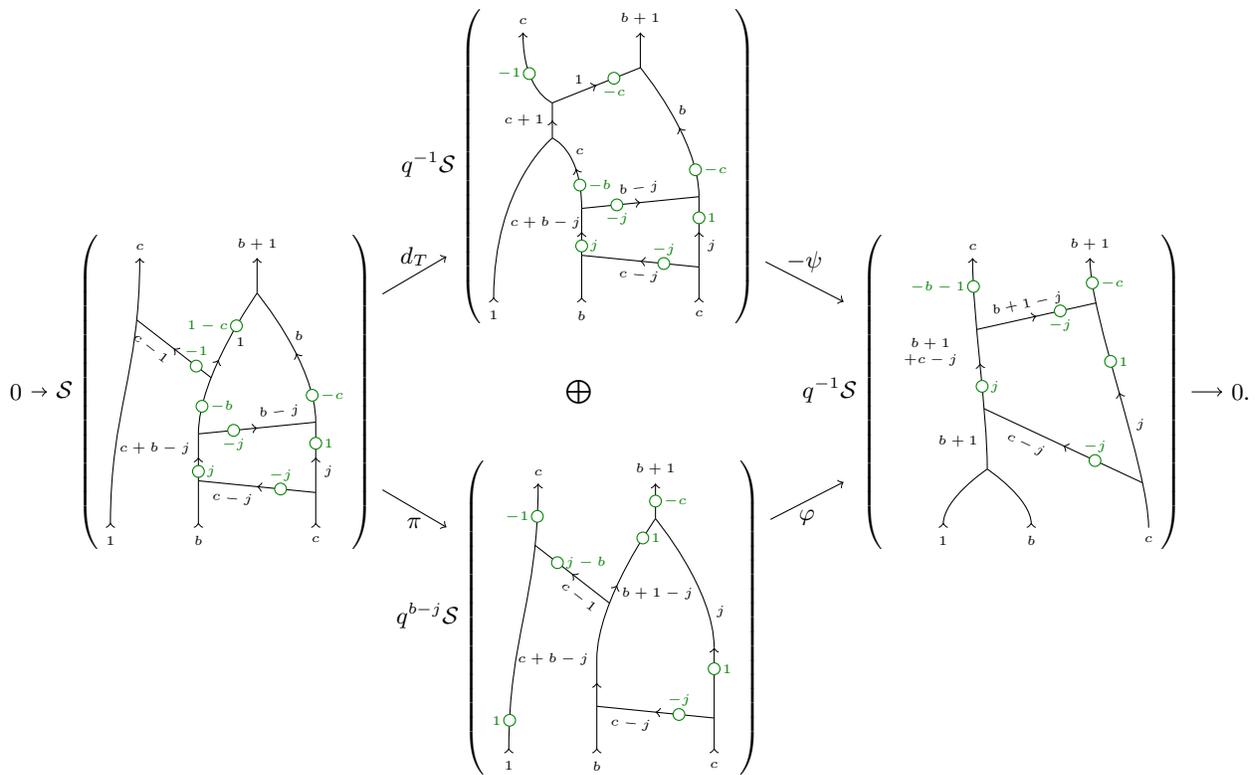
\[
  %
\]
As in the proof of the previous proposition, this is indeed a chain complex
since the square anti-commutes. The argument is essentially the same
as before and introducing notation, we rewrite the square
as follows:
\[
  %
 \ .
\]
Forgetting about the $H^\prime$-structure, we have that $M\simeq [b+1-j]
Q$ and $P\simeq N \oplus [b-j] Q$. Rewriting the maps using these
isomorphisms, we obtain:
\[
  %
\]
which is null-homotopic. This null-homotopic chain complex constitutes the homogeneous terms of the mapping cone of
\[
C(\pi) \xrightarrow{-\psi +\varphi} D. %
\]
This implies that $D$ and
$C(\pi)$ are isomorphic in the relative homotopy category. Keeping track of overall shifts, it follows that
\[
      \soergel\left(
  \NB{\tikz[font=\tiny, yscale =0.7]{}} 
    \right)
    \quad \simeq \quad
    \soergel\left(
  \NB{\tikz[font=\tiny, yscale =0.7]{}} 
    \right).
  \] 
\end{proof}

\subsection{Sliding twists}
\label{sec:sliding-twists}

The aim of this subsection is to prove that twists (or green dots) can
slide past crossings (this is made precise in
Proposition~\ref{prop:twist-slides}).

\begin{lem}\label{lem:twotwists}
  Let $a$ and $b$ be two non-negative integers and $c$ be an integer. Then there are isomorphisms of complexes of $H^\prime$-equivariant bimodules
  \begin{align}
    \soergel\left(\NB{\tikz[font=\tiny]{\begin{scope}[font=\tiny]
  \draw[->] (0.5, -0.5) ..controls +(0,0.3) and +(0,-0.3) .. (-0.5,
  0.5) node[pos=1, above] {$a$} coordinate[pos =0.2] (t1);
  \fill[white] (0,0) circle (2mm);
  \draw[->] (-0.5, -0.5) ..controls +(0,0.3) and +(0,-0.3) .. (0.5,
  0.5) node[pos=1, above] {$b$} coordinate[pos =0.2] (t2);
  \filldraw[draw= green!50!black, fill = white] (t2) circle (1mm)
  node[left, green!50!black] {$c$};
  \filldraw[draw= green!50!black, fill = white] (t1) circle (1mm)
  node[right, green!50!black] {$c$};
\end{scope}}}\right)
    \simeq
    \soergel\left(\NB{\tikz[font=\tiny]{\begin{scope}[font=\tiny]
  \draw[->] (0.5, -0.5) ..controls +(0,0.3) and +(0,-0.3) .. (-0.5,
  0.5) node[pos=1, above] {$a$} coordinate[pos =0.8] (t1);
  \fill[white] (0,0) circle (2mm);
  \draw[->] (-0.5, -0.5) ..controls +(0,0.3) and +(0,-0.3) .. (0.5,
  0.5) node[pos=1, above] {$b$} coordinate[pos =0.8] (t2);
  \filldraw[draw= green!50!black, fill = white] (t1) circle (1mm)
  node[left, green!50!black] {$c$};
  \filldraw[draw= green!50!black, fill = white] (t2) circle (1mm)
  node[right, green!50!black] {$c$};
\end{scope}}}\right)
    \qquad \text{and} \qquad 
    \soergel\left(\NB{\tikz[font=\tiny]{\begin{scope}[font=\tiny]
  \draw[->] (-0.5, -0.5) ..controls +(0,0.3) and +(0,-0.3) .. (0.5,
  0.5) node[pos=1, above] {$b$} coordinate[pos =0.2] (t2);
  \fill[white] (0,0) circle (2mm);
  \draw[->] (0.5, -0.5) ..controls +(0,0.3) and +(0,-0.3) .. (-0.5,
  0.5) node[pos=1, above] {$a$} coordinate[pos =0.2] (t1);
  \filldraw[draw= green!50!black, fill = white] (t2) circle (1mm)
  node[left, green!50!black] {$c$};
  \filldraw[draw= green!50!black, fill = white] (t1) circle (1mm)
  node[right, green!50!black] {$c$};
\end{scope}}}\right)
    \simeq
    \soergel\left(\NB{\tikz[font=\tiny]{\begin{scope}[font=\tiny]
  \draw[->] (-0.5, -0.5) ..controls +(0,0.3) and +(0,-0.3) .. (0.5,
  0.5) node[pos=1, above] {$b$} coordinate[pos =0.8] (t2);
  \fill[white] (0,0) circle (2mm);
  \draw[->] (0.5, -0.5) ..controls +(0,0.3) and +(0,-0.3) .. (-0.5,
  0.5) node[pos=1, above] {$a$} coordinate[pos =0.8] (t1);
  \filldraw[draw= green!50!black, fill = white] (t1) circle (1mm)
  node[left, green!50!black] {$c$};
  \filldraw[draw= green!50!black, fill = white] (t2) circle (1mm)
  node[right, green!50!black] {$c$};
\end{scope}}}\right).
  \end{align}
\end{lem}

\begin{proof}
  This follows directly from green dots migration on the diagrams
  defining Rickard complexes. For instance, for
  the second isomorphism, the composition 
\[
  \NB{ \begin{tikzpicture}[xscale= 6]
    \node (bot) at (-1,0) {$\soergel\left(\NB{\tikz[font=\tiny, scale
      =0.8]{\begin{scope}
  \coordinate (rt) at (+1,   1.2);
  \coordinate (lt) at (-1,   1.2);
  \coordinate (rb) at (+1,  -1.2);
  \coordinate (lb) at (-1,  -1.2);
  \coordinate (r1) at (+1, -0.6);
  \coordinate (l1) at (-1, -0.4);
  \coordinate (r2) at (+1,  0.6);
  \coordinate (l2) at (-1,  0.4);
  \draw[>-] (rb) -- (r1) node[pos=0, below] {$a$} coordinate [pos=0.7] (g1);
  \draw[>-] (lb) -- (l1) node[pos=0, below] {$b$} coordinate [pos=0.7] (g2); 
  \draw[->-] (r1) -- (r2) node[pos=0.3, right] {$j$}   coordinate [pos=0.7] (g3);
  \draw[->-] (l1) -- (l2) node[pos=0.3, left] {$a+b-j$}   coordinate [pos=0.7] (g4);
  \draw[->] (r2) -- (rt) node[pos=1, above] {$b$} coordinate [pos=0.3] (g5);
  \draw[->] (l2) -- (lt) node[pos=1, above] {$a$} coordinate [pos=0.3] (g6);
  \draw [->-] (r1) -- (l1) node[pos =0.6, below, sloped] {$a-j$} coordinate [pos=0.3] (g7);
  \draw [->-] (l2) -- (r2) node[pos =0.6, above, sloped] {$b-j$} coordinate [pos=0.3] (g8);
  \filldraw[draw= green!50!black, fill = white] (g6) circle (1mm)
  node[left, green!50!black] {$-b$};
  \filldraw[draw= green!50!black, fill = white] (g5) circle (1mm)
  node[right, green!50!black] {$-a$};
  \filldraw[draw= green!50!black, fill = white] (g4) circle (1mm)
  node[left, green!50!black] {$j$};
  \filldraw[draw= green!50!black, fill = white] (g3) circle (1mm)
  node[right, green!50!black] {$1$};
  \filldraw[draw= green!50!black, fill = white] (g7) circle (1mm)
  node[above, green!50!black] {$-j$};
  \filldraw[draw= green!50!black, fill = white] (g8) circle (1mm)
  node[below, green!50!black] {$-j$};
  \filldraw[draw= green!50!black, fill = white] (g1) circle (1mm)
  node[right, green!50!black] {$c$};
  \filldraw[draw= green!50!black, fill = white] (g2) circle (1mm)
  node[left, green!50!black] {$c$};
\end{scope}

}}\right)$} ;
    \node (mid) at ( 0,0) {$\soergel\left(\NB{\tikz[font=\tiny, scale
      =0.8]{\begin{scope}
  \coordinate (rt) at (+1,   1.2);
  \coordinate (lt) at (-1,   1.2);
  \coordinate (rb) at (+1,  -1.2);
  \coordinate (lb) at (-1,  -1.2);
  \coordinate (r1) at (+1, -0.6);
  \coordinate (l1) at (-1, -0.4);
  \coordinate (r2) at (+1,  0.6);
  \coordinate (l2) at (-1,  0.4);
  \draw[>-] (rb) -- (r1) node[pos=0, below] {$a$} coordinate [pos=0.7] (g1);
  \draw[>-] (lb) -- (l1) node[pos=0, below] {$b$} coordinate [pos=0.7] (g2); 
  \draw[->-] (r1) -- (r2) node[pos=0.3, right] {$j$}   coordinate [pos=0.7] (g3);
  \draw[->-] (l1) -- (l2) node[pos=0.3, left] {$a+b-j$}   coordinate [pos=0.7] (g4);
  \draw[->] (r2) -- (rt) node[pos=1, above] {$b$} coordinate [pos=0.3] (g5);
  \draw[->] (l2) -- (lt) node[pos=1, above] {$a$} coordinate [pos=0.3] (g6);
  \draw [->-] (r1) -- (l1) node[pos =0.6, below, sloped] {$a-j$} coordinate [pos=0.3] (g7);
  \draw [->-] (l2) -- (r2) node[pos =0.6, above, sloped] {$b-j$} coordinate [pos=0.3] (g8);
  \filldraw[draw= green!50!black, fill = white] (g6) circle (1mm)
  node[left, green!50!black] {$-b$};
  \filldraw[draw= green!50!black, fill = white] (g5) circle (1mm)
  node[right, green!50!black] {$-a$};
  \filldraw[draw= green!50!black, fill = white] (g4) circle (1mm)
  node[left, green!50!black] {$c+j$};
  \filldraw[draw= green!50!black, fill = white] (g3) circle (1mm)
  node[right, green!50!black] {$c+1$};
  \filldraw[draw= green!50!black, fill = white] (g7) circle (1mm)
  node[above, green!50!black] {$-j$};
  \filldraw[draw= green!50!black, fill = white] (g8) circle (1mm)
  node[below, green!50!black] {$-j$};
\end{scope}

}}\right)$};
    \node (top) at ( 1,0) {$\soergel\left(\NB{\tikz[font=\tiny, scale
      =0.8]{\begin{scope}
  \coordinate (rt) at (+1,   1.2);
  \coordinate (lt) at (-1,   1.2);
  \coordinate (rb) at (+1,  -1.2);
  \coordinate (lb) at (-1,  -1.2);
  \coordinate (r1) at (+1, -0.6);
  \coordinate (l1) at (-1, -0.4);
  \coordinate (r2) at (+1,  0.6);
  \coordinate (l2) at (-1,  0.4);
  \draw[>-] (rb) -- (r1) node[pos=0, below] {$a$} coordinate [pos=0.7] (g1);
  \draw[>-] (lb) -- (l1) node[pos=0, below] {$b$} coordinate [pos=0.7] (g2); 
  \draw[->-] (r1) -- (r2) node[pos=0.3, right] {$j$}   coordinate [pos=0.7] (g3);
  \draw[->-] (l1) -- (l2) node[pos=0.3, left] {$a+b-j$}   coordinate [pos=0.7] (g4);
  \draw[->] (r2) -- (rt) node[pos=1, above] {$b$} coordinate [pos=0.3] (g5);
  \draw[->] (l2) -- (lt) node[pos=1, above] {$a$} coordinate [pos=0.3] (g6);
  \draw [->-] (r1) -- (l1) node[pos =0.6, below, sloped] {$a-j$} coordinate [pos=0.3] (g7);
  \draw [->-] (l2) -- (r2) node[pos =0.6, above, sloped] {$b-j$} coordinate [pos=0.3] (g8);
  \filldraw[draw= green!50!black, fill = white] (g6) circle (1mm)
  node[left, green!50!black] {$c-b$};
  \filldraw[draw= green!50!black, fill = white] (g5) circle (1mm)
  node[right, green!50!black] {$c-a$};
  \filldraw[draw= green!50!black, fill = white] (g4) circle (1mm)
  node[left, green!50!black] {$j$};
  \filldraw[draw= green!50!black, fill = white] (g3) circle (1mm)
  node[right, green!50!black] {$1$};
  \filldraw[draw= green!50!black, fill = white] (g7) circle (1mm)
  node[above, green!50!black] {$-j$};
  \filldraw[draw= green!50!black, fill = white] (g8) circle (1mm)
  node[below, green!50!black] {$-j$};
\end{scope}

}}\right)$};    
  \draw[-to] (bot) -- (mid) node[above, pos= 0.5] {$\Id$};
  \draw[-to] (mid) -- (top) node[above, pos= 0.5] {$\Id$};
\end{tikzpicture}}
  \]
  gives an isomorphism of $H'$-equivariant Soergel bimodules which obviously commutes
  with the topological differential of the Rickard complex.
\end{proof}

\begin{lem}\label{lem:twist-slides-11}
Let $c$ be an integer. Then there are isomorphisms in the relative homotopy category
  \begin{align}
    \soergel\left(\NB{\tikz[font=\tiny]{\begin{scope}[font=\tiny]
  \draw[->] (0.5, -0.5) ..controls +(0,0.3) and +(0,-0.3) .. (-0.5,
  0.5) node[pos=1, above] {$1$} coordinate[pos =0.2] (t1);
  \fill[white] (0,0) circle (2mm);
  \draw[->] (-0.5, -0.5) ..controls +(0,0.3) and +(0,-0.3) .. (0.5,
  0.5) node[pos=1, above] {$1$} coordinate[pos =0.2] (t2);
  \filldraw[draw= green!50!black, fill = white] (t2) circle (1mm)
  node[left, green!50!black] {$c$};
\end{scope}}}\right)
    &\simeq
    \soergel\left(\NB{\tikz[font=\tiny]{\begin{scope}[font=\tiny]
  \draw[->] (0.5, -0.5) ..controls +(0,0.3) and +(0,-0.3) .. (-0.5,
  0.5) node[pos=1, above] {$1$} coordinate[pos =0.8] (t1);
  \fill[white] (0,0) circle (2mm);
  \draw[->] (-0.5, -0.5) ..controls +(0,0.3) and +(0,-0.3) .. (0.5,
  0.5) node[pos=1, above] {$1$} coordinate[pos =0.8] (t2);
  \filldraw[draw= green!50!black, fill = white] (t2) circle (1mm)
  node[right, green!50!black] {$c$};
\end{scope}}}\right),&\qquad& \qquad&
    \soergel\left(\NB{\tikz[font=\tiny]{\begin{scope}[font=\tiny]
  \draw[->] (0.5, -0.5) ..controls +(0,0.3) and +(0,-0.3) .. (-0.5,
  0.5) node[pos=1, above] {$1$} coordinate[pos =0.2] (t1);
  \fill[white] (0,0) circle (2mm);
  \draw[->] (-0.5, -0.5) ..controls +(0,0.3) and +(0,-0.3) .. (0.5,
  0.5) node[pos=1, above] {$1$} coordinate[pos =0.2] (t2);
  \filldraw[draw= green!50!black, fill = white] (t1) circle (1mm)
  node[right, green!50!black] {$c$};
\end{scope}}}\right)
    &\simeq
    \soergel\left(\NB{\tikz[font=\tiny]{\begin{scope}[font=\tiny]
  \draw[->] (0.5, -0.5) ..controls +(0,0.3) and +(0,-0.3) .. (-0.5,
  0.5) node[pos=1, above] {$1$} coordinate[pos =0.8] (t1);
  \fill[white] (0,0) circle (2mm);
  \draw[->] (-0.5, -0.5) ..controls +(0,0.3) and +(0,-0.3) .. (0.5,
  0.5) node[pos=1, above] {$1$} coordinate[pos =0.8] (t2);
  \filldraw[draw= green!50!black, fill = white] (t1) circle (1mm)
  node[left, green!50!black] {$c$};
\end{scope}}}\right), \\
    \soergel\left(\NB{\tikz[font=\tiny]{\begin{scope}[font=\tiny]
  \draw[->] (-0.5, -0.5) ..controls +(0,0.3) and +(0,-0.3) .. (0.5,
  0.5) node[pos=1, above] {$1$} coordinate[pos =0.2] (t2);
  \fill[white] (0,0) circle (2mm);
  \draw[->] (0.5, -0.5) ..controls +(0,0.3) and +(0,-0.3) .. (-0.5,
  0.5) node[pos=1, above] {$1$} coordinate[pos =0.2] (t1);
  \filldraw[draw= green!50!black, fill = white] (t2) circle (1mm)
  node[left, green!50!black] {$c$};
\end{scope}}}\right)
    &\simeq
    \soergel\left(\NB{\tikz[font=\tiny]{\begin{scope}[font=\tiny]
  \draw[->] (-0.5, -0.5) ..controls +(0,0.3) and +(0,-0.3) .. (0.5,
  0.5) node[pos=1, above] {$1$} coordinate[pos =0.8] (t2);
  \fill[white] (0,0) circle (2mm);
  \draw[->] (0.5, -0.5) ..controls +(0,0.3) and +(0,-0.3) .. (-0.5,
  0.5) node[pos=1, above] {$1$} coordinate[pos =0.8] (t1);
  \filldraw[draw= green!50!black, fill = white] (t2) circle (1mm)
  node[right, green!50!black] {$c$};
\end{scope}}}\right),%
    &\qquad & \qquad& 
    \soergel\left(\NB{\tikz[font=\tiny]{\begin{scope}[font=\tiny]
  \draw[->] (-0.5, -0.5) ..controls +(0,0.3) and +(0,-0.3) .. (0.5,
  0.5) node[pos=1, above] {$1$} coordinate[pos =0.2] (t2);
  \fill[white] (0,0) circle (2mm);
  \draw[->] (0.5, -0.5) ..controls +(0,0.3) and +(0,-0.3) .. (-0.5,
  0.5) node[pos=1, above] {$1$} coordinate[pos =0.2] (t1);
  \filldraw[draw= green!50!black, fill = white] (t1) circle (1mm)
  node[right, green!50!black] {$c$};
\end{scope}}}\right)
    &\simeq
    \soergel\left(\NB{\tikz[font=\tiny]{\begin{scope}[font=\tiny]
  \draw[->] (-0.5, -0.5) ..controls +(0,0.3) and +(0,-0.3) .. (0.5,
  0.5) node[pos=1, above] {$1$} coordinate[pos =0.8] (t2);
  \fill[white] (0,0) circle (2mm);
  \draw[->] (0.5, -0.5) ..controls +(0,0.3) and +(0,-0.3) .. (-0.5,
  0.5) node[pos=1, above] {$1$} coordinate[pos =0.8] (t1);
  \filldraw[draw= green!50!black, fill = white] (t1) circle (1mm)
  node[left, green!50!black] {$c$};
\end{scope}}}\right).
  \end{align}
\end{lem}

\begin{proof}
  Proving these isomorphisms essentially amounts to
  showing the usual ``dot sliding'' relation where the green dots are
  replaced by solid dots. 
  The proof follows closely that of \cite[Lemma 5.2]{KRWitt} and is adapted to fit
  our situation.

  It is only necessary to give the isomorphisms of the left column, since
  the other ones can be then deduced from Lemma~\ref{lem:twotwists}.
  We can also assume that $c=1$. 
  We first show 
  \begin{equation} \label{eq:homotopyslide1}
    \soergel\left(\NB{\tikz[font=\tiny]{\begin{scope}[font=\tiny]
  \draw[->] (0.5, -0.5) ..controls +(0,0.3) and +(0,-0.3) .. (-0.5,
  0.5) node[pos=1, above] {$1$} coordinate[pos =0.2] (t1);
  \fill[white] (0,0) circle (2mm);
  \draw[->] (-0.5, -0.5) ..controls +(0,0.3) and +(0,-0.3) .. (0.5,
  0.5) node[pos=1, above] {$1$} coordinate[pos =0.2] (t2);
  \filldraw[draw= green!50!black, fill = white] (t2) circle (1mm)
  node[left, green!50!black] {$1$};
\end{scope}}}\right)
    \simeq
    \soergel\left(\NB{\tikz[font=\tiny]{\begin{scope}[font=\tiny]
  \draw[->] (0.5, -0.5) ..controls +(0,0.3) and +(0,-0.3) .. (-0.5,
  0.5) node[pos=1, above] {$1$} coordinate[pos =0.8] (t1);
  \fill[white] (0,0) circle (2mm);
  \draw[->] (-0.5, -0.5) ..controls +(0,0.3) and +(0,-0.3) .. (0.5,
  0.5) node[pos=1, above] {$1$} coordinate[pos =0.8] (t2);
  \filldraw[draw= green!50!black, fill = white] (t2) circle (1mm)
  node[right, green!50!black] {$1$};
\end{scope}}}\right).
  \end{equation}
  Let $A=\Z[x_1,x_2]$ and $A[z]=\Z[x_1,x_2,z]$. We will also regard
  $A[z]$ as an $H^\prime$-equivariant $(A,A)$-bimodule by setting $\dif(z):=z^2$.
  Similary, we will denote the bimodule $B[z]:=B\otimes\Z[z]$ as an
  $H^\prime$-equivariant bimodule over $A[z]$, and thus over $A$ via
  restriction of scalars. To distinguish between the left and right
  $A$-actions, we will denote by $y_i$ the corresponding left
  $x_i$-action of $A$ (graphically, the action on the top) and
  still denote by $x_i$ the corresponding right action (graphically, the action on the bottom). In particular, on the bimodules $A$ and $A[z]$, the actions of $x_i$ and $y_i$ agree for $i=1,2$.
  
  Ignoring shifts, recall that
  \begin{align*}
  C_{x_1}:=    \soergel\left(\NB{\tikz[font=\tiny]{\begin{scope}[font=\tiny]
  \draw[->] (-0.5, -0.5) ..controls +(0,0.3) and +(0,-0.3) .. (0.5,
  0.5) node[pos=1, above] {$1$} coordinate[pos =0.2] (t2);
  \fill[white] (0,0) circle (2mm);
  \draw[->] (0.5, -0.5) ..controls +(0,0.3) and +(0,-0.3) .. (-0.5,
  0.5) node[pos=1, above] {$1$} coordinate[pos =0.2] (t1);
  \filldraw[draw= green!50!black, fill = white] (t2) circle (1mm)
  node[left, green!50!black] {$1$};
\end{scope}}}\right)
      =
    \left(B^{x_1} \xrightarrow{br}  A^{x_1}\right) \ , \qquad
    \text{and} \qquad
C_{y_2}:=    \soergel\left(\NB{\tikz[font=\tiny]{\begin{scope}[font=\tiny]
  \draw[->] (-0.5, -0.5) ..controls +(0,0.3) and +(0,-0.3) .. (0.5,
  0.5) node[pos=1, above] {$1$} coordinate[pos =0.8] (t2);
  \fill[white] (0,0) circle (2mm);
  \draw[->] (0.5, -0.5) ..controls +(0,0.3) and +(0,-0.3) .. (-0.5,
  0.5) node[pos=1, above] {$1$} coordinate[pos =0.8] (t1);
  \filldraw[draw= green!50!black, fill = white] (t2) circle (1mm)
  node[right, green!50!black] {$1$};
\end{scope}}}\right)
      =
    \left(B^{y_2} \xrightarrow{br}  A^{y_2}\right) \ .
  \end{align*}
We will also use the bimodule map determined by
\begin{equation}
rb: A\lra B, \quad     rb(1_A)=(y_2-x_1)\cdot 1_B.
\end{equation}

We can fit $C_{x_1}$ in a short exact sequence of
$H^\prime$-equivariant $(A,A)$-bimodules:
  \begin{equation}
  \NB{
  \begin{tikzpicture}[xscale =4, yscale=2]
    \node (CT) at (2,3) {$C_{x_1}$};
    \node (Ct) at (2,2) {$C'_{x_1, \lambda}$};
    \node (Cb) at (2,1) {$C''_{x_1,\lambda}$};
    \node (zeroT) at (2,3.5) {$0$};
    \node (zerob) at (2,0.5) {$0$};
    \draw[densely dotted] (CT) -- (1.2, 3);
    \draw[densely dotted] (Ct) -- (1.2, 2);
    \draw[densely dotted] (Cb) -- (1.2, 1);
    \draw[densely dotted, rounded corners] (-0.2, 3.2) rectangle (1.2, 2.8);
    \draw[densely dotted, rounded corners] (-1.2, 2.38) rectangle (1.2, 1.8);
    \draw[densely dotted, rounded corners] (-1.2, 1.38) rectangle
    (1.2,0.7);
    \draw [-to] (CT) -- (zeroT);
    \draw [-to] (Ct) -- (CT);
    \draw [-to] (Cb) -- (Ct);
    \draw [-to] (zerob) -- (Cb);
    \node (BT) at (0,3) {$B^{x_1}$};
    \node (AT) at (1,3) {$A^{x_1}$};
    \node (Zt) at (-1,2) {$B[z]^{t_\lambda}$};
    \node (Bt) at (0,2) {$B[z]^{z}\oplus A[z]^{t_\lambda}$};
    \node (At) at (1,2) {$A[z]^{z}$};
    \node (Zb) at (-1,1) {$B[z]^{t_\lambda}$};
    \node (Bb) at (0,1) {$B[z]^{2z+x_1}\oplus A[z]^{t_\lambda}$};
    \node (Ab) at (1,1) {$A[z]^{2z+x_1}$};
    \draw[-to] (BT) -- (AT) node[pos=0.5, above, scale = 0.7] {$br$};
    \draw[-to] (Bt) -- (At) node[pos=0.5, above, scale = 0.7] {$\left(br \,\,
        x_1 -z\right)$};
    \draw[-to] (Bb) -- (Ab) node[pos=0.5, above, scale = 0.7] {$\left(br \,\,
        1 \right)$};
    \draw[-to] (Zt) -- (Bt) node[pos=0.5, above, scale = 0.7]
    {$\left( \begin{array}{c}z-x_1 \\ br\end{array}\right)$};    
    \draw[-to] (Zb) -- (Bb) node[pos=0.5, above, scale = 0.7]
    {$\left( \begin{array}{c} 1 \\ -br\end{array}\right)$};
    \draw[-to] (Zb) -- (Zt) node[pos=0.5, left, scale=0.7] {$-\Id$};
    \draw[-to] (Ab) -- (At) node[pos=0.5, left, scale=0.7] {$x_1-z$};
    \draw[-to] (Bb) -- (Bt) node[pos=0.65, left, scale=0.7]
    {$\begin{pmatrix}  x_1-z & 0 \\ 0 & 1 \end{pmatrix}
      $};
    \draw[-to] (Bt) -- (BT) node[pos=0.5, left, scale=0.7]
    {$\left(z\mapsto x_1 \,\, 0 \right)$};
    \draw[-to] (At) -- (AT) node[pos=0.5, left, scale=0.7]
    {$z\mapsto x_1$};
     \draw[-to, green!50!black] ($(Bt)+(0.15, 0.1)$) .. controls
    +(-0.1,0.1) and +(0.1, 0.1) .. +(-0.3, 0)
    node[scale=0.7,green!50!black, pos=0.1, above, xshift=0.4cm]
    {$\lambda(z-x_1)  rb$};%
    \draw[-to, green!50!black] ($(Bb)+(0.15, -0.1)$) .. controls
    +(-0.1,-0.1) and +(0.1, -0.1) .. +(-0.3, 0)
    node[scale=0.7,green!50!black, pos=0.1, below] {$-\lambda rb$};%
  \end{tikzpicture}}
\ ,
\end{equation}
with $\lambda\in \Z$ and $t_\lambda= 2z+x_1 +\lambda(y_2-x_1)$.
Green arrows are meant to encode $H^\prime$-module structure twists. Namely the
twists on the $(A,A)$-bimodules
 \[
 \NB{\tikz[xscale =4, yscale=2]{\node (Bt) at (0,2)
     {$B[z]^{z}\oplus A[z]^{t_\lambda}$};      \draw[-to, green!50!black] ($(Bt)+(0.15, 0.1)$) .. controls
    +(-0.1,0.1) and +(0.1, 0.1) .. +(-0.3, 0)
    node[scale=0.7,green!50!black, pos=0.1, above] {$\lambda(z-x_1)  rb $};
}} \qquad \text{and} \qquad \NB{\tikz[xscale =4, yscale=2]{    \node (Bb) at (0,1) {$B[z]^{2z+x_1}\oplus
    A[z]^{t_\lambda}$};
      \draw[-to, green!50!black] ($(Bb)+(0.15, -0.1)$) .. controls
    +(-0.1,-0.1) and +(0.1, -0.1) .. +(-0.3, 0)
    node[scale=0.7,green!50!black, pos=0.1, below] {$-\lambda  rb $};
  }} \]
are given by the matrices
\[
  \begin{pmatrix}
    z & \lambda(z-x_1)  rb\\ 0 & t_\lambda
  \end{pmatrix} \qquad \text{and} \qquad
  \begin{pmatrix}
    2z+x_1 & -\lambda rb \\ 0 & t_\lambda
  \end{pmatrix}
\]
respectively. We let the reader check that all maps are indeed
$H^\prime$-equivariant. The complex $C''_{x_1,\lambda}$ is null-homotopic and 
this sequence splits when one forgets the $H^\prime$-module
structures. Hence $C_{x_1}$ and $C'_{x_1,\lambda}$ are isomorphic in the
relative homotopy category (see Lemma \ref{lem-construction-of-triangle}).

Similarly, for any $\lambda \in \Z$, one has an isomorphism
$C_{y_2}\cong C'_{y_2,\lambda}$ in the 
relative homotopy category with:
\[
C'_{y_2,\lambda} := \quad \left(\NB{\tikz[xscale =4, yscale=2]{
    \node (Zt) at (-1,2) {$B[z]^{s_\lambda}$};
    \node (Bt) at (0,2) {$B[z]^{z}\oplus A[z]^{s_\lambda}$};
    \node (At) at (1,2) {$A[z]^{z}$};
    \draw[-to] (Bt) -- (At) node[pos=0.5, above, scale = 0.7] {$\left(br \,\,
        y_2-z\right)$};
    \draw[-to] (Zt) -- (Bt) node[pos=0.5, above, scale = 0.7]
    {$\left( \begin{array}{c}z-y_2 \\ br\end{array}\right)$};    
     \draw[-to, green!50!black] ($(Bt)+(0.15, 0.1)$) .. controls
    +(-0.1,0.1) and +(0.1, 0.1) .. +(-0.3, 0)
    node[scale=0.7,green!50!black, pos=0.1, above] {$\lambda(z-y_2)  rb$}}}\right),
\]
where $s_\lambda = 2z+y_2 +\lambda(y_2 -x_1)$.
The complexes $C'_{x_1,0}$ and $C'_{y_2,1}$ are isomorphic as
$H'$-equivariant complexes of $(A,A)$-bimodules. An isomorphism is
given below:
\begin{align}
\NB{
  \begin{tikzpicture}[xscale =4, yscale=2]
    \node (Ct) at (2,2) {$C'_{x_1,0}$};
    \node (Cb) at (2,1) {$C'_{y_2,-1}$};
    \draw[densely dotted] (Cb) -- (1.2, 1);
    \draw[densely dotted] (Ct) -- (1.2, 2);
    \draw[densely dotted, rounded corners] (-1.2, 2.38) rectangle (1.2, 1.8);
    \draw[densely dotted, rounded corners] (-1.2, 1.38) rectangle  (1.2, 0.68);
    \draw[-to] (Cb) -- (Ct);
    \node (Zt) at (-1,2) {$B[z]^{2z+x_1}$};
    \node (Bt) at (0,2) {$B[z]^{z}\oplus A[z]^{2z+x_1}$};
    \node (At) at (1,2) {$A[z]^{z}$};
    \node (Zb) at (-1,1) {$B[z]^{2z+x_1}$};
    \node (Bb) at (0,1) {$B[z]^{z}\oplus A[z]^{2z+x_1}$};
    \node (Ab) at (1,1) {$A[z]^{z}$};
    \draw[-to] (Bt) -- (At) node[pos=0.5, above, scale = 0.7] {$\left(br \,\,
        x_1 -z\right)$};
    \draw[-to] (Bb) -- (Ab) node[pos=0.5, above, scale = 0.7] {$\left(br \,\,
        y_2 -z \right)$};
    \draw[-to] (Zt) -- (Bt) node[pos=0.5, above, scale = 0.7]
    {$\left( \begin{array}{c}z-x_1 \\ br\end{array}\right)$};    
    \draw[-to] (Zb) -- (Bb) node[pos=0.5, above, scale = 0.7]
    {$\left( \begin{array}{c} z-y_2 \\ br\end{array}\right)$};
    \draw[-to] (Zb) -- (Zt) node[pos=0.5, left, scale=0.7] {$\Id$};
    \draw[-to] (Ab) -- (At) node[pos=0.5, left, scale=0.7] {$\Id$};
    \draw[-to] (Bb) -- (Bt) node[pos=0.65, left, scale=0.7]
    {$\begin{pmatrix}  1 & rb \\ 0 & 1 \end{pmatrix}
      $};
    \draw[-to, green!50!black] ($(Bb)+(0.15, -0.1)$) .. controls
    +(-0.1,-0.1) and +(0.1, -0.1) .. +(-0.3, 0)
    node[scale=0.7,green!50!black, pos=0.1, below] {$(y_2-z)  rb$};
  \end{tikzpicture}} \ .
\end{align}
This proves that $C_{x_1} \cong C_{y_2}$. 

The twist slide for the negative crossing can be proved
analogously. It may also be deduced from that for the positive
crossing and invariance under Reidemeister II, as depicted below.
\[
\soergel\left(\NB{\NB{\tikz[yscale=1,font=\tiny]{}}}\right)
\cong
\soergel\left(\NB{\NB{\tikz[yscale=0.333,font=\tiny]{\begin{scope}[font=\tiny]
  \draw[->] (-0.5, -0.5) ..controls +(0,0.3) and +(0,-0.3) .. (0.5,
  0.5) node[pos=1, above] {$1$} coordinate[pos =0.2] (t2);
  \fill[white] (0,0) circle (2mm);
  \draw[->] (0.5, -0.5) ..controls +(0,0.3) and +(0,-0.3) .. (-0.5,
  0.5) node[pos=1, above] {$1$} coordinate[pos =0] (t1);
  \draw (0.5, -1.5) ..controls +(0,0.3) and +(0,-0.3) .. (-0.5,
  -0.5);%
  \fill[white] (0,-1) circle (2mm);
  \draw (-0.5, -1.5) ..controls +(0,0.3) and +(0,-0.3) .. (0.5,
  -0.5);%
  \draw (-0.5, -2.5) ..controls +(0,0.3) and +(0,-0.3) .. (0.5,
  -1.5) node[pos=0, below] {$1$};%
  \fill[white] (0,-2) circle (2mm);
  \draw (0.5, -2.5) ..controls +(0,0.3) and +(0,-0.3) .. (-0.5,
  -1.5) node[pos=0, below] {$1$};%
  \filldraw[draw= green!50!black, fill = white] (t1) circle (1mm and 3mm)
  node[right, green!50!black] {$c$};
\end{scope}}}}\right)
\cong
\soergel\left(\NB{\NB{\tikz[yscale=0.333,font=\tiny]{\begin{scope}[font=\tiny]
  \draw[->] (-0.5, -0.5) ..controls +(0,0.3) and +(0,-0.3) .. (0.5,
  0.5) node[pos=1, above] {$1$};%
  \fill[white] (0,0) circle (2mm);
  \draw[->] (0.5, -0.5) ..controls +(0,0.3) and +(0,-0.3) .. (-0.5,
  0.5) node[pos=1, above] {$1$};%
  \draw (0.5, -1.5) ..controls +(0,0.3) and +(0,-0.3) .. (-0.5,
  -0.5);%
  \fill[white] (0,-1) circle (2mm);
  \draw (-0.5, -1.5) ..controls +(0,0.3) and +(0,-0.3) .. (0.5,
  -0.5) coordinate[pos =0] (t1);
  \draw (-0.5, -2.5) ..controls +(0,0.3) and +(0,-0.3) .. (0.5,
  -1.5) node[pos=0, below] {$1$};%
  \fill[white] (0,-2) circle (2mm);
  \draw (0.5, -2.5) ..controls +(0,0.3) and +(0,-0.3) .. (-0.5,
  -1.5) node[pos=0, below] {$1$};%
  \filldraw[draw= green!50!black, fill = white] (t1) circle (1mm and 3mm)
  node[left, green!50!black] {$c$};
\end{scope}}}}\right)
\cong
\soergel\left(\NB{\NB{\tikz[yscale=1,font=\tiny]{}}}\right) \ .
\]
Here, the first and last isomorphisms arise from
Theorem~\ref{thm-braid-invariant}, while the middle one follows from the
previous case. This finishes the proof of the lemma.
\end{proof}

\begin{prop} \label{prop:twist-slides}
Let $a$ and $b$ be two non-negative integers and $c$ and be $d$
integers. Then one has the following isomorphism in the relative homotopy
category of $H^\prime$-equivariant bimodules: %
\begin{align}
    \soergel\left(\NB{\tikz[font=\tiny]{\begin{scope}[font=\tiny]
  \draw[->] (0.5, -0.5) ..controls +(0,0.3) and +(0,-0.3) .. (-0.5,
  0.5) node[pos=1, above] {$a$} coordinate[pos =0.2] (t1);
  \fill[white] (0,0) circle (2mm);
  \draw[->] (-0.5, -0.5) ..controls +(0,0.3) and +(0,-0.3) .. (0.5,
  0.5) node[pos=1, above] {$b$} coordinate[pos =0.2] (t2);
  \filldraw[draw= green!50!black, fill = white] (t2) circle (1mm)
  node[left, green!50!black] {$c$};
  \filldraw[draw= green!50!black, fill = white] (t1) circle (1mm)
  node[right, green!50!black] {$d$};
\end{scope}}}\right)
    \simeq
    \soergel\left(\NB{\tikz[font=\tiny]{\begin{scope}[font=\tiny]
  \draw[->] (0.5, -0.5) ..controls +(0,0.3) and +(0,-0.3) .. (-0.5,
  0.5) node[pos=1, above] {$a$} coordinate[pos =0.8] (t1);
  \fill[white] (0,0) circle (2mm);
  \draw[->] (-0.5, -0.5) ..controls +(0,0.3) and +(0,-0.3) .. (0.5,
  0.5) node[pos=1, above] {$b$} coordinate[pos =0.8] (t2);
  \filldraw[draw= green!50!black, fill = white] (t2) circle (1mm)
  node[right, green!50!black] {$c$};
  \filldraw[draw= green!50!black, fill = white] (t1) circle (1mm)
  node[left, green!50!black] {$d$};
\end{scope}}}\right)
    \qquad \text{and} \qquad 
    \soergel\left(\NB{\tikz[font=\tiny]{\begin{scope}[font=\tiny]
  \draw[->] (-0.5, -0.5) ..controls +(0,0.3) and +(0,-0.3) .. (0.5,
  0.5) node[pos=1, above] {$b$} coordinate[pos =0.2] (t2);
  \fill[white] (0,0) circle (2mm);
  \draw[->] (0.5, -0.5) ..controls +(0,0.3) and +(0,-0.3) .. (-0.5,
  0.5) node[pos=1, above] {$a$} coordinate[pos =0.2] (t1);
  \filldraw[draw= green!50!black, fill = white] (t2) circle (1mm)
  node[left, green!50!black] {$c$};
  \filldraw[draw= green!50!black, fill = white] (t1) circle (1mm)
  node[right, green!50!black] {$d$};
\end{scope}}}\right)
    \simeq
    \soergel\left(\NB{\tikz[font=\tiny]{\begin{scope}[font=\tiny]
  \draw[->] (-0.5, -0.5) ..controls +(0,0.3) and +(0,-0.3) .. (0.5,
  0.5) node[pos=1, above] {$b$} coordinate[pos =0.8] (t2);
  \fill[white] (0,0) circle (2mm);
  \draw[->] (0.5, -0.5) ..controls +(0,0.3) and +(0,-0.3) .. (-0.5,
  0.5) node[pos=1, above] {$a$} coordinate[pos =0.8] (t1);
  \filldraw[draw= green!50!black, fill = white] (t2) circle (1mm)
  node[right, green!50!black] {$c$};
  \filldraw[draw= green!50!black, fill = white] (t1) circle (1mm)
  node[left, green!50!black] {$d$};
\end{scope}}}\right).  
\end{align}
\end{prop}

\begin{proof}
  Both cases are similar, and so we only consider the second one. Using the
  blist hypothesis, we are reduced to proving:
\begin{align*}
    \soergel\left(\NB{\tikz[font=\tiny]{\begin{scope}[font=\tiny]
  \draw[->] (-0.5, -0.5) ..controls +(0,0.3) and +(0,-0.3) .. (0.5,
  0.5) node[pos=1, above] {$b$} coordinate[pos =0.2] (t2);
  \draw[>-] (-0.5, -1.5) -- +(0, 1) coordinate[pos=1] (b1)
  coordinate[pos=0.6] (b2) coordinate[pos=0.3] (b3) node [pos= 0, below] {$1$};
  \draw[>-] (-1.5, -1.5) -- (b1) node [pos= 0, below] {$1$};
  \draw[>-] (-1.1, -1.5) -- (b2) node [pos= 0, below] {$1$};
  \draw [decorate,decoration={brace,amplitude=2pt},yshift=-0.4cm] (-0.5,-1.5) -- (-1.5,-1.5) node [black,pos=0.5,below] {$b$};

  \node[below] at (-0.8, -1.5) {$\cdots$};
  \fill[white] (0,0) circle (2mm);
  \draw[->] (0.5, -0.5) ..controls +(0,0.3) and +(0,-0.3) .. (-0.5,
  0.5) node[pos=1, above] {$a$} coordinate[pos =0.2] (t1);
  \draw[>-] ( 0.5, -1.5) -- +(0, 1) coordinate[pos=1] (a1)
  coordinate[pos=0.6] (a2) coordinate[pos=0.3] (a3) node [pos= 0, below] {$1$};
  \draw[>-] ( 1.5, -1.5) -- (a1) node [pos= 0, below] {$1$};
  \draw[>-] ( 1.1, -1.5) -- (a2) node [pos= 0, below] {$1$};
  \node[below] at ( 0.8, -1.5) {$\cdots$};
  \draw [decorate,decoration={brace,amplitude=2pt},yshift=-0.4cm] (1.5,-1.5) -- (0.5,-1.5) node [black,pos=0.5,below] {$a$};
  \filldraw[draw= green!50!black, fill = white] (t2) circle (1mm)
  node[left, green!50!black] {$c$};
  \filldraw[draw= green!50!black, fill = white] (t1) circle (1mm)
  node[right, green!50!black] {$d$};
\end{scope}}}\right)
    \simeq
    \soergel\left(\NB{\tikz[font=\tiny]{\begin{scope}[font=\tiny]
  \draw[->] (-0.5, -0.5) ..controls +(0,0.3) and +(0,-0.3) .. (0.5,
  0.5) node[pos=1, above] {$b$} coordinate[pos =0.8] (t2);
  \draw[>-] (-0.5, -1.5) -- +(0, 1) coordinate[pos=1] (b1)
  coordinate[pos=0.6] (b2) coordinate[pos=0.3] (b3) node [pos= 0, below] {$1$};
  \draw[>-] (-1.5, -1.5) -- (b1) node [pos= 0, below] {$1$};
  \draw[>-] (-1.1, -1.5) -- (b2) node [pos= 0, below] {$1$};
  \draw [decorate,decoration={brace,amplitude=2pt},yshift=-0.4cm] (-0.5,-1.5) -- (-1.5,-1.5) node [black,pos=0.5,below] {$b$};

  \node[below] at (-0.8, -1.5) {$\cdots$};
  \fill[white] (0,0) circle (2mm);
  \draw[->] (0.5, -0.5) ..controls +(0,0.3) and +(0,-0.3) .. (-0.5,
  0.5) node[pos=1, above] {$a$} coordinate[pos =0.8] (t1);
  \draw[>-] ( 0.5, -1.5) -- +(0, 1) coordinate[pos=1] (a1)
  coordinate[pos=0.6] (a2) coordinate[pos=0.3] (a3) node [pos= 0, below] {$1$};
  \draw[>-] ( 1.5, -1.5) -- (a1) node [pos= 0, below] {$1$};
  \draw[>-] ( 1.1, -1.5) -- (a2) node [pos= 0, below] {$1$};
  \node[below] at ( 0.8, -1.5) {$\cdots$};
  \draw [decorate,decoration={brace,amplitude=2pt},yshift=-0.4cm] (1.5,-1.5) -- (0.5,-1.5) node [black,pos=0.5,below] {$a$};
  \filldraw[draw= green!50!black, fill = white] (t2) circle (1mm)
  node[right, green!50!black] {$c$};
  \filldraw[draw= green!50!black, fill = white] (t1) circle (1mm)
  node[left, green!50!black] {$d$};
\end{scope}
}}\right) .
\end{align*}
This is achieved using green dot migration, forkslide moves, and
Lemma~\ref{lem:twist-slides-11}. For
simplicity $\soergel( \cdot)$ is removed:
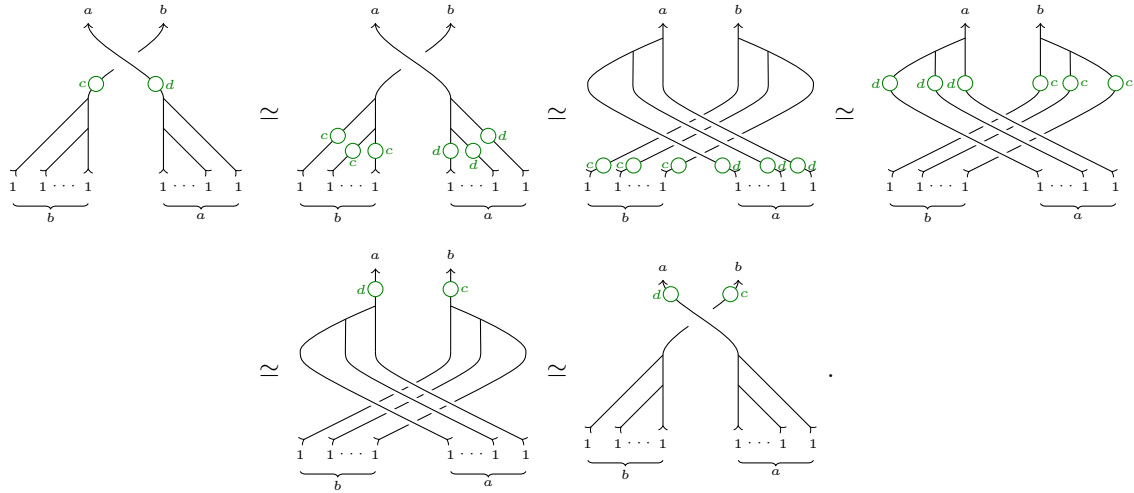
\begin{align*}
    \NB{\tikz[font=\tiny]{}}
  &\simeq
    \NB{\tikz[font=\tiny]{\begin{scope}[font=\tiny]
  \draw[->] (-0.5, -0.5) ..controls +(0,0.3) and +(0,-0.3) .. (0.5,
  0.5) node[pos=1, above] {$b$} coordinate[pos =0.2] (t2);
  \draw[>-] (-0.5, -1.5) -- +(0, 1) coordinate[pos=1] (b1)
  coordinate[pos=0.6] (b2) coordinate[pos=0.3] (tb3) node [pos= 0, below] {$1$};
  \draw[>-] (-1.5, -1.5) -- (b1) node [pos= 0, below] {$1$} coordinate[pos=0.5] (tb1);
  \draw[>-] (-1.1, -1.5) -- (b2) node [pos= 0, below] {$1$} coordinate[pos=0.5] (tb2);
  \draw [decorate,decoration={brace,amplitude=2pt},yshift=-0.4cm] (-0.5,-1.5) -- (-1.5,-1.5) node [black,pos=0.5,below] {$b$};

  \node[below] at (-0.8, -1.5) {$\cdots$};
  \fill[white] (0,0) circle (2mm);
  \draw[->] (0.5, -0.5) ..controls +(0,0.3) and +(0,-0.3) .. (-0.5,
  0.5) node[pos=1, above] {$a$} coordinate[pos =0.2] (t1);
  \draw[>-] ( 0.5, -1.5) -- +(0, 1) coordinate[pos=1] (a1)
  coordinate[pos=0.6] (a2) coordinate[pos=0.3] (ta3) node [pos= 0, below] {$1$};
  \draw[>-] ( 1.5, -1.5) -- (a1) node [pos= 0, below] {$1$} coordinate[pos=0.5] (ta1);
  \draw[>-] ( 1.1, -1.5) -- (a2) node [pos= 0, below] {$1$} coordinate[pos=0.5] (ta2);
  \node[below] at ( 0.8, -1.5) {$\cdots$};
  \draw [decorate,decoration={brace,amplitude=2pt},yshift=-0.4cm] (1.5,-1.5) -- (0.5,-1.5) node [black,pos=0.5,below] {$a$};
  \filldraw[draw= green!50!black, fill = white] (tb3) circle (1mm)
  node[right, green!50!black] {$c$};
  \filldraw[draw= green!50!black, fill = white] (tb2) circle (1mm)
  node[below, green!50!black] {$c$};
  \filldraw[draw= green!50!black, fill = white] (tb1) circle (1mm)
  node[left, green!50!black] {$c$};
  \filldraw[draw= green!50!black, fill = white] (ta3) circle (1mm)
  node[left, green!50!black] {$d$};
  \filldraw[draw= green!50!black, fill = white] (ta2) circle (1mm)
  node[below, green!50!black] {$d$};
  \filldraw[draw= green!50!black, fill = white] (ta1) circle (1mm)
  node[right, green!50!black] {$d$};
\end{scope}}}
  \simeq
    \NB{\tikz[font=\tiny]{\begin{scope}[font=\tiny]
  \node[ above] at (0.5,0.5) {$b$};
  \draw[>-] (-0.5, -1.5) .. controls +(0,0) and +(0, -0.2).. +(2, 1.2)
  node [pos= 0, below] {$1$} coordinate[pos=0.2] (tb3) .. controls
  +(0, 0.2) and +(0,0) .. (0.5, 0.3) coordinate[pos = 0.57] (bb);
  \draw[>-] (-1.5, -1.5) .. controls +(0,0) and +(0, -0.2).. +(2,1.2)
  node [pos= 0, below] {$1$} coordinate[pos=0.2] (tb1);
  \draw[->] (0.5, -0.3)-- (0.5, 0.5);
  \draw[>-] (-1.1, -1.5) .. controls +(0,0) and +(0, -0.2).. +(2,1.2)
  node [pos= 0, below] {$1$} coordinate[pos=0.2] (tb2) -- (bb);
  \draw [decorate,decoration={brace,amplitude=2pt},yshift=-0.4cm] (-0.5,-1.5) -- (-1.5,-1.5) node [black,pos=0.5,below] {$b$};
  \node[below] at (-0.8, -1.5) {$\cdots$};
  \fill[white] (0,0) circle (2mm);
  \node[above] at (-0.5, 0.5) {$a$}; %

  \draw[white, line width=1mm] ( 0.5, -1.5) .. controls +(0,0) and
  +(0, -0).. +(-2, 1);
  \draw[white, line width=1mm] ( 1.5, -1.5) .. controls +(0,0) and +(0, -0).. +(-2,1);
  \draw[white, line width=1mm] ( 1.1, -1.5) .. controls +(0,0) and
  +(0, -0).. +(-2,1);
  
  \draw[>-] ( 0.5, -1.5) .. controls +(0,0) and +(0, -0.2).. +(-2, 1.2)
  node [pos= 0, below] {$1$} coordinate[pos=0.2] (ta3) .. controls
  +(0, 0.2) and +(0,0) .. (-0.5, 0.3) coordinate[pos = 0.57] (aa);
  \draw[>-] ( 1.5, -1.5) .. controls +(0,0) and +(0, -0.2).. +(-2,1.2)
  node [pos= 0, below] {$1$} coordinate[pos=0.2] (ta1);%
  \draw[->] (-0.5, -0.3)-- (-0.5, 0.5);
  \draw[>-] ( 1.1, -1.5) .. controls +(0,0) and +(0, -0.2).. +(-2,1.2)
  node [pos= 0, below] {$1$} coordinate[pos=0.2] (ta2) -- (aa);

  \node[below] at ( 0.8, -1.5) {$\cdots$};
  \draw [decorate,decoration={brace,amplitude=2pt},yshift=-0.4cm] (1.5,-1.5) -- (0.5,-1.5) node [black,pos=0.5,below] {$a$};

  \filldraw[draw= green!50!black, fill = white] (tb3) circle (1mm)
  node[left, green!50!black] {$c$};
  \filldraw[draw= green!50!black, fill = white] (tb2) circle (1mm)
  node[left, green!50!black] {$c$};
  \filldraw[draw= green!50!black, fill = white] (tb1) circle (1mm)
  node[left, green!50!black] {$c$};
  \filldraw[draw= green!50!black, fill = white] (ta3) circle (1mm)
  node[right, green!50!black] {$d$};
  \filldraw[draw= green!50!black, fill = white] (ta2) circle (1mm)
  node[right, green!50!black] {$d$};
  \filldraw[draw= green!50!black, fill = white] (ta1) circle (1mm)
  node[right, green!50!black] {$d$};
\end{scope}}}
  \simeq
    \NB{\tikz[font=\tiny]{\begin{scope}[font=\tiny]
  \node[ above] at (0.5,0.5) {$b$};
  \draw[>-] (-0.5, -1.5) .. controls +(0,0) and +(0, -0.2).. +(2, 1.2)
  node [pos= 0, below] {$1$} coordinate[pos=1] (tb3) .. controls
  +(0, 0.2) and +(0,0) .. (0.5, 0.3) coordinate[pos = 0.57] (bb);
  \draw[>-] (-1.5, -1.5) .. controls +(0,0) and +(0, -0.2).. +(2,1.2)
  node [pos= 0, below] {$1$} coordinate[pos=1] (tb1);
  \draw[->] (0.5, -0.3)-- (0.5, 0.5);
  \draw[>-] (-1.1, -1.5) .. controls +(0,0) and +(0, -0.2).. +(2,1.2)
  node [pos= 0, below] {$1$} coordinate[pos=1] (tb2) -- (bb);
  \draw [decorate,decoration={brace,amplitude=2pt},yshift=-0.4cm] (-0.5,-1.5) -- (-1.5,-1.5) node [black,pos=0.5,below] {$b$};
  \node[below] at (-0.8, -1.5) {$\cdots$};
  \fill[white] (0,0) circle (2mm);
  \node[above] at (-0.5, 0.5) {$a$}; %

  \draw[white, line width=1mm] ( 0.5, -1.5) .. controls +(0,0) and
  +(0, -0).. +(-2, 1);
  \draw[white, line width=1mm] ( 1.5, -1.5) .. controls +(0,0) and +(0, -0).. +(-2,1);
  \draw[white, line width=1mm] ( 1.1, -1.5) .. controls +(0,0) and
  +(0, -0).. +(-2,1);
  
  \draw[>-] ( 0.5, -1.5) .. controls +(0,0) and +(0, -0.2).. +(-2, 1.2)
  node [pos= 0, below] {$1$} coordinate[pos=1] (ta3) .. controls
  +(0, 0.2) and +(0,0) .. (-0.5, 0.3) coordinate[pos = 0.57] (aa);
  \draw[>-] ( 1.5, -1.5) .. controls +(0,0) and +(0, -0.2).. +(-2,1.2)
  node [pos= 0, below] {$1$} coordinate[pos=1] (ta1);%
  \draw[->] (-0.5, -0.3)-- (-0.5, 0.5);
  \draw[>-] ( 1.1, -1.5) .. controls +(0,0) and +(0, -0.2).. +(-2,1.2)
  node [pos= 0, below] {$1$} coordinate[pos=1] (ta2) -- (aa);

  \node[below] at ( 0.8, -1.5) {$\cdots$};
  \draw [decorate,decoration={brace,amplitude=2pt},yshift=-0.4cm] (1.5,-1.5) -- (0.5,-1.5) node [black,pos=0.5,below] {$a$};

  \filldraw[draw= green!50!black, fill = white] (tb3) circle (1mm)
  node[right, green!50!black] {$c$};
  \filldraw[draw= green!50!black, fill = white] (tb2) circle (1mm)
  node[right, green!50!black] {$c$};
  \filldraw[draw= green!50!black, fill = white] (tb1) circle (1mm)
  node[right, green!50!black] {$c$};
  \filldraw[draw= green!50!black, fill = white] (ta3) circle (1mm)
  node[left, green!50!black] {$d$};
  \filldraw[draw= green!50!black, fill = white] (ta2) circle (1mm)
  node[left, green!50!black] {$d$};
  \filldraw[draw= green!50!black, fill = white] (ta1) circle (1mm)
  node[left, green!50!black] {$d$};
\end{scope}}} \\
  &\simeq
    \NB{\tikz[font=\tiny]{\begin{scope}[font=\tiny]
  \node[ above] at (0.5,0.8) {$b$};
  \draw[>-] (-0.5, -1.5) .. controls +(0,0) and +(0, -0.2).. +(2, 1.2)
  node [pos= 0, below] {$1$} coordinate[pos=1] (tb3) .. controls
  +(0, 0.2) and +(0,0) .. (0.5, 0.3) coordinate[pos = 0.57] (bb);
  \draw[>-] (-1.5, -1.5) .. controls +(0,0) and +(0, -0.2).. +(2,1.2)
  node [pos= 0, below] {$1$} coordinate[pos=1] (tb1);
  \draw[->] (0.5, -0.3)-- (0.5, 0.8) coordinate [pos = 0.75] (tb);
  \draw[>-] (-1.1, -1.5) .. controls +(0,0) and +(0, -0.2).. +(2,1.2)
  node [pos= 0, below] {$1$} coordinate[pos=1] (tb2) -- (bb);
  \draw [decorate,decoration={brace,amplitude=2pt},yshift=-0.4cm] (-0.5,-1.5) -- (-1.5,-1.5) node [black,pos=0.5,below] {$b$};
  \node[below] at (-0.8, -1.5) {$\cdots$};
  \fill[white] (0,0) circle (2mm);
  \node[above] at (-0.5, 0.8) {$a$}; %

  \draw[white, line width=1mm] ( 0.5, -1.5) .. controls +(0,0) and
  +(0, -0).. +(-2, 1);
  \draw[white, line width=1mm] ( 1.5, -1.5) .. controls +(0,0) and +(0, -0).. +(-2,1);
  \draw[white, line width=1mm] ( 1.1, -1.5) .. controls +(0,0) and
  +(0, -0).. +(-2,1);
  
  \draw[>-] ( 0.5, -1.5) .. controls +(0,0) and +(0, -0.2).. +(-2, 1.2)
  node [pos= 0, below] {$1$} coordinate[pos=1] (ta3) .. controls
  +(0, 0.2) and +(0,0) .. (-0.5, 0.3) coordinate[pos = 0.57] (aa);
  \draw[>-] ( 1.5, -1.5) .. controls +(0,0) and +(0, -0.2).. +(-2,1.2)
  node [pos= 0, below] {$1$} coordinate[pos=1] (ta1);%
  \draw[->] (-0.5, -0.3)-- (-0.5, 0.8) coordinate [pos = 0.75] (ta);
  \draw[>-] ( 1.1, -1.5) .. controls +(0,0) and +(0, -0.2).. +(-2,1.2)
  node [pos= 0, below] {$1$} coordinate[pos=1] (ta2) -- (aa);

  \node[below] at ( 0.8, -1.5) {$\cdots$};
  \draw [decorate,decoration={brace,amplitude=2pt},yshift=-0.4cm] (1.5,-1.5) -- (0.5,-1.5) node [black,pos=0.5,below] {$a$};

  \filldraw[draw= green!50!black, fill = white] (tb) circle (1mm)
  node[right, green!50!black] {$c$};
  \filldraw[draw= green!50!black, fill = white] (ta) circle (1mm)
  node[left, green!50!black] {$d$};
\end{scope}}}
  \simeq
      \NB{\tikz[font=\tiny]{}} .
\end{align*}  
\end{proof}

\subsection{Forktwists}
\label{sec:forktwists}
In this subsection we show that a fork resolves a nearby crossing up to a shift and twist in the $H^\prime$-structure.
\begin{lem}\label{lem:elementary-ft}
  The following isomorphisms hold in the relative homotopy category:
  \begin{align}
    \soergel\left(\NB{\tikz[scale=0.7, font= \tiny]{\begin{scope}
  \coordinate (t) at (0,0.5);
  \coordinate (o) at (0,0);
  \coordinate (l) at (-1,-1.5);
  \coordinate (r) at ( 1,-1.5);
  \draw[->] (o) -- (t) node [pos=1, above] {$2$};
  \draw[>-] (r) .. controls +(0, 0.3) and +(-1,-1).. (o)  coordinate[pos=0.56] (c) node[pos=0, below] {$1$};
  \fill[white] (c) circle (1mm);
  \draw[>-] (l) .. controls +(0, 0.3) and +(1,-1).. (o) node[pos=0, below] {$1$};
\end{scope}}}
    \right) 
    &\simeq
    q^{-1} \soergel\left(\NB{\tikz[scale=0.7, font= \tiny]{\begin{scope}
  \coordinate (t) at (0,0.5);
  \coordinate (o) at (0,-0.30);
  \coordinate (l) at (-1,-1.5);
  \coordinate (r) at ( 1,-1.5);
  \draw[->] (o) -- (t) node [pos=1, above] {$2$};
  \draw[>-] (l) .. controls +(0, 0.3) and +(-1,-1).. (o)
  node[pos=0, below] {$1$}  coordinate[pos=0.7] (twistl);
  \draw[>-] (r) .. controls +(0, 0.3) and +(1,-1).. (o) node[pos=0,
  below] {$1$} coordinate[pos=0.7] (twistr);
    \filldraw[draw= green!50!black, fill = white] (twistl) circle
  (1mm) node[left, green!50!black] {$1$};
  \filldraw[draw= green!50!black, fill = white] (twistr) circle
  (1mm) node[right, green!50!black] {$1$};

\end{scope}

}}
      \right), \label{eq:ft-positive} \\ %
    \soergel\left(\NB{\tikz[scale=0.7, font= \tiny]{\begin{scope}
  \coordinate (t) at (0,0.5);
  \coordinate (o) at (0,0);
  \coordinate (l) at (-1,-1.5);
  \coordinate (r) at ( 1,-1.5);
  \draw[->] (o) -- (t) node [pos=1, above] {$2$};
  \draw[>-] (l) .. controls +(0, 0.3) and +(1,-1).. (o) coordinate[pos=0.56] (c) node[pos=0, below] {$1$};
  \fill[white] (c) circle (1mm);
  \draw[>-] (r) .. controls +(0, 0.3) and +(-1,-1).. (o) node[pos=0, below] {$1$};
\end{scope}}}
    \right)
    &\simeq
    q^{} \soergel\left(\NB{\tikz[scale=0.7, font= \tiny]{\begin{scope}
  \coordinate (t) at (0,0.5);
  \coordinate (o) at (0,-0.30);
  \coordinate (l) at (-1,-1.5);
  \coordinate (r) at ( 1,-1.5);
  \draw[->] (o) -- (t) node [pos=1, above] {$2$};
  \draw[>-] (l) .. controls +(0, 0.3) and +(-1,-1).. (o)
  node[pos=0, below] {$1$}  coordinate[pos=0.7] (twistl);
  \draw[>-] (r) .. controls +(0, 0.3) and +(1,-1).. (o) node[pos=0,
  below] {$1$} coordinate[pos=0.7] (twistr);
    \filldraw[draw= green!50!black, fill = white] (twistl) circle
  (1mm) node[left, green!50!black] {$-1$};
  \filldraw[draw= green!50!black, fill = white] (twistr) circle
  (1mm) node[right, green!50!black] {$-1$};
\end{scope}

}}
    \right), \label{eq:ft-negative} \\
        \soergel\left(\NB{\tikz[scale=0.7, font= \tiny]{\begin{scope}[rotate=180]
  \coordinate (t) at (0,0.5);
  \coordinate (o) at (0,0);
  \coordinate (l) at (-1,-1.5);
  \coordinate (r) at ( 1,-1.5);
  \draw[-<] (o) -- (t) node [pos=1, below] {$2$};
  \draw[<-] (r) .. controls +(0, 0.3) and +(-1,-1).. (o)  coordinate[pos=0.56] (c) node[pos=0, above] {$1$};
  \fill[white] (c) circle (1mm);
  \draw[<-] (l) .. controls +(0, 0.3) and +(1,-1).. (o) node[pos=0, above] {$1$};
\end{scope}

}}
    \right)
    &\simeq
   q^{-1} \soergel\left(\NB{\tikz[scale=0.7, font= \tiny]{\begin{scope}[rotate=180]
  \coordinate (t) at (0,0.5);
  \coordinate (o) at (0,-0.30);
  \coordinate (l) at (-1,-1.5);
  \coordinate (r) at ( 1,-1.5);
  \draw[-<] (o) -- (t) node [pos=1, below] {$2$};
  \draw[<-] (l) .. controls +(0, 0.3) and +(-1,-1).. (o)  node[pos=0,
  above] {$1$} coordinate[pos=0.7] (twistl);
  \draw[<-] (r) .. controls +(0, 0.3) and +(1,-1).. (o) node[pos=0,
  above] {$1$} coordinate[pos=0.7] (twistr);
    \filldraw[draw= green!50!black, fill = white] (twistl) circle
  (1mm) node[right, green!50!black] {$1$};
  \filldraw[draw= green!50!black, fill = white] (twistr) circle
  (1mm) node[left, green!50!black] {$1$};
\end{scope}

}}
      \right) \label{eq:ft-positive-Y}%
                \\
    \soergel\left(\NB{\tikz[scale=0.7, font= \tiny]{\begin{scope}[rotate=180]
  \coordinate (t) at (0,0.5);
  \coordinate (o) at (0,0);
  \coordinate (l) at (-1,-1.5);
  \coordinate (r) at ( 1,-1.5);
  \draw[-<] (o) -- (t) node [pos=1, below] {$2$};
  \draw[<-] (l) .. controls +(0, 0.3) and +(1,-1).. (o)
  coordinate[pos=0.56] (c) node[pos=0, above] {$1$};
    \fill[white] (c) circle (1mm);
  \draw[<-] (r) .. controls +(0, 0.3) and +(-1,-1).. (o) node[pos=0, above] {$1$};

\end{scope}

}}
    \right)
    &\simeq
   q^{} \soergel\left(\NB{\tikz[scale=0.7, font= \tiny]{\begin{scope}[rotate=180]
  \coordinate (t) at (0,0.5);
  \coordinate (o) at (0,-0.30);
  \coordinate (l) at (-1,-1.5);
  \coordinate (r) at ( 1,-1.5);
  \draw[-<] (o) -- (t) node [pos=1, below] {$2$};
  \draw[<-] (l) .. controls +(0, 0.3) and +(-1,-1).. (o)  node[pos=0,
  above] {$1$} coordinate[pos=0.7] (twistl);
  \draw[<-] (r) .. controls +(0, 0.3) and +(1,-1).. (o) node[pos=0,
  above] {$1$} coordinate[pos=0.7] (twistr);
    \filldraw[draw= green!50!black, fill = white] (twistl) circle
  (1mm) node[right, green!50!black] {$-1$};
  \filldraw[draw= green!50!black, fill = white] (twistr) circle
  (1mm) node[left, green!50!black] {$-1$};
\end{scope}

}}
    \right) \label{eq:ft-negative-Y}.
  \end{align}
\end{lem}

\begin{proof}
Let us start with (\ref{eq:ft-positive}). One can fit
  \[
\soergel\left( \NB{\tikz[font= \tiny, scale =
  0.7]{\begin{scope}
  \coordinate (t) at (0,0.5);
  \coordinate (o) at (0,0);
  \coordinate (l) at (-1,-1.5);
  \coordinate (r) at ( 1,-1.5);
  \draw[->] (o) -- (t) node [pos=1, above] {$2$};
  \draw[>-] (r) .. controls +(0, 0.3) and +(-1,-1).. (o)  coordinate[pos=0.56] (c) node[pos=0, below] {$1$};
  \fill[white] (c) circle (1mm);
  \draw[>-] (l) .. controls +(0, 0.3) and +(1,-1).. (o) node[pos=0, below] {$1$};
\end{scope}}} \right)
=\NB{\tikz[scale =3.5]{
    \node (A) at (0,0) {$q^{-2} \soergel \left( \NB{\tikz[font= \tiny, scale =
  0.7]{\begin{scope}
  \coordinate (t) at (0,0.5);
  \coordinate (o1) at (0,0.1);
  \coordinate (o2) at (0,-0.6);
  \coordinate (o3) at (0,-1);
  \coordinate (l) at (-1,-1.5);
  \coordinate (r) at ( 1,-1.5);
  \draw[->] (o1) -- (t) node [pos=1, above] {$2$};
  \draw[->-] (o3) -- (o2) node [pos=0.5, left] {$2$};
  \draw[->-] (o2) .. controls +(0.3,0.3) and +(0.3, -0.3) .. (o1) node
 [pos= 0.5, right] {$1$};
   \draw[->-] (o2) .. controls +(-0.3,0.3) and +(-0.3, -0.3) .. (o1) node
 [pos= 0.5, left] {$1$};
 \draw[>-] (l) .. controls +(0, 0.3) and +(-0.5,-0.5).. (o3) coordinate[pos=0.56] (c) node[pos=0, below] {$1$};
  \draw[>-] (r) .. controls +(0, 0.3) and +(0.5,-0.5).. (o3) node[pos=0, below] {$1$};
\end{scope}

}} \right) $};
    \node (B) at (1.2,0) {$ q^{-3} \soergel \left( \NB{\tikz[font= \tiny, scale =
  0.7]{\begin{scope}
  \coordinate (t) at (0,0.5);
  \coordinate (o) at (0,0);
  \coordinate (l) at (-1,-1.5);
  \coordinate (r) at ( 1,-1.5);
  \draw[->] (o) -- (t) node [pos=1, above] {$2$};
  \draw[>-] (l) .. controls +(0, 0.3) and +(-1,-1).. (o)  coordinate[pos=0.56] (c) node[pos=0, below] {$1$};
  \draw[>-] (r) .. controls +(0, 0.3) and +(1,-1).. (o) node[pos=0, below] {$1$};
\end{scope}

}} \right)$};
    \draw[->] (A) -- (B) node[midway, above] {$\mapH$};
   }}
\quad\text{and}\quad
\soergel\left( \NB{\tikz[font= \tiny, scale =0.7]{}} \right)
\]
into the following short exact sequence of complexes of $H^\prime$-equivariant Soergel bimodules 
\[
\NB{\tikz[xscale =4, yscale = 3]{
    \node (i1) at (-1.8, 1) {$0$};
    \node (i2) at (-1.8, 0) {$0$};
    \node (i3) at ( 1.8, 1) {$0$};
    \node (i4) at ( 1.8, 0) {$0$};
\node (A) at (-1,1) {$ q^{-3} \soergel \left( \NB{\tikz[font= \tiny,
  scale=0.7]{}} \right) $};
\node (B) at (-1,0) {$q^{-3} \soergel \left( \NB{\tikz[font= \tiny,
  scale=0.7]{}} \right) $ };
\node (C) at (0,1) {$ q^{-2} \soergel \left( \NB{\tikz[font= \tiny,
  scale=0.7]{}} \right) $};
\node (D) at (0,0) {$ q^{-3} \soergel \left( \NB{\tikz[font= \tiny,
  scale=0.7]{}} \right) $};
\node (E) at (1,1) {$q^{-1} \soergel \left( \NB{\tikz[font= \tiny,
  scale=0.7]{}} \right) $};
\node (F) at (1,0) {$0$};
\draw[->] (A) -- (B) node [left, pos= 0.5] {$\mathrm{Id}$};
\draw[->] (C) -- (D) node [left, pos= 0.5] {$\mapH$};
\draw[->] (E) -- (F) node [right, pos= 0.5] {};
\draw[->] (A) -- (C) node [above, pos = 0.5] {$\mapB$};
\draw[->] (C) -- (E) node [above, pos = 0.5] {$\mapU$};
\draw[->] (B) -- (D) node [below, pos = 0.5] {$\mathrm{Id}$};
\draw[->] (D) -- (F) node [below, pos = 0.5] {};
\draw[->] (i1) -- (A);
\draw[->] (i2) -- (B);
\draw[->] (i1) -- (i2);
\draw[->] (E) -- (i3);
\draw[->] (F) -- (i4);
\draw[->] (i3) -- (i4);
\node[draw, densely dotted ,fit=(C) (D) ,inner sep=1ex,rectangle] {};
\node[draw, densely dotted ,fit=(A) (B) ,inner sep=1ex,rectangle] {};
\node[draw, densely dotted ,fit=(C) (D) ,inner sep=1ex,rectangle,
xshift =4cm] {};
}}
\]
whose first column is obviously null-homotopic. This gives an
isomorphism in the relative homotopy category.

  We now consider (\ref{eq:ft-negative}). One can fit
  \[
\soergel\left( \NB{\tikz[font= \tiny, scale =
  0.7]{\begin{scope}
  \coordinate (t) at (0,0.5);
  \coordinate (o) at (0,0);
  \coordinate (l) at (-1,-1.5);
  \coordinate (r) at ( 1,-1.5);
  \draw[->] (o) -- (t) node [pos=1, above] {$2$};
  \draw[>-] (l) .. controls +(0, 0.3) and +(1,-1).. (o) coordinate[pos=0.56] (c) node[pos=0, below] {$1$};
  \fill[white] (c) circle (1mm);
  \draw[>-] (r) .. controls +(0, 0.3) and +(-1,-1).. (o) node[pos=0, below] {$1$};
\end{scope}}} \right) = \NB{\tikz[scale =3.5]{
    \node (A) at (0,0) {$ q^3 \soergel \left( \NB{\tikz[font= \tiny, scale =
  0.7]{}}\right)$};
    \node (B) at (1.2,0) {$ q^2 \soergel \left(\NB{\tikz[font= \tiny, scale =
  0.7]{\begin{scope}[font=\tiny]
  \coordinate (t) at (0,0.5);
  \coordinate (o1) at (0,0.1);
  \coordinate (o2) at (0,-0.6);
  \coordinate (o3) at (0,-1);
  \coordinate (l) at (-1,-1.5);
  \coordinate (r) at ( 1,-1.5);
  \draw[->] (o1) -- (t) node [pos=1, above] {$2$};
  \draw[->-] (o3) -- (o2) node [pos=0.5, left] {$2$};
  \draw[->-] (o2) .. controls +(0.3,0.3) and +(0.3, -0.3) .. (o1) node
 [pos= 0.7, right] {$1$} coordinate[pos=0.2] (g1);
   \draw[->-] (o2) .. controls +(-0.3,0.3) and +(-0.3, -0.3) .. (o1) node
 [pos= 0.4, left] {$1$} coordinate[pos=0.8] (g2);
 \draw[>-] (l) .. controls +(0, 0.3) and +(-0.5,-0.5).. (o3) coordinate[pos=0.56] (c) node[pos=0, below] {$1$};
  \draw[>-] (r) .. controls +(0, 0.3) and +(0.5,-0.5).. (o3)
  node[pos=0, right] {$1$};
    \filldraw[draw= green!50!black, fill = white] (g1) circle
  (1mm) node[right, green!50!black] {$-1$};
      \filldraw[draw= green!50!black, fill = white] (g2) circle
  (1mm) node[left, green!50!black] {$-1$};
\end{scope}

}}\right) $ };
    \draw[->] (A) -- (B) node[midway, above] {$\mapX$};
   }}
\quad\text{and}\quad
\soergel\left( \NB{\tikz[font= \tiny, scale =0.7]{}} \right)
\]
into the following short exact sequence of complexes of $H^\prime$-equivariant bimodules
\[
\NB{\tikz[xscale = 4, yscale = 3]{
    \node (i1) at (-1.8, 1) {$0$};
    \node (i2) at (-1.8, 0) {$0$};
    \node (i3) at ( 1.8, 1) {$0$};
    \node (i4) at ( 1.8, 0) {$0$};
\node (E) at ( 1,1) {$ q^3 \soergel \left( \NB{\tikz[font= \tiny,
  scale=0.7]{}}\right) $};
\node (F) at ( 1,0) {$ q^3 \soergel \left( \NB{\tikz[font= \tiny,
  scale=0.7]{}} \right) $};
\node (D) at (0,0) {$q^2 \soergel \left( \NB{\tikz[font= \tiny,
  scale=0.7]{}} \right)$};
\node (C) at (0,1) {$ q^3 \soergel \left( \NB{\tikz[font= \tiny,
  scale=0.7]{}}\right) $};
\node (B) at (-1,0) {$q \soergel \left( \NB{\tikz[font= \tiny,
  scale=0.7]{}} \right)$};
\node (A) at (-1,1) {$0$};
\draw[->] (A) -- (B) node [left, pos= 0.5] {};
\draw[->] (C) -- (D) node [left, pos= 0.5] {$\mapX$};
\draw[->] (E) -- (F) node [right, pos= 0.5] {$\mathrm{Id}$};
\draw[->] (A) -- (C) node [above, pos = 0.5] {};
\draw[->] (C) -- (E) node [above, pos = 0.5] {$\mathrm{Id}$};
\draw[->] (B) -- (D) node [below, pos = 0.5] {$\mapB$};
\draw[->] (D) -- (F) node [below, pos = 0.5] {$\mapU$};
\draw[->] (i1) -- (A);
\draw[->] (i2) -- (B);
\draw[->] (i1) -- (i2);
\draw[->] (E) -- (i3);
\draw[->] (F) -- (i4);
\draw[->] (i3) -- (i4);
\node[draw, densely dotted ,fit=(C) (D) ,inner sep=1ex,rectangle] {};
\node[draw, densely dotted ,fit=(E) (F) ,inner sep=1ex,rectangle] {};
\node[draw, densely dotted ,fit=(C) (D) ,inner sep=1ex,rectangle,
xshift =-4cm] {};
}}
\]
whose last term is obviously null-homotopic. This gives an isomorphism in the relative homotopy category.

Isomorphisms (\ref{eq:ft-positive-Y}) and (\ref{eq:ft-negative-Y})
are analogous.
\end{proof}

\begin{cor}\label{cor:ft-ab}
  For any positive integers $a$ and $b$, the following isomorphisms hold in the relative homotopy category:
  \begin{align}
    \soergel\left(\NB{\tikz[font= \tiny]{\begin{scope}
  \coordinate (t) at (0,0.5);
  \coordinate (o) at (0,0);
  \coordinate (l) at (-1,-1.5);
  \coordinate (r) at ( 1,-1.5);
  \draw[->] (o) -- (t) node [pos=1, above] {$a+b$};
  \draw[>-] (r) .. controls +(0, 0.3) and +(-1,-1).. (o)  coordinate[pos=0.56] (c) node[pos=0, below] {$b$};
  \fill[white] (c) circle (1mm);
  \draw[>-] (l) .. controls +(0, 0.3) and +(1,-1).. (o) node[pos=0, below] {$a$};
\end{scope}}}
    \right) 
    &\simeq
    q^{-ab} \soergel\left(\NB{\tikz[font= \tiny]{\begin{scope}
  \coordinate (t) at (0,0.5);
  \coordinate (o) at (0,-0.30);
  \coordinate (l) at (-1,-1.5);
  \coordinate (r) at ( 1,-1.5);
  \draw[->] (o) -- (t) node [pos=1, above] {$a+b$};
  \draw[>-] (l) .. controls +(0, 0.3) and +(-1,-1).. (o)
  node[pos=0, below] {$a$}  coordinate[pos=0.7] (twistl);
  \draw[>-] (r) .. controls +(0, 0.3) and +(1,-1).. (o) node[pos=0,
  below] {$b$} coordinate[pos=0.7] (twistr);
    \filldraw[draw= green!50!black, fill = white] (twistl) circle
  (1mm) node[left, green!50!black] {$b$};
  \filldraw[draw= green!50!black, fill = white] (twistr) circle
  (1mm) node[right, green!50!black] {$a$};

\end{scope}

}}
      \right), \label{eq:ft-positive-ab} \\ %
    \soergel\left(\NB{\tikz[font= \tiny]{\begin{scope}
  \coordinate (t) at (0,0.5);
  \coordinate (o) at (0,0);
  \coordinate (l) at (-1,-1.5);
  \coordinate (r) at ( 1,-1.5);
  \draw[->] (o) -- (t) node [pos=1, above] {$a+b$};
  \draw[>-] (l) .. controls +(0, 0.3) and +(1,-1).. (o) coordinate[pos=0.56] (c) node[pos=0, below] {$a$};
  \fill[white] (c) circle (1mm);
  \draw[>-] (r) .. controls +(0, 0.3) and +(-1,-1).. (o) node[pos=0, below] {$b$};
\end{scope}}}
    \right)
    &\simeq
   q^{ab} \soergel\left(\NB{\tikz[font= \tiny]{\begin{scope}
  \coordinate (t) at (0,0.5);
  \coordinate (o) at (0,-0.30);
  \coordinate (l) at (-1,-1.5);
  \coordinate (r) at ( 1,-1.5);
  \draw[->] (o) -- (t) node [pos=1, above] {$a+b$};
  \draw[>-] (l) .. controls +(0, 0.3) and +(-1,-1).. (o)
  node[pos=0, below] {$a$}  coordinate[pos=0.7] (twistl);
  \draw[>-] (r) .. controls +(0, 0.3) and +(1,-1).. (o) node[pos=0,
  below] {$b$} coordinate[pos=0.7] (twistr);
  \filldraw[draw= green!50!black, fill = white] (twistl) circle
  (1mm) node[left, green!50!black] {$-b$};
  \filldraw[draw= green!50!black, fill = white] (twistr) circle
  (1mm) node[right, green!50!black] {$-a$};
\end{scope}

}}
    \right), \label{eq:ft-negative-ab} \\
        \soergel\left(\NB{\tikz[font= \tiny]{\begin{scope}[rotate=180]
  \coordinate (t) at (0,0.5);
  \coordinate (o) at (0,0);
  \coordinate (l) at (-1,-1.5);
  \coordinate (r) at ( 1,-1.5);
  \draw[-<] (o) -- (t) node [pos=1, below] {$a+b$};
  \draw[<-] (r) .. controls +(0, 0.3) and +(-1,-1).. (o)  coordinate[pos=0.56] (c) node[pos=0, above] {$a$};
  \fill[white] (c) circle (1mm);
  \draw[<-] (l) .. controls +(0, 0.3) and +(1,-1).. (o) node[pos=0,
  above] {$b$};
\end{scope}

}}
    \right)
    &\simeq
   q^{-ab} \soergel\left(\NB{\tikz[font= \tiny]{\begin{scope}[rotate=180]
  \coordinate (t) at (0,0.5);
  \coordinate (o) at (0,-0.30);
  \coordinate (l) at (-1,-1.5);
  \coordinate (r) at ( 1,-1.5);
  \draw[-<] (o) -- (t) node [pos=1, below] {$a+b$};
  \draw[<-] (l) .. controls +(0, 0.3) and +(-1,-1).. (o)  node[pos=0,
  above] {$b$} coordinate[pos=0.7] (twistl);
  \draw[<-] (r) .. controls +(0, 0.3) and +(1,-1).. (o) node[pos=0,
  above] {$a$} coordinate[pos=0.7] (twistr);
    \filldraw[draw= green!50!black, fill = white] (twistl) circle
  (1mm) node[right, green!50!black] {$a$};
  \filldraw[draw= green!50!black, fill = white] (twistr) circle
  (1mm) node[left, green!50!black] {$b$};
\end{scope}

}}
      \right), \label{eq:ft-positive-Y-ab}%
                \\
    \soergel\left(\NB{\tikz[font= \tiny]{\begin{scope}[rotate=180]
  \coordinate (t) at (0,0.5);
  \coordinate (o) at (0,0);
  \coordinate (l) at (-1,-1.5);
  \coordinate (r) at ( 1,-1.5);
  \draw[-<] (o) -- (t) node [pos=1, below] {$a+b$};
  \draw[<-] (l) .. controls +(0, 0.3) and +(1,-1).. (o)
  coordinate[pos=0.56] (c) node[pos=0, above] {$b$};
    \fill[white] (c) circle (1mm);
  \draw[<-] (r) .. controls +(0, 0.3) and +(-1,-1).. (o) node[pos=0, above] {$a$};

\end{scope}

}}
    \right)
    &\simeq
   q^{ab} \soergel\left(\NB{\tikz[font= \tiny]{\begin{scope}[rotate=180]
  \coordinate (t) at (0,0.5);
  \coordinate (o) at (0,-0.30);
  \coordinate (l) at (-1,-1.5);
  \coordinate (r) at ( 1,-1.5);
  \draw[-<] (o) -- (t) node [pos=1, below] {$a+b$};
  \draw[<-] (l) .. controls +(0, 0.3) and +(-1,-1).. (o)  node[pos=0,
  above] {$b$} coordinate[pos=0.7] (twistl);
  \draw[<-] (r) .. controls +(0, 0.3) and +(1,-1).. (o) node[pos=0,
  above] {$a$} coordinate[pos=0.7] (twistr);
    \filldraw[draw= green!50!black, fill = white] (twistl) circle
  (1mm) node[right, green!50!black] {$-a$};
  \filldraw[draw= green!50!black, fill = white] (twistr) circle
  (1mm) node[left, green!50!black] {$-b$};
\end{scope}

}}
    \right) \label{eq:ft-negative-Y-ab}.
  \end{align}
\end{cor}

\begin{proof}
  All cases are similar and may be proved by induction. We only give details
  for (\ref{eq:ft-positive-ab}). First we suppose that $b=1$ and argue
  by induction on $a$. The case $a=1$ is given by
  Lemma~\ref{lem:elementary-ft}. Using the blist hypothesis, it is
  enough to prove:
  \[
    \soergel\left(\NB{\tikz[font= \tiny]{\begin{scope}
  \coordinate (t) at (0,0.5);
  \coordinate (o) at (0,0);
  \coordinate (l) at (-0.75,-1.5);
  \coordinate (ll) at (-1.5,-1.5);
  \coordinate (r) at ( 1,-1.5);
  \draw[->] (o) -- (t) node [pos=1, above] {$a+2$};
  \draw[>-] (r) .. controls +(0, 0.3) and +(-1,-1).. (o)  coordinate[pos=0.6] (c) node[pos=0, below] {$1$};
  \fill[white] (c) circle (1mm);
  \draw[>-] (ll) .. controls +(0, 0.3) and +(1,-1).. (o) node[pos=0,
  below] {$a$} coordinate[pos =0.4] (i);
  \draw[>-] (l) .. controls +(0, 0.3) and +(0,-0).. (i) node[pos=0, below] {$1$};
\end{scope}

}}
    \right)
    \simeq
   q^{-a-1} \soergel\left(\NB{\tikz[font= \tiny]{\begin{scope}
  \coordinate (t) at (0,0.5);
  \coordinate (o) at (0,0);
  \coordinate (l) at (-0.75,-1.5);
  \coordinate (ll) at (-1.5,-1.5);
  \coordinate (r) at ( 1,-1.5);
  \draw[->] (o) -- (t) node [pos=1, above] {$a+2$};
  \draw[>-] (r) .. controls +(0, 0.3) and +(1,-1).. (o)
  coordinate[pos=0.6] (c) node[pos=0, below] {$1$} coordinate[pos=0.7] (twistr);
  \draw[>-] (ll) .. controls +(0, 0.3) and +(-1,-1).. (o) node[pos=0,
  below] {$a$} coordinate[pos =0.5] (i) coordinate[pos=0.8] (twistl);;
  \draw[>-] (l) .. controls +(0, 0.3) and +(0,-0).. (i) node[pos=0, below] {$1$};
  \filldraw[draw= green!50!black, fill = white] (twistl) circle
  (1mm) node[left, green!50!black] {$1$};
  \filldraw[draw= green!50!black, fill = white] (twistr) circle
  (1mm) node[right, green!50!black] {$a+1$};
\end{scope}

}}
    \right).
  \]
  This is obtained as follows:
  \begin{align*}
    \soergel\left(\NB{\tikz[font= \tiny]{}}
    \right)&\simeq
    \soergel\left(\NB{\tikz[font= \tiny]{\begin{scope}
  \coordinate (t) at (0,0.5);
  \coordinate (o) at (0,0);
  \coordinate (l) at (-0.75,-1.5);
  \coordinate (ll) at (-1.5,-1.5);
  \coordinate (r) at ( 1,-1.5);
  \draw[->] (o) -- (t) node [pos=1, above] {$a+2$};
  \draw[>-] (r) .. controls +(0, 0.3) and +(-1,-1).. (o)
  coordinate[pos=0.65] (c1) coordinate[pos=0.48] (c2) node[pos=0, below] {$1$} coordinate[pos=
  0.95] (oo);
  \fill[white] (c2) circle (1mm);
  \fill[white] (c1) circle (1mm);
  \draw[>-] (ll) .. controls +(0, 0.3) and +(0.8,-0.8).. (oo) node[pos=0,
  below] {$a$} coordinate[pos =0.4] (i);
  \draw[>-] (l) .. controls +(0, 0.3) and +(1.2,-1.2).. (o) node[pos=0,
  below] {$1$} coordinate[pos =0.4] (i);
\end{scope}

}}
    \right)\simeq
    q^{-a}\soergel\left(\NB{\tikz[font= \tiny]{\begin{scope}
  \coordinate (t) at (0,0.5);
  \coordinate (o) at (0,0);
  \coordinate (oo) at (-0.3,-0.30);
  \coordinate (l) at (-0.75,-1.5);
  \coordinate (ll) at (-1.5,-1.5);
  \coordinate (r) at ( 1,-1.5);
  \draw[->-] (oo) -- (o);
  \draw[->] (o) -- (t) node [pos=1, above] {$a+2$};
  \draw[>-] (r) .. controls +(0, 0.3) and +( 1,-1).. (oo)
  coordinate[pos=0.75] (c1)  node[pos=0, below] {$1$} coordinate[pos=0.92] (twistr);
  \fill[white] (c1) circle (1mm);
  \draw[>-] (ll) .. controls +(0, 0.3) and +(-0.8,-0.8).. (oo) node[pos=0,
  below] {$a$}  coordinate[pos = 0.6] (twistll);
  \draw[>-] (l) .. controls +(0, 0.3) and +(1.2,-1.2).. (o) node[pos=0,
  below] {$1$} ;
  \filldraw[draw= green!50!black, fill = white] (twistr) circle
  (1mm) node[below, green!50!black] {$a$};
  \filldraw[draw= green!50!black, fill = white] (twistll) circle
  (1mm) node[left, green!50!black] {$1$};

\end{scope}

}}
    \right)\\
           &\simeq
    q^{-a}\soergel\left(\NB{\tikz[font= \tiny]{\begin{scope}
  \coordinate (t) at (0,0.5);
  \coordinate (o) at (0,0);
  \coordinate (oo) at (0.3,-0.30);
  \coordinate (l) at (-0.75,-1.5);
  \coordinate (ll) at (-1.5,-1.5);
  \coordinate (r) at ( 1,-1.5);
  \draw[->-] (oo) -- (o);
  \draw[->] (o) -- (t) node [pos=1, above] {$a+2$};
  \draw[>-] (r) .. controls +(0, 0.3) and +( -1,-1).. (oo)
  coordinate[pos=0.585] (c1)  node[pos=0, below] {$1$} coordinate[pos=0.32] (twistr);
  \fill[white] (c1) circle (1mm);
  \draw[>-] (ll) .. controls +(0, 0.3) and +(-0.8,-0.8).. (o) node[pos=0,
  below] {$a$}  coordinate[pos = 0.6] (twistll);
  \draw[>-] (l) .. controls +(0, 0.3) and +(0.8,-0.8).. (oo) node[pos=0,
  below] {$1$} ;
  \filldraw[draw= green!50!black, fill = white] (twistr) circle
  (1mm) node[above, green!50!black] {$a$};
  \filldraw[draw= green!50!black, fill = white] (twistll) circle
  (1mm) node[left, green!50!black] {$1$};
\end{scope}

}}
    \right)\simeq
    q^{-a-1}\soergel\left(\NB{\tikz[font= \tiny]{\begin{scope}
  \coordinate (t) at (0,0.5);
  \coordinate (o) at (0,0);
  \coordinate (oo) at (0.3,-0.30);
  \coordinate (l) at (-0.75,-1.5);
  \coordinate (ll) at (-1.5,-1.5);
  \coordinate (r) at ( 1,-1.5);
  \draw[->-] (oo) -- (o);
  \draw[->] (o) -- (t) node [pos=1, above] {$a+2$};
  \draw[>-] (r) .. controls +(0, 0.3) and +( 0.8,-0.8).. (oo)
  coordinate[pos=0.585] (c1)  node[pos=0, below] {$1$} coordinate[pos=0.5] (twistr);
  \draw[>-] (ll) .. controls +(0, 0.3) and +(-0.8,-0.8).. (o) node[pos=0,
  below] {$a$}  coordinate[pos = 0.6] (twistll);
  \draw[>-] (l) .. controls +(0, 0.3) and +(-0.8,-0.8).. (oo) node[pos=0,
  below] {$1$} coordinate[pos = 0.6] (twistl);
  \filldraw[draw= green!50!black, fill = white] (twistr) circle
  (1mm) node[right, green!50!black] {$a+1$};
  \filldraw[draw= green!50!black, fill = white] (twistll) circle
  (1mm) node[left, green!50!black] {$1$};
  \filldraw[draw= green!50!black, fill = white] (twistl) circle
  (1mm) node[right, green!50!black] {$1$};

\end{scope}

}}
    \right)\\
           &\simeq
    q^{-a-1}\soergel\left(\NB{\tikz[font= \tiny]{}}
    \right).  
  \end{align*}
  The first isomorphism follows from a forkslide move and associativity.  The second isomorphism uses the induction hypothesis.  The third isomorphism is a consequence of a twist slide and associativity.  The fourth isomorphism uses Lemma \ref{lem:elementary-ft}.  The last isomorphism is a consequence of associativity and green dot migration.

  Then one can argue by induction on $b$ in a similar fashion: split
  the strand labeled $b+1$ and use the forkslide move. Then use the twist slide and the induction
  hypothesis to obtain the desired isomorphism.
\end{proof}

In the following corollary, $\Delta_k$ is the braid usually called the
\emph{positive half-twist} on $k$ strands. There are two notions of
twist at play here. We hope this will not create too much confusion.

\begin{cor}\label{cor:ft-half-twists}
  For any positive integer $k$, the following isomorphisms hold in the relative homotopy category:
  \begin{align}
    \soergel\left(\NB{\tikz[font= \tiny]{\begin{scope}
  \coordinate (t) at (0,1);
  \coordinate (o) at (0,0.5);
  \coordinate (a0) at (-1.5, -1);
  \coordinate (a1) at (-1, -1);
  \coordinate (a2) at (-0.5, -1);
  \node (a3) at (0.25, -1) {$\dots$};
  \coordinate (a5) at (1.5, -1);
  \coordinate (a4) at (1, -1);
  \coordinate (d0) at (-1.5, -2.5);
  \coordinate (d1) at (-1, -2.5);
  \coordinate (d2) at (-0.5, -2.5);
  \node (d3) at (0.25, -2.5) {$\dots$};
  \coordinate (d5) at (1.5, -2.5);
  \coordinate (d4) at (1, -2.5);
  \coordinate (B) at (-1.6, -1.2);
  \draw[->] (o) -- (t) node [above, pos = 1] {$k$};
  \draw (a0) .. controls +(0, 0.2) and +(0,0) .. (o) coordinate[pos
  =0.2] (b1) coordinate[pos =0.25] (b1) coordinate[pos =0.4] (b2) coordinate[pos =0.75] (b4);
  \draw[->-] (a5) .. controls +(0, 0.2)  and +(0,0) .. (o) ;
  \draw[->-] (a1) .. controls +(0, 0.15) and +(0,0) .. (b1);
  \draw[->-] (a2) .. controls +(0, 0.2)  and +(0,0) .. (b2);
  \draw[->-] (a4) .. controls +(0, 0.2)  and +(0,0) .. (b4);
  \draw[<-] (a0)  -- + (0,-0.2) node [pos =0, left] {$1$};;
  \draw (a1)  -- + (0,-0.2) node [pos =0, left] {$1$};
  \draw (a2)  -- + (0,-0.2) node [pos =0, right] {$1$};
  \draw (a4)  -- + (0,-0.2) node [pos =0, left] {$1$};
  \draw (a5)  -- + (0,-0.2) node [pos =0, right] {$1$};
  \draw (B) rectangle +(3.2,-0.8) node[pos =0.5, font=\small] {$\Delta_k$};
  \draw[>-] (d0)  -- + (0, 0.5) node [pos =0, below] {$1$};
  \draw[>-] (d1)  -- + (0, 0.5) node [pos =0, below] {$1$};
  \draw[>-] (d2)  -- + (0, 0.5) node [pos =0, below] {$1$};
  \draw[>-] (d4)  -- + (0, 0.5) node [pos =0, below] {$1$};
  \draw[>-] (d5)  -- + (0, 0.5) node [pos =0, below] {$1$};

\end{scope}

}}
    \right) 
    &\simeq
q^{\frac{-k(k-1)}{2}}    \soergel\left(\NB{\tikz[font= \tiny]{\begin{scope}
  \coordinate (t) at (0,1);
  \coordinate (o) at (0,-0);
  \coordinate (a0) at (-1.5, -2.5);
  \coordinate (a1) at (-1, -2.5);
  \coordinate (a2) at (-0.5, -2.5);
  \node (a3) at (0.25, -2.5) {$\dots$};
  \coordinate (a5) at (1.5, -2.5);
  \coordinate (a4) at (1, -2.5);
  \draw[->] (o) -- (t) node [above, pos = 1] {$k$} coordinate[pos
  =0.5] (g);
  \draw (a0) .. controls +(0, 0.2) and +(0,0) .. (o) coordinate[pos
  =0.2] (b1) coordinate[pos =0.25] (b1) coordinate[pos =0.4] (b2) coordinate[pos =0.75] (b4);
  \draw (a5) .. controls +(0, 0.2)  and +(0,0) .. (o) ;
  \draw (a1) .. controls +(0, 0.15) and +(0,0) .. (b1);
  \draw (a2) .. controls +(0, 0.2)  and +(0,0) .. (b2);
  \draw (a4) .. controls +(0, 0.2)  and +(0,0) .. (b4);
  \draw[-<] (a0)  -- + (0,-0.2) node [pos =1, below] {$1$};;
  \draw[-<] (a1)  -- + (0,-0.2) node [pos =1, below] {$1$};
  \draw[-<] (a2)  -- + (0,-0.2) node [pos =1, below] {$1$};
  \draw[-<] (a4)  -- + (0,-0.2) node [pos =1, below] {$1$};
  \draw[-<] (a5)  -- + (0,-0.2) node [pos =1, below] {$1$};
  \filldraw[draw= green!50!black, fill = white] (g) circle
  (1mm) node[left, green!50!black] {$k-1$};
\end{scope}

}}
      \right), \label{eq:ft-positive-delta} \\ %
    \soergel\left(\NB{\tikz[font= \tiny]{\begin{scope}
  \coordinate (t) at (0,1);
  \coordinate (o) at (0,0.5);
  \coordinate (a0) at (-1.5, -1);
  \coordinate (a1) at (-1, -1);
  \coordinate (a2) at (-0.5, -1);
  \node (a3) at (0.25, -1) {$\dots$};
  \coordinate (a5) at (1.5, -1);
  \coordinate (a4) at (1, -1);
  \coordinate (d0) at (-1.5, -2.5);
  \coordinate (d1) at (-1, -2.5);
  \coordinate (d2) at (-0.5, -2.5);
  \node (d3) at (0.25, -2.5) {$\dots$};
  \coordinate (d5) at (1.5, -2.5);
  \coordinate (d4) at (1, -2.5);
  \coordinate (B) at (-1.6, -1.2);
  \draw[->] (o) -- (t) node [above, pos = 1] {$k$};
  \draw (a0) .. controls +(0, 0.2) and +(0,0) .. (o) coordinate[pos
  =0.2] (b1) coordinate[pos =0.25] (b1) coordinate[pos =0.4] (b2) coordinate[pos =0.75] (b4);
  \draw[->-] (a5) .. controls +(0, 0.2)  and +(0,0) .. (o) ;
  \draw[->-] (a1) .. controls +(0, 0.15) and +(0,0) .. (b1);
  \draw[->-] (a2) .. controls +(0, 0.2)  and +(0,0) .. (b2);
  \draw[->-] (a4) .. controls +(0, 0.2)  and +(0,0) .. (b4);
  \draw[<-] (a0)  -- + (0,-0.2) node [pos =0, left] {$1$};;
  \draw (a1)  -- + (0,-0.2) node [pos =0, left] {$1$};
  \draw (a2)  -- + (0,-0.2) node [pos =0, right] {$1$};
  \draw (a4)  -- + (0,-0.2) node [pos =0, left] {$1$};
  \draw (a5)  -- + (0,-0.2) node [pos =0, right] {$1$};
  \draw (B) rectangle +(3.2,-0.8) node[pos =0.5, font=\small] {$\Delta_k^{-1}$};
  \draw[>-] (d0)  -- + (0, 0.5) node [pos =0, below] {$1$};
  \draw[>-] (d1)  -- + (0, 0.5) node [pos =0, below] {$1$};
  \draw[>-] (d2)  -- + (0, 0.5) node [pos =0, below] {$1$};
  \draw[>-] (d4)  -- + (0, 0.5) node [pos =0, below] {$1$};
  \draw[>-] (d5)  -- + (0, 0.5) node [pos =0, below] {$1$};

\end{scope}

}}
    \right)
    &\simeq
 q^{\frac{k(k-1)}{2}}    \soergel\left(\NB{\tikz[font= \tiny]{\begin{scope}
  \coordinate (t) at (0,1);
  \coordinate (o) at (0,-0);
  \coordinate (a0) at (-1.5, -2.5);
  \coordinate (a1) at (-1, -2.5);
  \coordinate (a2) at (-0.5, -2.5);
  \node (a3) at (0.25, -2.5) {$\dots$};
  \coordinate (a5) at (1.5, -2.5);
  \coordinate (a4) at (1, -2.5);
  \draw[->] (o) -- (t) node [above, pos = 1] {$k$} coordinate[pos
  =0.5] (g);
  \draw (a0) .. controls +(0, 0.2) and +(0,0) .. (o) coordinate[pos
  =0.2] (b1) coordinate[pos =0.25] (b1) coordinate[pos =0.4] (b2) coordinate[pos =0.75] (b4);
  \draw (a5) .. controls +(0, 0.2)  and +(0,0) .. (o) ;
  \draw (a1) .. controls +(0, 0.15) and +(0,0) .. (b1);
  \draw (a2) .. controls +(0, 0.2)  and +(0,0) .. (b2);
  \draw (a4) .. controls +(0, 0.2)  and +(0,0) .. (b4);
  \draw[-<] (a0)  -- + (0,-0.2) node [pos =1, below] {$1$};;
  \draw[-<] (a1)  -- + (0,-0.2) node [pos =1, below] {$1$};
  \draw[-<] (a2)  -- + (0,-0.2) node [pos =1, below] {$1$};
  \draw[-<] (a4)  -- + (0,-0.2) node [pos =1, below] {$1$};
  \draw[-<] (a5)  -- + (0,-0.2) node [pos =1, below] {$1$};
  \filldraw[draw= green!50!black, fill = white] (g) circle
  (1mm) node[left, green!50!black] {$1-k$};
\end{scope}

}}
    \right), \label{eq:ft-negative-delta} \\
        \soergel\left(\NB{\tikz[font= \tiny]{\begin{scope}[yscale=-1]
  \coordinate (t) at (0,1);
  \coordinate (o) at (0,0.5);
  \coordinate (a0) at (-1.5, -1);
  \coordinate (a1) at (-1, -1);
  \coordinate (a2) at (-0.5, -1);
  \node (a3) at (0.25, -1) {$\dots$};
  \coordinate (a5) at (1.5, -1);
  \coordinate (a4) at (1, -1);
  \coordinate (d0) at (-1.5, -2.5);
  \coordinate (d1) at (-1, -2.5);
  \coordinate (d2) at (-0.5, -2.5);
  \node (d3) at (0.25, -2.5) {$\dots$};
  \coordinate (d5) at (1.5, -2.5);
  \coordinate (d4) at (1, -2.5);
  \coordinate (B) at (-1.6, -1.2);
  \draw[-<] (o) -- (t) node [below, pos = 1] {$k$};
  \draw (a0) .. controls +(0, 0.2) and +(0,0) .. (o) coordinate[pos
  =0.2] (b1) coordinate[pos =0.25] (b1) coordinate[pos =0.4] (b2) coordinate[pos =0.75] (b4);
  \draw[-<-] (a5) .. controls +(0, 0.2)  and +(0,0) .. (o) ;
  \draw[-<-] (a1) .. controls +(0, 0.15) and +(0,0) .. (b1);
  \draw[-<-] (a2) .. controls +(0, 0.2)  and +(0,0) .. (b2);
  \draw[-<-] (a4) .. controls +(0, 0.2)  and +(0,0) .. (b4);
  \draw[>-] (a0)  -- + (0,-0.2) node [pos =0, left] {$1$};;
  \draw (a1)  -- + (0,-0.2) node [pos =0, left] {$1$};
  \draw (a2)  -- + (0,-0.2) node [pos =0, right] {$1$};
  \draw (a4)  -- + (0,-0.2) node [pos =0, left] {$1$};
  \draw (a5)  -- + (0,-0.2) node [pos =0, right] {$1$};
  \draw (B) rectangle +(3.2,-0.8) node[pos =0.5, font=\small] {$\Delta_k$};
  \draw[<-] (d0)  -- + (0, 0.5) node [pos =0, above] {$1$};
  \draw[<-] (d1)  -- + (0, 0.5) node [pos =0, above] {$1$};
  \draw[<-] (d2)  -- + (0, 0.5) node [pos =0, above] {$1$};
  \draw[<-] (d4)  -- + (0, 0.5) node [pos =0, above] {$1$};
  \draw[<-] (d5)  -- + (0, 0.5) node [pos =0, above] {$1$};

\end{scope}

}}
    \right)
    &\simeq
q^{\frac{-k(k-1)}{2}}     \soergel\left(\NB{\tikz[font= \tiny]{\begin{scope}[yscale=-1]
  \coordinate (t) at (0,1);
  \coordinate (o) at (0,-0);
  \coordinate (a0) at (-1.5, -2.5);
  \coordinate (a1) at (-1, -2.5);
  \coordinate (a2) at (-0.5, -2.5);
  \node (a3) at (0.25, -2.5) {$\dots$};
  \coordinate (a5) at (1.5, -2.5);
  \coordinate (a4) at (1, -2.5);
  \draw[-<] (o) -- (t) node [below, pos = 1] {$k$} coordinate[pos
  =0.5] (g);
  \draw (a0) .. controls +(0, 0.2) and +(0,0) .. (o) coordinate[pos
  =0.2] (b1) coordinate[pos =0.25] (b1) coordinate[pos =0.4] (b2) coordinate[pos =0.75] (b4);
  \draw (a5) .. controls +(0, 0.2)  and +(0,0) .. (o) ;
  \draw (a1) .. controls +(0, 0.15) and +(0,0) .. (b1);
  \draw (a2) .. controls +(0, 0.2)  and +(0,0) .. (b2);
  \draw (a4) .. controls +(0, 0.2)  and +(0,0) .. (b4);
  \draw[->] (a0)  -- + (0,-0.2) node [pos =1, above] {$1$};;
  \draw[->] (a1)  -- + (0,-0.2) node [pos =1, above] {$1$};
  \draw[->] (a2)  -- + (0,-0.2) node [pos =1, above] {$1$};
  \draw[->] (a4)  -- + (0,-0.2) node [pos =1, above] {$1$};
  \draw[->] (a5)  -- + (0,-0.2) node [pos =1, above] {$1$};
  \filldraw[draw= green!50!black, fill = white] (g) circle
  (1mm) node[left, green!50!black] {$k-1$};
\end{scope}

}}
      \right) \label{eq:ft-positive-Y-delta},%
                \\
    \soergel\left(\NB{\tikz[font= \tiny]{\begin{scope}[yscale=-1]
  \coordinate (t) at (0,1);
  \coordinate (o) at (0,0.5);
  \coordinate (a0) at (-1.5, -1);
  \coordinate (a1) at (-1, -1);
  \coordinate (a2) at (-0.5, -1);
  \node (a3) at (0.25, -1) {$\dots$};
  \coordinate (a5) at (1.5, -1);
  \coordinate (a4) at (1, -1);
  \coordinate (d0) at (-1.5, -2.5);
  \coordinate (d1) at (-1, -2.5);
  \coordinate (d2) at (-0.5, -2.5);
  \node (d3) at (0.25, -2.5) {$\dots$};
  \coordinate (d5) at (1.5, -2.5);
  \coordinate (d4) at (1, -2.5);
  \coordinate (B) at (-1.6, -1.2);
  \draw[-<] (o) -- (t) node [below, pos = 1] {$k$};
  \draw (a0) .. controls +(0, 0.2) and +(0,0) .. (o) coordinate[pos
  =0.2] (b1) coordinate[pos =0.25] (b1) coordinate[pos =0.4] (b2) coordinate[pos =0.75] (b4);
  \draw[-<-] (a5) .. controls +(0, 0.2)  and +(0,0) .. (o) ;
  \draw[-<-] (a1) .. controls +(0, 0.15) and +(0,0) .. (b1);
  \draw[-<-] (a2) .. controls +(0, 0.2)  and +(0,0) .. (b2);
  \draw[-<-] (a4) .. controls +(0, 0.2)  and +(0,0) .. (b4);
  \draw[>-] (a0)  -- + (0,-0.2) node [pos =0, left] {$1$};;
  \draw (a1)  -- + (0,-0.2) node [pos =0, left] {$1$};
  \draw (a2)  -- + (0,-0.2) node [pos =0, right] {$1$};
  \draw (a4)  -- + (0,-0.2) node [pos =0, left] {$1$};
  \draw (a5)  -- + (0,-0.2) node [pos =0, right] {$1$};
  \draw (B) rectangle +(3.2,-0.8) node[pos =0.5, font=\small] {$\Delta_k^{-1}$};
  \draw[<-] (d0)  -- + (0, 0.5) node [pos =0, above] {$1$};
  \draw[<-] (d1)  -- + (0, 0.5) node [pos =0, above] {$1$};
  \draw[<-] (d2)  -- + (0, 0.5) node [pos =0, above] {$1$};
  \draw[<-] (d4)  -- + (0, 0.5) node [pos =0, above] {$1$};
  \draw[<-] (d5)  -- + (0, 0.5) node [pos =0, above] {$1$};

\end{scope}

}}
    \right)
    &\simeq
q^{\frac{k(k-1)}{2}}     \soergel\left(\NB{\tikz[font= \tiny]{\begin{scope}[yscale=-1]
  \coordinate (t) at (0,1);
  \coordinate (o) at (0,-0);
  \coordinate (a0) at (-1.5, -2.5);
  \coordinate (a1) at (-1, -2.5);
  \coordinate (a2) at (-0.5, -2.5);
  \node (a3) at (0.25, -2.5) {$\dots$};
  \coordinate (a5) at (1.5, -2.5);
  \coordinate (a4) at (1, -2.5);
  \draw[-<] (o) -- (t) node [below, pos = 1] {$k$} coordinate[pos
  =0.5] (g);
  \draw (a0) .. controls +(0, 0.2) and +(0,0) .. (o) coordinate[pos
  =0.2] (b1) coordinate[pos =0.25] (b1) coordinate[pos =0.4] (b2) coordinate[pos =0.75] (b4);
  \draw (a5) .. controls +(0, 0.2)  and +(0,0) .. (o) ;
  \draw (a1) .. controls +(0, 0.15) and +(0,0) .. (b1);
  \draw (a2) .. controls +(0, 0.2)  and +(0,0) .. (b2);
  \draw (a4) .. controls +(0, 0.2)  and +(0,0) .. (b4);
  \draw[->] (a0)  -- + (0,-0.2) node [pos =1, above] {$1$};;
  \draw[->] (a1)  -- + (0,-0.2) node [pos =1, above] {$1$};
  \draw[->] (a2)  -- + (0,-0.2) node [pos =1, above] {$1$};
  \draw[->] (a4)  -- + (0,-0.2) node [pos =1, above] {$1$};
  \draw[->] (a5)  -- + (0,-0.2) node [pos =1, above] {$1$};
  \filldraw[draw= green!50!black, fill = white] (g) circle
  (1mm) node[left, green!50!black] {$1-k$};
\end{scope}

}}
    \right) \label{eq:ft-negative-Y-delta}.
  \end{align}
\end{cor}

\begin{proof}
  All cases are similar and proved by induction on $k$. We only give details
  for (\ref{eq:ft-positive-delta}).  The case $k=1$ is trivial and the
  case $k=2$ is given by
  Lemma~\ref{lem:elementary-ft}.

    \begin{align*}
    \soergel\left(\NB{\tikz[font= \tiny]{\begin{scope}
  \coordinate (t) at (0,1);
  \coordinate (o) at (0,0.5);
  \coordinate (a0) at (-1.5, -1);
  \coordinate (a1) at (-1, -1);
  \coordinate (a2) at (-0.5, -1);
  \node (a3) at (0.25, -1) {$\dots$};
  \coordinate (a5) at (1.5, -1);
  \coordinate (a4) at (1, -1);
  \coordinate (d0) at (-1.5, -2.5);
  \coordinate (d1) at (-1, -2.5);
  \coordinate (d2) at (-0.5, -2.5);
  \node (d3) at (0.25, -2.5) {$\dots$};
  \coordinate (d5) at (1.5, -2.5);
  \coordinate (d4) at (1, -2.5);
  \coordinate (B) at (-1.6, -1.2);
  \draw[->] (o) -- (t) node [above, pos = 1] {$k+1$};
  \draw (a0) .. controls +(0, 0.2) and +(0,0) .. (o) coordinate[pos
  =0.2] (b1) coordinate[pos =0.25] (b1) coordinate[pos =0.4] (b2) coordinate[pos =0.75] (b4);
  \draw[->-] (a5) .. controls +(0, 0.2)  and +(0,0) .. (o) ;
  \draw[->-] (a1) .. controls +(0, 0.15) and +(0,0) .. (b1);
  \draw[->-] (a2) .. controls +(0, 0.2)  and +(0,0) .. (b2);
  \draw[->-] (a4) .. controls +(0, 0.2)  and +(0,0) .. (b4);
  \draw[<-] (a0)  -- + (0,-0.2) node [pos =0, left] {$1$};;
  \draw (a1)  -- + (0,-0.2) node [pos =0, left] {$1$};
  \draw (a2)  -- + (0,-0.2) node [pos =0, right] {$1$};
  \draw (a4)  -- + (0,-0.2) node [pos =0, left] {$1$};
  \draw (a5)  -- + (0,-0.2) node [pos =0, right] {$1$};
  \draw (B) rectangle +(3.2,-0.8) node[pos =0.5, font=\small] {$\Delta_{k+1}$};
  \draw[>-] (d0)  -- + (0, 0.5) node [pos =0, below] {$1$};
  \draw[>-] (d1)  -- + (0, 0.5) node [pos =0, below] {$1$};
  \draw[>-] (d2)  -- + (0, 0.5) node [pos =0, below] {$1$};
  \draw[>-] (d4)  -- + (0, 0.5) node [pos =0, below] {$1$};
  \draw[>-] (d5)  -- + (0, 0.5) node [pos =0, below] {$1$};
\end{scope}

}}
    \right) 
      &\simeq
    \soergel\left(\NB{\tikz[font= \tiny]{\begin{scope}
  \coordinate (t) at (0,1);
  \coordinate (o) at (0,0.5);
  \coordinate (a0) at (-1.5, -1);
  \coordinate (a1) at (-1, -1);
  \coordinate (a2) at (-0.5, -1);
  \node (a3) at (0.25, -1) {$\dots$};
  \coordinate (a5) at (1.5, -1);
  \coordinate (a4) at (1, -1);
  \coordinate (d0) at (-1.5, -2.5);
  \coordinate (d1) at (-1, -2.5);
  \coordinate (d2) at (-0.5, -2.5);
  \node (dd) at (0.75, -2.5) {$\dots$};
  \coordinate (d5) at (1.5, -2.5);
  \coordinate (d4) at (1, -2.5);
  \coordinate (d3) at (0, -2.5);
  \coordinate (B) at (-1.6, -1.2);
  \draw[->] (o) -- (t) node [above, pos = 1] {$k+1$};
  \draw (a0) .. controls +(0, 0.2) and +(0,0) .. (o) coordinate[pos
  =0.2] (b1) coordinate[pos =0.25] (b1) coordinate[pos =0.4] (b2) coordinate[pos =0.75] (b4);
  \draw[->-] (a5) .. controls +(0, 0.2)  and +(0,0) .. (o) ;
  \draw[->-] (a1) .. controls +(0, 0.15) and +(0,0) .. (b1);
  \draw[->-] (a2) .. controls +(0, 0.2)  and +(0,0) .. (b2);
  \draw[->-] (a4) .. controls +(0, 0.2)  and +(0,0) .. (b4);
  \draw[<-] (a0)  -- + (0,-0.2) node [pos =0, left] {$1$};;
  \draw (a1)  -- + (0,-0.2) node [pos =0, left] {$1$};
  \draw (a2)  -- + (0,-0.2) node [pos =0, right] {$1$};
  \draw (a4)  -- + (0,-0.2) node [pos =0, left] {$1$};
  \draw (a5)  -- + (0,-0.2) node [pos =0, right] {$1$};
  \draw (B) rectangle +(2.7,-0.8) node[pos =0.5, font=\small] {$\Delta_{k}$};
  \draw[>-] (d1)  .. controls +(0,0.2) and +(0,-0.20) .. + (-0.5, 0.5) node [pos =0, below] {$1$};
  \draw[>-] (d2)  .. controls +(0,0.2) and +(0,-0.20) .. + (-0.5, 0.5) node [pos =0, below] {$1$};
  \draw[>-] (d3)  .. controls +(0,0.2) and +(0,-0.20) .. + (-0.5, 0.5) node [pos =0, below] {$1$};
  \draw[>-] (d5)  .. controls +(0,0.2) and +(0,-0.20) .. + (-0.5, 0.5) node [pos =0, below] {$1$};
  \draw[line width=1mm, white] (d0)  .. controls +(0,0.3) and +(0,-0.3) .. ++(3, 0.5);
  \draw[>-] (d0)  .. controls +(0,0.3) and +(0,-0.3) .. ++(3, 0.5)
  node [pos =0, below] {$1$} -- +(0,0.8);

\end{scope}}}
        \right) \simeq
         q^{\frac{-k(k-1)}{2}}   \soergel\left(\NB{\tikz[font=
        \tiny]{\begin{scope}
  \coordinate (t) at (0,1);
  \coordinate (o) at (0,0.5);
  \coordinate (a0) at (-1.5, -1);
  \coordinate (a1) at (-1, -1);
  \coordinate (a2) at (-0.5, -1);
  \coordinate (a5) at (1.5, -1);
  \coordinate (a4) at (1, -1);
  \coordinate (d0) at (-1.5, -2.5);
  \coordinate (d1) at (-1, -2.5);
  \coordinate (d2) at (-0.5, -2.5);
  \node (dd) at (0.75, -2.5) {$\dots$};
  \coordinate (d5) at (1.5, -2.5);
  \coordinate (d4) at (1, -2.5);
  \coordinate (d3) at (0, -2.5);
  \coordinate (B) at (-1.6, -1.2);
  \draw[->] (o) -- (t) node [above, pos = 1] {$k+1$};
  \draw (a0) .. controls +(0, 0.2) and +(0,0) .. (o) coordinate[pos
  =0.2] (b1) coordinate[pos =0.25] (b1) coordinate[pos =0.4] (b2)
  coordinate[pos =0.7] (b4) coordinate[pos= 0.8] (twist) ;
  \draw[->-] (a5) .. controls +(0, 0.2)  and +(0,0) .. (o) ;
  \draw[>-] (d1)  .. controls +(0,0.2) and +(0,-0.20) ..  (a0) node [pos =0, below] {$1$};
  \draw[>-] (d2)  .. controls +(0,0.2) and +(0,-0.20) ..  (b1) node [pos =0, below] {$1$};
  \draw[>-] (d3)  .. controls +(0,0.2) and +(0,-0.20) .. (b2) node [pos =0, below] {$1$};
  \draw[>-] (d5)  .. controls +(0,0.2) and +(0,-0.20) ..  (b4) node [pos =0, below] {$1$};
  \draw[line width=1mm, white] (d0)  .. controls +(0,0.3) and +(0,-0.3) .. (a5);
  \draw[>-] (d0)  .. controls +(0,0.3) and +(0,-0.3) .. (a5) node [pos =0, below] {$1$};
\filldraw[draw= green!50!black, fill = white] (twist) circle
  (1mm) node[left, green!50!black] {$k-1$};
\end{scope}

}} \right) \\
      &\simeq
   q^{\frac{-k(k+1)}{2}} \soergel\left(\NB{\tikz[font= \tiny]{\begin{scope}
  \coordinate (t) at (0,1);
  \coordinate (o) at (0.3,-0.5);
  \coordinate (oo) at (0,-0);
  \coordinate (a0) at (-1.5, -2.5);
  \coordinate (a1) at (-1, -2.5);
  \coordinate (a2) at (-0.5, -2.5);
  \node (a3) at (0.25, -2.5) {$\dots$};
  \coordinate (a5) at (1.5, -2.5);
  \coordinate (a4) at (1, -2.5);
  \draw (o) -- (oo) coordinate[pos=0.5] (twistr); 
  \draw[->] (oo) -- (t) node [above, pos = 1] {$k+1$} coordinate[pos
  =0.5] (g);
  \draw (a0) .. controls +(0, 0.2) and +(0,0) .. (oo) coordinate[pos
  =0.2] (b1) coordinate[pos = 0.5] (twistl);%
  \draw (a5) .. controls +(0, 0.2)  and +(0,0) .. (o) ;
  \draw (a1) .. controls +(0, 0.2) and +(0,0) .. (o) coordinate[pos =0.3] (b2)
  coordinate[pos =0.7] (b4);
  \draw (a2) .. controls +(0, 0.15)  and +(0,0) .. (b2);
  \draw (a4) .. controls +(0, 0.2)  and +(0,0) .. (b4);
  \draw[-<] (a0)  -- + (0,-0.2) node [pos =1, below] {$1$};;
  \draw[-<] (a1)  -- + (0,-0.2) node [pos =1, below] {$1$};
  \draw[-<] (a2)  -- + (0,-0.2) node [pos =1, below] {$1$};
  \draw[-<] (a4)  -- + (0,-0.2) node [pos =1, below] {$1$};
  \draw[-<] (a5)  -- + (0,-0.2) node [pos =1, below] {$1$};
  \filldraw[draw= green!50!black, fill = white] (twistl) circle
   (1mm) node[left, green!50!black] {$k$};
  \filldraw[draw= green!50!black, fill = white] (twistr) circle
   (1mm) node[right, green!50!black] {$k-1$};
   \filldraw[draw= green!50!black, fill = white] (a1) circle
   (1mm) node[left, green!50!black] {$1$};
  \filldraw[draw= green!50!black, fill = white] (a2) circle
   (1mm) node[right, green!50!black] {$1$};
  \filldraw[draw= green!50!black, fill = white] (a4) circle
   (1mm) node[left, green!50!black] {$1$};
   \filldraw[draw= green!50!black, fill = white] (a5) circle
   (1mm) node[right, green!50!black] {$1$};
\end{scope}

}}
        \right) \simeq
        q^{\frac{-k(k+1)}{2}}    \soergel\left(\NB{\tikz[font=
        \tiny]{\begin{scope}
  \coordinate (t) at (0,1);
  \coordinate (o) at (0,-0);
  \coordinate (a0) at (-1.5, -2.5);
  \coordinate (a1) at (-1, -2.5);
  \coordinate (a2) at (-0.5, -2.5);
  \node (a3) at (0.25, -2.5) {$\dots$};
  \coordinate (a5) at (1.5, -2.5);
  \coordinate (a4) at (1, -2.5);
  \draw[->] (o) -- (t) node [above, pos = 1] {$k+1$} coordinate[pos
  =0.5] (g);
  \draw (a0) .. controls +(0, 0.2) and +(0,0) .. (o) coordinate[pos
  =0.2] (b1) coordinate[pos =0.25] (b1) coordinate[pos =0.4] (b2) coordinate[pos =0.75] (b4);
  \draw (a5) .. controls +(0, 0.2)  and +(0,0) .. (o) ;
  \draw (a1) .. controls +(0, 0.15) and +(0,0) .. (b1);
  \draw (a2) .. controls +(0, 0.2)  and +(0,0) .. (b2);
  \draw (a4) .. controls +(0, 0.2)  and +(0,0) .. (b4);
  \draw[-<] (a0)  -- + (0,-0.2) node [pos =1, below] {$1$};;
  \draw[-<] (a1)  -- + (0,-0.2) node [pos =1, below] {$1$};
  \draw[-<] (a2)  -- + (0,-0.2) node [pos =1, below] {$1$};
  \draw[-<] (a4)  -- + (0,-0.2) node [pos =1, below] {$1$};
  \draw[-<] (a5)  -- + (0,-0.2) node [pos =1, below] {$1$};
  \filldraw[draw= green!50!black, fill = white] (g) circle
  (1mm) node[left, green!50!black] {$k$};
\end{scope}

}} \right)
    \end{align*}
\end{proof}
The first isomorphism only uses the definition of the positive half-twist.  The second isomorphism holds by the induction hypothesis.  The third isomorphism follows from repeated use of Lemma \ref{lem:elementary-ft}, twist slides, associativity, and green dot migration.  The last isomorphism is a direct consequence of green dot migration and associativity.

\section{Main theorem}
\label{sec:coloredmarkII}
In this section we construct a categorification of the colored
Jones polynomial evaluated at a root of unity. Links are framed,
oriented and colored by
non-negative integers, and the integer $k$ represents the $k$th
symmetric power of the standard $U_q(\mathfrak{sl}_2)$-representation. These
colored links are presented as braid closures.  In particular braids
we consider are colored and the colorings of their top and bottom boundaries
match so that their closures give rise to well-defined colored
links. In what follows, such braids are simply called \emph{colored
  braids}. The usual Alexander and Markov theorems have natural analogs in this
context. One can deduce them immediately from the classical Alexander
and Markov theorems. For stating the first Markov move, one actually
needs to consider two colored braids whose top and bottom colors do not
necessarily match, but for which both concatenations are well-defined.

\subsection{Construction}
\label{sec:construction}
Here we repeat the construction of Section~\ref{sec:def-uncolored} in
the colored case.
Let $\beta$ be a colored braid diagram on $r$ strands (such colored
braids will be assumed to have matching top and bottom boundaries). It
is composed of colored crossings. As prescribed by \eqref{eqn:posab}
and \eqref{eqn:negab}, consider the hyperrectangle $R_\beta$ obtained by
replacing each colored crossing by the corresponding 
$H^\prime$-equivariant complex of Soergel bimodules given in Definition
\ref{def:rickard}, and flatten it to obtain an $H^\prime$-equivariant
complex of Soergel bimodules $\soergel(\beta)$. In the uncolored case (Section~\ref{sec:def-uncolored}), $\soergel(\beta)$ is denoted $T_\beta$.

The Hochschild homology $\mHH_\bullet^{\dif}(\soergel(\beta))$ of
$\soergel(\beta)$ is triply-graded by its Hochschild degree (or
$a$-degree), its quantum degree (or $q$-degree) and its topological
degree (or $t$-degree). It is endowed with two super differentials: the
Cautis differential $d_C$ with degree $(-1, 2, 0)$ and the topological
differential $d_t$ with degree $(0,0,1)$. We now form a chain complex with respect to the
\emph{total differential} $d_T:=d_t+d_C$, yielding a complex $\left(
  \mHH_\bullet^\dif(\soergel(\beta)), d_T \right)$ of
$H^\prime$-modules. Because of the disagreement of the degrees of
$d_t$ and $d_C$, we are forced to collapse the $a$-grading onto $t$ and
$q$ by
\begin{equation}\label{eqnatqdegreecollapse-r1}
  a=t^{-1}q^2.
\end{equation}
As a result, the complex $\left( \mHH_\bullet^\dif(\soergel(\beta)), d_T \right)$ is a 
bigraded complex of $H^\prime$-modules.

\begin{defn}
\label{def-HHH-colored}
Let $\beta$ be a colored braid on $r$ strands colored by $a_1, \dots , a_r$.
\begin{enumerate}
   \item[(1)]  The \emph{$H^\prime$-module} of $\beta$ is the bigraded
  $H^\prime$-representation
  \[
    \mC^{\prime}(\beta):=q^{-{\sum_{i=1}^r a_i}} \mHT\left( \mHH_\bullet^\dif(\soergel(\beta)), d_T \right) .
  \]
Here the notation $\mHT$ on the right hand side emphasizes that the homology is taken with respect to the total differential $d_T$.  
 \item[(2)] The \emph{$H$-module} of $\beta$ is the bigraded $H$-representation (see Corollary \ref{cor-p-nilpotency})
 \[
    \mC^{\dif}(\beta):=\mC^\prime(\beta)^{\mathrm{free}}\otimes_\Z \F_p,
  \]
  where $\mC^\prime(\beta)^{\mathrm{free}}$ denotes the free part of the $\Z$-module.
 \item[(3)]
  The \emph{slash cohomology} of $\beta$ is the slash
  cohomology (see \eqref{eqnslashcohomology}) of $\mC^\dif(\beta)$:
  \[
    \mH^{\dif}_/(\beta):=\mH_/ \left( \mC^\dif(\beta)
    \right).
  \]
  By definition, it is the image of the bigraded module $\mC^{\dif}(\beta)$ in
  the stable category $H\udmod$.
\end{enumerate}
\end{defn}

The definition of $\mC^{\prime}$ extends straightforwardly to diagram
of knotted MOY graphs.

\subsection{Invariance}
\label{sec:invariance-col}

\begin{prop}\label{prop:braidrel-col}
  In the relative homotopy category, the $H^\prime$-equivariant complex of
  Soergel bimodule $\soergel(\beta)$ is invariant under braid relations.
\end{prop}
\begin{proof}
  One only need to prove that:
  \[
\soergel\left(\NB{\tikz[scale =0.8]{\begin{scope}[font=\tiny]
  \draw (0.5, -0.5) ..controls +(0,0.3) and +(0,-0.3) .. (-0.5,
  0.5);%
  \fill[white] (0,0) circle (2mm);
  \draw (-0.5, -0.5) ..controls +(0,0.3) and +(0,-0.3) .. (0.5,
  0.5);%
  \draw[->] (-0.5, 0.5) ..controls +(0,0.3) and +(0,-0.3) .. (0.5,
  1.5) node[pos=1, above] {$a$} coordinate[pos =0.2] (t1);
  \fill[white] (0,1) circle (2mm);
  \draw[->] (0.5, 0.5) ..controls +(0,0.3) and +(0,-0.3) .. (-0.5,
  1.5) node[pos=1, above] {$b$} coordinate[pos =0.2] (t2);
\end{scope}}} \right) \simeq
\soergel\left(\NB{\tikz[scale =0.8]{}} \right) \simeq
\soergel\left(\NB{\tikz[scale =0.8]{\begin{scope}[font=\tiny]
  \draw (-0.5, -0.5) ..controls +(0,0.3) and +(0,-0.3) .. (0.5,
  0.5);%
  \fill[white] (0,0) circle (2mm);
  \draw (0.5, -0.5) ..controls +(0,0.3) and +(0,-0.3) .. (-0.5,
  0.5);%
  \draw[->] (0.5, 0.5) ..controls +(0,0.3) and +(0,-0.3) .. (-0.5,
  1.5) node[pos=1, above] {$a$} coordinate[pos =0.2] (t2);
  \fill[white] (0,1) circle (2mm);
  \draw[->] (-0.5, 0.5) ..controls +(0,0.3) and +(0,-0.3) .. (0.5,
  1.5) node[pos=1, above] {$b$} coordinate[pos =0.2] (t1);
\end{scope}}} \right) 
\]
and
\[
\soergel\left(\NB{\tikz[scale =0.8]{\begin{scope}[font=\tiny]
  \draw (0.5, -0.5) ..controls +(0,0.3) and +(0,-0.3) .. (-0.5,
  0.5);%
  \fill[white] (0,0) circle (2mm);
  \draw (-0.5, -0.5) ..controls +(0,0.3) and +(0,-0.3) .. (0.5,
  0.5);%
  \draw (1.5, 0.5) ..controls +(0,0.3) and +(0,-0.3) .. (0.5,
  1.5);%
  \fill[white] (1,1) circle (2mm);
  \draw (0.5, 0.5) ..controls +(0,0.3) and +(0,-0.3) .. (1.5,
  1.5);%
  \draw[->] (0.5, 1.5) ..controls +(0,0.3) and +(0,-0.3) .. (-0.5,
  2.5) node[pos=1, above] {$c$} coordinate[pos =0.2] (t2);
  \fill[white] (0,2) circle (2mm);
  \draw[->] (-0.5, 1.5) ..controls +(0,0.3) and +(0,-0.3) .. (0.5,
  2.5) node[pos=1, above] {$b$} coordinate[pos =0.2] (t1);
  \draw (-0.5, 0.5) -- (-0.5, 1.5);
  \draw (1.5, -0.5) -- (1.5, 0.5);
  \draw[->] (1.5, 1.5) -- (1.5, 2.5) node[pos=1, above] {$a$};
\end{scope}}} \right) \simeq
\soergel\left(\NB{\tikz[scale =0.8]{\begin{scope}[font=\tiny]
  \draw (1.5, -0.5) ..controls +(0,0.3) and +(0,-0.3) .. ( 0.5,
  0.5);%
  \fill[white] (1,0) circle (2mm);
  \draw (0.5, -0.5) ..controls +(0,0.3) and +(0,-0.3) .. (1.5,
  0.5);%
  \draw (0.5, 0.5) ..controls +(0,0.3) and +(0,-0.3) .. (-0.5,
  1.5);%
  \fill[white] (0,1) circle (2mm);
  \draw (-0.5, 0.5) ..controls +(0,0.3) and +(0,-0.3) .. (0.5,
  1.5);%
  \draw[->] (1.5, 1.5) ..controls +(0,0.3) and +(0,-0.3) .. (0.5,
  2.5) node[pos=1, above] {$b$} coordinate[pos =0.2] (t2);
  \fill[white] (1,2) circle (2mm);
  \draw[->] (0.5, 1.5) ..controls +(0,0.3) and +(0,-0.3) .. (1.5,
  2.5) node[pos=1, above] {$a$} coordinate[pos =0.2] (t1);
  \draw (1.5, 0.5) -- (1.5, 1.5);
  \draw (-0.5, -0.5) -- (-0.5, 0.5);
  \draw[->] (-0.5, 1.5) -- (-0.5, 2.5) node[pos=1, above] {$c$};
\end{scope}}} \right) \ .
\]
Thanks to the blist hypothesis, this amounts to showing that:
  \[
\soergel\left(\NB{\tikz[scale =0.8]{\begin{scope}[font=\tiny]
  \draw (0.5, -0.5) ..controls +(0,0.3) and +(0,-0.3) .. (-0.5,
  0.5);%
  \fill[white] (0,0) circle (2mm);
  \draw (-0.5, -0.5) ..controls +(0,0.3) and +(0,-0.3) .. (0.5,
  0.5);%
  \draw[->] (-0.5, 0.5) ..controls +(0,0.3) and +(0,-0.3) .. (0.5,
  1.5) node[pos=1, above] {$a$} coordinate[pos =0.2] (t1);
  \fill[white] (0,1) circle (2mm);
  \draw[->] (0.5, 0.5) ..controls +(0,0.3) and +(0,-0.3) .. (-0.5,
  1.5) node[pos=1, above] {$b$} coordinate[pos =0.2] (t2);
  \draw (0.5, -0.5) -- (1,-1) node[pos=1, below] {$1$};
  \draw (0.5, -0.5) -- (0.5, -1)node[pos=1, below] {$1$};
  \node at (0.75, -1 ) {$\dots$};
  \draw (-0.5, -0.5) -- (-1,-1) node[pos=1, below] {$1$};
  \draw (-0.5, -0.5) -- (-0.5, -1) node[pos=1, below] {$1$};
  \node at (-0.75, -1 ) {$\dots$};

\end{scope}}} \right) \simeq
\soergel\left(\NB{\tikz[scale =0.8]{\begin{scope}[font=\tiny]
  \draw[->] (0.5, -0.5) ..controls +(0,0.3) and +(0,-0.3) .. (0.5,
  1.5) node[pos=1, above] {$b$} coordinate[pos =0.2] (t1);
  \draw[->] (-0.5, -0.5) ..controls +(0,0.3) and +(0,-0.3) .. (-0.5,
  1.5) node[pos=1, above] {$a$} coordinate[pos =0.2] (t2);
    \draw (0.5, -0.5) -- (1,-1) node[pos=1, below] {$1$};
  \draw (0.5, -0.5) -- (0.5, -1)node[pos=1, below] {$1$};
  \node at (0.75, -1 ) {$\dots$};
  \draw (-0.5, -0.5) -- (-1,-1) node[pos=1, below] {$1$};
  \draw (-0.5, -0.5) -- (-0.5, -1) node[pos=1, below] {$1$};
  \node at (-0.75, -1 ) {$\dots$};
\end{scope}}} \right) \simeq
\soergel\left(\NB{\tikz[scale =0.8]{\begin{scope}[font=\tiny]
  \draw (-0.5, -0.5) ..controls +(0,0.3) and +(0,-0.3) .. (0.5,
  0.5);%
  \fill[white] (0,0) circle (2mm);
  \draw (0.5, -0.5) ..controls +(0,0.3) and +(0,-0.3) .. (-0.5,
  0.5);%
  \draw[->] (0.5, 0.5) ..controls +(0,0.3) and +(0,-0.3) .. (-0.5,
  1.5) node[pos=1, above] {$a$} coordinate[pos =0.2] (t2);
  \fill[white] (0,1) circle (2mm);
  \draw[->] (-0.5, 0.5) ..controls +(0,0.3) and +(0,-0.3) .. (0.5,
  1.5) node[pos=1, above] {$b$} coordinate[pos =0.2] (t1);
  \draw (0.5, -0.5) -- (1,-1) node[pos=1, below] {$1$};
  \draw (0.5, -0.5) -- (0.5, -1)node[pos=1, below] {$1$};
  \node at (0.75, -1 ) {$\dots$};
  \draw (-0.5, -0.5) -- (-1,-1) node[pos=1, below] {$1$};
  \draw (-0.5, -0.5) -- (-0.5, -1) node[pos=1, below] {$1$};
  \node at (-0.75, -1 ) {$\dots$};

\end{scope}}} \right) 
\]
and
\[
\soergel\left(\NB{\tikz[scale =0.8]{\begin{scope}[font=\tiny]
  \draw (0.5, -0.5) ..controls +(0,0.3) and +(0,-0.3) .. (-0.5,
  0.5);%
  \fill[white] (0,0) circle (2mm);
  \draw (-0.5, -0.5) ..controls +(0,0.3) and +(0,-0.3) .. (0.5,
  0.5);%
  \draw (1.5, 0.5) ..controls +(0,0.3) and +(0,-0.3) .. (0.5,
  1.5);%
  \fill[white] (1,1) circle (2mm);
  \draw (0.5, 0.5) ..controls +(0,0.3) and +(0,-0.3) .. (1.5,
  1.5);%
  \draw[->] (0.5, 1.5) ..controls +(0,0.3) and +(0,-0.3) .. (-0.5,
  2.5) node[pos=1, above] {$c$} coordinate[pos =0.2] (t2);
  \fill[white] (0,2) circle (2mm);
  \draw[->] (-0.5, 1.5) ..controls +(0,0.3) and +(0,-0.3) .. (0.5,
  2.5) node[pos=1, above] {$b$} coordinate[pos =0.2] (t1);
  \draw (-0.5, 0.5) -- (-0.5, 1.5);
  \draw (1.5, -0.5) -- (1.5, 0.5);
  \draw[->] (1.5, 1.5) -- (1.5, 2.5) node[pos=1, above] {$a$};
  \draw (1.5, -0.5) -- (2,-1) node[pos=1, below] {$1$};
  \draw (1.5, -0.5) -- (1.5, -1)node[pos=1, below] {$1$};
  \node at (1.75, -1 ) {$\dots$};
  \draw (-0.5, -0.5) -- (-1,-1) node[pos=1, below] {$1$};
  \draw (-0.5, -0.5) -- (-0.5, -1) node[pos=1, below] {$1$};
  \node at (-0.75, -1 ) {$\dots$};
  \draw (0.5, -0.5) -- (0.25,-1) node[pos=1, below] {$1$};
  \draw (0.5, -0.5) -- (0.75, -1) node[pos=1, below] {$1$};
  \node at (0.5, -1 ) {$\dots$};
\end{scope}}} \right) \simeq
\soergel\left(\NB{\tikz[scale =0.8]{\begin{scope}[font=\tiny]
    \draw (1.5, -0.5) ..controls +(0,0.3) and +(0,-0.3) .. ( 0.5,
  0.5);%
  \fill[white] (1,0) circle (2mm);
  \draw (0.5, -0.5) ..controls +(0,0.3) and +(0,-0.3) .. (1.5,
  0.5);%
  \draw (0.5, 0.5) ..controls +(0,0.3) and +(0,-0.3) .. (-0.5,
  1.5);%
  \fill[white] (0,1) circle (2mm);
  \draw (-0.5, 0.5) ..controls +(0,0.3) and +(0,-0.3) .. (0.5,
  1.5);%
  \draw[->] (1.5, 1.5) ..controls +(0,0.3) and +(0,-0.3) .. (0.5,
  2.5) node[pos=1, above] {$b$} coordinate[pos =0.2] (t2);
  \fill[white] (1,2) circle (2mm);
  \draw[->] (0.5, 1.5) ..controls +(0,0.3) and +(0,-0.3) .. (1.5,
  2.5) node[pos=1, above] {$a$} coordinate[pos =0.2] (t1);
  \draw (1.5, 0.5) -- (1.5, 1.5);
  \draw (-0.5, -0.5) -- (-0.5, 0.5);
  \draw[->] (-0.5, 1.5) -- (-0.5, 2.5) node[pos=1, above] {$c$};
  \draw (1.5, -0.5) -- (2,-1) node[pos=1, below] {$1$};
  \draw (1.5, -0.5) -- (1.5, -1)node[pos=1, below] {$1$};
  \node at (1.75, -1 ) {$\dots$};
  \draw (-0.5, -0.5) -- (-1,-1) node[pos=1, below] {$1$};
  \draw (-0.5, -0.5) -- (-0.5, -1) node[pos=1, below] {$1$};
  \node at (-0.75, -1 ) {$\dots$};
  \draw (0.5, -0.5) -- (0.25,-1) node[pos=1, below] {$1$};
  \draw (0.5, -0.5) -- (0.75, -1) node[pos=1, below] {$1$};
  \node at (0.5, -1 ) {$\dots$};
\end{scope}}} \right) \ .
\]
Both identities follows from forkslide isomorphisms, and the uncolored version of
the statement (Theorem~\ref{thm-braid-invariant}).
\end{proof}

\begin{cor}
  The bigraded $H^\prime$-module $\mC^{\prime}(\beta)$, the bigraded
  $H$-module $\mC^{\dif}(\beta)$ and the slash cohomology
$\mH^{\dif}_/(\beta)$ of $\beta$  are invariant under braid relations.
\end{cor}

\begin{proof}
Since the Cautis differential $d_C$ and the topological differential $d_t$ commute, there is a converging spectral sequence starting with the $E_2$ page:
\begin{equation}
\mH_\bullet( \mH_\bullet(\mHH_\bullet^\dif (\soergel(\beta)), d_t),d_C)=\mH_\bullet (\mHHH^\dif(\widehat{\beta}),d_C)  \Rightarrow  \mHH_\bullet^\dif(\soergel(\beta),d_T)=\mC^\prime(\beta) .
\end{equation}
 Here, the $\mHHH^{\dif}(\widehat{\beta})$ is the colored HOMFLYPT homology of $\beta$ as a framed invariant of the braid closure $\widehat{\beta}$, which is now equipped with a compatible $H^\prime$-structure, thanks to Proposition \ref{prop:braidrel-col}. Since $\mHHH^\dif$ is invariant under braid moves, and $d_C$ is also well-defined on $\mHHH^\dif(\widehat{\beta})$ (see, for instance \cite[Lemma 6.5]{RW}), the braid invariance follows.
\end{proof}

As in Section~\ref{secmarkov}, we obtain invariance under the first
Markov move  from the trace-like property of Hochschild homology (see Proposition \ref{HHcyclprop}).

\begin{prop}\label{prop:markov1col}
  Let $\beta_1$ and $\beta_2$ be two colored braids (top and bottom
  colors of $\beta_1$, respectively $\beta_2$, do not need to match) on $n$
  strands.  Then
  $\mC^{\prime}(\beta_1 \beta_2) \cong \mC^{\prime}(\beta_2
  \beta_1)$. \qedhere
\end{prop}

The proof of invariance under the second Markov move, uses a trick due to
Cautis \cite[proof of Proposition 5.1]{Cautisremarks} and is based on diagrammatic
manipulations. Let us rephrase the invariance result (Proposition~\ref{prop-Markov-II-for-HHH}) in the uncolored
case in a way that will suit the diagrammatic framework. It says that there are isomorphisms of bigraded $H^\prime$-modules:
  \begin{equation}\label{eqn:r1-1-col}
    \mC^{\prime}\left( \NB{\tikz[font=\tiny]{\begin{scope}
  \draw[>-] (-1, -1) -- (-1, -0.5)  node [pos =0, below] {$1$};
  \draw[<-] (-1,  1) -- (-1,  0.5);
  \draw (0, -0.5) .. controls +(0, 0.3) and +(0, -0.3) .. (-1, 0.5)  
  coordinate[pos=  0.5] (m);
  \fill[white] (m) circle (1mm);
  \draw (-1, -0.5) .. controls +(0, 0.3) and +(0, -0.3) .. (0, 0.5);
  \draw (0, 0.5) arc (180:0:0.5);
  \draw (0, -0.5) arc (-180:0:0.5);
  \draw[->-] (1, 0.5) --(1 , - 0.5);
\end{scope}}} \right)  \simeq
    \mC^{\prime}\left( \NB{\tikz[font=\tiny]{\begin{scope}
  \draw[>->] (-1, -1) -- (-1, 1)  node [pos =0, below] {$1$}
  coordinate[pos =0.5] (twist);
  \filldraw[draw= green!50!black, fill = white] (twist) circle
  (1mm) node[left, green!50!black] {$2$};
\end{scope}}} \right) \qquad
    \text{and} \qquad
    \mC^{\prime}\left( \NB{\tikz[font=\tiny]{\begin{scope}
  \draw[>-] (-1, -1) -- (-1, -0.5)  node [pos =0, below] {$1$};
  \draw[<-] (-1,  1) -- (-1,  0.5);
  \draw (-1, -0.5) .. controls +(0, 0.3) and +(0, -0.3) .. (0, 0.5)
  coordinate[pos=  0.5] (m);
  \fill[white] (m) circle (1mm);
  \draw (0, -0.5) .. controls +(0, 0.3) and +(0, -0.3) .. (-1, 0.5);
  \draw (0, 0.5) arc (180:0:0.5);
  \draw (0, -0.5) arc (-180:0:0.5);
  \draw[->-] (1, 0.5) --(1 , - 0.5);
\end{scope}}} \right)  \simeq
    \mC^{\prime}\left( \NB{\tikz[font=\tiny]{\begin{scope}
  \draw[>->] (-1, -1) -- (-1, 1)  node [pos =0, below] {$1$}
  coordinate[pos =0.5] (twist);
  \filldraw[draw= green!50!black, fill = white] (twist) circle
  (1mm) node[left, green!50!black] {$-2$};
\end{scope}}} \right).
  \end{equation}

\begin{cor}\label{cor:r1-a-col}
  For any positive integer $a$, one has:
  \[
    \mC^{\prime}\left( \NB{\tikz[font=\tiny]{\begin{scope}
  \draw[>-] (-1, -1) -- (-1, -0.5)  node [pos =0, below] {$a$};
  \draw[<-] (-1,  1) -- (-1,  0.5);
  \draw (0, -0.5) .. controls +(0, 0.3) and +(0, -0.3) .. (-1, 0.5)  
  coordinate[pos=  0.5] (m);
  \fill[white] (m) circle (1mm);
  \draw (-1, -0.5) .. controls +(0, 0.3) and +(0, -0.3) .. (0, 0.5);
  \draw (0, 0.5) arc (180:0:0.5);
  \draw (0, -0.5) arc (-180:0:0.5);
  \draw[->-] (1, 0.5) --(1 , - 0.5);
\end{scope}}} \right)  \simeq
q^{a-a^2}    \mC^{\prime}\left( \NB{\tikz[font=\tiny]{\begin{scope}
  \draw[>->] (-1, -1) -- (-1, 1)  node [pos =0, below] {$a$}
  coordinate[pos =0.5] (twist);
  \filldraw[draw= green!50!black, fill = white] (twist) circle
  (1mm) node[left, green!50!black] {$2a$};
\end{scope}}} \right) \qquad
    \text{and} \qquad
    \mC^{\prime}\left( \NB{\tikz[font=\tiny]{\begin{scope}
  \draw[>-] (-1, -1) -- (-1, -0.5)  node [pos =0, below] {$a$};
  \draw[<-] (-1,  1) -- (-1,  0.5);
  \draw (-1, -0.5) .. controls +(0, 0.3) and +(0, -0.3) .. (0, 0.5)
  coordinate[pos=  0.5] (m);
  \fill[white] (m) circle (1mm);
  \draw (0, -0.5) .. controls +(0, 0.3) and +(0, -0.3) .. (-1, 0.5);
  \draw (0, 0.5) arc (180:0:0.5);
  \draw (0, -0.5) arc (-180:0:0.5);
  \draw[->-] (1, 0.5) --(1 , - 0.5);
\end{scope}}} \right)  \simeq
q^{a^2-a}    \left( \NB{\tikz[font=\tiny]{\begin{scope}
  \draw[>->] (-1, -1) -- (-1, 1)  node [pos =0, below] {$a$}
  coordinate[pos =0.5] (twist);
  \filldraw[draw= green!50!black, fill = white] (twist) circle
  (1mm) node[left, green!50!black] {$-2a$};
\end{scope}}} \right).
    \]
\end{cor}

\begin{proof}
  Both cases are analogous, and we only prove the first one and argue by
  induction on $a$. Due to the blist hypothesis it is enough to
  prove
  \[
    \mC^{\prime}\left( \NB{\tikz[font=\tiny]{\begin{scope}
  \draw[>-] (-1.2, -1) .. controls +(0, 0.3) and +(0,0) .. (-1,
  -0.5) node [pos = 0, below] {$a$};
  \draw[>-] (-0.8, -1) .. controls +(0, 0.3) and +(0,0) .. (-1, -0.5) node [pos =0, below] {$1$};
  \draw[<-] (-1,  1) -- (-1,  0.5) node [pos =0, above] {$a+1$};
  \draw (0, -0.5) .. controls +(0, 0.3) and +(0, -0.3) .. (-1, 0.5)  
  coordinate[pos=  0.5] (m);
  \fill[white] (m) circle (1mm);
  \draw (-1, -0.5) .. controls +(0, 0.3) and +(0, -0.3) .. (0, 0.5);
  \draw (0, 0.5) arc (180:0:0.5);
  \draw (0, -0.5) arc (-180:0:0.5);
  \draw[->-] (1, 0.5) --(1 , - 0.5);
\end{scope}}} \right)
    \simeq
 q^{-a^2-a}   \mC^{\prime}\left( \NB{\tikz[font=\tiny]{\begin{scope}
  \draw[>-] (-1.2, -1) .. controls +(0, 0.3) and +(0,0) .. (-1,
  -0.5) node [pos = 0, below] {$a$};
  \draw[>-] (-0.8, -1) .. controls +(0, 0.3) and +(0,0) .. (-1, -0.5) node [pos =0, below] {$1$};
  \draw[->] (-1,  -0.5) -- (-1,  1) node [pos =1, above] {$a+1$}
  coordinate[pos=0.5] (twist);
  \filldraw[draw= green!50!black, fill = white] (twist) circle
  (1mm) node[left, green!50!black] {$2a+2$};
\end{scope}
}} \right).
  \]
  This is a straightforward computation:
  \begin{align*}
    \mC^{\prime}\left( \NB{\tikz[font=\tiny]{\begin{scope}
  \draw[>-] (-1.2, -1) .. controls +(0, 0.3) and +(0,0) .. (-1,
  -0.5) node [pos = 0, below] {$a$};
  \draw[>-] (-0.8, -1) .. controls +(0, 0.3) and +(0,0) .. (-1, -0.5) node [pos =0, below] {$1$};
  \draw[<-] (-1,  1) -- (-1,  0.5) node [pos =0, above] {$a+1$};
  \draw (0, -0.5) .. controls +(0, 0.3) and +(0, -0.3) .. (-1, 0.5)  
  coordinate[pos=  0.5] (m);
  \fill[white] (m) circle (1mm);
  \draw (-1, -0.5) .. controls +(0, 0.3) and +(0, -0.3) .. (0, 0.5);
  \draw (0, 0.5) arc (180:0:0.5);
  \draw (0, -0.5) arc (-180:0:0.5);
  \draw[->-] (1, 0.5) --(1 , - 0.5);
\end{scope}}} \right)
    &\simeq
      \mC^{\prime}\left( \NB{\tikz[font=\tiny]{\begin{scope}
  \draw[>-] (-1.2, -1) .. controls +(0, 0.3) and +(0,0) .. (-1.2,
  -0.5) node [pos = 0, below] {$a$};
  \draw[>-] (-0.8, -1) .. controls +(0, 0.3) and +(0,0) .. (-0.8,
  -0.5) node [pos =0, below] {$1$};  
  \draw[<-] (-1,  1) -- (-1,  0.8) node [pos =0, above] {$a+1$};
  \draw (-1.2, 0.5) .. controls +(0, 0.3) and +(0,0) .. (-1, 0.8);
    \draw (-0.8, 0.5) .. controls +(0, 0.3) and +(0,0) .. (-1, 0.8);
  \draw (-0.2, -0.5) .. controls +(0, 0.3) and +(0, -0.3) .. (-1.2,
  0.5) coordinate[pos=  0.5] (m1) coordinate[pos=  0.37] (m2);
  \draw ( 0.2, -0.5) .. controls +(0, 0.3) and +(0, -0.3) .. (-0.8, 0.5)
  coordinate[pos=  0.5] (m3) coordinate[pos=  0.63] (m4);
  \fill[white] (m1) circle (0.8mm);
  \fill[white] (m2) circle (0.8mm);
  \fill[white] (m3) circle (0.8mm);
  \fill[white] (m4) circle (0.8mm);
  \draw (-1.2, -0.5) .. controls +(0, 0.3) and +(0, -0.3) .. (-0.2, 0.5);
  \draw (-0.8, -0.5) .. controls +(0, 0.3) and +(0, -0.3) .. ( 0.2, 0.5);
  \draw (0.2, 0.5) arc (180:0:0.3);
  \draw (0.2, -0.5) arc (-180:0:0.3);
  \draw (-0.2, 0.5) arc (180:0:0.7);
  \draw (-0.2, -0.5) arc (-180:0:0.7);
  \draw[->-] (0.8, 0.5) --(0.8 , - 0.5);
  \draw[->-] (1.2, 0.5) --(1.2 , - 0.5);
\end{scope}}} \right)
    \simeq
      \mC^{\prime}\left( \NB{\tikz[font=\tiny]{\begin{scope}
  \draw[>-] (-1.2, -1) .. controls +(0, 0.3) and +(0,0) .. (-1.2,
  -0.5) node [pos = 0, below] {$a$};
  \draw[>-] (-0.8, -1) .. controls +(0, 0.3) and +(0,0) .. (-0.8,
  -0.5) node [pos =0, below] {$1$};  
  \draw[<-] (-1,  1) -- (-1,  0.8) node [pos =0, above] {$a+1$};
  \draw (-1.2, 0.5) .. controls +(0, 0.3) and +(0,0) .. (-1, 0.8);
  \draw (-0.8, 0.5) .. controls +(0, 0.3) and +(0,0) .. (-1, 0.8);
  \draw (-0.2, -0.5) .. controls +(0, 0.3) and +(0, -0.3) .. (-1.2,
  0.5) coordinate[pos=  0.5] (m1) coordinate[pos=  0.37] (m2);
  \draw ( -0.3, 0) .. controls +(0, 0.1) and +(0, -0.3) .. (-0.8, 0.5)
  coordinate[pos=  0.425] (m3) coordinate[pos=  0] (twist);
  \fill[white] (m1) circle (0.8mm);
  \fill[white] (m2) circle (0.8mm);
  \fill[white] (m3) circle (0.8mm);
  \draw (-1.2, -0.5) .. controls +(0, 0.3) and +(0, -0.3) .. (-0.2, 0.5);
  \draw (-0.8, -0.5) .. controls +(0, 0.3) and +(0, -0.1) .. ( -0.3, 0);
  \draw (-0.2, 0.5) arc (180:0:0.6);
  \draw (-0.2, -0.5) arc (-180:0:0.6);
  \draw[->-] (1, 0.5) --(1 , - 0.5);
    \filldraw[draw= green!50!black, fill = white] (twist) circle
  (1mm) node[right, green!50!black] {$2$};
\end{scope}}} \right)
    \simeq
      \mC^{\prime}\left( \NB{\tikz[font=\tiny]{\begin{scope}
  \draw[>-] (-1.2, -1) .. controls +(0, 0.3) and +(0,0) .. (-1.2,
  -0.5) node [pos = 0, below] {$a$};
  \draw[>-] (-0.8, -1) .. controls +(0, 0.3) and +(0,0) .. (-0.8,
  -0.5) node [pos =0, below] {$1$};  
  \draw[<-] (-1,  1) -- (-1,  0.8) node [pos =0, above] {$a+1$};
  \draw (-1.2, 0.5) .. controls +(0, 0.3) and +(0,0) .. (-1, 0.8);
  \draw (-0.8, 0.5) .. controls +(0, 0.3) and +(0,0) .. (-1, 0.8);
  \draw (-0.2, -0.5) .. controls +(0, 0.3) and +(0, -0.3) .. (-1.2,
  0.5) coordinate[pos=  0.5] (m1) coordinate[pos=  0.7] (m2);
    \fill[white] (m2) circle (0.8mm);
  \draw (-0.8, -0.5) .. controls +(0, 0.3) and +(0, -0.1) .. ( -1.3,
  0) coordinate[pos=  0.425] (m3);
  \fill[white] (m3) circle (0.8mm);
  \draw ( -1.3, 0) .. controls +(0, 0.1) and +(0, -0.3) .. (-0.8, 0.5)
   coordinate[pos=  0] (twist);
  \fill[white] (m1) circle (0.8mm);

  \draw (-1.2, -0.5) .. controls +(0, 0.3) and +(0, -0.3) .. (-0.2, 0.5);

  \draw (-0.2, 0.5) arc (180:0:0.6);
  \draw (-0.2, -0.5) arc (-180:0:0.6);
  \draw[->-] (1, 0.5) --(1 , - 0.5);
    \filldraw[draw= green!50!black, fill = white] (twist) circle
  (1mm) node[left, green!50!black] {$2$};
\end{scope}}} \right)
  \\
    &\simeq
      q^{a-a^2}\mC^{\prime}\left( \NB{\tikz[font=\tiny]{\begin{scope}
   \draw[>-] (-0.8, -1) .. controls +(0, 0.3) and +(0,-0.3) .. (-1.2,
  -0.25) node [pos =0, below] {$1$} coordinate[pos=0.5] (m1) coordinate[pos=1] (twist1);  
  \fill[white] (m1) circle (0.8mm);
  \draw[>-] (-1.2, -1) .. controls +(0, 0.3) and +(0,-0.3) .. (-0.8,
  -0.25) node [pos = 0, below] {$a$} coordinate[pos=1] (twist2);;
  \draw (-0.8, -.25) .. controls +(0, 0.3) and +(0,-0.3) .. (-1.2,
  0.5) coordinate[pos = 0.5] (m2);
  \fill[white] (m2) circle (0.8mm);
  \draw (-1.2, -.25) .. controls +(0, 0.3) and +(0,-0.3) .. (-0.8,
  0.5); 
  \draw[<-] (-1,  1) -- (-1,  0.8) node [pos =0, above] {$a+1$};
  \draw (-1.2, 0.5) .. controls +(0, 0.3) and +(0,0) .. (-1, 0.8);
  \draw (-0.8, 0.5) .. controls +(0, 0.3) and +(0,0) .. (-1, 0.8);
    \filldraw[draw= green!50!black, fill = white] (twist2) circle
  (1mm) node[right, green!50!black] {$2a$};
    \filldraw[draw= green!50!black, fill = white] (twist1) circle
  (1mm) node[left, green!50!black] {$2$};
\end{scope}
}} \right)
      \simeq
     q^{-a^2} \mC^{\prime}\left( \NB{\tikz[font=\tiny]{\begin{scope}
   \draw[>-] (-0.8, -1) .. controls +(0, 0.3) and +(0,-0.3) .. (-1.2,
  -0.25) node [pos =0, below] {$1$} coordinate[pos=0.5] (m1) coordinate[pos=1] (twist1);  
  \fill[white] (m1) circle (0.8mm);
  \draw[>-] (-1.2, -1) .. controls +(0, 0.3) and +(0,-0.3) .. (-0.8,
  -0.25) node [pos = 0, below] {$a$} coordinate[pos=1] (twist2);
  \draw (-0.8, -.25) .. controls +(0, 0.3) and +(0,-0.3) .. (-0.8,
  0.5) coordinate[pos = 0.5] (m2);
  \draw (-1.2, -.25) .. controls +(0, 0.3) and +(0,-0.3) .. (-1.2,
  0.5); 
  \draw[<-] (-1,  1) -- (-1,  0.8) node [pos =0, above] {$a+1$};
  \draw (-1.2, 0.5) .. controls +(0, 0.3) and +(0,0) .. (-1, 0.8);
  \draw (-0.8, 0.5) .. controls +(0, 0.3) and +(0,0) .. (-1, 0.8);
    \filldraw[draw= green!50!black, fill = white] (twist2) circle
  (1mm) node[right, green!50!black] {$2a+1$};
    \filldraw[draw= green!50!black, fill = white] (twist1) circle
  (1mm) node[left, green!50!black] {$a+2$};
\end{scope}
}} \right)
      \simeq
     q^{-a^2} \mC^{\prime}\left( \NB{\tikz[font=\tiny]{\begin{scope}
   \draw[>-] (-0.8, -1) .. controls +(0, 0.3) and +(0,-0.3) .. (-0.8,
  -0.25) node [pos =0, below] {$1$} coordinate[pos=0.5] (m1) coordinate[pos=1] (twist1);  
  \draw[>-] (-1.2, -1) .. controls +(0, 0.3) and +(0,-0.3) .. (-1.2,
  -0.25) node [pos = 0, below] {$a$} coordinate[pos=1] (twist2);
  \draw (-0.8, -.25) .. controls +(0, 0.3) and +(0,-0.3) .. (-1.2,
  0.5) coordinate[pos = 0.5] (m2);
  \fill[white] (m2) circle (0.8mm);
  \draw (-1.2, -.25) .. controls +(0, 0.3) and +(0,-0.3) .. (-0.8,
  0.5); 
  \draw[<-] (-1,  1) -- (-1,  0.8) node [pos =0, above] {$a+1$};
  \draw (-1.2, 0.5) .. controls +(0, 0.3) and +(0,0) .. (-1, 0.8);
  \draw (-0.8, 0.5) .. controls +(0, 0.3) and +(0,0) .. (-1, 0.8);
    \filldraw[draw= green!50!black, fill = white] (twist1)circle
  (1mm) node[right, green!50!black] {$a+2$};
    \filldraw[draw= green!50!black, fill = white] (twist2) circle
  (1mm) node[left, green!50!black] {$2a+1$};
\end{scope}
}} \right)
  \\ &
   \simeq
   q^{-a^2-a} \mC^{\prime}\left( \NB{\tikz[font=\tiny]{\begin{scope}
  \draw[>-] (-1.2, -1) .. controls +(0, 1.3) and +(0,0) .. (-1,
   0.5) node [pos = 0, below] {$a$} coordinate[pos =0.2] (twist1);
  \draw[>-] (-0.8, -1) .. controls +(0, 1.3) and +(0,0) .. (-1,  0.5)
  node [pos =0, below] {$1$} coordinate[pos =0.2] (twist2);
  \draw[->] (-1,  0.5) -- (-1,  1) node [pos =1, above] {$a+1$}
  coordinate[pos=0.5] (twist);
  \filldraw[draw= green!50!black, fill = white] (twist1) circle
  (1mm) node[left, green!50!black] {$-2(a+1)$};
    \filldraw[draw= green!50!black, fill = white] (twist2) circle
  (1mm) node[right, green!50!black] {$-2(a+1)$};
\end{scope}
}} \right)
    \simeq
  q^{-a^2-a}  \mC^{\prime}\left( \NB{\tikz[font=\tiny]{}} \right).
  \end{align*}
  Note that in the fifth and seventh isomorphisms above, we used Corollary \ref{cor:ft-ab}.
\end{proof}

As a direct consequence of Proposition~\ref{prop:markov1col} and
Corollary~\ref{cor:r1-a-col}, we obtain the following result.

\begin{thm}
  \label{thm:main-col}
  The bigraded $H^\prime$-module $\mC^{\prime}(\beta)$, the bigraded
  $H$-module $\mC^{\dif}(\beta)$ and the slash cohomology
$\mH^{\dif}_/(\beta)$ of $\beta$ depend only on the colored, oriented and framed
link represented by the closure of $\beta$.
\end{thm}

Both $\mC^{\prime}(\beta)$, and $\mC^{\dif}(\beta)$ are graded by
$q$ and $t$ degrees.  
By the graded Euler characteristic of
these objects, we mean:
\[
\chi_q( \mC^{\prime}(\beta) )= \sum_{i\in \Z} (-1)^i
\rkq (\mC^{\prime}(\beta)_i)\quad \text{and}\quad
\chi_q( \mC^{\dif}(\beta) )= \sum_{i\in \Z} (-1)^i
\mathrm{dim}_q (\mC^{\dif}(\beta)_i) ,
\]
where $\mC^{\prime}(\beta)_i$ (respectively $\mC^{\dif}(\beta)_i$) is the
$q$-graded homogeneous subspace of $\mC^{\prime}(\beta)$ (respectively
$\mC^{\dif}(\beta)$) of $t$-degree $i$.

\subsection{Relation to colored Jones polynomial}

These graded Euler characteristics are both equal to the colored Jones
polynomial of the underlying colored framed links. We give a few
details about the normalization of the colored Jones polynomial that we use.
The uncolored Jones polynomial is such that it is not sensitive to
framing and satisfies:
\[
  q^2 \Jones \left(\NB{\tikz[yscale=0.8, xscale = 0.6]{\begin{scope}[font=\tiny]
  \draw[->] (0.5, -0.5) ..controls +(0,0.3) and +(0,-0.3) .. (-0.5,
  0.5) node[pos=1, above] {$1$} coordinate[pos =0.2] (t1);
  \fill[white] (0,0) circle (2mm);
  \draw[->] (-0.5, -0.5) ..controls +(0,0.3) and +(0,-0.3) .. (0.5,
  0.5) node[pos=1, above] {$1$} coordinate[pos =0.2] (t2);
\end{scope}}}  \right)
-
  q^{-2}\Jones \left(\NB{\tikz[yscale=0.8, xscale =
    0.6]{\begin{scope}[font=\tiny]
  \draw[->] (-0.5, -0.5) ..controls +(0,0.3) and +(0,-0.3) .. (0.5,
  0.5) node[pos=1, above] {$1$} coordinate[pos =0.2] (t2);
  \fill[white] (0,0) circle (2mm);
  \draw[->] (0.5, -0.5) ..controls +(0,0.3) and +(0,-0.3) .. (-0.5,
  0.5) node[pos=1, above] {$1$} coordinate[pos =0.2] (t1);
\end{scope}}} \right) =  (q - q^{-1}) \Jones \left( \NB{\tikz[yscale=0.8, xscale = 0.6]{\begin{scope}[font=\tiny]
  \draw[->] (-0.5, -0.5) ..controls +(0,0.3) and +(0,-0.3) .. (-0.5,
  0.5) node[pos=1, above] {$1$} coordinate[pos =0.2] (t2);
  \draw[->] (0.5, -0.5) ..controls +(0,0.3) and +(0,-0.3) .. ( 0.5,
  0.5) node[pos=1, above] {$1$} coordinate[pos =0.2] (t1);
\end{scope}}} \right)
\]
It extends to diagrams of entangled MOY graphs and can be computed
using the following rules for crossings:
\begin{align}
  \Jones \left(\NB{\tikz[yscale=0.8, xscale = 0.6]{}} \right)
  &= q^{-2} \Jones \left(\NB{\tikz[yscale=0.4, xscale = 0.6, font
    =\tiny]{}} \right) - q^{-3} \Jones \left(\NB{\tikz[yscale=0.8,
    xscale = 0.6]{}} \right) \label{eq:crossing-plus} , \\
    \Jones \left(\NB{\tikz[yscale=0.8, xscale = 0.6]{}} \right)
  &= q^{2} \Jones  \left( \NB{\tikz[yscale=0.4, xscale = 0.6, font
    =\tiny]{}} \right) - q^{3} \Jones\left(\NB{\tikz[yscale=0.8,
    xscale = 0.6]{}} \right) ,  \label{eq:crossing-minus}
\end{align}
and the following identities (and their mirror images) on MOY graphs known as MOY calculus:
\begin{align}\label{eq:extrelcircle}
  \Jones\left(\vcenter{\hbox{\tikz[scale= 0.5, font = \tiny]{\draw[->]
  (0,0) arc(0:360:1cm) node[right] {{$\!a\!$}};}}}\right)= [a+1],
\end{align}
\begin{align} \label{eq:extrelass}
   \Jones\left(\stgamma\right) = \Jones\left(\stgammaprime\right),
 \end{align}
\begin{align} \label{eq:extrelbin1} 
\Jones\left(\digonab\right) = \arraycolsep=2.5pt
  \begin{bmatrix}
    a+b \\ a
  \end{bmatrix}%
\Jones\left(\vertab \right),
\end{align}
\begin{align} \label{eq:extrelbin2}
\arraycolsep=2.5pt
\Jones\left(\baddigonab \right) = 
  \begin{bmatrix}
    a+b+1 \\ b
  \end{bmatrix}%
\Jones\left(\verta\right) ,
\end{align}
\begin{align}
 \Jones\left(\squarea\right) = \Jones\left(\twoverta\right) +
  [a+3]\Jones\left(\doubleYa \right), \label{eq:extrelsquare1}
\end{align}
\begin{align}
  \Jones\left(\squarec\right)= \Jones\left(\squared
  \right)\label{eq:extrelsquare3} + [b-a] \Jones\left( \vertavertb\right).
\end{align}

Our version of the colored Jones polynomial is invariant under
forkslide moves. This characterizes completely the colored Jones polynomial $\Jones$ since one can blist
every component, use forkslide moves, rules for uncolored crossings and MOY
calculus to compute. Its behavior with respect to framing is given
by:
\[
    \Jones \left(\NB{\tikz[font=\tiny, scale = 0.5]{\begin{scope}
  \draw[>-] (-1, -1) -- (-1, -0.5)  node [pos =0, below] {$a$};
  \draw[<-] (-1,  1) -- (-1,  0.5);
  \draw (0, -0.5) .. controls +(0, 0.3) and +(0, -0.3) .. (-1, 0.5)  
  coordinate[pos=  0.5] (m);
  \fill[white] (m) circle (1mm);
  \draw (-1, -0.5) .. controls +(0, 0.3) and +(0, -0.3) .. (0, 0.5);
  \draw (0, 0.5) arc (180:0:0.5);
  \draw (0, -0.5) arc (-180:0:0.5);
  \draw[->-] (1, 0.5) --(1 , - 0.5);
\end{scope}}} \right) =
q^{a-a^2} \Jones \left(\NB{\tikz[font=\tiny, scale =0.5]{}} \right) \qquad
    \text{and} \qquad
    \Jones \left(\NB{\tikz[font=\tiny, scale =0.5]{\begin{scope}
  \draw[>-] (-1, -1) -- (-1, -0.5)  node [pos =0, below] {$a$};
  \draw[<-] (-1,  1) -- (-1,  0.5);
  \draw (-1, -0.5) .. controls +(0, 0.3) and +(0, -0.3) .. (0, 0.5)
  coordinate[pos=  0.5] (m);
  \fill[white] (m) circle (1mm);
  \draw (0, -0.5) .. controls +(0, 0.3) and +(0, -0.3) .. (-1, 0.5);
  \draw (0, 0.5) arc (180:0:0.5);
  \draw (0, -0.5) arc (-180:0:0.5);
  \draw[->-] (1, 0.5) --(1 , - 0.5);
\end{scope}}} \right)  =
q^{a^2-a}   \Jones  \left( \NB{\tikz[font=\tiny, scale =0.5]{}} \right).
\]

Recall from \eqref{eqn:cyclotomicEuler} that the cyclotomic Euler characteristic of a $p$-complex is identified with its image in the Grothendieck ring $\mathbb{O}_p=K_0(H\udmod)$.

\begin{prop}\label{prop:categorification}
  Let $\beta$ be a colored braid and denote by $\widehat{\beta}$ the
  framed link obtained by closing up $\beta$. Then
  \[
\chi_q( \mC^{\prime}(\beta) ) = \chi_q( \mC^{\dif}(\beta)) =
\Jones(\widehat{\beta}).
  \]
In particular the cyclotomic Euler characteristic $\chi_{\mathbb{O}}(\mH^\dif_/(\beta))$ is equal to
$\Jones(\widehat{\beta})$ specialized at a $2p$th root of unity.
\qedhere
\end{prop}

\begin{proof}
  It suffices to show the statement for $\chi_q(\mC^\prime(\beta))$. Since the Euler characteristic does not change under taking homology, 
$\chi_q(\mC^\prime(\beta)) $ is equal to the Euler characteritic of $\mHH^\dif_\bullet (\soergel (\beta))$
where we identify $a=-q^2$ via the grading collapse \eqref{eqnatqdegreecollapse-r1}. Without the grading collapse, the triply graded homology has its Euler characteristic equal to the colored HOMFLYPT polynomial. Under setting $a=-q^2$, the HOMFLYPT polynomial specializes to the colored Jones polynomial $\Jones(\widehat{\beta})$.
\end{proof}

\begin{rem}[Unframed invariants] \label{rmk:unframed}
The framed colored Jones polynomial $\Jones$ can be renormalized to become an
invariant of unframed links. Consider a diagram for a link $L$ and denote by $X$ its
set of crossings. Each crossing $x \in X$ has a sign $s(x)$.  Suppose the crossing $x$ involves two distinct strands colored by $a$ and $b$.  Then define $\eta_x:=
s(x)\delta_{ab} (a^2-a)$ and $\eta_L:= \sum_{x \in X} \eta_x$. The
Laurent polynomial $J(L):= q^{\eta_L}\Jones(L)$ is an invariant of
oriented colored links.

Similarly, the homological invariants
$\mC^{\prime}$,  $\mC^{\dif}$ and $\mH^{\dif}_/$
can be renormalized to obtain invariants of oriented colored links.
First shift the $q$-degree by $\eta_L$. For each component $C$ of $L$ of color
$a$, add a twist (a green dot) of multiplicity $(-2a(n_+(C) -n_-(C))$,
where $n_+(C)$ (respectively $n_-(C)$) denotes the number of positive
(respectively negative) crossings of $C$ (forgetting about the other
components of $L$). 
\end{rem}

\section{Example: a Hopf link} \label{sec:hopflink}
In this section we compute the colored homology of a Hopf link where the components are colored by $1$ and $2$.

Throughout this section we will let $A=A_{(2,1)}$. %
It should be possible to slightly generalize the calculation which follows to the case of the Hopf link where one components is colored $a$ and the other is colored $1$.  Ignoring the $H^\prime$-structure, this should be fairly straightforward.  With $H^\prime $, this case is technically more involved than the special case we present below.

Sto{\v{s}}i{\'c} \cite{StosicHH} calculated the Hochschild homology of a certain bimodule that will appear in our computations.
It would be interesting to compute the colored homologies of arbitrary torus links and knots building upon \cite{EH1, HM, Hog1, Mel1}.

\subsection{The complex}
To begin our computation, first note that the Hopf link is the closure
of a simple crossing concatenated with itself.  Thus we consider the
corresponding complex of Soergel bimodules $\rickardp{2}{1} \circ
\rickardp{1}{2}$ where we recall (up to some overall internal and cohomological shifts that we will add in later),
      \begin{equation}
        \begin{array}{crcl}
        \rickardp{1}{2} = &
       \soergel \left(    \NB{\tikz[font=\tiny, scale=0.6]{\begin{scope}
  \coordinate (bl) at (-0.5, -1);
  \coordinate (br) at ( 0.5, -1);
  \coordinate (bm) at (  0,-0.3);
  \coordinate (tl) at (-0.5,  1);
  \coordinate (tr) at ( 0.5,  1);
  \coordinate (tm) at (  0, 0.3);
  \draw[>-]  (bl) .. controls +( 0, 0.5) and +(0,0) .. (bm)
  node[below, pos = 0] {$2$};
  \draw[>-]  (br) .. controls +( 0, 0.5) and +(0,0) .. (bm)
  node[below, pos = 0] {$1$};
  \draw[<-]  (tl) .. controls +( 0, -0.5) and +(0,0) .. (tm)
  node[above, pos = 0] {$1$};
  \draw[<-]  (tr) .. controls +( 0, -0.5) and +(0,0) .. (tm)
  node[above, pos = 0] {$2$};
  \draw [->-] (bm) -- (tm) node[left, pos = 0.5] {$3$};
\end{scope}}} \right)  & \to & q^{-1} \soergel \left( \NB{\tikz[font=\tiny, scale=0.6]{\begin{scope}
  \coordinate (bl) at (-0.5, -1);
  \coordinate (br) at ( 0.5, -1);
  \coordinate (tl) at (-0.5,  1);
  \coordinate (tr) at ( 0.5,  1);
  \draw[>->] (bl) -- (tl) node[pos = 0, below] {$2$} node[pos = 1,
  above] {$1$} coordinate[pos = 0.4] (ml);
  \draw[>->] (br) -- (tr) node[pos = 0, below] {$1$} node[pos = 1, above] {$2$} coordinate[pos = 0.6] (mr);
  \draw[->-] (ml) -- (mr) node [pos= 0.5, above] {$1$};
\end{scope}}}  \right)
        \end{array}
      \end{equation}
      
            \begin{equation}
        \begin{array}{crcl}
        \rickardp{2}{1} = &
          \soergel \left(  \NB{\tikz[font=\tiny, scale=0.6]{\begin{scope}
  \coordinate (bl) at (-0.5, -1);
  \coordinate (br) at ( 0.5, -1);
  \coordinate (bm) at (  0,-0.3);
  \coordinate (tl) at (-0.5,  1);
  \coordinate (tr) at ( 0.5,  1);
  \coordinate (tm) at (  0, 0.3);
  \draw[>-]  (bl) .. controls +( 0, 0.5) and +(0,0) .. (bm)
  node[below, pos = 0] {$1$};
  \draw[>-]  (br) .. controls +( 0, 0.5) and +(0,0) .. (bm)
  node[below, pos = 0] {$2$};
  \draw[<-]  (tl) .. controls +( 0, -0.5) and +(0,0) .. (tm)
  node[above, pos = 0] {$2$};
  \draw[<-]  (tr) .. controls +( 0, -0.5) and +(0,0) .. (tm)
  node[above, pos = 0] {$1$};
  \draw [->-] (bm) -- (tm) node[left, pos = 0.5] {$3$};
\end{scope}}} \right) & \to & q^{-1} \soergel \left(  \NB{\tikz[font=\tiny, scale=0.6]{\begin{scope}
  \coordinate (bl) at (-0.5, -1);
  \coordinate (br) at ( 0.5, -1);
  \coordinate (tl) at (-0.5,  1);
  \coordinate (tr) at ( 0.5,  1);
  \draw[>->] (bl) -- (tl) node[pos = 0, below] {$1$} node[pos = 1,
  above] {$2$} coordinate[pos = 0.6] (ml);
  \draw[>->] (br) -- (tr) node[pos = 0, below] {$2$} node[pos = 1, above] {$1$} coordinate[pos = 0.4] (mr);
  \draw[->-] (mr) -- (ml) node [pos= 0.5, above] {$1$};
\end{scope}

}}  \right)
        \end{array} .
      \end{equation}
Then $\rickardp{2}{1} \circ \rickardp{1}{2}=$
\[
\begin{tikzpicture}[xscale = 4.6, yscale = 2.2, font = \small]
    \node (A) at (-1.25, 0) { \soergel $\left(   \NB{\tikz[font= \tiny, scale =0.6]{\begin{scope}
  \coordinate (o4) at (0,.6);
  \coordinate (o3) at (0,0.3);
  \coordinate (o2) at (0,-0.3);
  \coordinate (o1) at (0,-.6);
  \coordinate (bl) at (-.5,-1);
  \coordinate (br) at ( .5,-1);
  \coordinate (ul) at (-.5,1);
  \coordinate (ur) at (.5,1);
  \draw[->-] (o3) -- (o4) node [pos=.5, left] {$3$};
  \draw[->-] (o1) -- (o2) node [pos=0.5, left] {$3$};
  \draw[->-] (o2) .. controls +(0.3,0.3) and +(0.3, -0.3) .. (o3) node
 [pos= 0.5, right] {$2$};
   \draw[->-] (o2) .. controls +(-0.3,0.3) and +(-0.3, -0.3) .. (o3) node
 [pos= 0.5, left] {$1$};
 \draw[>-] (bl) .. controls +(0, 0.3) and +(-0,-0).. (o1) coordinate[pos=0.56] (c) node[pos=0, below] {$2$};
  \draw[>-] (br) .. controls +(0, 0.3) and +(0,-0).. (o1) node[pos=0, below] {$1$};
  \draw[<-] (ul) .. controls +(0, -0.3) and +(-0,-0).. (o4) coordinate[pos=0.56] (c) node[pos=0, above] {$2$};
  \draw[<-] (ur) .. controls +(0, -0.3) and +(0,-0).. (o4) node[pos=0, above] {$1$};
\end{scope}}} \right)$  };
       \node (C) at (1.35, 0) {$q^{-2}$ \soergel $ \left(\NB{\tikz[font= \tiny, scale =0.6]{\begin{scope}
  \coordinate (bl) at (-0.5, -1);
  \coordinate (br) at ( 0.5, -1);
  \coordinate (tl) at (-0.5,  1);
  \coordinate (tr) at ( 0.5,  1);
    \coordinate (ml) at (-0.5,  -.8);
        \coordinate (Ml) at (-0.5,  .8);
 \coordinate (mr) at (0.5,  -.6);
\coordinate (Mr) at (0.5,  .6);

   \draw[>->] (bl) -- (tl) node[pos = 0, below] {$2$} node[pos = 1,
  above] {$2$};

    \draw[>->] (br) -- (tr) node[pos = 0, below] {$1$} node[pos = 1, above] {$1$};
  
  \draw[->-] (ml) -- (mr) node [pos= 0.5, above] {$1$};
    \draw[->-] (Ml) -- (Mr) node [pos= 0.5, above] {$1$};

\end{scope}}} \right)$ };
    \node (m) at (0,0) {$q^{-1} \soergel  \left(\NB{\tikz[font= \tiny, scale =0.6]{\begin{scope}

    \coordinate (bl) at (-0.5, -1);
  \coordinate (br) at ( 0.5, -1);
  \coordinate (bm) at (  0,-0.6);
  \coordinate (tm) at (  0, -0.3);
  \coordinate (tl) at (-0.5,  0);
  \coordinate (tr) at ( 0.5,  0); 
    \coordinate (ttl) at (  -.5, 1);
        \coordinate (ttr) at (  .5, 1);

 \draw[->] (tl) -- (ttl) node[pos = 0, below] {$ $} node[pos = 1,
  above] {$2$} coordinate[pos = 0.6]  coordinate[pos = 0.6] (ml) ;
  
   \draw[->] (tr) -- (ttr) node[pos = 0, below] {$ $} node[pos = 1,
  above] {$1$} coordinate[pos = 0.6] coordinate[pos = 0.4] (mr);
  
   \draw[->] (mr) -- (ml) node[pos = 0, right] {$  $} node[pos = .5,
  above] {$ 1 $} coordinate[pos = 0.6];

  \draw[>-]  (bl) .. controls +( 0, 0.3) and +(0,0) .. (bm)
  node[below, pos = 0] {$2$};
  \draw[>-]  (br) .. controls +( 0, 0.3) and +(0,0) .. (bm)
  node[below, pos = 0] {$ 1$};
  \draw[]  (tl) .. controls +( 0, -0.3) and +(0,0) .. (tm)
  node[above, pos = 0] {$ $};
  \draw[]  (tr) .. controls +( 0, -0.3) and +(0,0) .. (tm)
  node[above, pos = 0] {$ $};
  \draw [->-] (bm) -- (tm) node[left, pos = 0.5] {$3$};
\end{scope}}} \right)  \bigoplus  q^{-1} \soergel \left( \NB{\tikz[font= \tiny, scale
      =0.6]{\begin{scope}

    \coordinate (bl) at (-0.5, 0);
  \coordinate (br) at ( 0.5, 0);
  \coordinate (bm) at (  0,.3);
  \coordinate (tl) at (-0.5,  1);
  \coordinate (tr) at ( 0.5,  1);
  \coordinate (tm) at (  0, 0.6);
  
      \coordinate (bbl) at (  -.5, -1);
        \coordinate (bbr) at (  .5, -1);

 \draw[] (bl) -- (bbl) node[pos = 0, left] {$  $} node[pos = 1,
  below] {$ 2$} coordinate[pos = .6]  coordinate[pos = 0.6] (ml) ;
  
   \draw[] (br) -- (bbr) node[pos = 0, below] {$ $} node[pos = 1,
  below] {$1$} coordinate[pos = 0.6] coordinate[pos = 0.4] (mr);
  
   \draw[<-] (mr) -- (ml) node[pos = 0, right] {$  $} node[pos = .5,
  above] {$ 1 $} coordinate[pos = 0.6];

  \draw[>-]  (bl) .. controls +( 0, 0.3) and +(0,0) .. (bm)
  node[below, pos = 0] {$ $};
  \draw[>-]  (br) .. controls +( 0, 0.3) and +(0,0) .. (bm)
  node[below, pos = 0] {$ $};
  \draw[<-]  (tl) .. controls +( 0, -0.3) and +(0,0) .. (tm)
  node[above, pos = 0] {$2$};
  \draw[<-]  (tr) .. controls +( 0, -0.3) and +(0,0) .. (tm)
  node[above, pos = 0] {$1$};
  \draw [->-] (bm) -- (tm) node[left, pos = 0.5] {$3$};
\end{scope}}} \right) $};
        \draw[-to] (A) -- (m) node[pos=0.5, above] {$\begin{pmatrix} \mapH \\-\mapH \end{pmatrix} $}; 
          \draw[-to] (m) -- (C) node[pos=0.5, above] {$\begin{pmatrix} \mapH & \mapH \end{pmatrix} $}; 
 \end{tikzpicture} .
\]
Ignoring the $H^\prime$-structure, there is a direct sum decomposition of the rightmost bimodule 
  \[
  \soergel \left( \NB{\tikz[font= \tiny, scale=0.6]{}} \right)
    \cong
  \soergel \left(  \NB{\tikz[font= \tiny, scale=0.6]{\begin{scope}
  \coordinate (bl) at (-0.5, -1);
  \coordinate (br) at ( 0.5, -1);
  \coordinate (bm) at (  0,-0.3);
  \coordinate (tl) at (-0.5,  1);
  \coordinate (tr) at ( 0.5,  1);
  \coordinate (tm) at (  0, 0.3);
  \draw[>-]  (bl) .. controls +( 0, 0.5) and +(0,0) .. (bm)
  node[below, pos = 0] {$2$};
  \draw[>-]  (br) .. controls +( 0, 0.5) and +(0,0) .. (bm)
  node[below, pos = 0] {$1$};
  \draw[<-]  (tl) .. controls +( 0, -0.5) and +(0,0) .. (tm)
  node[above, pos = 0] {$2$};
  \draw[<-]  (tr) .. controls +( 0, -0.5) and +(0,0) .. (tm)
  node[above, pos = 0] {$1$};
  \draw [->-] (bm) -- (tm) node[left, pos = 0.5] {$3$};
\end{scope}}} \right)
    \oplus 
q^{-2} \soergel \left(   \NB{\tikz[font= \tiny, scale=0.6]{\begin{scope}
  \coordinate (bl) at (-0.5, -1);
  \coordinate (br) at ( 0.5, -1);
  \coordinate (tl) at (-0.5,  1);
  \coordinate (tr) at ( 0.5,  1);
    \coordinate (ml) at (-0.5,  -.8);
        \coordinate (Ml) at (-0.5,  .8);
 \coordinate (mr) at (0.5,  -.6);
\coordinate (Mr) at (0.5,  .6);

   \draw[>->] (bl) -- (tl) node[pos = 0, below] {$2$} node[pos = 1,
  above] {$2$};

    \draw[>->] (br) -- (tr) node[pos = 0, below] {$1$} node[pos = 1, above] {$1$};

\end{scope}}} \right)
    .
  \]
While this is not an $H^\prime$-isomorphism, some of the maps needed for the isomorphism above do respect the $H^\prime$-structure.
In particular,
the map 
\[
\soergel \left( \NB{\tikz[font= \tiny, scale=0.6]{}} \right)
    \longrightarrow
\soergel \left( \NB{\tikz[font= \tiny, scale=0.6]{}} \right)
\]
where the generator gets sent to the generator is an $H^\prime$-module map.
Additionally, the unzip map $(\mapU \circ \mapH)$ 
\[
\soergel \left( \NB{\tikz[font= \tiny, scale=0.6]{}} \right)
    \longrightarrow
q^{-2} \soergel \left( \NB{\tikz[font= \tiny, scale=0.6]{\begin{scope}
  \coordinate (bl) at (-0.5, -1);
  \coordinate (br) at ( 0.5, -1);
  \coordinate (tl) at (-0.5,  1);
  \coordinate (tr) at ( 0.5,  1);
    \coordinate (ml) at (-0.5,  -.8);
        \coordinate (Ml) at (-0.5,  .8);
 \coordinate (mr) at (0.5,  -.6);
\coordinate (Mr) at (0.5,  .6);

   \draw[>->] (bl) -- (tl) node[pos = 0, below] {$2$} node[pos = 1,
  above] {$2$} coordinate[pos = 0.5] (ga) ;
   \filldraw[draw= green!50!black, fill = white] (ga) circle (1mm)
  node[left, green!50!black] {$1$};

    \draw[>->] (br) -- (tr) node[pos = 0, below] {$1$} node[pos = 1, above] {$1$};

\end{scope}}} \right)
\]
respects the $H^\prime$-structure.  Note that $\mapH$ acts on the portion of the diagram with external edges labeled by $1$ and $\mapU$ acts on the portion of the resulting diagram where the external edges are labeled by $2$.

Below we will repeatedly use Lemma \ref{lem-construction-of-triangle} to simplify the complex in the relative homotopy category. First,
we have a short exact sequence of complexes of $H^\prime$-equivariant bimodules
\begin{equation} \label{hopfses1}
\NB{\tikz[xscale = 3, yscale = 3]{
    \node (i1) at (-2.65, 1) {$0$};
    \node (i2) at (-2.65, 0) {$0$};
    \node (i3) at ( 2.5, 1) {$0$};
    \node (i4) at ( 2.5, 0) {$0$};
   \node (i5) at (-2.65, 2) {$0$};
        \node (i6) at (-1.75, 2) {$ q^{-2} \soergel \left( \NB{\tikz[font= \tiny,
  scale=0.6]{}} \right) $};
        \node (i7) at (0, 2) {$ q^{-2} \soergel \left( \NB{\tikz[font= \tiny,
  scale=0.6]{}} \right) \oplus q^{-2} \soergel \left( \NB{\tikz[font= \tiny,
  scale=0.6]{}} \right)  $};
    \node (i8) at (1.75, 2) {$q^{-2} \soergel \left( \NB{\tikz[font= \tiny,
  scale=0.6]{}} \right) $};
          \node (i9) at (2.5, 2) {$0$};
\node (A) at (-1.75,1) {$ \soergel \left( \NB{\tikz[font= \tiny,
  scale=0.6]{\begin{scope}
  \coordinate (o4) at (0,.6);
  \coordinate (o3) at (0,0.3);
  \coordinate (o2) at (0,-0.3);
  \coordinate (o1) at (0,-.6);
  \coordinate (bl) at (-.5,-1);
  \coordinate (br) at ( .5,-1);
  \coordinate (ul) at (-.5,1);
  \coordinate (ur) at (.5,1);
  \draw[->-] (o3) -- (o4) node [pos=.5, left] {$3$};
  \draw[->-] (o1) -- (o2) node [pos=0.5, left] {$3$};
  \draw[->-] (o2) .. controls +(0.3,0.3) and +(0.3, -0.3) .. (o3) node
 [pos= 0.5, right] {$2$};
   \draw[->-] (o2) .. controls +(-0.3,0.3) and +(-0.3, -0.3) .. (o3) node
 [pos= 0.5, left] {$1$};
 \draw[>-] (bl) .. controls +(0, 0.3) and +(-0,-0).. (o1) coordinate[pos=0.56] (c) node[pos=0, below] {$2$};
  \draw[>-] (br) .. controls +(0, 0.3) and +(0,-0).. (o1) node[pos=0, below] {$1$};
  \draw[<-] (ul) .. controls +(0, -0.3) and +(-0,-0).. (o4) coordinate[pos=0.56] (c) node[pos=0, above] {$2$};
  \draw[<-] (ur) .. controls +(0, -0.3) and +(0,-0).. (o4) node[pos=0, above] {$1$};
\end{scope}}} \right) $};
\node (B) at (-1.75,0) {$ \soergel \left( \NB{\tikz[font= \tiny, scale=0.6]{\begin{scope}
  \coordinate (bl) at (-0.5, -1);
  \coordinate (br) at ( 0.5, -1);
  \coordinate (bm) at (  0,-0.3);
  \coordinate (tl) at (-0.5,  1);
  \coordinate (tr) at ( 0.5,  1);
  \coordinate (tm) at (  0, 0.3);
  \draw[>-]  (bl) .. controls +( 0, 0.5) and +(0,0) .. (bm)
  node[below, pos = 0] {$2$};
  \draw[>-]  (br) .. controls +( 0, 0.5) and +(0,0) .. (bm)
  node[below, pos = 0] {$1$};
  \draw[<-]  (tl) .. controls +( 0, -0.5) and +(0,0) .. (tm)
  node[above, pos = 0] {$2$};
  \draw[<-]  (tr) .. controls +( 0, -0.5) and +(0,0) .. (tm)
  node[above, pos = 0] {$1$};
  \draw [->-] (bm) -- (tm) node[left, pos = 0.75] {$3$}    coordinate[pos = 0.25] (ga);
    \filldraw[draw= green!50!black, fill = white] (ga) circle (1mm)  node[right, green!50!black] {$1$};
\end{scope}}} \right)  \oplus  q^{2} \soergel \left( \NB{\tikz[font= \tiny,
  scale=0.6]{}} \right) $};
\draw[green!50!black, ->] ($(B)+(-0.1,-0.15)$) .. controls +(0.05,
-0.05) and +(-0.05, -0.05) .. ($(B)+(0.15,-0.15)$)
node[font =\tiny, below, pos= 0.5] {$-1$};
\node (C) at (0,1) {$ q^{-1}\soergel \left( \NB{\tikz[font= \tiny,
  scale=0.6]{}} \right) \oplus q^{-1} \soergel \left( \NB{\tikz[font= \tiny,
  scale=0.6]{}} \right)$};
\node (D) at (0,0) {$ \soergel \left( \NB{\tikz[font= \tiny,
  scale=0.6]{\begin{scope}
  \coordinate (bl) at (-0.5, -1);
  \coordinate (br) at ( 0.5, -1);
  \coordinate (bm) at (  0,-0.3);
  \coordinate (tl) at (-0.5,  1);
  \coordinate (tr) at ( 0.5,  1);
  \coordinate (tm) at (  0, 0.3);
  \draw[>-]  (bl) .. controls +( 0, 0.5) and +(0,0) .. (bm)
  node[below, pos = 0] {$2$};
  \draw[>-]  (br) .. controls +( 0, 0.5) and +(0,0) .. (bm)
  node[below, pos = 0] {$1$};
  \draw[<-]  (tl) .. controls +( 0, -0.5) and +(0,0) .. (tm)
  node[above, pos = 0] {$2$} coordinate[pos = 0.25] (ga) ;
    \filldraw[draw= green!50!black, fill = white] (ga) circle (1mm)
  node[left, green!50!black] {$1$};

  \draw[<-]  (tr) .. controls +( 0, -0.5) and +(0,0) .. (tm)
  node[above, pos = 0] {$1$};
  \draw [->-] (bm) -- (tm) node[left, pos = 0.5] {$3$};

\end{scope}}} \right) \oplus  \soergel \left( \NB{\tikz[font= \tiny,
  scale=0.6]{\begin{scope}
  \coordinate (bl) at (-0.5, -1);
  \coordinate (br) at ( 0.5, -1);
  \coordinate (bm) at (  0,-0.3);
  \coordinate (tl) at (-0.5,  1);
  \coordinate (tr) at ( 0.5,  1);
  \coordinate (tm) at (  0, 0.3);
  \draw[>-]  (bl) .. controls +( 0, 0.5) and +(0,0) .. (bm)
  node[below, pos = 0] {$2$} coordinate[pos = 0.25] (ga);
    \filldraw[draw= green!50!black, fill = white] (ga) circle (1mm)
  node[left, green!50!black] {$1$};
  \draw[>-]  (br) .. controls +( 0, 0.5) and +(0,0) .. (bm)
  node[below, pos = 0] {$1$};
  \draw[<-]  (tl) .. controls +( 0, -0.5) and +(0,0) .. (tm)
  node[above, pos = 0] {$2$}  ;

  \draw[<-]  (tr) .. controls +( 0, -0.5) and +(0,0) .. (tm)
  node[above, pos = 0] {$1$};
  \draw [->-] (bm) -- (tm) node[left, pos = 0.5] {$3$};

\end{scope}}} \right) $};
\node (E) at (1.75,1) {$q^{-2} \soergel \left( \NB{\tikz[font= \tiny,
  scale=0.6]{}} \right) $};
\node (F) at (1.75,0) {$ q^{-2} \soergel \left( \NB{\tikz[font= \tiny,
  scale=0.6]{}} \right) $};
\draw[->] (A) -- (B) node [left, pos= 0.5] {$C_3$};
\draw[->] (C) -- (D) node [left, pos= 0.5] {$C_1$};
\draw[->] (E) -- (F) node [right, pos= 0.5] {$\mapU \circ \mapH$};
\draw[->] (A) -- (C) node [above, pos = 0.5] {$\begin{pmatrix} \mapH \\ -\mapH \end{pmatrix}$};
\draw[->] (C) -- (E) node [above, pos = 0.5] {$\begin{pmatrix} \mapH & \mapH \end{pmatrix}$};
\draw[->] (B) -- (D) node [above, pos = 0.5] {$C_2$};
\draw[->] (D) -- (F) node [above, pos = 0.5] {$\begin{pmatrix} \mapH & \mapH \end{pmatrix}$};
\draw[->] (i1) -- (A);
\draw[->] (i2) -- (B);
\draw[->] (i1) -- (i2);
\draw[->] (E) -- (i3);
\draw[->] (F) -- (i4);
\draw[->] (i3) -- (i4);
\draw[->] (i5) -- (i6);
\draw[->] (i6) -- (i7) node [above, pos = .5] {$\begin{pmatrix} \Id \\ -\Id \end{pmatrix}$};
\draw[->] (i7) -- (i8) node [above, pos = .5] {$\begin{pmatrix} \Id & \Id \end{pmatrix}$};
\draw[->] (i8) -- (i9);
\draw[->] (i5) -- (i1);
\draw[->] (i6) -- (A) node [right, pos = .5] {$\mapB$};
\draw[->] (i7) -- (C) node [right, pos = 0.5] {$C_0$};
\draw[->] (i8) -- (E)node [right, pos = 0.5] {$ \mapH \circ \mapU \circ \mapA$};
\draw[->] (i9) -- (i3);

\node[draw, densely dotted ,fit=(i5) (i6) (i7) (i8) (i9) ,inner sep=1ex,rectangle] {};
\node[draw, densely dotted ,fit=(i1) (i3) (A) (C) (E) ,inner sep=1ex,rectangle] {};
\node[draw, densely dotted ,fit=(i2) (B) (D) (F) (i4) ,inner sep=1ex,rectangle] {};
}}
\end{equation}
where we define
\[C_0 = 
\begin{pmatrix}
\mapA \circ \mapB & 0 \\
0 & \mapA \circ \mapB
\end{pmatrix} , 
\quad \quad 
C_1 = 
\begin{pmatrix}
\mapU \circ \mapA & 0 \\
0 & \mapU \circ \mapA
\end{pmatrix} , \quad \quad
C_2=
\begin{pmatrix}
\Id & {\NB{\tikz[font= \tiny,
  scale=0.6]{\begin{scope}
  \coordinate (bl) at (-0.5, -1);
  \coordinate (br) at ( 0.5, -1);
  \coordinate (bm) at (  0,-0.3);
  \coordinate (tl) at (-0.5,  1);
  \coordinate (tr) at ( 0.5,  1);
  \coordinate (tm) at (  0, 0.3);
  \draw[>-]  (bl) .. controls +( 0, 0.5) and +(0,0) .. (bm)
  node[below, pos = 0] {$2$};
  \draw[>-]  (br) .. controls +( 0, 0.5) and +(0,0) .. (bm)
  node[below, pos = 0] {$1$};
  \draw[<-]  (tl) .. controls +( 0, -0.5) and +(0,0) .. (tm)
  node[above, pos = 0] {$2$};
  \draw[<-]  (tr) .. controls +( 0, -0.5) and +(0,0) .. (tm)
  node[above, pos = 0] {$1$} coordinate[pos = 0.25] (gb) ;
  \filldraw[draw= black!50!black, ] (gb) circle (1mm)
  node[left, black!50!black] {$e_1$};
  \draw[->-] (bm) -- (tm) node[left, pos = 0.5] {$3$};

\end{scope}}}} \\
-\Id & - {\NB{\tikz[font= \tiny,
  scale=0.6]{\begin{scope}
  \coordinate (bl) at (-0.5, -1);
  \coordinate (br) at ( 0.5, -1);
  \coordinate (bm) at (  0,-0.3);
  \coordinate (tl) at (-0.5,  1);
  \coordinate (tr) at ( 0.5,  1);
  \coordinate (tm) at (  0, 0.3);
  \draw[>-]  (bl) .. controls +( 0, 0.5) and +(0,0) .. (bm)
  node[below, pos = 0] {$2$};
  \draw[>-]  (br) .. controls +( 0, 0.5) and +(0,0) .. (bm)
  node[below, pos = 0] {$1$} coordinate[pos = 0.25] (gb) ;
  \filldraw[draw= black!50!black, ] (gb) circle (1mm)
  node[left, black!50!black] {$e_1$};
  \draw[<-]  (tl) .. controls +( 0, -0.5) and +(0,0) .. (tm)
  node[above, pos = 0] {$2$};

  \draw[<-]  (tr) .. controls +( 0, -0.5) and +(0,0) .. (tm)
  node[above, pos = 0] {$1$};
  \draw [->-] (bm) -- (tm) node[left, pos = 0.5] {$3$};

\end{scope}}}}
\end{pmatrix} ,
\]
and $C_3$ is defined by
\[
{\NB{\tikz[font= \tiny,
  scale=0.6]{\begin{scope}
  \coordinate (o4) at (0,.6);
  \coordinate (o3) at (0,0.3);
  \coordinate (o2) at (0,-0.3);
  \coordinate (o1) at (0,-.6);
  \coordinate (bl) at (-.5,-1);
  \coordinate (br) at ( .5,-1);
  \coordinate (ul) at (-.5,1);
  \coordinate (ur) at (.5,1);
  \draw[->-] (o3) -- (o4) node [pos=.5, left] {$3$};
  \draw[->-] (o1) -- (o2) node [pos=0.5, left] {$3$};
  \draw[->-] (o2) .. controls +(0.3,0.3) and +(0.3, -0.3) .. (o3) node
 [pos= 0.5, right] {$2$};
   \draw[->-] (o2) .. controls +(-0.3,0.3) and +(-0.3, -0.3) .. (o3) node
 [pos= 0.5, left] {$1$};
 \draw[>-] (bl) .. controls +(0, 0.3) and +(-0,-0).. (o1) coordinate[pos=0.56] (c) node[pos=0, below] {$2$};
  \draw[>-] (br) .. controls +(0, 0.3) and +(0,-0).. (o1) node[pos=0, below] {$1$};
  \draw[<-] (ul) .. controls +(0, -0.3) and +(-0,-0).. (o4) coordinate[pos=0.56] (c) node[pos=0, above] {$2$};
  \draw[<-] (ur) .. controls +(0, -0.3) and +(0,-0).. (o4) node[pos=0, above] {$1$};
\end{scope}}}}
 \mapsto
\begin{pmatrix}
  0 \\
  0
\end{pmatrix} ,
\quad \quad 
{\NB{\tikz[font= \tiny,
  scale=0.6]{\begin{scope}
  \coordinate (o4) at (0,.6);
  \coordinate (o3) at (0,0.3);
  \coordinate (o2) at (0,-0.3);
  \coordinate (o1) at (0,-.6);
  \coordinate (bl) at (-.5,-1);
  \coordinate (br) at ( .5,-1);
  \coordinate (ul) at (-.5,1);
  \coordinate (ur) at (.5,1);
  \draw[->-] (o3) -- (o4) node [pos=.5, left] {$3$};
  \draw[->-] (o1) -- (o2) node [pos=0.5, left] {$3$};
  \draw[->-] (o2) .. controls +(0.3,0.3) and +(0.3, -0.3) .. (o3) node
 [pos= 0.5, right] {$e_1$}  coordinate[pos= 0.5] (dot2);
 \draw[->-] (o2) .. controls +(-0.3,0.3) and +(-0.3, -0.3) .. (o3)
 coordinate[pos= 0.5] (dot);
 \draw[>-] (bl) .. controls +(0, 0.3) and +(-0,-0).. (o1) coordinate[pos=0.56] (c) node[pos=0, below] {$2$};
  \draw[>-] (br) .. controls +(0, 0.3) and +(0,-0).. (o1) node[pos=0, below] {$1$};
  \draw[<-] (ul) .. controls +(0, -0.3) and +(-0,-0).. (o4) coordinate[pos=0.56] (c) node[pos=0, above] {$2$};
  \draw[<-] (ur) .. controls +(0, -0.3) and +(0,-0).. (o4) node[pos=0, above] {$1$};
 \filldraw[draw= black!50!black, ] (dot2) circle (1mm)
  node[left, black!50!black] {$ $};
\end{scope}}}}
 \mapsto
\begin{pmatrix}
 {\NB{\tikz[font= \tiny,
  scale=0.6]{}}} \\
0
\end{pmatrix} ,
\quad \quad 
{\NB{\tikz[font= \tiny,
  scale=0.6]{\begin{scope}
  \coordinate (o4) at (0,.6);
  \coordinate (o3) at (0,0.3);
  \coordinate (o2) at (0,-0.3);
  \coordinate (o1) at (0,-.6);
  \coordinate (bl) at (-.5,-1);
  \coordinate (br) at ( .5,-1);
  \coordinate (ul) at (-.5,1);
  \coordinate (ur) at (.5,1);
  \draw[->-] (o3) -- (o4) node [pos=.5, left] {$3$};
  \draw[->-] (o1) -- (o2) node [pos=0.5, left] {$3$};
  \draw[->-] (o2) .. controls +(0.3,0.3) and +(0.3, -0.3) .. (o3) node
 [pos= 0.5, right] {$e_2$}  coordinate[pos= 0.5] (dot2);
 \draw[->-] (o2) .. controls +(-0.3,0.3) and +(-0.3, -0.3) .. (o3)
 coordinate[pos= 0.5] (dot);
 \draw[>-] (bl) .. controls +(0, 0.3) and +(-0,-0).. (o1) coordinate[pos=0.56] (c) node[pos=0, below] {$2$};
  \draw[>-] (br) .. controls +(0, 0.3) and +(0,-0).. (o1) node[pos=0, below] {$1$};
  \draw[<-] (ul) .. controls +(0, -0.3) and +(-0,-0).. (o4) coordinate[pos=0.56] (c) node[pos=0, above] {$2$};
  \draw[<-] (ur) .. controls +(0, -0.3) and +(0,-0).. (o4) node[pos=0, above] {$1$};
 \filldraw[draw= black!50!black, ] (dot2) circle (1mm)
  node[left, black!50!black] {$ $};
\end{scope}}}}
 \mapsto
\begin{pmatrix}
0 \\
 {\NB{\tikz[font= \tiny,
  scale=0.6]{}}}
\end{pmatrix} .
\]
The bimodules in the first two rows of \eqref{hopfses1} are equipped with the natural $H^\prime$-structure.
In the bottom row of \eqref{hopfses1}, the bimodules in the middle and right are
equipped with the $H^\prime$-structures as indicated by green dots in the diagrams.  
The $H^\prime$-structure on the bimodule in the leftmost position is a
bit more complicated. The green dots and the green arrow are meant to
encode that the $H^\prime$-action is twisted by the following matrix:
\[
\begin{pmatrix}
 {\NB{\tikz[font= \tiny,
  scale=0.6]{\begin{scope}
  \coordinate (bl) at (-0.5, -1);
  \coordinate (br) at ( 0.5, -1);
  \coordinate (bm) at (  0,-0.3);
  \coordinate (tl) at (-0.5,  1);
  \coordinate (tr) at ( 0.5,  1);
  \coordinate (tm) at (  0, 0.3);
    \coordinate (gm) at (  0, 0);
  \draw[>-]  (bl) .. controls +( 0, 0.5) and +(0,0) .. (bm);
  \draw[>-]  (br) .. controls +( 0, 0.5) and +(0,0) .. (bm)
  node[below, pos = 0] {$1$};
  \draw[<-]  (tl) .. controls +( 0, -0.5) and +(0,0) .. (tm)
  node[above, pos = 0] {$2$}  ;
  \draw[<-]  (tr) .. controls +( 0, -0.5) and +(0,0) .. (tm)
  node[above, pos = 0] {$1$};
  \draw [->-] (bm) -- (tm) node[left, pos = 0.5] {$3$} coordinate[pos = 0.25] (ga);
    \fill[black] (gm) circle (1mm)  node[right] {$e_1$};

\end{scope}}}} 
  & 0
\\
- {\NB{\tikz[font= \tiny,
  scale=0.6]{}}} & {\NB{\tikz[font= \tiny,
  scale=0.6]{}}} 
\end{pmatrix} .
\]
The top row of \eqref{hopfses1} is contractible.  Thus in the relative homotopy category, 
$\rickardp{2}{1} \circ \rickardp{1}{2}$ is isomorphic to the bottom row of \eqref{hopfses1} with the $H^\prime$-structure described above.

Next, we have a short exact sequence of complexes of $H^\prime$-equivariant bimodules
\begin{equation} \label{hopfses3}
\NB{\tikz[xscale = 3.2, yscale = 3]{
    \node (i1) at (-2.25, 1) {$0$};
    \node (i2) at (-2.25, 0) {$0$};
    \node (i3) at ( 2.6, 1) {$0$};
    \node (i4) at ( 2.6, 0) {$0$.};
    \node (i5) at (-2.25, 2) {$0$};
    \node (i9) at (2.6, 2) {$0$};
        \node (i6) at (-1.25, 2) {$ q^2 \soergel \left( \NB{\tikz[font= \tiny,
  scale=0.6]{\begin{scope}
  \coordinate (bl) at (-0.5, -1);
  \coordinate (br) at ( 0.5, -1);
  \coordinate (bm) at (  0,-0.3);
  \coordinate (tl) at (-0.5,  1);
  \coordinate (tr) at ( 0.5,  1);
  \coordinate (tm) at (  0, 0.3);
  \draw[>-]  (bl) .. controls +( 0, 0.5) and +(0,0) .. (bm)
  node[below, pos = 0] {$2$};
  \draw[>-]  (br) .. controls +( 0, 0.5) and +(0,0) .. (bm)
  node[below, pos = 0] {$1$};
  \draw[<-]  (tl) .. controls +( 0, -0.5) and +(0,0) .. (tm)
  node[above, pos = 0] {$2$} coordinate[pos = 0.25] (ga) ;
    \filldraw[draw= green!50!black, fill = white] (ga) circle (1mm)
  node[left, green!50!black] {$1$};

  \draw[<-]  (tr) .. controls +( 0, -0.5) and +(0,0) .. (tm)
  node[above, pos = 0] {$1$} coordinate[pos = 0.25] (gb) ;
    \filldraw[draw= green!50!black, fill = white] (gb) circle (1mm)
  node[left, green!50!black] {$2$};
  \draw [->-] (bm) -- (tm) node[left, pos = 0.5] {$3$};
 
\end{scope}}} \right)$ };
        \node (i7) at (.5, 2) {$ \soergel \left( \NB{\tikz[font= \tiny,
  scale=0.6]{}} \right) $};
    \node (i8) at (2, 2) {$q^{-2} \soergel \left( \NB{\tikz[font= \tiny,
  scale=0.6]{}} \right) $};
\node (A) at (-1.25,1) {$ \soergel \left( \NB{\tikz[font= \tiny, scale=0.6]{}} \right)  \oplus  q^{2} \soergel \left( \NB{\tikz[font= \tiny,
  scale=0.6]{}} \right) $ };
\node (B) at (-1.25,0) {$ \soergel \left( \NB{\tikz[font= \tiny,
  scale=0.6]{}} \right) $};
\node (D) at (.5,0) {$ \soergel \left( \NB{\tikz[font= \tiny,
  scale=0.6]{}} \right) $} ;
\node (C) at (.5,1) {$ \soergel \left( \NB{\tikz[font= \tiny,
  scale=0.6]{}} \right) \oplus \soergel \left( \NB{\tikz[font= \tiny,
  scale=0.6]{}} \right) $};
\draw[green!50!black, ->] ($(A)+(-0.1,-0.15)$) .. controls +(0.05,
-0.05) and +(-0.05, -0.05) .. ($(A)+(0.15,-0.15)$)
node[font =\tiny, below, pos= 0.5] {$-1$};
\node (E) at (2,1) {$q^{-2} \soergel \left( \NB{\tikz[font= \tiny,
  scale=0.6]{}} \right) $};
\node (F) at (2,0) {$0$};
\draw[->] (A) -- (B) node [left, pos= 0.5] {$C_5$};
\draw[->] (C) -- (D) node [left, pos= 0.5] {$\begin{pmatrix} \Id & 0 \end{pmatrix}$};
\draw[->] (E) -- (F) node [right, pos= 0.5] {$ $};
\draw[->] (A) -- (C) node [above, pos = 0.5] {$C_2$};
\draw[->] (C) -- (E) node [above, pos = 0.5] {$\begin{pmatrix} \mapH & \mapH \end{pmatrix}$};
\draw[->] (B) -- (D) node [above, pos = 0.5] {$ \Id $};
\draw[->] (D) -- (F) node [above, pos = 0.5] {$ $};
\draw[->] (i1) -- (A);
\draw[->] (i2) -- (B);
\draw[->] (i1) -- (i2);
\draw[->] (E) -- (i3);
\draw[->] (F) -- (i4);
\draw[->] (i3) -- (i4);
\draw[->] (i5) -- (i6);
\draw[->] (i6) -- (i7) coordinate[pos=.5] (iii);
\node[above] at (iii) {${\NB{\tikz[font= \tiny,
  scale=0.5]{}}}
  -
  {\NB{\tikz[font= \tiny,
  scale=0.5]{}}}$};
\draw[->] (i7) -- (i8) node [above, pos = .5] {$ \mapH $};
\draw[->] (i8) -- (i9) node [above, pos = .5] {$ $};
\draw[->] (i5) -- (i1);
\draw[->] (i6) -- (A) node [left, pos = .5] {$ C_4$};
\draw[->] (i7) -- (C) node [right, pos = 0.5] {$\begin{pmatrix} 0 \\ \Id  \end{pmatrix}$};
\draw[->] (i8) -- (E) node [right, pos = 0.5] {$\Id $};
\draw[->] (i9) -- (i3);
\node[draw, densely dotted ,fit=(i5) (i6) (i7) (i8) (i9) ,inner sep=1ex,rectangle] {};
\node[draw, densely dotted ,fit=(i1) (i3) (A) (C) (E) ,inner sep=1ex,rectangle] {};
\node[draw, densely dotted ,fit=(i2) (B) (D) (F) (i4) ,inner sep=1ex,rectangle] {};
}}
\end{equation}
The maps $C_4$ and $C_5$ are defined by:
\[
C_4=
\begin{pmatrix}
- {\NB{\tikz[font= \tiny,
  scale=0.6]{}}} \\
   {\NB{\tikz[font= \tiny,
  scale=0.6]{}}} \\
\end{pmatrix},
\quad \quad
C_5 =
\begin{pmatrix}
\Id & {\NB{\tikz[font= \tiny,
  scale=0.6]{}}}
\end{pmatrix} 
\]
Since the bottom row of \eqref{hopfses3} is contractible, we obtain an isomorphism in the relative homotopy category (incorporating back in overall shifts), 
\begin{equation} \label{hopfreduced}
 \NB{ \tikz[xscale = 3.5, yscale = 3]{
 \node (i5) at (-2.45, 2) {$\rickardp{2}{1} \circ \rickardp{1}{2} \cong$};
  \node (i6) at (-1.5, 2) {$  q^{-6}\left( q^2 \soergel \left( \NB{\tikz[font= \tiny,
  scale=0.6]{}} \right)\right.$};
     \node (i7) at (0, 2) {$ \soergel \left( \NB{\tikz[font= \tiny,
  scale=0.6]{}} \right) $};
    \node (i8) at (1, 2) {$ \left.q^{-2}\soergel\left(\NB{\tikz[font= \tiny,
  scale=0.6]{}} \right)\right)$}  ;
\draw[->] (i6) -- (i7) coordinate[pos= .5] (iii);
\node[above] at (iii) {${\NB{\tikz[font= \tiny,
  scale=0.5]{}}}
  -
  {\NB{\tikz[font= \tiny,
  scale=0.5]{}}}$};
\draw[->] (i7) -- (i8) node [above, pos = .5] {$ \mapH $};
    }}    
    \ ,
\end{equation}
where the leftmost term sits in cohomological degree $0$.

\subsection{Hochschild homology of the complex}

Now that we have simplified the complex associated to the braid used to calculate the homology of the Hopf link, we begin with the calculation of the Hochschild homologies of various Soergel bimodules appearing in the simplified complex.  

The Soergel bimodules appearing are all bimodules over
$A=\Z[e_1,e_2,x_3]$ where $e_1$ and $e_2$ are elementary symmetric functions in variables $x_1$ and $x_2$.
We also let $B$ denote $(A,A)$-bimodule associated to the MOY graph:
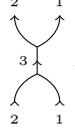
\[
\NB{\tikz[font= \tiny,
  scale=0.6]{}}
  .
\]
In order to calculate the Hochschild homologies of $A$ and $B$, we first recall the Koszul resolution of $A$ as an $(A,A)$-bimodule.  

\begin{equation} \label{koszul21a}
\bigotimes_{i=1}^2 
\left(
 \Z[e_i] \otimes \Z[e_i]  \xrightarrow{\, e_i\otimes 1-1\otimes e_i\, }%
\Z^{}[e_i] \otimes \Z^{}[e_i] 
\right) \otimes
\left(
 \Z[x_3] \otimes \Z[x_3]  \xrightarrow{\,x_3\otimes 1-1\otimes x_3 \,}%
\Z^{}[x_3] \otimes \Z^{}[x_3] 
\right)
\ .
\end{equation}
By Section \ref{sec:koszul-resolution-p}, this resolution could be made to be $H^\prime$-equivariant.  One may have to twist terms in higher Hochschild degrees, but the explicit form of the twists turn out to be irrelevant in this calculation so we ignore them below. 

Then to compute $\mHH_{\bullet}^{\dif}(B)$, we tensor \eqref{koszul21a} with $B$ and obtain
\begin{equation} \label{koszul21b}
\bigotimes_{i=1}^2 
\left(
 \Z[e_i] \otimes \Z[e_i]   \xrightarrow{\, e_i\otimes 1-1\otimes e_i \,}%
\Z^{}[e_i] \otimes \Z^{}[e_i] 
\right) \otimes
\left(
 \Z[x_3] \otimes \Z[x_3]   \xrightarrow{\,x_3\otimes 1-1\otimes x_3 \,}%
\Z^{}[x_3] \otimes \Z^{}[x_3] 
\right) \otimes_{(A,A)} B
\ .
\end{equation}
Using relations in $B$, this is isomorphic to 
\begin{equation} \label{koszul21c}
\bigotimes_{i=1}^2 
\left(
 \Z[e_i] \otimes \Z[e_i]   \xrightarrow{\quad 0 \quad}%
\Z^{}[e_i] \otimes \Z^{}[e_i] 
\right) \otimes
\left(
 \Z[x_3] \otimes \Z[x_3]   \xrightarrow{\, x_3\otimes 1-1\otimes x_3 \,}%
\Z^{}[x_3] \otimes \Z^{}[x_3] 
\right) \otimes_{(A,A)} B
\ .
\end{equation}
Using \cite[Theorem 2]{StosicHH}, this is isomorphic to 
\begin{equation}
\mHH_{\bullet}^{\dif}(B)
\cong \Lambda(de_1, de_2, f dx_3) \otimes A
\ ,
\end{equation}
where the $H^\prime$-structure in Hochschild degree zero is inherited from $A$, and
\begin{equation}
f=x_3^2-e_1 x_3 + e_2.    
\end{equation}
It is straightforward to calculate the Cautis differential on this homology
\begin{equation}
d_C(de_1)=e_1^2-2e_2,
\quad \quad
d_C(de_2)=e_1 e_2,
\quad \quad
d_C(dx_3)=x_3^2 .
\end{equation}
The homology of $\mHH^\dif_{\bullet}(B)$ with respect to $d_C$ is concentrated in Hochschild degree $0$ and we obtain
\begin{equation}
\mH(\mHH^\dif_\bullet(B),d_C) \cong
A / I , \quad  \quad \quad I=\langle e_1^2-2e_2, e_1 e_2, f x_3^2 \rangle .
\end{equation}
As a $\Z$-module, there is an isomorphism in Hochschild degree zero
\begin{equation}
\mH^{}(\mHH^\dif_\bullet (B),d_C)^{\mathrm{free}} \cong \Z \langle e_i x_3^j | 0 \leq i \leq 2, 0 \leq j \leq 3  \rangle .
\end{equation}

In order to compute the Hochschild homology of $A$, tensor the Koszul resolution of $A$ as a bimodule with $A$.  As a result of the tensoring, all the differentials in the Koszul resolution vanish and we obtain
\begin{equation} \label{HH21R}
\mHH_{\bullet}^{\dif}(A) \cong \Lambda(de_1, de_2, dx_3) \otimes A \ ,
\end{equation}
where the $H^\prime$-structure in Hochschild degree zero is inherited from $A$.
The homology of $\mHH^\dif_{\bullet}(A)$ with respect to $d_C$ is free
and concentrated in Hochschild degree $0$ and we obtain
\begin{equation}
\mH(\mHH_\bullet^\dif(A),d_C) \cong
A / I , \quad  \quad \quad I=\langle e_1^2-2e_2, e_1 e_2,  x_3^2 \rangle .
\end{equation}
As a $\Z$-module, there is an isomorphism in Hochschild degree zero
\begin{equation}
\mH(\mHH_\bullet^\dif(A),d_C)^{\mathrm{free}}= \mH(\mHH_\bullet^\dif(A),d_C) \cong \Z \langle e_i x_3^j | 0 \leq i \leq 2, 0 \leq j \leq 1  \rangle .
\end{equation}
Now that we have computed the homology of the various pieces of \eqref{hopfreduced}, we now consider the induced topological differentials in \eqref{hopfreduced} after taking Hochschild homology.  The leftmost differential vanishes.  
The rightmost differential becomes a projection map.  In particular, we obtain
\begin{equation}
  \ker \mH(\mHH^\dif(\mapH)) = \Z \langle x_3^2, e_1 x_3^2, e_2 x_3^2, x_3^3, e_1 x_3^3, e_2 x_3^3 \rangle \ .  \end{equation}

We actually need to compute $\mHT$ which is the homology with respect to the total differential $d_T$.  Here,
$d_T=d_C+\mHH^\dif(\mapH)$.
Since $\mHH^\dif(\mapH)$ is surjective here, the spectral sequence for the double complex collapses on the second page and we get that the homology with respect to the total differential $d_T=d_C+\mHH^\dif(\mapH)$, as a $\Z$-module is:
\[
\mHT \left( \mHH_\bullet^\dif(B) \rightarrow \mHH_\bullet^\dif(A) , d_T\right)
\cong
\Z \langle x_3^2, e_1 x_3^2, e_2 x_3^2, x_3^3, e_1 x_3^3, e_2 x_3^3 \rangle .
\]

Next we base change from $\Z$ to $\F_p$.
We also put in an overall shift of $q^{-6}$ coming from tensoring the Rickard complexes along with an overall shift of $q^{-3} $ coming from taking Hochschild homology. We then get that the homology of the Hopf link is isomorphic to
\begin{equation}
q^{-9} \left(q^2  \F_p \langle e_i x_3^j | 0 \leq i \leq 2, 0 \leq j \leq 3  \rangle
\oplus
t^{} \F_p \langle x_3^2, e_1 x_3^2, e_2 x_3^2, x_3^3, e_1 x_3^3, e_2 x_3^3 \rangle \right)
\end{equation}
where 
\begin{equation}
\partial(1)=e_1+2x_3 \ , \quad \quad \partial(x_3^2)= e_1 x_3^2 +2x_3^3 .
\end{equation}
In summary, the $p$-complexes are: %
  \[
    \NB{\tikz[]{\begin{scope}[xscale = 2, yscale=1.5]
  \node (e0x0)    at (0,2) {$1$};
  \node (e0x1)    at (1,2) {$x_3$};
  \node (e0x2)    at (2,2) {$x_3^2$};
  \node (e0x3)    at (3,2) {$x_3^3$};
  \node (e1x0)    at (0,1) {$e_1$};
  \node (e1x1)    at (1,1) {$e_1x_3$};
  \node (e1x2)    at (2,1) {$e_1x_3^2$};
  \node (e1x3)    at (3,1) {$e_1x_3^3$};
  \node (e2x0)    at (0,0) {$e_2$};
  \node (e2x1)    at (1,0) {$e_2 x_3$};
  \node (e2x2)    at (2,0) {$e_2 x_3^2$};
  \node (e2x3)    at (3,0) {$e_2x_3^3$};
  \begin{scope}[font=\tiny]
    \draw[->] (e0x0) -- (e0x1) node [pos=0.5, above] {$2$};
    \draw[->] (e0x1) -- (e0x2) node [pos=0.5, above] {$3$};
    \draw[->] (e0x2) -- (e0x3) node [pos=0.5, above] {$4$};
    \draw[->] (e1x0) -- (e1x1) node [pos=0.5, above] {$2$};
    \draw[->] (e1x1) -- (e1x2) node [pos=0.5, above] {$3$};
    \draw[->] (e1x2) -- (e1x3) node [pos=0.3, above] {$4$}
    coordinate[pos=0.5](xing);
    \fill[white] (xing) circle (0.5mm);
    \draw[->] (e2x0) -- (e2x1) node [pos=0.5, below] {$2$};
    \draw[->] (e2x1) -- (e2x2) node [pos=0.5, below] {$3$};
    \draw[->] (e2x2) -- (e2x3) node [pos=0.5, below] {$4$};
    \draw[->] (e0x0) -- (e1x0) node [pos=0.5, left] {$1$};
    \draw[->] (e1x0) -- (e2x0) node [pos=0.5, left] {$2$};
    \draw[->] (e0x1) -- (e1x1) node [pos=0.5, left] {$1$};
    \draw[->] (e1x1) -- (e2x1) node [pos=0.5, left] {$2$};
    \draw[->] (e0x2) -- (e1x2) node [pos=0.5, right] {$1$};
    \draw[->] (e1x2) -- (e2x2) node [pos=0.5, right] {$2$};
    \draw[->] (e0x3) -- (e1x3) node [pos=0.5, right] {$6$};
    \draw[->] (e1x3) -- (e2x3) node [pos=0.5, right] {$12$};
    \draw[->] (e0x3) -- (e2x2) node [pos=0.7, right] {$-5$};    
\end{scope}
\end{scope}

}} \qquad \text{and} \qquad
    \NB{\tikz[]{\begin{scope}[xscale = 2, yscale=1.5]
  \node (e0x2)   at (0,2) {$x_3^2$};
  \node (e0x3)   at (1,2) {$x_3^3$};
  \node (e1x2) at (0,1) {$e_1x_3^2$};
  \node (e1x3) at (1,1) {$e_1x_3^3$};
  \node (e2x2) at (0,0) {$e_2x_3^2$};
  \node (e2x3) at (1,0) {$e_2x_3^3$};
  \begin{scope}[font=\tiny]
  \draw[->] (e0x2) -- (e0x3) node [pos=0.5, above] {$2$}; %
  \draw[->] (e1x2) -- (e1x3) node [pos=0.3, above] {$2$}    coordinate[pos=0.5](xing);
    \fill[white] (xing) circle (0.5mm);
  \draw[->] (e2x2) -- (e2x3) node [pos=0.5, below] {$2$};
  \draw[->] (e0x2) -- (e1x2) node [pos=0.5, left] {$1$}; %
  \draw[->] (e1x2) -- (e2x2) node [pos=0.5, left] {$2$};
  \draw[->] (e0x3) -- (e1x3) node [pos=0.5, right] {$4$};
  \draw[->] (e1x3) -- (e2x3) node [pos=0.5, right] {$8$};
  \draw[->] (e0x3) -- (e2x2) node [pos=0.7, right] {$-3$};
\end{scope}
\end{scope}

}} .
    \]
By carefully choosing bases for these vector spaces, one could compute that in the stable category, the homology of the Hopf link depends upon $p$ as follows.\\
If $p=2$:
\begin{align*}
\mH^{\dif}_/ &\cong  q^{-9} \left( q^2((q^4+q^6+q^8+q^{10})V_0 \oplus (1+q^2+q^4+q^6)V_1)
\oplus
t((q^4+q^6+q^8+q^{10})V_0 \oplus q^8 V_1) 
\right)  \\
&\cong q^{-9} \left( q^2(q^4+q^6+q^8+q^{10})V_0 
\oplus
t(q^4+q^6+q^8+q^{10})V_0  
\right)
,
\end{align*}
if $p=3$:
\begin{align*}
\mH^{\dif}_/ &\cong q^{-9} \left( q^2 (V_2 \oplus q^2 V_2 \oplus q^4 V_2 \oplus q^6 V_2) 
\oplus
t^{}(q^4 V_2 \oplus q^6 V_2) \right) \\ & \cong 0 ,
\end{align*}
if $p=5$:
\begin{align*}
\mH^{\dif}_/ &\cong q^{-9} \left(
q^2 (V_4 \oplus q^2 V_4 \oplus q^4 V_1)
\oplus
t^{}(q^4 V_3 \oplus q^6 V_1) \right) \\
&\cong q^{-9} \left(
q^2 ( q^4 V_1)
\oplus
t^{}(q^4 V_3 \oplus q^6 V_1) \right)
,
\end{align*}
if $p \neq 3, 5$:
\[\mH^{\dif}_/ \cong  q^{-9} \left( q^2 (V_5 \oplus q^2 V_3 \oplus q^4 V_1) \oplus
t^{}(q^4 V_3 \oplus q^6 V_1) \right).
\]

\addcontentsline{toc}{section}{References}

\bibliographystyle{alphaurl}
\bibliography{extracted}

\noindent Y.~Q.: { \sl \small Department of Mathematics, University of Virginia, Charlottesville, VA 22904, USA} \newline \noindent {\tt \small email: yq2dw@virginia.edu}

\vspace{0.1in}

\noindent L.-H.~R.:  {\sl \small RMATH, Maison du Nombre, 6 avenue de
la Fonte, L-4365 Esch-sur-Alzette, Luxembourg }
\newline \noindent {\tt \small email: louis-hadrien.robert@uni.lu  }

\vspace{0.1in}

\noindent J.~S.:  {\sl \small Department of Mathematics, CUNY Medgar Evers, Brooklyn, NY, 11225, USA}\newline \noindent {\tt \small email: jsussan@mec.cuny.edu \newline 
\sl \small Mathematics Program, The Graduate Center, CUNY, New York, NY, 10016, USA}\newline \noindent {\tt \small email: jsussan@gc.cuny.edu}

\vspace{0.1in}

\noindent E.~W.: { \sl \small Univ Paris Diderot, IMJ-PRG, UMR 7586 CNRS, F-75013, Paris, France} \newline \noindent {\tt \small email: wagner@imj-prg.fr}

\end{document}